\documentclass[11pt]{smfart}
\usepackage{xups19,xups19-01}
\let\subsection\Subsection

\makeatletter
\def\Ddots{\mathinner{\mkern1mu\raise\p@
\vbox{\kern7\p@\hbox{.}}\mkern2mu
\raise4\p@\hbox{.}\mkern2mu\raise7\p@\hbox{.}\mkern1mu}}
\makeatother

\begin{document}
\frontmatter

\title{Une introduction aux périodes}
\author{Javier Fres\'an}
\address{CMLS, \'Ecole polytechnique\\ F-91128 Palaiseau cedex, France}
\email{javier.fresan@polytechnique.edu}
\urladdr{http://javier.fresan.perso.math.cnrs.fr/}

\thanks{L'auteur a bénéficié du soutien du projet ANR-18-CE40-0017 \og Périodes en Géométrie Arithmétique et Motivique\fg de l'Agence Nationale de la Recherche.}

\begin{abstract}
Les périodes sont des nombres complexes dont la partie réelle et la partie imaginaire s'écrivent comme des intégrales d'une fonction rationnelle sur un domaine défini par des inégalités polynomiales, le tout à coefficients rationnels. Selon une conjecture de Kontsevich et Zagier, n'importe quelle relation algébrique entre ces nombres devrait pouvoir se déduire des règles évidentes du calcul intégral : l'additivité, le changement de variables et la formule de Stokes. Dans un premier temps, j'explique la définition des périodes et quelques propriétés élémentaires qui s'ensuivent, en les illustrant par maints exemples. Ensuite, je me dirige doucement vers l’interprétation de ces nombres comme les coefficients de l'accouplement d'intégration entre la cohomologie de de~Rham algébrique et l'homo\-logie singulière des variétés algébriques définies sur $\QQ$, point de vue qui est à l'origine de toutes les percées récentes dans leur étude.
\end{abstract}

\maketitle
\vspace*{-1.5\baselineskip}
\enlargethispage{2\baselineskip}
{\let\\\relax\tableofcontents}

\mainmatter

\section{Un peu d'histoire}

Il sera question dans ces notes d'une classe de nombres, les \textit{périodes}, dont l'origine remonte aux premières heures du calcul intégral, car avant de devenir objets du désir des arithméticiens ce furent surtout les périodes de révolution des planètes ou le temps que met un pendule à revenir à sa position initiale...

\subsection{Les ovales de Newton} On prendra pour point de départ le lemme XXVIII du premier livre des \textit{Principia} de Newton, qui s'énonce ainsi dans la traduction française de la marquise du Châtelet~\cite{chatelet}:
\begin{quote}
Les parties quelconques de toute figure ovale, déterminées par les coordonnées ou par d'autres droites tirées à volonté, ne peuvent jamais être trouvées par aucune équation d'un nombre fini de termes et de dimensions.
\end{quote}
Autrement dit: si l'on se donne un ovale dans le plan, l'aire~$S(a, b, c)$ du segment délimité par une droite d'équation~\hbox{$ax+by=c$} n'est pas une fonction algébrique des paramètres, en ce sens qu'il n'existe pas de polynôme non nul à coefficients complexes en quatre variables~$P$ satisfaisant à l'égalité $P(a, b, c, S(a, b, c))=0$ pour toutes valeurs~$a, b, c$.
\begin{figure}[ht]
\begin{center}
\includegraphics[width=0.65\textwidth]{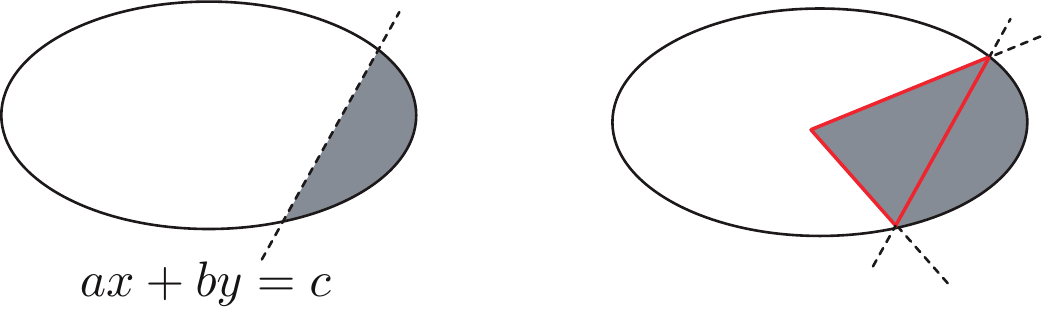}
\end{center}
\caption{Aires d'un segment et d'un secteur d'ovale}
\label{fig:1}
\end{figure}

Il revient au même de dire que l'aire du secteur coupé par deux demi\nobreakdash-droites quelconques émanant d'un point à l'intérieur de l'ovale n'est pas une fonction algébrique des données: comme le montre la figure \ref{fig:1}, on peut passer de l'une à l'autre
en retranchant un triangle dont l'aire s'exprime algébriquement en fonction des paramètres.

\begin{figure}[ht]
\includegraphics[width=0.48\textwidth]{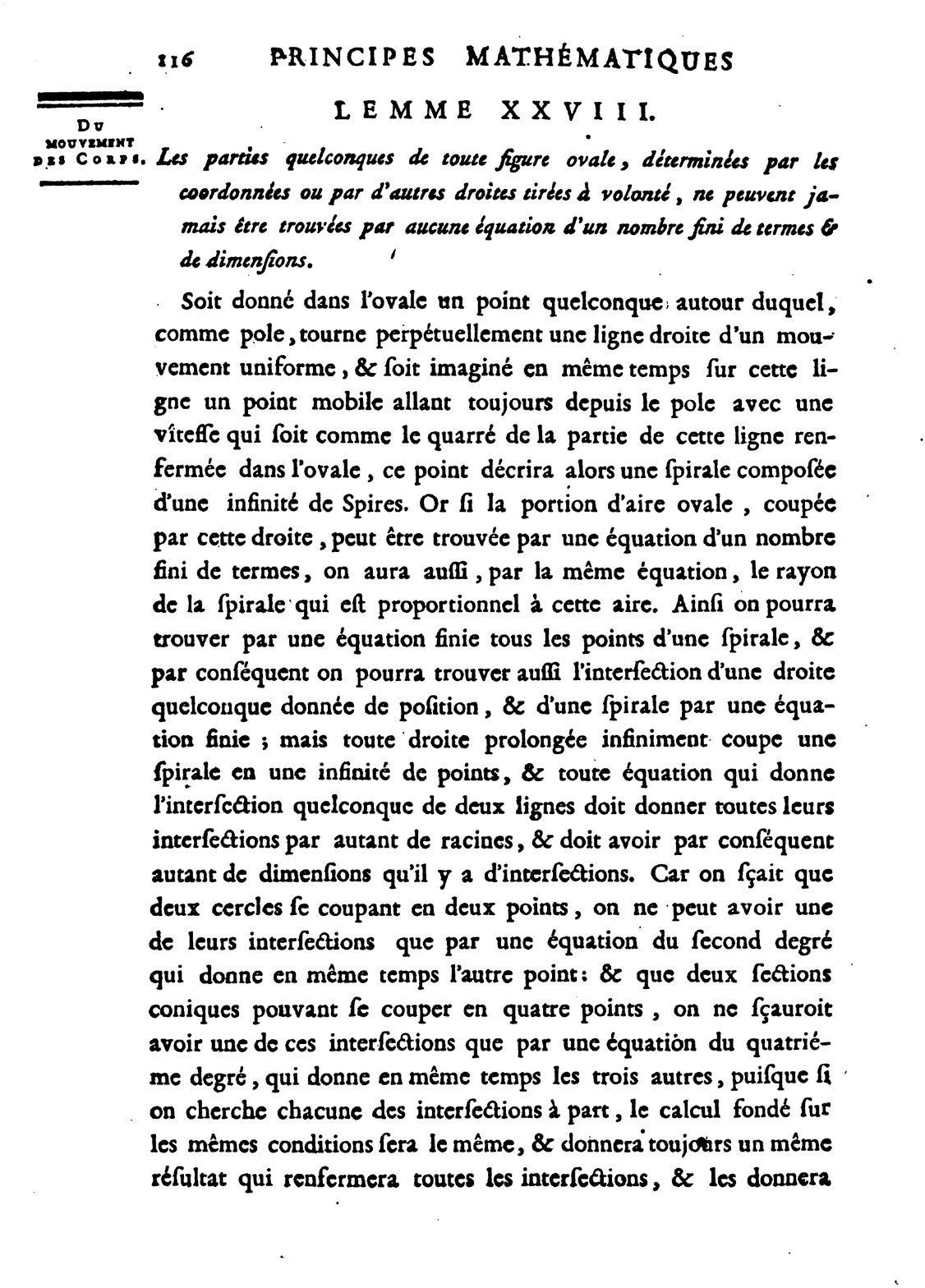}
\includegraphics[width=0.48\textwidth]{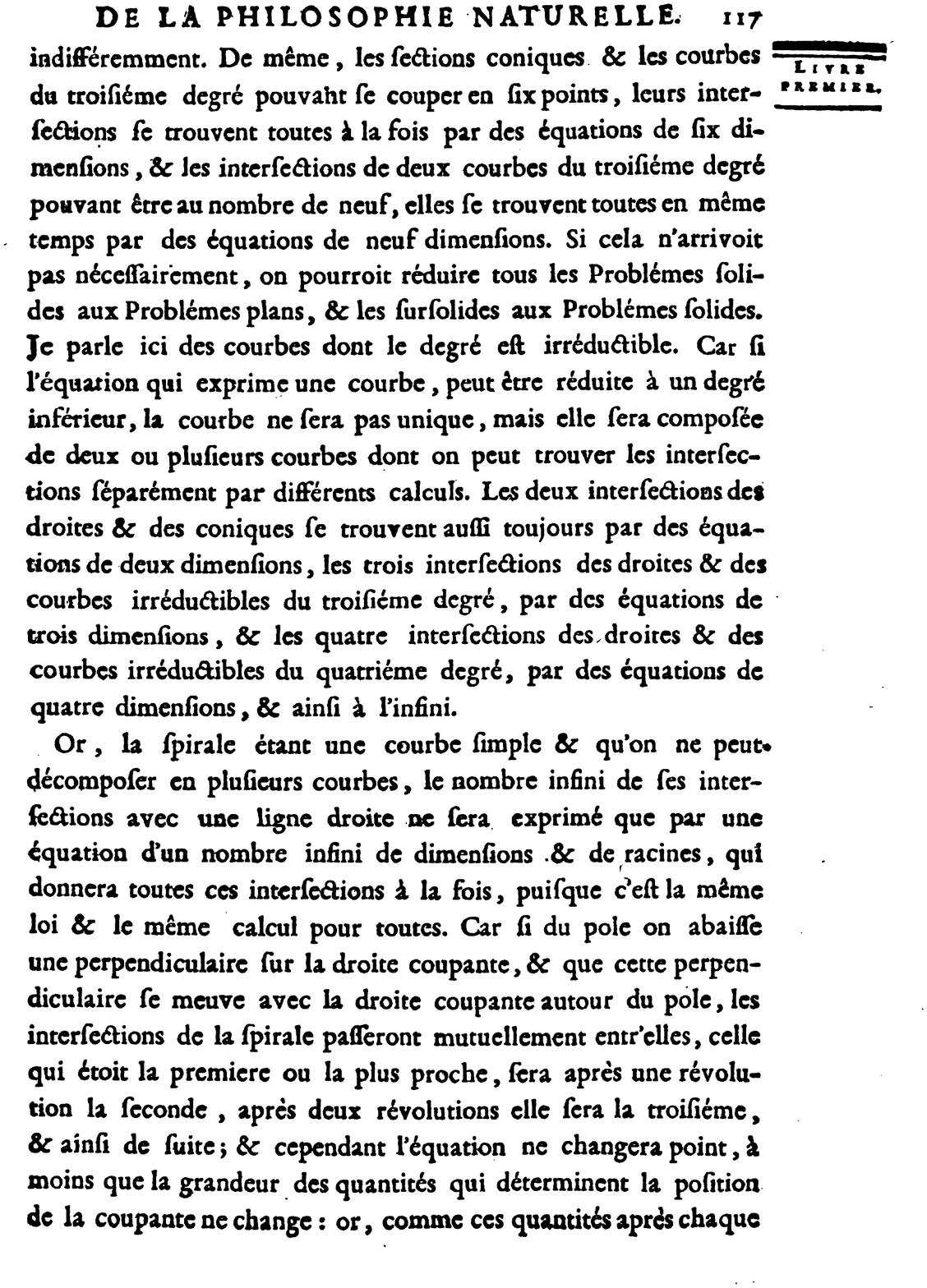}
\caption{Lemme XXVIII des \emph{Principia} de Newton}
\end{figure}

Par l'intermédiaire des lois de Kepler, selon lesquelles les planètes tracent des orbites elliptiques autour du soleil en balayant des aires égales dans des temps égaux, Newton en déduit que leur position ne peut pas être déterminée algébriquement en fonction du temps:
\begin{quote}
De là on voit que l'aire elliptique décrite autour du foyer ne peut pas être exprimée dans un temps donné par une équation finie [...]
\end{quote}
Son lemme de transcendance lui sert ainsi à justifier que la solution qu'il apporte au problème de \og trouver pour un temps donné le lieu d'un corps qui se meut dans une trajectoire elliptique donnée\fg ne soit pas ce qu'il appelle une courbe \og géométriquement rationnelle\fg, c'est-à-dire \emph{algébrique} dans la terminologie actuelle, mais \og géométriquement irrationnelle\fg, c'est-à-dire \emph{transcendante}\footnote{À savoir, la cycloïde. Et \og comme la description de cette courbe est difficile\fg, il invente dans la foulée la méthode d'approximation... de Newton!}.

Fixons un point $O$ à l'intérieur de l'ovale et une demi-droite l'ayant pour origine. Il s'agit de voir que la fonction $S(\ell)$ qui, à une autre demi\nobreakdash-droite~$\ell$ partant de ce même point, associe l'aire du secteur qu'elles découpent sous l'ovale est transcendante: si l'on exprime tout en termes de l'angle~$\theta$ décrivant le secteur, les fonctions $S(\theta)$ et~$\tan(\theta)$ ne satisfont à aucune relation~polynomiale non triviale.
\begin{figure}[ht]
\begin{center}
\includegraphics[width=0.35\textwidth]{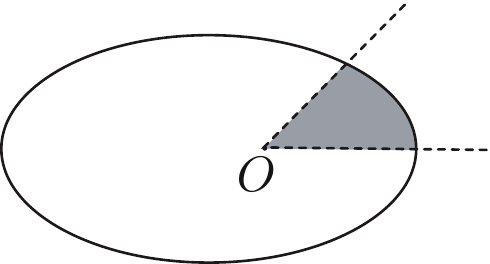}
\end{center}
\caption{Partie de l'ovale délimitée par deux demi-droites}
\label{fig2}
\end{figure}

Voici la démonstration que propose Newton. Faisons tourner \og perpétuellement d'un mouvement uniforme\fg la demi-droite $\ell$ et considérons sur elle \og un point mobile allant toujours depuis le pôle avec une vitesse qui soit comme le carré de la partie de cette ligne renfermée dans l'ovale\fg. La première condition signifie que l'angle~$\theta(t)$ est proportionnel au temps $t$ (on les supposera égaux pour simplifier) et la seconde se traduit par l'équation différentielle
\begin{displaymath}
\frac{dR}{d\theta}=r^2,
\end{displaymath}
où $R(\theta)$ désigne la position du point qui bouge et $r(\theta)$ la longueur de la partie de la droite renfermée dans l'ovale. \og Ce point décrira alors une spirale composée d'une infinité de spires\fg. Par exemple, si l'ovale est un cercle de centre $O$, alors la fonction $r(\theta)$ est constante et l'équation différentielle décrit une spirale d'Archimède~\hbox{$R=r^2 \theta+\mathrm{cst}$.} Par ailleurs, comme l'aire d'un triangle rectangle d'hypoténuse $r$ et d'angle infinitésimal $d\theta$ vaut $r^2d\theta/2$ à l'ordre un, la fonction $S(\theta)$ satisfait à l'équation différentielle
\begin{displaymath}
\frac{dS}{d\theta}=\frac{r^2}{2}.
\end{displaymath}
Les fonctions $R(\theta)$ et $2S(\theta)$ diffèrent donc par une constante: si~l'une est algébrique, l'autre aussi. Or, la spirale n'est pas une courbe algébrique car elle coupe toute droite en une infinité de points\footnote{\og Or si la portion d'aire ovale, coupée par cette droite, peut être trouvée par une équation d'un nombre fini de termes, on aura aussi, par la même équation, le rayon de la spirale qui est proportionnel à cette aire. Ainsi on pourra trouver par une équation finie tous les points d'une spirale, et par conséquent on pourra trouver aussi l'intersection d'une droite quelconque donnée de position, et d'une spirale par une équation finie; mais toute droite prolongée infiniment coupe une spirale en une infinité de points, et toute équation qui donne l'intersection quelconque de deux lignes doit donner toutes leurs intersections par autant de racines, et doit avoir par conséquent autant de dimensions qu'il y a d'intersections\fg. La discussion qui suit ce passage chez Newton laisse entendre qu'il connaissait une forme faible du théorème de Bézout, notamment l'énoncé que deux courbes planes de degrés~$m$ et $n$ sans composante commune s'intersectent en au plus $mn$ points (pour une reconstitution de la preuve qu'il aurait pu~donner, voir~\cite[note~(48),~p.\,116]{Arnold}).}.
\begin{figure}[ht]
\begin{center}
\includegraphics[width=0.75\textwidth]{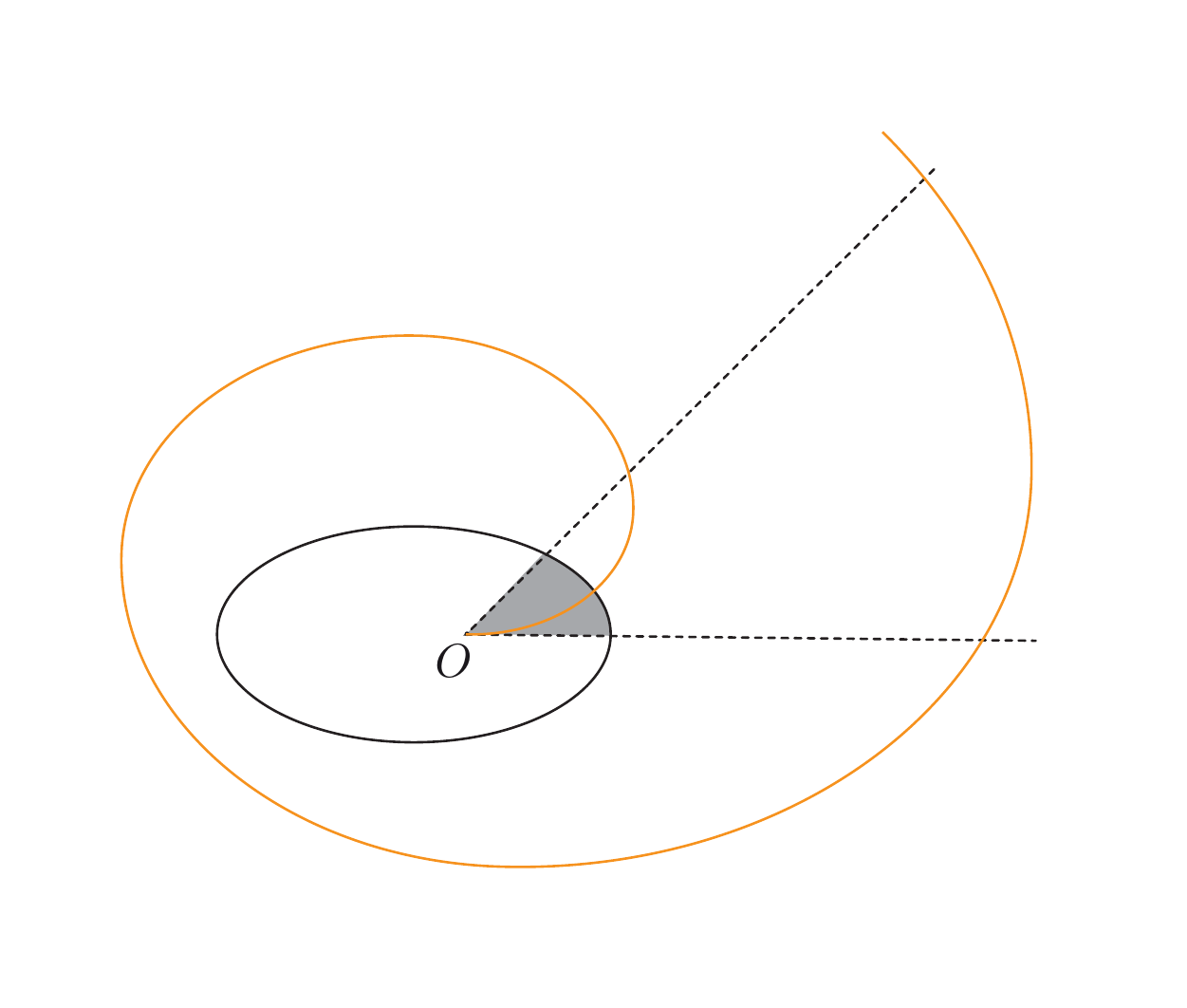}
\end{center}
\caption{Toute droite coupe une infinité de fois la spirale décrite par le point mobile}
\end{figure}

Aussi convaincant que ce raisonnement puisse paraître au premier abord, on y trouve plusieurs sources d'ambiguïté qui, ne passant pas inaperçues aux lecteurs contemporains, donnèrent lieu à une polémique qui ne serait étouffée que trois cent ans après la publication des \textit{Principia}~\cite{pourciau}. Qu'entend Newton par \og figure ovale\fg? Qu'entend\nobreakdash-il par \og jamais\fg? Jakob Bernoulli accepte la démonstration d'emblée et s'en inspire pour ses travaux sur la spirale logarithmique. Leibniz et Huygens n'y croient pas et se lancent dans un va-et-vient de contre\nobreakdash-exemples dans leur correspondance. Voici ce que Huygens écrit à Leibniz dans une lettre datée à La Haye le 5 mai 1691:
\begin{quote}
Je ne vois pas qu'on puisse accorder sa proposition pag.\,105 à Mr. Newton, parce que, ne considérant aucunement la nature de ce qu'il appelle ovale, mais seulement que c'est une ligne fermée tout au tour, il n'exclut pas même le carré ou le triangle\footnote{Cette lettre et celles qui suivent sont accessibles sur le site de la bibliothèque numérique des lettres néerlandaises \url{https://www.dbnl.org/tekst/huyg003oeuv10_01/}, où elles sont numérotées 2664 (2 mars 1691), 2667 (26 mars~1691), 2676 (20~avril 1691) et 2680 (5 mai~1691). J'ai modernisé la graphie et la ponctuation.}. 
\end{quote}

En effet, si les carrés ont droit de cité parmi les ovales, l'aire du secteur décrit par un angle $\theta$ peut dépendre algébriquement de sa tangente (par exemple, elle vaut~$\tan(\theta)/2$ pour le carré de côté~$2$ centré en l'origine lorsque $\theta$ est compris entre~$0$ et $\pi/4$), et de même pour le triangle.
\begin{figure}[ht]
\vspace{3mm}
\begin{center}
\includegraphics[width=0.7\textwidth]{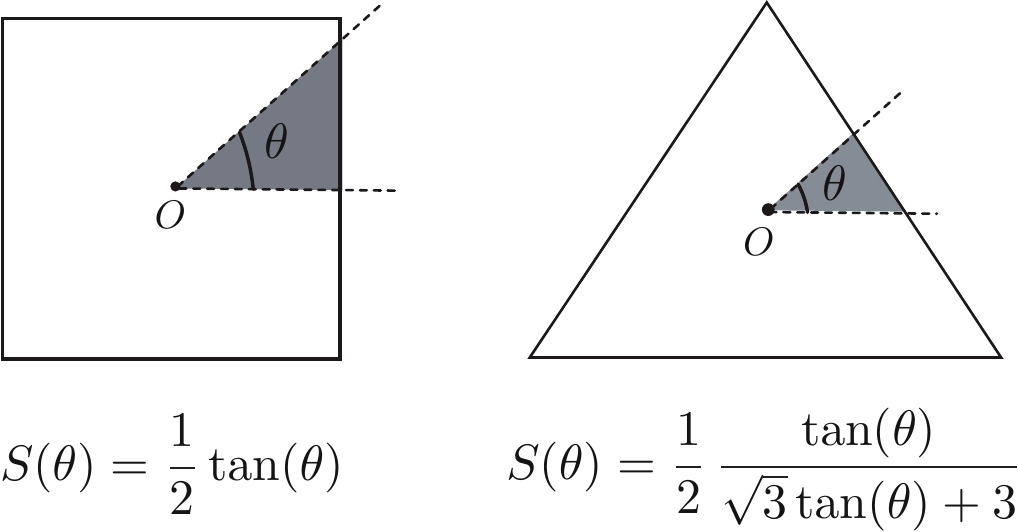}
\end{center}
\caption{Le lemme de Newton ne vaut ni pour le carré ni pour le triangle}
\end{figure} Ces potentiels contre-exemples venaient s'ajouter à un autre, l'ovale délimité par les paraboles $y=x^2-1$ et $y=1-x^2$, que Huygens avait déjà évoqué dans une lettre du 26 mars. Et à Leibniz d'objecter, dans sa réponse envoyée le 20 avril depuis Hanovre:
\begin{quote}
Considérez [...] si contre votre instance des deux portions égales de parabole sur une même base, Monsieur Newton pour soutenir l'impossibilité de la quadrature des ovales ne pourrait répondre qu'une telle ovale serait fausse et non pas composée d'une même ligne recourante, comme il semble que son raisonnement demande, puisqu'une parabole continuée ne tombe pas dans l'autre.
\end{quote}
\begin{figure}[ht]
\vspace{3mm}
\begin{center}
\begin{tikzpicture}[scale=0.5]
\draw[domain=-2:-1,smooth,variable=\x, dashed] plot ({\x},{1 -\x*\x});
\draw[domain=-1:1,smooth,variable=\x] plot ({\x},{1 -\x*\x});
\draw[domain=1:2,smooth,variable=\x, dashed] plot ({\x},{1 -\x*\x});

\draw[domain=-2:-1,smooth,variable=\x, dashed] plot ({\x},{\x*\x-1});
\draw[domain=-1:1,smooth,variable=\x] plot ({\x},{\x*\x-1});
\draw[domain=1:2,smooth,variable=\x, dashed] plot ({\x},{\x*\x-1});
\end{tikzpicture}
\hspace{2cm}
\begin{tikzpicture}[scale=1.5]
\draw[domain=-5:5,smooth,variable=\t,samples=101]
plot ({(\t*\t -1)/(\t*\t+1)},{2*\t*(\t*\t-1)/(\t*\t+1)/(\t*\t+1)});
\draw[domain=-5:5,smooth,variable=\t,samples=101]
plot ({-(\t*\t -1)/(\t*\t+1)},{2*\t*(\t*\t-1)/(\t*\t+1)/(\t*\t+1)});
\end{tikzpicture}
\end{center}
\caption{La double parabole et la lemniscate de Huygens}
\label{fig:lemniscate}
\end{figure}
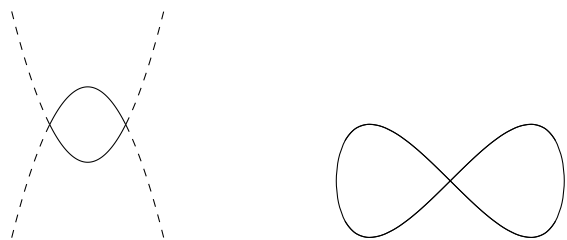

Par contre, Leibniz ne voit pas d'obstruction à inclure la \emph{lemniscate} d'équation $y^2=x^2(1-x^2)$, dont Huygens lui avait parlé dans des lettres précédentes, parmi les contre-exemples:
\begin{quote}
Mais votre ligne qui fait $8$\footnote{Comme on peut le voir dans sa lettre 2667, Huygens considère plutôt la courbe d'équation $x^2=y^2(1-y^2)$, qui est celle de la figure \ref{fig:lemniscate} tournée de $90$ degrés; cette courbe est aussi connue sous le nom de lemniscate de Gerono.} est véritablement recourante, et son raisonnement y est applicable, quoiqu'elle n'ait pas justement la forme d'une ovale, et selon lui, elle ne devrait pas être généralement quadrable. Il serait bon de considérer son raisonnement en lui même, pour voir où gît le manquement.
\end{quote} En effet, étant donné un angle $\theta \in [0, \pi/4)$ et posant $t=\tan(\theta)$ pour abréger la notation, la droite $y=tx$ intersecte la lemniscate au point d'abscisse $\sqrt{1-t^2}$, et l'aire du secteur décrit par $\theta$ est la somme de l'aire d'un triangle rectangle de base $\sqrt{1-t^2}$ et de hauteur~$t\sqrt{1-t^2}$ et de l'aire sous la fonction $y=x\sqrt{1-x^2}$ entre $x=\sqrt{1-t^2}$ et $x=1$, ce qui résulte en l'expression algébrique
\begin{displaymath}
\frac{1}{2}\,t(1-t^2)+\int^1_{\sqrt{1-t^2}}x\sqrt{1-x^2}\,dx=\frac{1}{2}\,t-\frac{1}{6}\,t^3.
\end{displaymath}

Il doit bien avoir une faute dans le raisonnement du maître! Leibniz écrit \og error\fg à côté du lemme dans sa copie des \emph{Principia}\footnote{L'exemplaire annoté de la main de Leibniz fait partie des collections de la Fondation Martin Bodmer et a été numérisé:\par\nopagebreak\url{https://bodmerlab.unige.ch/fr/constellations/early-modern-english-books/mirador/1072068344?page=002}.\par Je remercie Morgane Cariou, Rosa Navarro Durán et José Antonio Pascual de leur aide à déchiffrer les marginalia de Leibniz.}, ce qui ne l'empêchera pas de faire un compliment extraordinaire à Newton le 2 mars~1691 lorsqu'il le compare implicitement à Homère en citant l'\emph{Art poétique} d'Horace:
\begin{quote}
Puisque la première achevée retourne en elle même, en forme de 8, on en peut juger que le théorème de Mr. Newton p.\,105, qui pretend qu'il n'y a point de courbe recourante (de la géométrie ordinaire) indéfiniment quadrable, ne saurait subsister, et qu'il y a quelque faute dans sa démonstration. Mais je ne l'en estime pas moins; \emph{Opere in longo fas est obrepere somnum.}
\end{quote}

Newton, a-t-il cédé au sommeil dans son long ouvrage? De fait, ses arguments s'appliquent aussi au carré, au triangle et à la lemniscate; quelque soit la définition d'un ovale, ils montrent que l'aire des secteurs découpés ne peut pas être donnée par une fonction \emph{globale} des paramètres, pour la simple raison qu'une fonction algébrique non constante n'est jamais périodique. C'était sans doute clair pour Leibniz, qui en avait lui-même déduit quelque peu avant, et c'est à cette occasion qu'il a introduit le terme, que les fonctions trigonométriques sont transcendantes. Dans le cas du carré, l'expression algébrique que nous avons trouvée ne décrit l'aire que pour des angles compris entre~$0$ et $\pi/4$; lorsque l'on dépasse le premier sommet, il faut la remplacer par une autre. La vraie question qui se pose est donc celle de l'algébricité \emph{locale}: un angle $\theta_0$ étant donné, existe-t-il un intervalle ouvert autour de $\theta_0$ sur lequel l'aire $S(\theta)$ est une fonction algébrique de la tangente de $\theta$? C'est la possibilité que la fonction change avec l'angle qui donne tout son sel au lemme de Newton\dots  

Partant du principe que le sens du mot \emph{ovale} est celui courant à cette époque, à~savoir une courbe plane, fermée, connexe et convexe contenue dans le lieu des zéros d'un polynôme réel en deux variables, la réponse dépend alors de la régularité des ovales. Si l'on impose pour seule restriction qu'ils puissent être tracés de façon continue, on peut sans peine construire des exemples où l'aire est une fonction localement (mais non globalement) algébrique des paramètres, on l'a vu. On dira qu'une courbe plane est de classe~$\cC^\infty$ si, dans un petit disque autour de chacun de ses points, elle est défi\-nie par le graphe d'une fonction indéfiniment dérivable. Aucun des contre\nobreakdash-exemples proposés par Huygens et Leibniz n'est un ovale de classe~$\cC^\infty$, puisqu'ils possèdent tous des points singuliers en lesquels la dérivée n'existe pas.

\begin{theoreme}[Newton] L'aire des segments tracés sur un ovale de classe $\cC^\infty$ n'est pas une fonction localement algébrique.
\end{theoreme}

Comme l'explique Arnold \cite{Arnold}, cette assertion bien plus forte résulte de la non-algébricité globale en utilisant le fait qu'un ovale de classe~$\cC^\infty$ est \emph{analytique}, c'est-à-dire qu'il est défini autour de chaque point par une fonction développable en série entière\footnote{L'hypothèse que les ovales sont contenus dans le lieu des zéros d'un polynôme réel en deux variables exclut les exemples classiques de fonctions de classe $\cC^\infty$ qui ne sont pas développables en série entière. L'analyticité est alors une conséquence d'un autre théorème de Newton: l'existence des développements en série de Puiseux $y=a_1x^{1/n}+a_2 x^{2/n}+\cdots$ pour les branches des courbes algébriques.}. En effet, supposons par l'absurde que l'aire est une fonction localement, mais non globalement, algébrique des paramètres; il existe alors un point sur l'ovale tel que deux expressions algébriques distinctes sont nécessaires pour décrire l'aire d'un côté et de l'autre de ce point. Or, l'analyticité de l'ovale implique que l'aire est une fonction analytique des paramètres; en développant en série entière, on voit que les deux expressions coïncident sur un petit ouvert et sont donc partout égales. L'aire serait ainsi une fonction globalement algébrique, contradiction. 

Revenons maintenant sur le sens du mot \og jamais\fg, en nous appuyant sur un autre extrait de la lettre de Leibniz écrite à Hanovre entre le 10 et le~20 avril 1691:
\begin{quote}
Quant au cercle et à l'ellipse, l'impossibilité de leur quadrature générale est assez démontrée, mais je n'ai pas encore vu qu'on ait donné aucune démonstration pour prouver que le cercle entier, ou quelque portion déterminée, n'est pas quadrable.
\end{quote} Les fonctions trigonométriques sont transcendantes en tant que fonctions, mais qu'en est-il de leurs valeurs, par exemple le nombre~$\pi$? Peut-on démontrer que l'aire d'une partie donnée d'une figure ovale est un nombre transcendant, c'est-à-dire qu'il n'est pas racine d'un polynôme non nul à coefficients rationnels? Il s'agit là de deux exemples de \emph{périodes}, un terme qui s'est petit à petit écarté de son sens historique initial de \og période d'une fonction périodique \fg (les fonctions trigonométriques, les fonctions elliptiques, etc.) pour englober tous les nombres qui apparaissent comme des intégrales de fonctions algébriques sur des domaines définis par des inégalités polynomiales; n'importe quel premier cours d'analyse en est plein. En un sens, ce sont les nombres accessibles par des méthodes géométriques. Affranchis de leur lien avec la physique (aires balayées par des planètes, périodes d'un pendule, etc.) et des fonctions dont ils sont des valeurs spéciales ou des périodes, ils sont devenus un objet d'étude à part entière, avec pour horizon inaccessible la question de décrire toutes les relations algébriques entre eux. Bien plus difficile que celui pour les fonctions\footnote{Comme le témoigne, par exemple, le fait qu'il ait fallu attendre jusqu'en 1882 pour que la question de Leibniz sur la transcendance de $\pi$ trouve une réponse chez Lindemann. Les fonctions hypergéométriques constituent un bel exemple de cette dichotomie entre la transcendance d'une fonction et celle de ses valeurs spéciales. Soient $a, b, c$ des nombres rationnels qui ne sont pas des entiers négatifs ou nuls. La \textit{fonction hypergéométrique} de paramètres $a, b, c$ est définie par la série entière
\begin{displaymath}
_{2}F_1\biggl(\begin{matrix} a \ b \\ c \end{matrix} \mid z \biggr)=1+\frac{ab}{c}\,z+\frac{a(a+1)b(b+1)}{c(c+1)}\,\frac{z^2}{2!}+\cdots
\end{displaymath}
qui converge absolument dans le cercle unité $|z|<1$. Ces fonctions peuvent être algébriques, comme le montre l'exemple
\begin{displaymath}
_{2}F_1\biggl(\begin{matrix} a \ 1 \\ 1 \end{matrix} \mid z \biggr)=1+az+\frac{a(a+1)}{2!}\,z^2+\cdots=(1-z)^{-a}
\end{displaymath}
quel que soit l'entier $a$, mais mis à part une liste de cas exceptionnels, elles sont en général transcendantes. Néanmoins, ces fonctions transcendantes prennent parfois des valeurs algébriques en des nombres algébriques; c'est le cas des valeurs
\begin{equation}\label{eqn:Beukers}
_{2}F_1\biggl(\begin{matrix} \frac{1}{12} \ \frac{5}{12} \\ \frac{1}{2} \end{matrix} \mid \frac{1323}{1331} \biggr)=\frac{3}{4}\sqrt[4]{11}, \ _{2}F_1\biggl(\begin{matrix} \frac{1}{12} \ \frac{7}{12} \\ \frac{2}{3} \end{matrix} \mid \frac{6400}{64009} \biggr)=\frac{2}{3}\sqrt[6]{253}
\end{equation}
trouvées par Beukers et Wolfart \cite{BW}. Les valeurs des fonctions hypergéométriques en des nombres algébriques sont reliées aux périodes par le biais de la représentation intégrale d'Euler
\begin{displaymath}
_{2}F_1\biggl(\begin{matrix} a \ b \\ c \end{matrix} \mid z \biggr)=\frac{\int_0^1 t^{a-1} (1-t)^{c-a-1}(1-tz)^{-b} dt}{\int_0^1 t^{a-1} (1-t)^{c-a-1} dt} \quad (c>a>0),
\end{displaymath}
qui montre qu'elles sont des quotients d'intégrales le long d'un chemin de formes différentielles sur une courbe algébrique. Des identités telles que \eqref{eqn:Beukers} fournissent donc une relation non triviale entre périodes.}, ce problème admet néanmoins une solution conjecturale étonnamment simple. En lisant entre les lignes, il~est tentant d'imaginer Leibniz conjecturant que les périodes ont tendance à être transcendantes et qu'il n'y a en général d'autres relations algébriques parmi elles que celles qui s'expliquent par les propriétés formelles de l'intégration. Le but de ces notes est d'expliquer deux manières de rendre cette intuition rigoureuse, la conjecture de Kontsevich-Zagier et la conjecture des périodes de Grothendieck, et comment, en les supposant vraies, il est possible d'étendre aux périodes la théorie de Galois pour les nombres algébriques.

\subsection{La période d'un pendule} Considérons un pendule de longueur $\ell$ et de masse~$m$ qui oscille sans frottement dans le plan vertical, sa position étant repérée par l'angle~$\theta$.
\begin{figure}[ht]
\centerline{\includegraphics[width=0.2\textwidth]{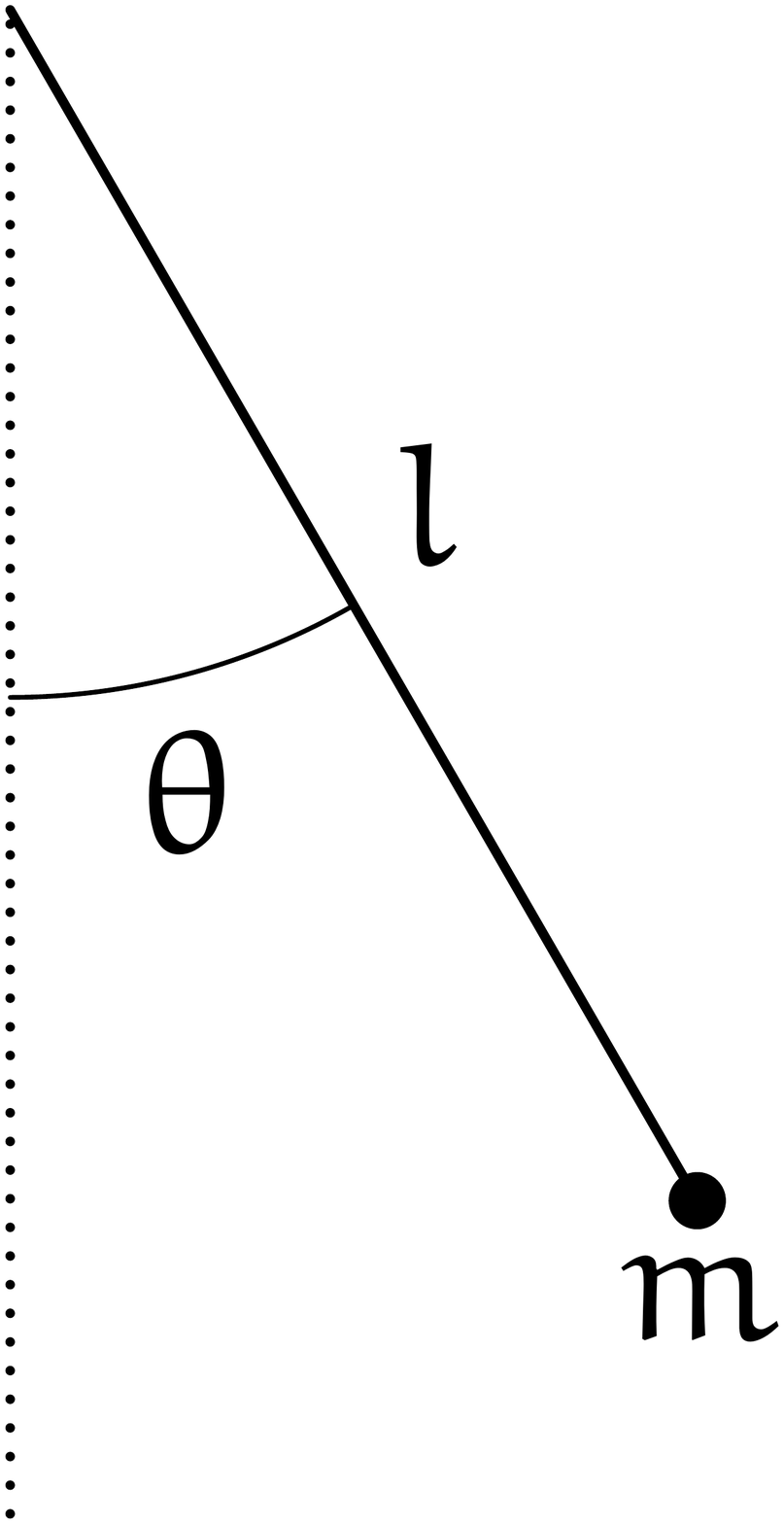}}
\caption{Un pendule de longueur $\ell$ et de masse $m$}
\end{figure}
Son mouvement est décrit par la deuxième loi de \hbox{Newton}~$F=ma$ reliant la force $F$ et l'accélération tangentielle $a$, ce qui donne dans ce cas $a=-g \sin \theta$, où $g$ désigne l'accélération de la pesanteur à l'endroit de l'expérience. En exprimant l'accélération comme la dérivée seconde de la longueur d'arc~$\ell \theta$, on trouve par suite l'équation différentielle de second ordre
\begin{equation}\label{eqn:motion}
\frac{d^2 \theta}{dt^2}+\frac{g}{\ell} \sin \theta=0,
\end{equation}
qui n'admet pas de solution en termes de fonctions élémentaires. Si l'on ne s'intéresse qu'aux oscillations de petite amplitude, c'est-à-dire aux solutions dont la valeur initiale~$\theta_0$ en $t=0$ est petite en module, on peut remplacer $\sin \theta$ par son développement limité $\theta$ à l'ordre $1$; l'équation qui en résulte a alors pour solution $\theta=\theta_0\cos(\sqrt{\sfrac{g}{\ell}}\, t)$, d'où la valeur approchée $T=2\pi \sqrt{\sfrac{\ell}{g}}$ pour la période du pendule.

Pour calculer la valeur exacte de $T$, une possibilité consiste à multiplier \eqref{eqn:motion} par~$\sfrac{d\theta}{dt}$, puis intégrer par parties. On trouve l'équation
\begin{displaymath}
\frac{d\theta}{dt}=\sqrt{\frac{2g}{\ell}(\cos \theta-\cos \theta_0)}
\end{displaymath}
qui, par inversion, permet d'exprimer la période du pendule comme
\begin{align*}
T&=4\sqrt{\frac{\ell}{2g}} \int_0^{\theta_0} \frac{d\theta }{\sqrt{\cos \theta-\cos \theta_0}} \\
&=2\sqrt{\frac{\ell}{g}} \int_0^{\theta_0} \frac{d\theta}{\sqrt{\sin^2(\sfrac{\theta_0}{2}) -\sin^2(\sfrac{\theta}{2})}},
\end{align*}
où l'on a utilisé l'identité trigonométrique $\cos \theta=1-2\sin^2(\sfrac{\theta}{2})$. Intro\-duisant le \textit{module elliptique} $k=\sin(\sfrac{\theta_0}{2})$, on aboutit à l'\textit{intégrale elliptique de première espèce}
\begin{align*}
T&=4\sqrt{\frac{\ell}{g}}\int_0^{\sfrac{\pi}{2}} \frac{d\phi}{\sqrt{1-k^2\sin^2 \phi}} \\
&=4\sqrt{\frac{\ell}{g}}\int_0^1 \frac{dx}{\sqrt{(1-x^2)(1-k^2 x^2)}}
\end{align*}
par les changements de variable $\sin(\sfrac{\theta}{2})=k\sin \phi$ et $x=\sin \phi$. Si~$k$ est un nombre algébrique, c'est un exemple paradigmatique de période au sens de la théorie des nombres.

\subsection*{Guide de lecture} Ces notes reprennent les deux exposés que j'ai donnés aux Journées~X\nobreakdash-UPS \emph{Périodes et transcendance} le 15 et le 16 avril 2019. Je n'ai pas su résister la tentation d'aller bien au-delà de mon cahier des charges et d'écrire le texte que, rétrospectivement, j'aurais voulu avoir eu à ma disposition quand je m'initiais dans le domaine il y a dix ans. En fonction des connaissances préalables et de l'énergie que l'on est prêt à y investir, il admet plusieurs niveaux de lecture: 
\begin{itemize}
\item Les sections \ref{sec:2} et \ref{sec:Kontsevich-Zagier}, demandant le moins de prérequis, introduisent la notion de période et la conjecture de Kontsevich-Zagier; l'accent est surtout mis sur les exemples. Vers la fin, j'énonce quelques théorèmes récents dans la direction de cette conjecture et j'esquisse les contours que pourrait prendre une théorie de Galois pour les périodes. 
\item Ceux qui souhaitent approfondir ce dernier point de vue sont alors invités à continuer la lecture par la section \ref{sec:4}, au moins jusqu'à l'exemple \ref{exmp:lacet0}, les trois premiers numéros de la section \ref{sec:5} et le début de la section \ref{sec:accoupl}, jusqu'à l'exemple \ref{example:nombrepi}. Le matériel ici demande une certaine aisance avec les espaces topologiques et les anneaux. Dans un premier temps, de l'introduction aux variétés algébriques on pourra ne retenir que l'anneau des fonctions et la définition de lissité. 
\item Il y a alors deux possibilités pour continuer la lecture: ou bien aborder les variantes de l'interprétation cohomologique des périodes présentées dans la sections \ref{sec:variante-relative}, au moins jusqu'à l'exemple~\ref{exmp:log-cohomologique}, et la section~\ref{sec:periodexp}; ou bien se concentrer sur les périodes des courbes elliptiques, en étudiant le reste des sections~\ref{sec:4} et~\ref{sec:5}, puis les deux derniers numéros de la section~\ref{sec:accoupl} (ceux qui souhaitent encore rallonger le chemin pourront ensuite lire les sections~\ref{sec:Legendre} et~\ref{sec:functionaltrans}). Dans les deux cas, un peu de familiarité avec ces objets de base de l'algèbre homologique (complexes, suites exactes, etc.) est supposée; en ce sens, les passages les plus exigeants sont l'exemple~\ref{exmp:dilogs} et les dernières pages de la section~\ref{sec:5}. Pour comprendre les constructions faisant intervenir des variétés projectives, il est aussi utile d'avoir en tête la notion de carte d'une variété différentielle. 

\item La section \ref{sec:period-conjecture} est enfin de nature plus avancée. Je l'ai écrite en imaginant ce que je répondrais à la question \og sur quoi je travaille?\fg à des étudiants de M2 en théorie des nombres et géométrie algébrique ou à des collègues dans des sujets plus éloignés, s'ils avaient une heure à m'accorder et qu'ils n'avaient pas peur de ne pas tout comprendre. 
\end{itemize}

\subsection*{Remerciements} Je~tiens à remercier les organisateurs et les participants de ces rencontres de m'avoir permis de porter un regard frais sur ma recherche. Ma vision des périodes a été enrichie par les échanges que j'ai eu la chance d'avoir, au cours des années, avec Yves André, Joseph Ayoub, Jean\nobreakdash-Benoît Bost, Francis Brown, José Ignacio Burgos Gil, Clément Dupont et Peter Jossen ; qu'ils soient tous ici remerciés! C'est notam\-ment Yves André qui m'a parlé en premier du lemme de Newton, Peter Jossen qui m'a appris la preuve du théorème \ref{thm:transfonct}, et avec Clément Dupont que j'ai souvent discuté du dilogarithme. Merci enfin à Olivier Benoist, Clément Dupont, Tiago Jardim da Fonseca et Marco Maculan pour leurs commentaires sur des versions préliminaires de ce texte, et à Peter Jossen, Marco Maculan et Claude Sabbah pour leur aide dans la préparation des figures l'illustrant.

\section{Définition et exemples}\label{sec:2}

Dans cette section, on présente une définition élémentaire de la notion de période qui a été dégagée par Kontsevich et Zagier \cite{KZ} au tournant du siècle dernier et on en déduit quelques conséquences immédiates, comme le fait que les périodes forment un sous-anneau dénombrable de celui des nombres complexes. On donne ensuite plusieurs exemples de périodes (les nombres algébriques, le nombre~$\pi$, les logarithmes des nombres rationnels, les intégrales elliptiques, les valeurs zêta multiples...) sur lesquels on reviendra au fur et à mesure que l'on disposera d'outils plus sophistiqués pour les étudier. Conjecturalement, des nombres tels que~$e$, la racine carrée de $\pi$ ou la constante~$\gamma$ d'Euler n'en font pas partie, mais il est néanmoins possible de les traiter sur le même pied que les périodes en en assouplissant la définition; on arrive de cette manière à la notion de période exponentielle, aussi évoquée par Kontsevich et Zagier à la fin de leur article. On conclut en discutant brièvement la question de trouver des exemples de nombres qui ne sont pas des périodes.

\subsection{Définition et premières propriétés} Soit $n \geq 1$ un entier. Une \textit{fonction rationnelle} à coefficients rationnels en $n$ variables est un quotient $f=P \slash Q$ de polynômes $P$ et~$Q$ dans $\QQ[x_1, \ldots, x_n]$, avec $Q$ non nul. On dit qu'un sous-ensemble $S \subset \RR^n$ est $\QQ$-\textit{semi\nobreakdash-algébrique} s'il est de la forme
\begin{displaymath}
S=\{(x_1, \ldots, x_n) \in \RR^n \mid P(x_1, \ldots, x_r) \geq 0 \}
\end{displaymath}
pour un polynôme $P \in \QQ[x_1, \ldots, x_n]$, ou s'il peut être obtenu à partir de ceux-ci en itérant un nombre fini de fois les opérations réunion, intersection et complémentaire. Par exemple, tous les sous-ensembles de~$\RR^n$ définis par un nombre fini d'équations et d'inéquations
\begin{displaymath}
P_i(x_1, \dots, x_n)=0, \qquad Q_j(x_1, \dots, x_n) \geq 0,
\end{displaymath}
avec $P_i$ et $Q_j$ des polynômes à coefficients rationnels en $n$ variables sont $\QQ$-semi-algébriques. Comme on le verra dans l'exemple~\ref{exmp:algebrique}, les sous-ensembles $\QQ$-semi\nobreakdash-algébriques de~$\RR$ sont les réunions finies de points algébriques réels et d'intervalles ouverts à bornes algébriques réelles ou infinies. La \textit{dimension} d'un ensemble $\QQ$-semi\nobreakdash-algébrique non vide \hbox{$S \subset \RR^n$} est le plus grand entier $d \geq 0$ tel que $S$ contient un ouvert homéomorphe à une boule dans $\RR^d$ (par exemple, $S \subset \RR$ est de dimension $1$ s'il contient un intervalle ouvert et de dimension $0$ sinon). D'après un théorème de Tarski et Seidenberg \cite[Th.\,2.2.1]{real-algebraic-geometry}, l'image d'un sous-ensemble $\QQ$-semi-algébrique de $\RR^{n+1}$ par la projection vers les~$n$ premières coordonnées est un sous-ensemble $\QQ$-semi-algébrique de~$\RR^n$. Par exemple, l'image par la projection vers la première coordonnée du sous-ensemble de $\RR^2$ défini par l'équation $x=y^2$ est l'ensemble $\{x \geq 0\}$ des nombres réels positifs. 

\begin{definition}[Kontsevich-Zagier]\label{defn:periodeKZ} Une \textit{période} est un nombre complexe dont la partie réelle et la partie imaginaire peuvent s'écrire comme des intégrales absolument convergentes
\begin{equation}\label{eqn:defperiod}
\int_{S} f(x_1, \ldots, x_n) \, dx_1\cdots dx_n,
\end{equation}
où $f$ est une fonction rationnelle à coefficients rationnels en $n$ variables et $S \subset \RR^n$ est un ensemble $\QQ$-semi-algébrique.
\end{definition}

Comme tant l'intégrande que le domaine d'intégration sont obtenus à partir de polynômes à coefficients rationnels, les périodes forment un sous-ensemble \textit{dénombrable} de celui des nombres complexes; un nombre complexe \og pris au hasard\fg n'est donc pas une période. Cependant, on verra plus tard que beaucoup de constantes de ce qui pourrait être appelé la vie mathématique courante sont des périodes.

\begin{lemma}\label{lemme:difference} Les périodes sont les nombres complexes dont la partie réelle et la partie imaginaire peuvent s'écrire comme des différences
$$\mathrm{vol}(S_1)-\mathrm{vol}(S_2),$$ où $S_1$ et $S_2$ sont des ensembles $\QQ$\nobreakdash-semi\nobreakdash-algébriques de volume fini.
\end{lemma}

\begin{proof} Soit $\int_S f(x_1, \ldots, x_n) dx_1\cdots dx_n$ une période. En décomposant le domaine d'intégration $S$ comme la réunion des sous-ensembles $S_+$ et $S_{-}$ où $f$ prend, respectivement, des valeurs positives et négatives, on peut récrire
\begin{align*}
\int_S f(x_1, \ldots, x_n)dx_1\cdots dx_n=\int_{S_+} &f(x_1, \ldots, x_n)dx_1\cdots dx_n \\
&-\int_{S_{-}} -f(x_1, \ldots, x_n)dx_1\cdots dx_n.
\end{align*}
Comme l'intégrale de départ est supposée \textit{absolument} convergente, les deux dernières intégrales sont des nombres réels positifs. Afin de les écrire comme des volumes, on introduit une nouvelle variable $t$ et on considère les sous-ensembles de $\RR^{n+1}$ définis par
\begin{align*}
S_1=\{(x_1,\dots, x_n, t) \in S_+ \times \RR \mid 0 \leq t \leq f(x_1, \ldots, x_n)\}, \\
S_2=\{(x_1,\dots, x_n, t) \in S_- \times \RR \mid 0 \geq t \geq f(x_1, \ldots, x_n)\}.
\end{align*}
Il s'agit d'ensembles $\QQ$-semi-algébriques car $f$ est une fonction rationnelle; en posant $f=\sfrac{P}{Q}$, on peut par exemple écrire $S_1$ comme l'intersection de la demi-droite $\{0 \leq t \}$ avec l'ensemble
\begin{displaymath}
\left( \{ tQ-P \leq 0\} \cap \{Q \geq 0 \} \right) \cup (\{tQ-P \geq 0\} \cap \{Q \leq 0\}).
\end{displaymath}
Par construction, le volume de $S_1$ est égal à
\begin{displaymath}
\vol(S_1)=\int_{S_1} dx_1\cdots dx_n dt=\int_{S_+} f(x_1, \ldots, x_n) dx_1\cdots dx_n,
\end{displaymath}
et de même pour $S_2$ en remplaçant $S_+$ par $S_{-}$ et $f$ par $-f$. On a donc écrit la période de départ sous la forme $\vol(S_1)-\vol(S_2)$. Le fait, aussi simple, que n'importe quelle différence de volumes d'ensembles~$\QQ$\nobreakdash-semi\nobreakdash-algébriques est une période est laissé en exercice.
\end{proof}

Avec plus de travail, on peut en fait démontrer que toute période est, au signe près, le volume d'un ensemble~$\QQ$\nobreakdash-semi\nobreakdash-algébrique (autrement dit, qu'une différence positive de tels volumes est encore un tel volume), que l'on peut de surcroît supposer compact \cite{viu-sos}. 

\begin{lemma}\label{lem:periodalgebra} Les périodes forment un sous-anneau de celui des nombres complexes contenant les nombres rationnels.
\end{lemma}

\begin{proof} En écrivant $r=\int_{0 \leq x \leq r} dx$ ou $r=\int_{r \leq x \leq 0} dx$ selon que~$r$ est positif ou négatif, on voit que tout nombre rationnel~$r$ est une période. Il suffit donc de démontrer que l'ensemble des périodes est stable par addition et par multiplication. Dans le deuxième cas, il s'agit d'une application immédiate du théorème de Fubini
\begin{multline*}
\left( \int_{S} f(x_1, \dots, x_n) dx_1 \cdots dx_n \right) \cdot \left(\int_{T} g(y_1, \dots, y_m)dy_1 \cdots y_m\right) \\
=\int_{S \times T} f(x_1, \dots, x_n)g(y_1, \dots, y_m) dx_1 \cdots dx_n dy_1 \cdots dy_m,
\end{multline*}
car un produit de fonctions rationnelles est une fonction rationnelle et un produit de sous-ensembles $\QQ$-semi-algébriques en est un aussi.

Pour démontrer que les périodes sont stables par addition, on peut supposer sans perte de généralité que les deux intégrales ont même nombre de variables. En effet, si l'on avait disons \hbox{$m<n$}, il suffirait de voir $g$ comme une fonction de $n$ variables et d'utiliser l'identité
\begin{displaymath}
\int_T g(x_1, \ldots, x_m)dx_1\cdots dx_m=\int_{T \times [0, 1]^{n-m}} g(x_1, \dots, x_n) dx_1 \cdots dx_n
\end{displaymath}
pour ramener l'intégrale à un sous-ensemble de~$\RR^n$. Supposons donc que les deux périodes sont données par des intégrales faisant intervenir des fonctions de $n$ variables. Après les avoir écrites comme des différences de volumes (lemme~\ref{lemme:difference}), il suffit de voir que
\begin{equation}\label{eqn:difference}
\vol(S_1)+\vol(T_1)-\vol(S_2)-\vol(T_2) \\
\end{equation}
est encore une différence de volumes. Or, en choisissant des intervalles disjoints $I_1, I_2, I_3, I_4 \subset \RR$ de longueur $1$ et d'extrémités rationnelles, pour qu'ils soient $\QQ$-semi-algébriques, on peut récrire \eqref{eqn:difference} comme
\begin{displaymath}
\vol((S_1 \times I_1) \cup (T_1 \times I_2))-\vol((S_2 \times I_3) \cup (T_2 \times I_4)).
\end{displaymath}
La somme de deux périodes est donc une période.
\end{proof}

\subsection{Exemples}

On vient de voir que les nombres rationnels sont des périodes. Voici quelques exemples plus intéressants:

\begin{exemple}[les nombres algébriques]\label{exmp:algebrique} On dit qu'un nombre complexe $\alpha$ est \textit{algébrique} s'il existe un polynôme non nul à coefficients rationnels $P \in \QQ[x]$ dont $\alpha$ est racine:~$P(\alpha)=0$. Le~seul polynôme unitaire et irréductible ayant cette propriété est alors appe\-lé le polynôme minimal de $\alpha$, et son degré le degré de~$\alpha$. La~partie réelle et la partie imaginaire d'un nombre algébrique étant des réels algébriques, pour démontrer que $\alpha$ est une période il suffit de traiter le cas où $\alpha$ est réel strictement positif. Soient donc $\alpha>0$ un nombre algébrique et $P$ son polynôme minimal. Pour voir que $\alpha=\int_0^\alpha dx$ est une période, il suffit de démontrer que l'ensemble $\{0 \leq x \leq \alpha\}$ est~$\QQ$\nobreakdash-semi\nobreakdash-algébrique. Comme $\alpha$ est une racine simple de $P$, il existe des nombres rationnels $0<r<\alpha<s$ tels que la fonction $P$ soit strictement monotone dans l'intervalle~$[r, s]$. L'écriture
\begin{displaymath}
\{0 \leq x \leq \alpha \}=\{0 \leq x \leq r\} \cup \bigl(\{r \leq x \leq s \} \cap \{\varepsilon P(x) \leq 0 \}\bigr),
\end{displaymath}
avec $\varepsilon=1$ si $P$ est croissante et $\varepsilon=-1$ si $P$ est décroissante, permet alors de conclure. De là, on déduit également que les sous-ensembles $\QQ$\nobreakdash-semi-algébriques de $\RR$ sont les réunions finies de points algébriques et d'intervalles ouverts à bornes algébriques ou infinies.
\end{exemple}

\begin{exemple}[le nombre $\pi$]\label{exmp:pi} La définition de $\pi$ comme l'aire du cercle unité montre qu'il s'agit d'une période:
\begin{displaymath}
\pi=\int_{x^2+y^2 \leq 1} dxdy.
\end{displaymath}
Comme ce nombre est transcendant d'après un théorème de Lindemann \cite{Lin-pi}, l'ensemble des périodes est strictement inclus entre celui des nombres algébriques et celui de tous les nombres complexes. On aurait pu choisir maintes autres formules intégrales, par exemple
\begin{displaymath}
\pi=\int_{-\infty}^\infty \frac{dx}{1+x^2}=\int_0^\infty \frac{2dx}{1+x^2},
\end{displaymath}
ce qui met en évidence le fait qu'il existe en général beaucoup de façons différentes de représenter un même nombre comme période. Selon la conjecture de Kontsevich-Zagier (section~\ref{sec:Kontsevich-Zagier}), on devrait pouvoir naviguer à travers toutes ces représentations avec pour seule boussole les règles élémentaires du calcul différentiel, à savoir l'additivité des intégrales, le changement de variables et la formule de Stokes.
\end{exemple}

\begin{remarque}[périodes \og non effectives\fg]\label{noneffectif} Comme les périodes sont stables par produit, toutes les puissances positives de $\pi$ sont des périodes. En revanche, on conjecture que le nombre $\sfrac{1}{\pi}$ n'en est pas une. Pour des raisons qui deviendront plus claires dans l'approche cohomologique des périodes, il peut être convenable de considérer le sous-anneau de $\CC$ formé des nombres complexes de la forme $\sfrac{z}{\pi^n}$ où $z$ est une période et $n \geq 0$ un entier. Le mot \og période\fg est parfois aussi employé pour désigner cette classe \emph{a priori} plus large de nombres; quand on veut mettre l'accent sur cette distinction, on appelle périodes \emph{effectives} les nombres de la définition \ref{defn:periodeKZ}.
\end{remarque}

\begin{exemple}[logarithmes]\label{exmp:logarithme} Soit $q>1$ un nombre rationnel. La représentation intégrale
\begin{displaymath}
\log(q)=\int_1^q \frac{dx}{x}
\end{displaymath}
montre que le logarithme réel de $q$ est une période. Il en va de même pour toutes les déterminations du logarithme d'un nombre algébrique non nul~$\alpha$, c'est\nobreakdash-à\nobreakdash-dire pour n'importe quel nombre complexe~$\beta$ satis\-faisant à \hbox{$\exp(\beta)=\alpha$.} D'après le théorème d'Hermite\nobreakdash-Lindemann (exemple~\ref{exmp:LW}), un tel $\beta$ non nul est transcendant. Ce résultat a été renforcé de façon spectaculaire par Baker en~1967, lorsqu'il a démontré qu'une combinaison linéaire à coefficients algébriques de logarithmes de nombres algébriques non nuls est transcendante pourvu qu'elle ne soit pas nulle, voir par exemple \cite[Ch.\,2]{baker}.
\end{exemple}

\begin{exemple}[valeurs zêta]\label{exmp:zeta(2)} Les valeurs spéciales
\begin{equation}\label{eqn:zetaasseries}
\zeta(n)=\sum_{k=1}^\infty \frac{1}{k^n}
\end{equation}
de la fonction zêta de Riemann aux entiers $n \geq 2$ sont également des périodes. En effet, ces nombres admettent la représentation intégrale
\begin{displaymath}
\zeta(n)=\int_{[0, 1]^n} \frac{dx_1\cdots dx_n}{1-x_1 x_2 \cdots x_n},
\end{displaymath}
comme on le voit en développant l'intégrande comme série géométrique, puis en intégrant terme à terme $n$ fois.

Pour $n$ pair, on sait depuis Euler (1734) que $\zeta(n)$ est un multiple rationnel de $\pi^n$; en voici les premières valeurs:
\begin{displaymath}
\zeta(2)=\frac{\pi^2}{6}, \quad \zeta(4)=\frac{\pi^4}{90}, \quad \zeta(6)=\frac{\pi^6}{945}, \quad \text{etc.}
\end{displaymath}
Il s'ensuit que ces nombres sont transcendants. On conjecture que c'est aussi le cas pour les valeurs de la fonction zêta aux entiers impairs et même que, contrairement à ce qui précède, celles-ci ne sont pas algébriquement reliées au nombre $\pi$. Très peu de résultats sont connus dans cette direction, hormis l'irrationalité de $\zeta(3)$, démontrée par Apéry \cite{apery} en 1978, et l'existence d'une infinité de nombres irrationnels parmi~$\zeta(5), \zeta(7), \dots$, obtenue par Ball et Rivoal \cite{ball-rivoal} en 2001 et qui a déjà fait l'objet de l'exposé \cite{colmez} aux Journées X-UPS.

Plus généralement, étant donnés un entier $\ell \geq 1$ et des entiers~$n_1, \ldots, n_\ell \geq 1$ avec $n_1 \geq 2$ (cette dernière condition assurant la convergence de la série ci-dessous), la \emph{valeur zêta multiple}
\begin{equation}\label{eqn:MZVasseries}
\zeta(n_1, \ldots, n_\ell)=\sum_{k_1>k_2>\dots>k_\ell \geq 1} \frac{1}{k_1^{n_1}\dots k_\ell^{n_\ell}}
\end{equation}
est une période. En effet, si l'on considère les polynômes
\begin{displaymath}
P_0(x)=x, \qquad P_ 1(x)=1-x
\end{displaymath}
et on associe à un $n$-uplet $(a_1, \ldots, a_n) \in \{0, 1\}^n$ l'intégrale
\begin{equation}\label{eqn:intrepsMZV}
I(a_1, \ldots, a_n)=\int_{1 \geq x_1 \geq \cdots \geq x_n \geq 0} \frac{dx_1\cdots dx_n}{P_{a_1}(x_1)\cdots P_{a_n}(x_n)},
\end{equation}
qui converge si et seulement si $a_1=0$ et $a_n=1$, alors la valeur zêta multiple \eqref{eqn:MZVasseries} est égale à
\begin{displaymath}
\zeta(n_1, \ldots, n_\ell)=I(\underbrace{0, \ldots, 0}_{n_1-1}, 1, \underbrace{0, \ldots, 0}_{n_2-1}, \ldots, \underbrace{0, \ldots, 0}_{n_\ell-1}, 1).
\end{displaymath}

Il est conjecturé que la dimension $d_k$ du $\QQ$-espace vectoriel engendré par les valeurs zêta multiples $\zeta(n_1, \ldots, n_\ell)$ avec $n_1+\cdots+n_\ell=k$ est donnée par la série génératrice
\[
\sum_{k=0}^\infty d_k t^k=\frac{1}{1-t^2-t^3},
\]
avec la convention $d_0=1$ et $d_1=0$. L'interprétation de ces nombres comme périodes et les progrès en théorie des motifs ont permis à \hbox{Deligne}-Goncharov \cite{deligne-goncharov} et à Terasoma \cite{terasoma} de prouver que le côté droit de cette égalité conjecturale fournit une borne supérieure pour les dimensions, puis à Brown \cite{brown} que toute valeur zêta multiple peut s'écrire comme une combinaison $\QQ$-linéaire de celles où les $n_i$ ne prennent que les valeurs $2$ et $3$. Je renvoie à \cite{mzvbook} ou à l'exposé de Dupont dans ce volume \cite{dupont} pour beaucoup plus sur ces nombres.
\end{exemple}

\subsection{Plus de souplesse: intégrales de fonctions algébriques}\label{sec:algebrique} Dans la définition \ref{defn:periodeKZ}, il n'est question que d'intégrales de formes différentielles de degré maximal, c'est-à-dire égal au nombre de variables. L'ensemble des périodes reste inchangé si l'on admet également des intégrales de la forme
\begin{displaymath}
\int_S f(x_1, \ldots, x_n) dx_{i_1} \cdots dx_{i_d},
\end{displaymath}
où le domaine d'intégration est un sous-ensemble $\QQ$-semi-algébrique \emph{orienté}~$S \subset \RR^n$ de dimension~\hbox{$d$}, c'est-à-dire muni d'une orientation du sous-ensemble de $S$ dont les points admettent un voisinage ouvert~$U$ dans $\RR^n$ tel que $S \cap U$ soit une variété différentielle de dimension $d$. Une conséquence immédiate est que l'on peut remplacer, dans la définition de Kontsevich-Zagier, les coefficients rationnels par des coefficients algébriques et les fonctions rationnelles par des fonctions algébriques sans changer la classe des périodes. Par exemple, pour tout nombre rationnel $\lambda$ distinct de $0$ et $1$, l'intégrale elliptique
\begin{displaymath}
\int_0^1 \frac{dx}{\sqrt{x(x-1)(x-\lambda)}}
\end{displaymath}
n'est pas à première vue une période au sens de la définition \ref{defn:periodeKZ} car l'intégrande n'est pas un quotient de polynômes, mais en introduisant une nouvelle variable $y$ on peut la récrire comme
\begin{displaymath}
\int_{\substack{y^2=x(x-1)(x-\lambda) \\ 0 \leq x \leq 1}} \frac{dx}{y}
\end{displaymath}
(intégrale d'une forme de degré $1$ sur un sous-ensemble de $\RR^2$). D'après un théorème de Schneider~\cite{Schneider}, ces nombres sont transcendants. De même, pour voir que les valeurs de la fonction bêta
\begin{equation}\label{eqn:betafunction}
\mathrm{B}(a, b)=\int_0^1 t^{a-1}(1-t)^{b-1} dt \qquad (\mathrm{Re}(a)>0,\,\mathrm{Re}(b)>0)
\end{equation}
sont des périodes lorsque $a$ et $b$ sont rationnels, on pose $a=\sfrac{r}{d}$ et~$b=\sfrac{s}{d}$ et on récrit $\mathrm{B}\hspace{-.5mm}\left(a, b \right)$ comme l'intégrale
\begin{equation}\label{eqn:betaperiods}
d\int_{\substack{x^d+y^d=1 \\ 0 \leq x \leq 1}} x^{r-1} y^{s-d} dx
\end{equation} par le biais du changement de variables $x=t^{1/d}$ et $y=(1-t)^{1/d}$. 
Dans un travail ultérieur \cite{Schneider}, Schneider a aussi démontré la transcendance des valeurs bêta dès que $a$ et $b$ ne sont pas des entiers.

\begin{lemme} Les sous-ensembles $S \subset \RR^n$ définis par un nombre fini d'équations et d'inéquations polynomiales à coefficients algébriques sont $\QQ$-semi-algébriques.
\end{lemme}

\begin{proof} L'argument est une variante de celui de l'exemple~\ref{exmp:algebrique}. Il suffit de traiter le cas où $S$ est défini par des inéquations $P_i(x_1, \dots, x_n) \geq 0$, avec des polynômes $P_i$ à coefficients dans une extension finie $L$ de $\QQ$. Notons $d$ son degré. Par le théorème de l'élément primitif, il existe un nombre algébrique~$\alpha$ tel que~$1, \alpha, \dots, \alpha^{d-1}$ soit une base de $L$ sur $\QQ$. En écrivant les coefficients des $P_i$ dans cette base, on trouve des polynômes~\hbox{$\tilde{P}_i \in \QQ[x_1, \dots, x_n, t]$} satisfaisant à $\tilde{P}_i(x_1, \dots, x_n, \alpha)=P_i(x_1, \dots, x_n)$. Soit $f$ le polynôme minimal de~$\alpha$. Comme dans l'exemple \ref{exmp:algebrique}, on peut trouver un intervalle à bornes rationnelles $r<\alpha<s$ sur lequel $f$ n'a pas d'autre racine que $\alpha$. L'ensemble $S$ est alors l'image par la projection vers les~$n$ premières coordonnées de l'ensemble $\QQ$-semi-algébrique
\[
\{(x_1, \dots, x_n, t) \in \RR^{n+1} \,|\, f(t)=0, \,r<t<s,\, \tilde{P}_i(x_1,\dots, x_n, t) \geq 0 \}
\]
et il est donc lui-même $\QQ$-semi-algébrique par le théorème de Tarski et Seidenberg énoncé au début de cette section. 
\end{proof}

\subsection{Les périodes exponentielles}

Conjecturalement, le nombre $e$, la constante $\gamma$ d'Euler ou les valeurs spéciales de la fonction gamma ne sont pas des périodes: si ces nombres avaient une représentation intégrale comme il faut, on le saurait depuis longtemps! On peut pourtant les traiter sur un même pied que les périodes à condition d'élargir la définition comme suit.

\begin{definition}[Kontsevich-Zagier]\label{def:perexp} Une \textit{période exponentielle} est un nombre complexe dont la partie réelle et la partie imaginaire peuvent s'écrire comme des intégrales absolument convergentes de la forme
\begin{displaymath}
\int_S e^{f(x_1, \dots, x_n)} g(x_1, \dots, x_n) dx_1 \cdots dx_n,
\end{displaymath}
où $f$ et $g$ sont des fonctions rationnelles à coefficients rationnels en~$n$ variables et~\hbox{$S \subset \RR^n$} est un ensemble $\QQ$-semi-algébrique.
\end{definition}

En choisissant pour $f$ la fonction identiquement nulle, on voit que les périodes sont aussi des périodes exponentielles; pour les distinguer, on les appellera parfois \og usuelles\fg ou \og classiques\fg. Quitte à peut-être restreindre légèrement les domaines d'intégration permis dans la définition, on peut démontrer \cite{ominimal} que la partie réelle et la partie imaginaire d'une période exponentielle sont des différences de volumes de sous-ensembles de~$\RR^n$ définissables par des formules faisant intervenir la fonction exponentielle réelle et la restriction de la fonction sinus à l'intervalle~$[0, 1]$; c'est l'analogue du lemme \ref{lemme:difference}, à l'exception près que l'on ne s'attend pas à ce qu'un tel volume soit toujours une période exponentielle.

\begin{exemple}[exponentielles]\label{exmp:LW}
Si~$\alpha$ est un nombre algébrique non nul, la représentation intégrale
\begin{displaymath}
e^\alpha=\int_{x \leq \alpha} e^{x} dx
\end{displaymath}
montre que~$e^\alpha$ est une période exponentielle. Les résultats de transcendance sur ces nombres remontent à Hermite, qui a prouvé la transcendance de $e$ dans quatre notes aux \emph{Comptes Rendus de l'Académie des Sciences} \cite{hermite} publiées en 1873. En établissant la transcendance de~$\pi$ neuf ans plus tard, Lindemann a expliqué comment son argument permettait plus généralement de démontrer ce qui est connu aujourd'hui comme le \emph{théorème d'Hermite-Lindemann}: la transcendance de $e^{\alpha}$ quel que soit le nombre algébrique non nul $\alpha$ (compte tenu de l'identité $e^{i\pi}=-1$, la transcendance de $\pi$ en est un cas particulier). Il énonce aussi, sous une autre forme, ce qui est maintenant appelé le \emph{théorème de Lindemann-Weierstrass}: si~$\alpha_1, \dots, \alpha_n$ sont des nombres algébriques linéairement indépendants sur $\QQ$, alors leurs exponentielles~$e^{\alpha_1}, \dots, e^{\alpha_n}$ sont des nombres algébriquement indépendants. On s'attend même à ce que les nombres $e^{\alpha}$ soient transcendants sur le corps engendré par toutes les périodes usuelles.
\end{exemple}

\begin{exemple}[la racine carrée de $\pi$]
L'intégrale de Gauss
\begin{displaymath}
\int_{-\infty}^\infty e^{-x^2} dx=\sqrt{\pi}
\end{displaymath}
est une période exponentielle. Indiquons la preuve de cette identité. Comme l'intégrale est positive, il suffit de montrer que son carré vaut~$\pi$, ce qui revient à calculer
\begin{equation}\label{eqn:gauss}
\int_{-\infty}^\infty \int_{-\infty}^\infty e^{-x^2-y^2} dx dy
\end{equation} par le théorème de Fubini. 
Or, en passant en coordonnées polaires $(x, y)=(r\cos \theta, r\sin \theta)$, cette intégrale se récrit
\begin{displaymath}
\int_0^\infty \int_0^{2\pi} e^{-r^2} rdr d\theta=\pi \int_0^\infty \underbrace{2re^{-r^2}dr}_{d(-e^{-r^2})}=\pi.
\end{displaymath}
Cet exemple montre que des périodes \emph{a priori} exponentielles telles que \eqref{eqn:gauss} peuvent en fait être égales à des périodes classiques.
\end{exemple}

\begin{exemple}[valeurs gamma]\label{exmp:valeursgamma} Plus généralement, les valeurs de la fonction gamma en des nombres rationnels sont des périodes exponentielles. Cette fonction est définie, pour tout nombre complexe~$s$ de partie réelle strictement positive, par l'intégrale convergente
\[
\Gamma(s)=\int_0^\infty t^{s-1} e^{-t} dt.
\]
Elle est holomorphe sur le domaine $\{\mathrm{Re}(s)>0 \}$; en intégrant par parties, on obtient l'identité $\Gamma(s+1)=s\Gamma(s)$ qui permet de la prolonger en une fonction méromorphe sur tout le plan complexe, dont les pôles se trouvent aux entiers $s=0, -1, \dots$ On en déduit que gamma prend la valeur $\Gamma(n)=(n-1)!$ pour tout entier $n \geq 1$.

Les valeurs en des arguments $s \in \QQ \setminus \ZZ$ sont plus intéressantes. Pour les comprendre, on se ramène immédiatement au cas où $s$ est strictement positif, et même compris entre $0$ et $1$. Pour un tel $s=\sfrac{a}{n}$, avec $n \geq 2$ ne divisant pas $a$, la représentation intégrale
\begin{displaymath}
\Gamma(\sfrac{a}{n})=\int_0^\infty t^{\sfrac{a}{n}-1}e^{-t} dt=\int_0^\infty n e^{-x^n} x^{a-1} dx
\end{displaymath}
montre que $\Gamma(\sfrac{a}{n})$ est une période exponentielle. À nouveau, on ne croit pas que ces nombres soient des périodes au sens classique, mais contrairement à ce qui est attendu pour $e^\alpha$ (exemple \ref{exmp:LW}), leurs puissances $n$-ièmes le sont: grâce à la formule d'Euler
\begin{equation}\label{eqn:Eulerformula}
\Gamma(r)\Gamma(s)=\Gamma(r+s)\mathrm{B}\hspace{-.5mm}\left(r, s \right)
\end{equation}
reliant la fonction gamma et la fonction bêta, on peut les écrire comme le produit télescopique
\begin{displaymath}
\Gamma(\sfrac{a}{n})^n=(a-1)!\prod_{k=1}^{n-1} \mathrm{B}(\sfrac{a}{n}, \sfrac{ak}{n}),
\end{displaymath}
où chaque facteur est une période d'après la formule \eqref{eqn:betaperiods}.

On conjecture que ces valeurs de la fonction gamma sont toutes des nombres transcendants, mais ceci n'est connu que pour\footnote{Par contre, comme les valeurs spéciales de la fonction bêta sont transcendantes d'après le théorème de Schneider, l'identité \eqref{eqn:Eulerformula} entraîne qu'au moins l'un des nombres $\Gamma(a/n)$ et $\Gamma(2a/n)$ est transcendant.} \hbox{$n=2, 3, 4$} (section~\ref{sec:grothendieck-elliptic}). Lang et Rohrlich \cite[Ch.\,24]{motifs} ont conjecturé que, le dénominateur $n$ étant fixé, toutes les relations algébriques entre les~$\Gamma(\sfrac{a}{n})$ découlent des trois relations fonctionnelles 
\begin{equation}\label{eqn:functionalgamma}
\begin{aligned}
\Gamma(s+1)&=s\Gamma(s), \\
\Gamma(s)\Gamma(1-s)&=\frac{\pi}{\sin(\pi s)}, \\
\prod_{a=0}^{n-1} \Gamma(s+\sfrac{a}{n})&=\frac{(2\pi)^{\sfrac{(n-1)}{2}}}{n^{ns-\sfrac{1}{2}}}\,\Gamma(ns).
\end{aligned}
\end{equation}
Ces relations entraînent une borne supérieure pour le nombre d'éléments algébriquement indépendants dans le corps engendré par $2\pi i$ et les $\Gamma(\sfrac{a}{n})$: pour tout $n \geq 3$, il y en a au plus $1+\sfrac{\varphi(n)}{2}$, où $\varphi$ est la fonction indicatrice d'Euler. D'après la conjecture de Lang-Rohrlich, il y en aurait exactement autant.
\end{exemple}

\begin{exemple}[la constante $\gamma$ d'Euler] La constante $\gamma$ d'Euler est d'habitude définie comme la limite
\begin{displaymath}
\gamma=\lim_{n \to \infty}\Bigl(1+\frac{1}{2}+\ldots+\frac{1}{n}-\log(n) \Bigr)
\end{displaymath}
des écarts entre les sommes partielles de la série harmonique et le logarithme, mais ce point de vue n'est guère utile pour démontrer qu'il s'agit d'une période exponentielle, comme Belkale et Brosnan l'ont observé en premier \cite{belkale}. En revanche, $\gamma$ est aussi l'opposé de la dérivée de la fonction gamma en $s=1$, d'où la représentation
\begin{displaymath}
\gamma=-\Gamma'(1)=-\int_0^{+\infty} e^{-x}\log(x) dx.
\end{displaymath}
Le logarithme n'est pas une fonction algébrique, mais la substitution $\log(x)=\int_1^x \sfrac{dy}{y}$ et un changement de variables mènent à l'expression
\begin{equation}\label{eqn:gamma-per-exponentielle}
\gamma=\int_{0}^1\int_0^1 e^{-xy}dxdy-\int_1^\infty\int_1^\infty e^{-xy}dxdy,
\end{equation}
qui est clairement une période exponentielle! On conjecture que la constante d'Euler est un nombre transcendant, mais à l'heure actuelle on ne sait même pas s'il est irrationnel\footnote{Par contre, il résulte du \emph{théorème de Siegel-Shidlovskii}, tel qu'il est par exemple expliqué dans la contribution de Rivoal \cite{rivoal} à ce volume, que la première intégrale dans l'expression de $\gamma$ comme période exponentielle est un nombre transcendant: c'est la valeur en $z=1$ de la fonction $\sum_{n \geq 0} (-1)^n z^n/(n+1)(n+1)!$, qui est une fonction $E$ satisfaisant à une équation différentielle non singulière en~$z=1$.}.
\end{exemple}

\subsection{Et des nombres qui ne sont pas des périodes?}

Avant de nous lancer à la recherche de nombres qui ne sont pas des périodes, faisons un pas en arrière et examinons la question, par certains côtés analogue, de l'existence de nombres transcendants~\cite{walds-hist}. De~nos jours, la preuve la plus rapide repose sur la non\nobreakdash-dénombrabilité des nombres réels, et \emph{a fortiori} des nombres complexes, établie par Cantor \cite{cantor, cantor2} pour la première fois en 1874, puis en 1891 à l'aide de l'argument diagonal\footnote{Contrairement à une idée reçue, ces démonstrations de Cantor sont constructives : partant d'une énumération explicite des nombres algébriques, elles fournissent des algorithmes pour écrire des nombres transcendants en base binaire avec autant de précision que voulu \cite{Gray}. }. Comme les nombres algébriques forment un sous-ensemble dénombrable de celui des nombres complexes, non seulement il y a des nombres transcendants mais en fait la plupart des nombres complexes le sont. Cependant, d'un point de vue historique cet argument est postérieur de trente ans à la première démonstration de l'existence des nombres transcendants par Liouville, et même à celle d'Hermite de la transcendance du nombre $e$, publiée un an avant. En~1844, Liouville démontre que tout nombre réel algébrique irrationnel est mal approché par les rationnels, au sens suivant:

\begin{thm}[Liouville]\label{thm:Liouville} Soit $\alpha$ un nombre réel algébrique de degré~$d$. Il existe un nombre réel~$C(\alpha)>0$ tel que l'inégalité
\begin{equation}\label{eqn:liouville}
\Bigl|\alpha-\frac{p}{q}\Bigr|\geq \frac{C(\alpha)}{q^d}
\end{equation}
soit vraie pour tout rationnel $\sfrac{p}{q}$ distinct de $\alpha$, avec $p$ et $q$ des entiers premiers entre eux et~$q>0$.
\end{thm}

\begin{proof} Il n'y a rien à démontrer si $|\alpha-\sfrac{p}{q}|\geq 1 \geq \sfrac{1}{q^d}$. Supposons donc $|\alpha-\sfrac{p}{q}|< 1$ et soit $P \in \ZZ[X]$ un polynôme irréductible de degré $d$ annulant $\alpha$. Posant $M=\max_{x \in [\alpha-1, \alpha+1]} | P'(x)|$, le théorème des accroissements finis donne l'inégalité
\begin{displaymath}
\left|P(\sfrac{p}{q})\right| \leq M \left|\alpha-\sfrac{p}{q}\right|.
\end{displaymath}
Vu que $\sfrac{p}{q}$ est différent de $\alpha$ et que $P$ est irréductible, $q^d P(\sfrac{p}{q})$ est un entier non nul, d'où~$|q^d P(\sfrac{p}{q})| \geq 1$. L'inégalité~\hbox{$|\alpha-\sfrac{p}{q}| \geq \sfrac{1}{Mq^d}$} est donc vraie et l'énoncé vaut avec~$C(\alpha)=\min\{1, M^{-1}\}$.
\end{proof}

Pour construire un nombre transcendant, il suffit ainsi d'exhiber un nombre réel avec des approximations rationnelles violant l'inégalité~\eqref{eqn:liouville}, et Liouville donne l'exemple du nombre
\begin{equation}\label{eqn:nombreLiouville}
x=\sum_{k=1}^\infty 10^{-k!}=0,11000100000000000000000100\dots
\end{equation}
En effet, en posant $p_n=\sum_{k=1}^n 10^{n!-k!}$ et $q_n=10^{n!}$, l'inégalité
\[
0<x-\frac{p_n}{q_n}=\sum_{k=n+1}^{\infty} 10^{-k!}<10^{-(n+1)!}<\frac{1}{q_n^n}
\]
est satisfaite pour tout $n$, ce qui serait en contradiction, pour $n$ assez grand, avec l'inégalité $|x-\sfrac{p_n}{q_n}|\geq \sfrac{C(\alpha)}{q^d}$ si $x$ était algébrique de degré $d$. Plus généralement, les \emph{nombres de Liouville} sont les réels $x$ tels qu'il existe une suite de nombres rationnels $(\sfrac{p_n}{q_n})$ satisfaisant~à
\begin{displaymath}
0<\Bigl|x-\frac{p_n}{q_n}\Bigr|<\frac{1}{q_n^n}.
\end{displaymath}
Par exemple, le nombre $\sum_{k \geq 0} \varepsilon_n 2^{-n!}$ est de Liouville quelle que soit la suite~$(\varepsilon_n)$ de $0$ et de $1$ contenant une infinité de $1$. Il ne coûte rien de se poser la question suivante, dont la réponse devrait très vraisemblablement être négative\footnote{Pour le même prix, on pourrait conjecturer que la \emph{mesure d'irrationalité} (c'est\nobreakdash-à\nobreakdash-dire l'infimum $\mu(x)$ de l'ensemble des nombres réels $\mu$ pour lesquels il existe $C>0$ tel que l'égalité $\left|x-\sfrac{p}{q}\right|\geq \sfrac{C}{q^\mu}$ soit vraie pour tout $\sfrac{p}{q}$ distinct de~$x$) de toute période exponentielle réelle $x$ vaut $2$. Les nombres de Liouville sont les nombres réels à mesure d'irrationalité infinie.}:

\begin{question}\label{question:Liouville} Est-ce qu'il y a des nombres de Liouville parmi les périodes ou, plus généralement, parmi les périodes exponentielles?
\end{question}

Comme les nombres rationnels sont denses dans $\RR$, tout réel positif~$\alpha$ peut s'écrire comme une limite
\[
\alpha=\lim_{n \to +\infty} \frac{a_n}{b_n+1},
\]
où $(a_n)$ et $(b_n)$ sont des suites d'entiers naturels (mettre $b_n+1$ au déno\-minateur est juste un moyen de garantir sa non-annulation). Une classe $\mathcal{E}$ de fonctions de $\NN$ dans $\NN$ étant fixée, on dit que~$\alpha$ est \emph{calculable par rapport à $\mathcal{E}$} si une telle limite peut être réalisée par des fonctions dans la classe $\mathcal{E}$, en ce sens précis qu'il existe des fonctions~$a, b, c \in \mathcal{E}$ telles que, pour tout entier~$m$, l'inégalité
\begin{displaymath}
\left| \alpha-\frac{a(n)}{b(n)+1} \right|<\frac{1}{m}
\end{displaymath}
soit satisfaite dès que $n \geq c(m)$. Pour construire des nombres qui ne sont pas des périodes, Yoshinaga \cite{yoshinaga} a eu l'idée de démontrer que toutes les périodes réelles sont calculables par rapport à la classe des \emph{fonctions élémentaires}. Il s'agit des fonctions d'une variable dans la plus petite classe de fonctions $\{f\colon \NN^r \to \NN \,|\,r\geq 1\}$ contenant les fonctions constantes, la fonction successeur $x \mapsto x+1$, la soustraction modifiée $(x, y) \mapsto \max(x-y, 0)$ et les projections $(x_1, \dots, x_r) \mapsto x_i$, et qui soit stable par les opérations de composition et de somme et produit bornés: les fonctions qui à $(n, x_1, \dots, x_r) \in \NN^{r+1}$ associent
\[
\sum_{i=0}^n f(i, x_1,\dots,x_r) \quad \text{et}\quad \prod_{i=0}^n f(i, x_1,\dots,x_r)
\]
sont élémentaires si $f$ l'est. Tent et Ziegler \cite{TZ} ont ensuite démontré que l'on pouvait se passer du produit borné; on appelle fonctions élémentaires \emph{inférieures} celles dans cette classe restreinte.

\begin{thm}[Yoshinaga, Tent-Ziegler] La partie réelle et la partie imaginaire d'une période sont des nombres calculables par rapport à la classe des fonctions élémentaires inférieures.
\end{thm}

À partir d'une énumération explicite de l'ensemble des fonctions élémentaires, Yoshinaga démontre que le~nombre
\[
0,388832221773824641256243009581\dots
\]
n'est pas une période par une variante de l'argument diagonal de Cantor. En revanche, le nombre de Liouville \eqref{eqn:nombreLiouville} est calculable\footnote{Le seul résultat dans la direction de la question \ref{question:Liouville} que je connaisse est un théorème récent de Lairez et Sertöz, qui montrent que le nombre \hbox{$\sum_{n \geq 0} (2\uparrow\uparrow 3n)^{-1}$} n'est pas le quotient de deux périodes sur une surface lisse dans $\PP^3$ définie par un polynôme homogène de degré $4$ à coefficients rationnels \cite{lairez}. La notation $2\uparrow\uparrow 3n$ désigne ici une tour de $3n$ puissances de $2$.}. 

\section{La conjecture de Kontsevich-Zagier}\label{sec:Kontsevich-Zagier}

La même période peut admettre, on l'a vu, de nombreuses repré\-sentations intégrales. Après avoir formalisé la notion de période, Kontsevich et Zagier énoncent la conjecture qu'il est toujours possible de passer d'une représentation à une autre par une combinaison des trois relations élémentaires suivantes:
\begin{enumerate}
\item[(i)] \emph{Additivité:} pour $f$ et $g$ des fonctions rationnelles et $S$ et $T$ des domaines $\QQ$\nobreakdash-semi-algébriques disjoints, on a les relations
\begin{displaymath}
\int_S (f+g)=\int_S f +\int_S g, \qquad \int_{S \cup T} f=\int_{S} f +\int_{T} f.
\end{displaymath}

\item[(ii)] \emph{Changement de variables:} pour tout changement de variables algébrique $\varphi$, on a la relation
\begin{align*}
\int_S f(x_1, \ldots, x_n)&dx_1\dots dx_n \\
&=\int_{\varphi^{-1}(S)} f(\varphi(y_1), \ldots, \varphi(y_n)) |J_\varphi|dy_1 \dots dy_n,
\end{align*}
où $J_\varphi$ désigne le jacobien de $\varphi$, c'est-à-dire le déterminant de la matrice des dérivées partielles de cette fonction.

\item[(iii)] \emph{Formule de Stokes:} pour toute sous-variété différentielle à bord orientée $S$ de $\RR^n$ et pour toute forme différentielle~$\omega$ à coefficients des fonctions rationnelles sur $S$, on a la relation
\begin{displaymath}
\int_S d\omega=\int_{\partial S} \omega, 
\end{displaymath} où $\partial S$ désigne le bord de $S$, muni de l'orientation induite par celle de~$S$, et $d\omega$ la dérivée extérieure de la forme $\omega$. 
\end{enumerate}

\vspace{.1cm}

\begin{conjecture}[Kontsevich-Zagier]\label{conj:KZ} Toutes les relations linéaires entre périodes découlent de l'additivité, du changement de variables et de la formule de Stokes.
\end{conjecture}

Bien que l'énoncé ci-dessus ne concerne \emph{a priori} que les relations \textit{linéaires}, il s'agit en fait d'une conjecture sur toutes les relations \textit{algé\-bri\-ques}: comme les périodes forment un sous-anneau de $\CC$ qui contient celui des nombres algébriques, toute relation polynomiale à coefficients algébriques donne lieu à une relation linéaire.

\begin{exemple} Revenons à l'exemple \ref{exmp:pi} et voyons comment les représentations intégrales de la période $\pi$ qui s'y trouvent sont reliées entre elles. Appliquée à la sous-variété $S=\{x^2+y^2 \leq 1\}$ de $\RR^2$ et à la forme différentielle $\omega=xdy$, la formule de Stokes donne
\[
\pi=\int_{x^2+y^2 \leq 1} dxdy=\int_{x^2+y^2=1} xdy=2\int_{-1}^1 \sqrt{1-x^2} dx,    
\] où la dernière égalité découle de la relation d'additivité en écrivant le cercle~$\partial S$, avec son orientation anti-horaire induite par celle de~$S$, comme la réunion des demi\nobreakdash-cercles paramétrés par $x=\sqrt{1-y^2}$ et~$x=-\sqrt{1-y^2}$ pour $-1 \leq y \leq 1$. Par additivité et changement de variables, l'intégrale ainsi obtenue est égale à 
\begin{align}
&\int_{-1}^1\sqrt{1-x^2} dx+\int_{-1}^0 \sqrt{1-x^2} dx+\int_{0}^1 \sqrt{1-x^2} dx \\
&\quad=\int_{-1}^1 \frac{1-t^2}{\sqrt{1-t^2}} dt+\int_{-1}^0 \frac{t^2}{\sqrt{1-t^2}} dt +\int_{0}^1 \frac{t^2}{\sqrt{1-t^2}} dt=\int_{-1}^1 \frac{dt}{\sqrt{1-t^2}}
\end{align} (on a posé $x=-\sqrt{1-t^2}$ dans la deuxième intégrale et $x=\sqrt{1-t^2}$ dans la troisième). De là on obtient enfin
\begin{align}
\int_{-1}^1 \frac{dt}{\sqrt{1-t^2}}&=\int_{-1}^0 \frac{dt}{\sqrt{1-t^2}}+\int_{0}^1 \frac{dt}{\sqrt{1-t^2}} \\
&=\int_{-\infty}^0 \frac{du}{1+u^2}+\int_0^{\infty} \frac{du}{1+u^2}  \\
&=\int_{-\infty}^\infty \frac{du}{1+u^2},
\end{align} par la relation d'additivité et les changements de variables 
\[
t=-\frac{1}{\sqrt{1-u^2}} \quad\text{et}\quad t=\frac{1}{\sqrt{1-u^2}}
\] sur la première et la deuxième intégrale respectivement. 
\end{exemple}

\begin{exemple}\label{exmp:KZ-log} Soient $a, b >1$ des nombres rationnels. Au vu de l'exemple \ref{exmp:logarithme}, l'identité $\log(ab)=\log(a)+\log(b)$ est une relation entre périodes qui devrait, selon la conjecture de Kontsevich-Zagier, découler des règles ci-dessus. Voici comment:\enlargethispage{-3\baselineskip}
\begin{align}
\log(ab)&=\int_1^{ab} \frac{dx}{x} \\
&=\int_1^a \frac{dx}{x}+\int_a^{ab} \frac{dx}{x} \qquad \mbox{(additivité)} \\
&=\int_1^a \frac{dx}{x}+\int_1^b \frac{dy}{y} \qquad \mbox{(changement de variables $x=ay$)} \\
&=\log(a)+\log(b).
\end{align}
\end{exemple}

\begin{exemple} Soient $m, n \geq 2$ des entiers. La représentation comme série \eqref{eqn:zetaasseries} des valeurs spéciales de la fonction zêta de Riemann donne lieu à l'identité de périodes
\begin{equation}\label{eqn:identityMZV}
\zeta(m)\zeta(n)=\zeta(m, n)+\zeta(n, m)+\zeta(n+m),
\end{equation}
comme on le voit en décomposant le domaine de sommation dans le produit $\zeta(m)\zeta(n)=\sum_{k_1, k_2 \geq 1} k_1^{-m} k_2^{-n}$ selon que $k_1$ est supérieur, inférieur ou égal à $k_2$. Il n'est pas \emph{a priori} évident que cette relation puisse se déduire des règles ci-dessus, comme le prédit la conjecture de Kontsevich-Zagier, car elle est obtenue à partir d'une représentation comme série plutôt que comme intégrale.

Néanmoins, le changement de variables $x_i=t_1\cdots t_{i}$ transforme le domaine d'intégration dans la formule \eqref{eqn:intrepsMZV} en l'hypercube $[0, 1]^{n+m}$ et fournit les représentations alternatives
\begin{align}
\zeta(m)\zeta(n)&=\int_{[0, 1]^{n+m}} \frac{dt_1\cdots dt_{n+m}}{(1-t_1\cdots t_m)(1-t_{m+1}\cdots t_{n+m})} \\
\zeta(m, n)&=\int_{[0, 1]^{n+m}} \frac{t_1\cdots t_m\,dt_1\cdots dt_{n+m}}{(1-t_1\cdots t_m)(1-t_1\cdots t_{n+m})} \\
\zeta(n, m)&=\int_{[0, 1]^{n+m}} \frac{t_{m+1}\cdots t_{m+n}\,dt_1\cdots dt_{n+m} }{(1-t_{m+1}\cdots t_{m+n})(1-t_1\cdots t_{n+m})} \\
\zeta(n+m)&=\int_{[0, 1]^{n+m}} \frac{dt_1\cdots dt_{n+m}}{1-t_1\cdots t_{n+m}},
\end{align}
à partir desquelles la relation \eqref{eqn:identityMZV} découle par additivité de l'identité
\begin{displaymath}
\frac{1}{(1-a)(1-b)}=\frac{a}{(1-a)(1-ab)}+\frac{b}{(1-b)(1-ab)}+\frac{1}{1-ab}
\end{displaymath}
appliquée à $a=t_1\cdots t_m$ et $b=t_{m+1}\cdots t_{m+n}$.
\end{exemple}

\subsection{Une version précisée} L'énoncé de la conjecture \ref{conj:KZ} est certes un peu flou, surtout en ce qui concerne la manière d'obtenir une relation entre périodes à partir des trois règles du calcul intégral. Le but de cette section est de donner une version beaucoup plus précise de la conjecture suivant les travaux d'Ayoub~\cite{AyoubEMS, Ayoub, AyoubTohoku}. Pour chaque entier $n \geq 0$, notons
\begin{displaymath}
\overline{\DD}^n=\{(z_1, \ldots, z_n) \in \CC^n \mid |z_i| \leq 1 \text{ pour tout } i=1, \dots, n\}
\end{displaymath}
le polydisque fermé dans $\CC^n$ et $\mathcal{O}(\overline{\DD}^n)$ l'anneau des séries entières
\begin{displaymath}
f=\sum_{j_1, \dots, j_n \geq 0} a_{j_1, \dots, j_n} z_1^{j_1}\cdots z_n^{j_n}
\end{displaymath}
qui ont polyrayon de convergence $> 1$, c'est\nobreakdash-à\nobreakdash-dire qui convergent partout sur un voisinage ouvert de $\overline{\DD}^n$. Soit $\mathcal{O}_{\textup{$\QQ$-alg}}(\overline{\DD}^n)$ le sous-anneau de~$\mathcal{O}(\overline{\DD}^n)$ des séries entières qui sont \textit{algébriques} sur~$\QQ(z_1, \ldots, z_n)$, c'est-à-dire qui satisfont à une équation polynomiale dont les coefficients sont des fonctions rationnelles à coefficients rationnels. En regardant une fonction de $n$ variables comme fonction de $n+1$ variables, on obtient des inclusions
\begin{equation}\label{eqn:inclusions}
\mathcal{O}_{\textup{$\QQ$-alg}}(\overline{\DD}^n) \subset \mathcal{O}_{\textup{$\QQ$-alg}}(\overline{\DD}^{n+1}). 
\end{equation}
On désigne par $\mathcal{O}_{\textup{$\QQ$-alg}}(\overline{\DD}^\infty)$ la réunion de tous les $\mathcal{O}_{\textup{$\QQ$-alg}}(\overline{\DD}^n)$. À un élément $f \in \mathcal{O}_{\textup{$\QQ$-alg}}(\overline{\DD}^n)$ on peut associer son intégrale
\begin{align*}
\int_{[0, 1]^n} f(z_1, \dots, z_n)dz_1\dots dz_n,
\end{align*}
qui est une période au sens de Kontsevich-Zagier (section \ref{sec:algebrique}). Comme la valeur de cette intégrale ne change pas si l'on regarde $f$ comme une fonction de~$n+1$ variables et on intègre sur $[0, 1]^{n+1}$, on dispose d'une application~$\QQ$\nobreakdash-linéaire
\begin{equation}\label{eqn:integration-infinie}
\int_{[0, 1]^\infty}\colon \mathcal{O}_{\textup{$\QQ$-alg}}(\overline{\DD}^\infty) \longrightarrow \CC.
\end{equation}
On peut démontrer (mais ce n'est pas facile!) que l'image de cette application coïncide avec l'anneau des périodes.

\begin{conjecture}\label{conj:KZ-Ayoub} Le noyau de l'application \eqref{eqn:integration-infinie} est le $\QQ$-sous-espace vectoriel de $\mathcal{O}_{\textup{$\QQ$-alg}}(\overline{\DD}^\infty)$ engendré par les éléments de la forme
\begin{equation}\label{eqn:conjKZgenerateurs}
\frac{\partial g}{\partial z_i}-g|_{z_i=1}+g|_{z_i=0},
\end{equation}
où $g$ est une fonction dans $\mathcal{O}_{\textup{$\QQ$-alg}}(\overline{\DD}^\infty)$ et $i \geq 1$ est un entier.
\end{conjecture}

\begin{remarque} Comme l'explique Ayoub \cite[Rem.\,1.2]{Ayoub}, si à la place des fonctions algébriques on considère toutes les fonctions holomorphes, on voit sans peine que le noyau de l'application
\begin{displaymath}
\int_{[0, 1]^n} \colon \mathcal{O}(\overline{\DD}^n) \longrightarrow \CC
\end{displaymath}
est le $\CC$-espace vectoriel engendré par les $\sfrac{\partial g}{\partial z_i}-g|_{z_i=1}+g|_{z_i=0}$. Raisonnons par récurrence sur $n$. Le cas $n=0$ est évident. Prenons donc~$n \geq 1$ et un élément $f=f(z_1, \ldots, z_n)$ du noyau. Considérons une primitive $g$ de $f$ par rapport à la dernière variable, c'est-à-dire une fonction~$g \in \mathcal{O}(\overline{\DD}^n)$ satisfaisant à $\sfrac{\partial g}{\partial z_n}=f$ (elle existe toujours, vu que le polyrayon de la série entière obtenue en intégrant terme à terme par rapport à $z_n$ est encore $>1$). Alors la fonction
\begin{displaymath}
h=f-\left( \sfrac{\partial g}{\partial z_n}-g|_{z_n=1}+g|_{z_n=0} \right)
\end{displaymath}
appartient au noyau et ne dépend que des variables $z_1, \ldots, z_{n-1}$. On peut donc la voir comme un élément dans $\mathcal{O}(\overline{\DD}^{n-1})$ et, par hypothèse de récurrence, l'écrire comme une combinaison linéaire d'éléments de la forme \eqref{eqn:conjKZgenerateurs}; il s'ensuit que $f$ l'est également. Ce raisonnement, cela va sans dire, est voué à l'échec dans le cas algébrique car une fonction algébrique n'admet pas en général de primitive algébrique; c'est déjà le cas pour la fonction $1/(z-2)=\sum_{n \geq 0} (z/2)^n $. 
\end{remarque}

\begin{remarque} Un des aspects les plus frappants de la conjecture~\ref{conj:KZ-Ayoub} est que les relations de changement de variables ont disparu de l'énoncé, qui ne garde finalement qu'une forme simple de la formule de Stokes. Ayoub explique comment les retrouver en une variable \cite[Rem.\,1.5]{Ayoub}. Étant donnés $f_1, f_2 \in \mathcal{O}_{\textup{$\QQ$-alg}}(\overline{\DD}^\infty)$, écrivons $f_1 \equiv f_2$ si $f_1-f_2$ appartient à l'espace vectoriel engendré par les éléments de la forme \eqref{eqn:conjKZgenerateurs}. Soit~$u$ un changement de variables satisfaisant à $u(0)=0$ et \hbox{$u(1)=1$.} En dérivant la fonction
\begin{displaymath}
g(z_1, z_2)=(u'(z_1)z_2+1-z_2)f(u(z_1)z_2+z_1(1-z_2))
\end{displaymath}
par rapport à la variable $z_2$, on trouve la relation
\begin{align}
u'(z_1)&f(u(z_1))-f(z_1) \equiv (u'(z_1)-1)f(u(z_1)z_2+z_1(1-z_2)) \\
&+(u'(z_1)z_2+1-z_2)f'(u(z_1)z_2+z_1(1-z_2))(u(z_1)-z_1),
\end{align}
dont le côté droit est égal à $\sfrac{\partial h}{\partial z_1}-h|_{z_1=1}+h|_{z_1=0}$ pour la fonction
\begin{displaymath}
h(z_1, z_2)=(u(z_1)-z_1)f(u(z_1)z_2+z_1(1-z_2)).
\end{displaymath}
La relation de changement de variables $\int_0^1 f(z)dz=\int_0^1 u'(z)f(u(z))dz$ découle donc de la formule de Stokes, mais pour le démontrer on a eu besoin d'introduire une nouvelle variable.
\end{remarque}

\subsection{Ce qui est connu} J'aurais du mal à imaginer une conjecture plus désespérément hors de portée que celle de Kontsevich et Zagier; il faut donc s'attendre à un paragraphe court... Quand il s'agit d'intégrales de fonctions algébriques d'une variable, on peut relier les périodes aux groupes algébriques et utiliser des techniques poussées en théorie de la transcendance, notamment le théorème du sous-groupe analytique de W\"ustholz \cite{baker-wustholz}, pour déterminer toutes les relations \textit{linéaires} entre ces nombres. Mais, contrairement à la conjecture générale, ceci ne suffit pas à contrôler les relations algébriques car les périodes en une variable ne sont \emph{pas} stables par multiplication.

\begin{thm}[Huber-W\"ustholz]\label{thm:HW} La conjecture de Kontsevich-Zagier vaut pour les périodes \textit{en une variable}\footnote{Dans le langage de la section \ref{sec:variante-relative}, on peut rendre \og scientifique\fg cette notion en considérant l'ensemble des nombres qui apparaissent comme des coefficients de l'accouplement de périodes entre la cohomologie de de~Rham $\rH^1_{\dR}(X, D)$ et l'homologie de Betti $\rH_1^\Betti(X, D)$ d'une courbe affine $X$ relatives à un ensemble fini de points $D \subset X$, tout étant défini par des équations à coefficients algébriques.}.
\end{thm}

Ce résultat est démontré dans \cite{huber-wustholz}. Les auteurs en tirent une carac\-térisation des périodes algébriques en une variable, que l'on peut estimer être une réponse complète à la question de Leibniz dans sa lettre à Huygens \cite{wust-leib}. Tous les termes intervenant dans l'énoncé ci-dessous seront expliqués en détail dans les sections suivantes.

\begin{thm}[Huber-W\"ustholz]\label{thm:reponseLeibniz} Soit $X$ une courbe affine lisse définie sur $\overline\QQ$. Pour toute $1$-forme différentielle $\omega$ et toute combinaison linéaire $\gamma$ de chemins $[0, 1] \to X(\CC)$ de classe $\mathcal{C}^\infty$ dont les points initiaux et finaux sont dans $X(\overline\QQ)$, la période $\int_\gamma \omega$ est algébrique si et seulement s'il existe une fonction $f$ et une $1$-forme différentielle~$\omega'$ sur $X$ satisfaisant aux égalités $\omega=df+\omega'$ et $\int_\gamma \omega'=0$.
\end{thm}

Si $\gamma$ n'a pas de bord, l'intégrale $\int_\gamma df$ est nulle par la formule de Stokes; la période $\int_\gamma\omega$ est donc ou bien nulle ou bien transcendante. Ce cas, qui comprend entre autre les théorèmes de Schneider sur la transcendance des intégrales elliptiques et des valeurs de la fonction bêta en des arguments rationnels, avait été démontré par Wüstholz à la fin des années 80 comme corollaire de son théorème du sous-groupe analytique \cite{WusICM}. La principale contribution de \cite{huber-wustholz} est de mener à bien les dévissages géométriques nécessaires pour pouvoir appliquer le théorème du sous-groupe analytique également quand $\gamma$ a un bord.

\begin{exemple} Examinons les implications de ce théorème pour les intégrales des fonctions \emph{rationnelles} en une variable, un cas pour lequel le théorème de Baker sur les combinaisons linéaires de logarithmes (exemple~\ref{exmp:logarithme}) est en fait suffisant \cite{vdP}. Soient $P, Q \in \QQ[x]$ des polynômes sans facteur commun et soit
\begin{displaymath}
Q(x)=q(x-\alpha_1)^{n_1}\cdots(x-\alpha_\ell)^{n_\ell}
\end{displaymath}
la factorisation de $Q$ dans $\CC[x]$. La fraction rationnelle~$\sfrac{P}{Q}$ admet la décomposition en éléments simples
\begin{equation}\label{eqn:elementssimples}
\frac{P}{Q}=T+\sum_{i=1}^\ell \Bigl[\frac{r_i}{x-\alpha_i}+\frac{r_{i, 2}}{(x-\alpha_i)^2}+\cdots+\frac{r_{i, n_i}}{(x-\alpha_i)^{n_i}} \Bigr]
\end{equation}
où $T$ est le quotient de la division euclidienne de $P$ par $Q$ et tous les nombres $r_i, r_{i, 2}, \dots$ sont algébriques (par exemple, $r_i$ est le résidu de~$\sfrac{P}{Q}$ au point $x=\alpha_i$). Comme $T$ et tous les termes faisant intervenir des pôles d'ordre $\geq 2$ sont la dérivée d'une fonction rationnelle, le théorème \ref{thm:reponseLeibniz} est dans ce cas l'énoncé:

\begin{coro} La période $\int_\gamma \sfrac{P(x)}{Q(x)}dx$ est algébrique si et seulement si l'égalité $\sum_{i=1}^{n} r_i\int_\gamma \spfrac{dx}{x-a_i}=0$ est satisfaite.
\end{coro}

En particulier, lorsque le degré de $P$ est strictement inférieur à celui de $Q$ et que les zéros de $Q$ sont tous distincts, seuls les termes~$\sfrac{r_i}{(x-\alpha_i)}$ apparaissent dans la décomposition \eqref{eqn:elementssimples}, et il s'ensuit que l'intégrale $\int_\gamma \sfrac{P(x)}{Q(x)}dx$ est ou bien zéro ou bien transcendante. Un exemple célèbre est la période
\begin{displaymath}
\int_0^1 \frac{dx}{x^3+1}=\frac{1}{3}\int_0^1 \Bigl(\frac{1}{x+1}-\frac{x-2}{x^2-x+1}\Bigr)dx=\frac{1}{3}\Bigl(\log 2+\frac{\pi}{\sqrt{3}}\Bigr),
\end{displaymath}
dont on ne savait même pas si elle était irrationnelle quand Siegel a publié son livre \cite{siegel} sur les nombres transcendants en 1949.
\end{exemple}

Dans une autre direction, Ayoub~\cite{Ayoub} a récemment démontré une variante \og géométrique\fg de la conjecture de Kontsevich-Zagier, sous la forme précisée du numéro précédent, où les périodes sont remplacées par des \emph{séries de périodes}. Pour l'énoncer, considérons l'anneau
\begin{displaymath}
\mathcal{O}^\dagger(\overline{\DD}^n)=\mathcal{O}(\overline{\DD}^n)\lpr t\rpr
\end{displaymath}
des séries de Laurent formelles
\[
F=\sum_{i \gg-\infty} f_i(z_1, \ldots, z_n) t^i
\]
dont les coefficients sont des fonctions \hbox{$f_i \in \mathcal{O}(\overline{\DD}^n)$,} ainsi que le sous-anneau~$\mathcal{O}^\dagger_{\textup{$\CC$-alg}}(\overline{\DD}^n)$ formé de celles qui sont algébriques sur le corps $\CC(z_1, \dots, z_n, t)$. Posons $\mathcal{O}^\dagger_{\textup{$\CC$-alg}}(\overline{\DD}^\infty)=\bigcup_{n \geq 0} \mathcal{O}^\dagger_{\textup{$\CC$-alg}}(\overline{\DD}^n)$ et considérons l'application~$\CC$\nobreakdash-linéaire
\begin{equation}\label{eqn:applicationAyoub}
\begin{aligned}
\mathcal{O}^\dagger_{\textup{$\CC$-alg}}(\overline{\DD}^\infty) &\longrightarrow \CC\lpr t\rpr \\
\sum_{i \gg -\infty} f_i t^i &\longmapsto \sum_{i \gg -\infty} \biggl(\int_{[0, 1]^\infty} f_i \biggr) t^i.
\end{aligned}
\end{equation}
Si, dans ce qui précède, on remplaçait les nombres complexes par les nombres rationnels, les coefficients $\int_{[0, 1]^\infty} f_i$ seraient des périodes.

\begin{thm}[Ayoub]\label{thm:Ayoub} Le noyau de \eqref{eqn:applicationAyoub} est le $\CC$-sous-espace vectoriel de $\mathcal{O}^\dagger_{\textup{$\CC$-alg}}(\overline{\DD}^\infty)$ engendré par les éléments suivants:
\begin{enumerate}
\item tous les éléments de la forme
\begin{displaymath}
\frac{\partial F}{\partial z_i}-F|_{z_i=1}+F|_{z_i=0},
\end{displaymath}
où $F$ est une fonction dans $\mathcal{O}^\dagger_{\textup{$\CC$-alg}}(\overline{\DD}^\infty)$ et $i \geq 1$ est un entier;
\item tous les éléments de la forme
\begin{displaymath}
\biggl(g-\int_{[0, 1]^\infty} g \biggr) \cdot F,
\end{displaymath}
où $g, F \in \mathcal{O}^\dagger_{\textup{alg}}(\overline{\DD}^\infty)$ sont des fonctions satisfaisant à
\begin{displaymath}
\frac{\partial g}{\partial t}=0, \qquad \frac{\partial g}{\partial z_i} \cdot \frac{\partial F}{\partial z_i}=0 \quad (i \geq 1)
\end{displaymath}
(dit autrement, $g$ ne dépend pas de l'indéterminée $t$ et $F$ et $g$ ne dépendent pas simultanément d'une même variable $z_i$).
\end{enumerate}
\end{thm}

\subsection{Vers une théorie de Galois pour les périodes}\label{sec:Galoistheoryheuristics}

Reprenons l'exemple \ref{exmp:algebrique}, où l'on a vu qu'un nombre algébrique est une période à l'aide de son polynôme minimal. Les autres racines de ce polynôme sont appelés les \emph{conjugués} du nombre algébrique; d'après Galois, on peut leur associer un groupe de permutations qui constitue depuis les débuts de la théorie un des outils les plus puissants pour les étudier \cite{Yves4}. Voici une manière d'y penser. 

Soit~$P$ un polynôme de degré $n 
\geq 1$ à coefficients rationnels. Si $P$ est irréductible, l'équation $P(x)=0$ admet exactement $n$ solutions complexes. On voudrait les numéroter $\alpha_1, \ldots, \alpha_n$, mais comme il n'y a pas de moyen naturel de distinguer les unes des autres, cela requiert un choix ; l'ambiguïté dans ce choix est encodée dans le groupe symétrique $\mathfrak{S}_n$ permutant les $\alpha_j$. Certaines relations algébriques entre les racines (par exemple, celles exprimant les coefficients de $P$ comme des fonctions symétriques en les $\alpha_j$) sont préservées par toutes les permutations, mais il se peut que d'autres relations ne le soient pas. Par exemple, les racines du polynôme
\begin{equation}\label{eqn:exmpcycl}
x^4+x^3+x^2+x+1=\frac{x^5-1}{x-1} 
\end{equation} sont les racines primitives cinquièmes de l'unité et, si l'on choisit la numérotation $\alpha_j=\zeta^j$ avec $\zeta=\exp(2\pi i/5)$ pour $1 \leq j \leq 4$, les seules permutations qui respectent les relations $\alpha_j=\alpha^j_1$ sont celles dans le sous-groupe engendré par le $4$-cycle $(1\,2\,4\,3)$, qui envoie $\zeta$ sur $\zeta^2$. 

On est ainsi amené à définir le \emph{groupe de Galois} de $P$ comme le sous-groupe $G_P \subset \mathfrak{S}_n$ formé des permutations $g$ préservant toutes les relations algébriques entre les racines, c'est-à-dire satisfaisant~à 
\[
Q(g(\alpha_1), \ldots, g(\alpha_n))=0
\] quel que soit le polynôme~$Q \in \QQ[x_1, \ldots, x_n]$ avec $Q(\alpha_1, \ldots, \alpha_n)=0$. C'est ce qu'il faut pour que l'action de $G_P$ sur le sous-anneau 
\begin{equation}\label{eqn:sousanneaueng}
\QQ[\alpha_1, \dots, \alpha_n]=\{Q(\alpha_1, \dots, \alpha_n) \  | \ Q \in \QQ[x_1, \dots, x_n] \} \subset \CC
\end{equation} soit bien définie, c'est-à-dire indépendante de la manière de représenter ses éléments comme des polynômes en les $\alpha_j$. 

Que le groupe de Galois soit grand, par exemple $\mathfrak{S}_n$ en entier, veut alors dire qu'il y a peu de relations entre les racines, par exemple seulement celles \og évidentes\fg ; c'est le cas pour la plupart des polynômes $P$. Que le groupe de Galois soit petit, par exemple le groupe cyclique d'ordre $4$ par opposition à $\mathfrak{S}_4$ pour le polynôme \eqref{eqn:exmpcycl}, veut dire qu'il y a beaucoup de relations entre les racines. On peut néanmoins démontrer que $G_P$ contient toujours assez d'éléments pour permuter transitivement les racines et pour qu'une expression polynomiale en les $\alpha_j$ soit invariante par tout le groupe si et seulement s'il s'agit d'un nombre rationnel. Ces intuitions sont quantifiées par une égalité entre la dimension de \eqref{eqn:sousanneaueng} et l'ordre du groupe de Galois: 
\begin{equation}\label{eqn:dimension-ordre}
\dim_\QQ \QQ[\alpha_1, \dots, \alpha_n]=|G_P|. 
\end{equation} 

Deux questions, à première vue indépendantes, se posent lorsque l'on essaie d'étendre ces idées au-delà des nombres algébriques: 
\begin{itemize}
\item quels seraient les conjugués d'un nombre transcendant\footnote{Galois a lui-même envisagé ce genre des questions à la fin de sa lettre-testament: \og Mes principales méditations depuis quelque temps étaient dirigées sur l'application à l'analyse transcendante de la théorie de l'ambiguïté. Il s'agissait de voir a priori dans une relation entre quantités ou fonctions transcendantes quels échanges on pouvait faire, quelles quantités on pouvait substituer aux quantités données sans que la relation pût cesser d'avoir lieu. Cela fait reconnaître tout de suite l'impossibilité de beaucoup d'expressions que l'on pouvait chercher. \fg }? 
\item peut-on associer des groupes de Galois aux polynômes en plusieurs variables?  
\end{itemize} On verra que le rêve de Grothendieck d'une théorie des motifs fournit de nos jours une réponse inconditionnelle à la seconde question; supposant connue la conjecture de Kontsevich-Zagier, elle permet de donner un sens aux conjugués des périodes. 

Comme l'explique André \cite{Yves2}, il semblerait que l'on dispose d'heuristiques pour deviner les conjugués de certains nombres transcendants. Par exemple, bien que $\pi$ ne soit racine d'aucun polynôme non nul à coefficients rationnels, il l'est du \og polynôme\fg de degré infini
\begin{displaymath}
\prod_{n \in \ZZ \setminus \{0\}} \left(1-\frac{x}{n\pi} \right)=\frac{\sin x}{x}=1-\frac{x^2}{6}+\frac{x^4}{120}+\cdots \in \QQ\lcr x\rcr. 
\end{displaymath}
Cela suggère que les conjugués de $\pi$ devraient être ses multiples entiers non nuls, ou plutôt rationnels si l'on vise à une action transitive du groupe de Galois, qui serait donc le groupe multiplicatif~$\QQ^\times$. Comme dire que $\pi$ est transcendant revient à dire que l'espace vectoriel $\QQ[\pi]$ est de dimension infinie sur $\QQ$, le fait de trouver un groupe de Galois infini laisse la porte ouverte à une généralisation de~\eqref{eqn:dimension-ordre}. 

Cette intuition est pourtant trompeuse en général: d'après un théorème d'Hurwitz \cite{hurwitz}, tout nombre complexe est zéro d'une infinité non dénombrable de séries entières à coefficients rationnels, et il n'y a pas moyen d'en choisir une \og minimale\fg. Pire: on peut toujours trouver de telles séries qui ne s'annulent que sur le nombre en question et sur son conjugué complexe, et des séries qui s'annulent également sur n'importe quel autre nombre donné. 

Une meilleure heuristique, cette fois-ci pour les logarithmes, provient de la \emph{monodromie}, c'est-à-dire de l'ambiguïté dans le choix d'un chemin d'intégration dans la représentation $\log(q)=\int_1^q dx/x$. Si, à la place de la droite allant de $1$ à $q$, on choisit un chemin qui tourne autour de $0$, le résultat change par un multiple de $2\pi i$. Cela suggère que les conjugués de $\log(q)$ devraient au moins contenir $\log(q)+(2\pi i)\QQ$. 

Quant aux polynômes en plusieurs variables, on se heurte d'abord à la difficulté que les solutions dans $\CC^n$ des équations qu'ils définissent ne forment plus un ensemble fini, ni même discret, mais continu. Par exemple, l'équation $xy-1=0$ admet pour solutions tous les couples~$(x, y) \in \CC^2$ avec $x$ non nul et~$y$ son inverse; on a donc affaire au plan complexe épointé~$\CC^\times$ plongé dans $\CC^2$ comme une hyperbole, un des premiers exemples de \emph{variété algébrique}. Une théorie de Galois en plusieurs variables doit tenir compte de la géométrie de ces espaces.

Suivant Grothendieck, le rôle du groupe de permutations de l'ensemble fini de racines sera joué par des groupes de matrices agissant sur des $\QQ$-espaces vectoriels de dimension finie, à savoir les groupes d'\emph{homologie singulière} $\rH_p^{\Betti}(X)$ des variétés algébriques $X$. Ces espaces vectoriels sont des invariants de nature topologique, dont les éléments sont représentés par des applications continues $\sigma \colon \Delta^p \to X$ que l'on appelle des cycles. Par exemple, le premier groupe d'homologie singulière de $\CC^\times$ est le~$\QQ$\nobreakdash-espace vectoriel de dimension $1$ engendré par le lacet $\gamma$ qui tourne une fois autour de l'origine dans le sens anti-horaire. Dans ce cas, il y a aussi une façon algébrique de détecter le point manquant : le fait que la forme différentielle $dx/x$ n'admet pas de primitive algébrique, c'est-à-dire ne soit pas la dérivée d'un polynôme en~$x$ et~$1/x$. En général, les obstructions de cette nature donnent lieu à d'autres~$\QQ$\nobreakdash-espaces vectoriels de dimension finie: les groupes de \emph{cohomologie de de Rham}~$\rH^p_{\dR}(X)$, dont les éléments sont représentés par des formes différentielles algébriques~$\omega$. 

Grothendieck \cite{GrodR} a relié ces deux espaces en démontrant que l'intégration fournit un accouplement parfait 
\begin{equation}\label{eqn:accouplGroth}
\begin{aligned}
\rH^p_{\dR}(X) \times \rH_p^\Betti(X) &\longrightarrow \CC \\
([\omega], [\sigma]) &\longmapsto \int_{\sigma} \omega, 
\end{aligned} 
\end{equation} c'est-à-dire que $\rH^p_{\dR}(X)$ et $\rH_p^\Betti(X)$ ont même dimension et que, quelles que soient les bases $\{[\omega_i]\}$ et~$\{[\sigma_j]\}$ de ces espaces, la matrice 
\[
\biggl(\int_{\sigma_j} \omega_i \biggr)
\] est inversible. Comme on peut démontrer que tous les coefficients de ces matrices sont des périodes, il est coutume d'appeler \eqref{eqn:accouplGroth} l'\emph{accouplement de périodes}. Par exemple, pour $X=\CC^\times$, il s'agit de la matrice de taille $1$ ayant pour coefficient $\int_\gamma dx/x=2\pi i$. La donnée des espaces vectoriels $\rH^p_{\dR}(X)$ et $\rH_p^\Betti(X)$ et de leur accouplement de périodes est une première approximation de la notion de \emph{motif}. 

Réciproquement, le point de vue plus sophistiqué sur les périodes évoqué à plusieurs reprises consiste à réinterpréter le domaine d'intégration dans la définition \ref{defn:periodeKZ} comme la classe d'un cycle en homologie singulière sur une certaine variété algébrique; l'intégrande, comme la classe d'une forme différentielle en cohomologie de de~Rham sur cette même variété; et~le processus d'intégration, comme l'accouplement entre ces deux espaces. On verra que c'est toujours possible, à condition d'élargir un peu le cadre ci-dessus. Partant d'une seule période, on obtient alors d'autres périodes qu'il aurait été impossible de deviner en gardant le point de vue élémentaire: tous les coefficients de toutes les matrices de l'accouplement par rapport à tous les choix possibles de bases. C'est parmi celles-ci que se trouveront les conjugués que l'on souhaite définir pour la période de départ. 

Si l'ambiguïté dans le choix d'une numérotation des racines faisait naturellement surgir le groupe symétrique, celle dans le choix de bases suggère que l'on considère des groupes d'automorphismes linéaires, disons du $\QQ$\nobreakdash-espace vectoriel $\rH_p^\Betti(X)$. On voudrait faire agir un tel automorphisme $g$ sur les périodes par le biais de la formule
\begin{equation}\label{eqn:actionGalois}
g\cdot \int_\sigma \omega=\int_{g(\sigma)} \omega,  
\end{equation} mais pour que cette action soit bien définie il faudrait savoir que le nombre de droite ne dépend pas de la représentation intégrale de la période $\int_\sigma \omega$. D'après la conjecture de Kontsevich-Zagier, toutes les relations algébriques entre périodes devraient découler de l'additivité, du changement de variables et de la formule de Stokes. 

Ce sont toutes les trois des relations d'\og origine géométrique\fg. Par exemple, on peut penser à un changement de variables comme à un morphisme $\varphi \colon X \to Y$ entre variétés algébriques. Un tel morphisme induit, par composition et tiré en arrière, des applications linéaires 
\[
\varphi_\ast \colon \rH_p^\Betti(X) \to \rH_p^\Betti(Y) \quad \text{et}\quad \varphi^\ast \colon \rH_{\dR}^p(Y) \to \rH^p_{\dR}(X)
\]en homologie singulière et en cohomologie de de Rham, et la relation de changement de variables exprime leur compatibilité 
\[
\int_{\varphi_\ast(\sigma)} \omega=\int_\sigma \varphi^\ast(\omega) 
\] avec l'accouplement de périodes. On définit alors, simultanément pour toutes les variétés algébriques, des \textit{groupes de Galois motiviques}
\[
G_p^X \subset \GL(\rH_p^{\Betti}(X))
\] comme les sous-groupes des automorphismes linéaires compatibles aux contraintes imposées par l'additivité, le changement de variables et la formule de Stokes, par exemple satisfaisant à 
\[
\varphi_\ast \circ g=h \circ \varphi_\ast
\] pour tout morphisme $\varphi\colon X \to Y$ et pour tous $g \in G_p^X$ et $h \in G_p^Y$. Bien que nous ayons été guidés par la conjecture de Kontsevich-Zagier, cette définition du groupe de Galois motivique est inconditionnelle. 

Comme les groupes de Galois motiviques sont en général infinis, leur taille doit être interprétée comme une dimension. Celle-ci est la différence entre la dimension du groupe ambiant $\GL(\rH_p^{\Betti}(X))$ et le nombre d'équations nécessaires pour le définir. Qu'un groupe de Galois motivique soit de grande dimension suggère qu'il y a peu de relations entre les périodes; qu'il soit de petite dimension, qu'il y en a beaucoup. Pour rendre cette idée plus précise, fixons une variété~$X$ et un entier $p$ et considérons le sous-anneau  
\begin{equation}\label{eqn:anneauperiodes}
\QQ\biggr[\Bigl(\int_{\sigma_j} \omega_i \Bigr)_{1 \leq 1, j \leq n}\biggl] \subset \CC
\end{equation} engendré par les coefficients de la matrices de l'accouplement \eqref{eqn:accouplGroth}, où $n$ désigne la dimension de $\rH_p^{\Betti}(X)$ ; contrairement aux périodes elles-mêmes, ce sous-anneau ne dépend pas du choix des bases. Plutôt qu'à sa dimension sur~$\QQ$, en général infinie, on s'intéresse à son \emph{degré de transcendance}. Il s'agit d'un entier compris entre $0$ et $n^2$, qui est égal à~$r$ si l'on peut trouver~$r$ éléments algébriquement indépendants parmi les~$\int_{\sigma_j} \omega_i$ mais n'importe quels~$r+1$ éléments sont reliés par une équation algébrique à coefficients dans~$\QQ$. L'analogue de \eqref{eqn:dimension-ordre} est alors une conjecture dont l'origine remonte à une note de bas de page dans l'article fondateur~\cite{GrodR} de Grothendieck.  

\begin{conjecture}[de périodes de Grothendieck] Le degré de transcendance du sous-anneau \eqref{eqn:anneauperiodes} de $\CC$ coïncide avec la dimension du groupe de Galois motivique $G_p^X$. 
\end{conjecture}

Avec les définitions données, cette conjecture s'avère être équivalente à celle de Kontsevich-Zagier. La supposant vraie, l'action \eqref{eqn:actionGalois} du groupe de Galois sur les périodes est bien définie. À la différence des nombres algébriques, cette action n'est plus transitive en général. Les \emph{conjugués} d'une période~$\int_\sigma \omega$ sont alors les éléments de son orbite sous l'action du groupe de Galois motivique, c'est-à-dire tous les périodes de la forme $g\cdot\int_\sigma \omega$. On verra vers la fin de ces notes (exemples~\ref{exmp:Galoispi} et~\ref{eqn:Galoislogs}) que ce point de vue est compatible avec les premières tentatives de réponse pour les conjugués de $\pi$ et de $\log(q)$. 

\section{Homologie singulière}\label{sec:4}

Le but de cette section et les suivantes est de rendre précis le point de vue sur les périodes évoqué ci-dessus. Cette approche plus sophistiquée, qui est à l'origine de tous les \hbox{progrès} récents vers la conjecture de Kontsevich-Zagier, nous permettra notamment d'associer de manière rigoureuse à une période d'autres périodes parmi lesquelles se trouvent les analogues des conjugués d'un nombre algébrique.

\subsection{Définition et premiers exemples}\label{sec:homsing} Soit $p \geq 0$ un entier. Le \textit{$p$-simplexe} $\Delta^p$ est l'ensemble
\begin{displaymath}
\Delta^p=\{(x_0, \dots, x_p) \in \RR^{p+1} \mid x_i \geq 0, \, x_0+\dots+x_p=1 \}
\end{displaymath}
muni de la topologie de sous-espace de $\RR^{p+1}$. Comme le montre la figure \ref{fig:Delta}, le $0$\nobreakdash-simplexe $\Delta^0$ est réduit à un point, le $1$-simplexe $\Delta^1$ est homéomorphe à l'intervalle~$[0, 1]$ par l'application \hbox{$t \mapsto (1-t, t)$}, le 2\nobreakdash-simplexe $\Delta^2$ est un triangle dans le plan $x_0+x_1+x_2=1$, etc.

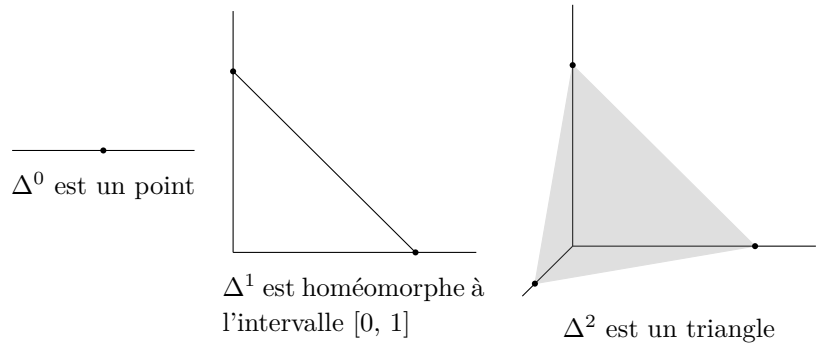
\begin{figure}[ht]
\centering
\begin{subfigure}{.25\textwidth}
\centering
\begin{tikzpicture}[scale=.8]
\def\laxis{1.5}
\def\ltriangle{3}
\def\r{.05};
\def\sqrtoftwo{1.4142}
\draw (-\laxis,0) -- (\laxis,0);
\fill (0,0) circle (\r);
\end{tikzpicture}
\caption*{$\Delta^0$ est un point}
\end{subfigure}\
\begin{subfigure}{.3\textwidth}
\centering
\begin{tikzpicture}[scale=.8]
\def\laxis{4}
\def\ltriangle{3}
\def\r{.05};
\def\sqrtoftwo{1.4142}
\draw (0,0) -- (\laxis,0);
\draw (0,0) -- (0,\laxis);
\draw (\ltriangle,0) -- (0,\ltriangle);
\fill (\ltriangle,0) circle (\r);
\fill (0,\ltriangle) circle (\r);
\end{tikzpicture}
\caption*{$\Delta^1$ est homéomorphe à l'intervalle [0, 1]}
\end{subfigure}\
\begin{subfigure}{.4\textwidth}
\centering
\begin{tikzpicture}[scale=.8]
\def\laxis{4}
\def\ltriangle{3}
\def\sqrtoftwo{1.4142}
\def\r{.05};
\draw (0,0) -- (\laxis,0);
\draw (0,0) -- (0,\laxis);
\draw (0,0) -- (-{\laxis*(2*\sqrtoftwo/2-1)/2}, -{\laxis*(2*\sqrtoftwo/2-1)/2});
\filldraw [opacity=.25,gray] (\ltriangle,0) -- (0,\ltriangle) --
(-{\ltriangle*(2*\sqrtoftwo/2-1)/2},-{\ltriangle*(2*\sqrtoftwo/2-1)/2}) -- cycle;
\fill (\ltriangle,0) circle (\r);
\fill (0,\ltriangle) circle (\r);
\fill (-{\ltriangle*(2*\sqrtoftwo/2-1)/2},-{\ltriangle*(2*\sqrtoftwo/2-1)/2}) circle (\r);
\end{tikzpicture}
\caption*{$\Delta^2$ est un triangle}
\end{subfigure}
\caption{Premiers exemples de $p$-simplexes}
\label{fig:Delta}
\end{figure}

Lorsque $p$ varie, les $p$-simplexes sont reliés par des applications \textit{face}~$\delta^i_p \colon \Delta^p \to \Delta^{p+1}$, définies pour $i=0,\dots, p+1$ par la formule
\begin{displaymath}
\delta^i_p(x_0, \dots, x_p)=(x_0, \dots, x_{i-1}, 0, x_{i}, \dots, x_p).
\end{displaymath}
Les faces encodent les différentes manières de plonger un $p$-simplexe dans le bord d'un $(p+1)$-simplexe. Par exemple, on peut voir le point~$\Delta^0$ comme l'une des deux extrémités du segment~$\Delta^1$ et celui-ci, à son tour, comme l'un des trois côtés du triangle~$\Delta^2$.

Soit $M$ un espace topologique. Un \textit{$p$-simplexe singulier} dans~$M$ est une application continue $\sigma \colon \Delta^p \to M$ et une \textit{$p$-chaîne singulière} dans~$M$ est une combinaison linéaire à coefficients entiers d'un nombre fini de $p$-simplexes singuliers. Comme la seule contrainte imposée aux applications $\sigma$ est d'être continues, leurs images dans $M$ peuvent par exemple avoir des auto\nobreakdash-intersections et sont en général loin d'être lisses, d'où le nom \og singulier\fg (voir figure \ref{fig:raisonsing}). Le groupe des $p$\nobreakdash-chaînes singulières dans $M$ est noté
\begin{displaymath}
C_p(M)=\bigoplus_{\sigma\hspace{-.3mm}\colon\hspace{-.3mm}\Delta^p \to M} \ZZ \, \sigma.
\end{displaymath}
Par exemple, $C_0(M)$ et $C_1(M)$ sont les groupes abéliens libres engendrés, respectivement, par les points de~$M$ et par les chemins joignant des points, distincts ou pas, de $M$.

\begin{figure}[ht]
\begin{center}
$\begin{array}{c}\includegraphics[scale=.6]{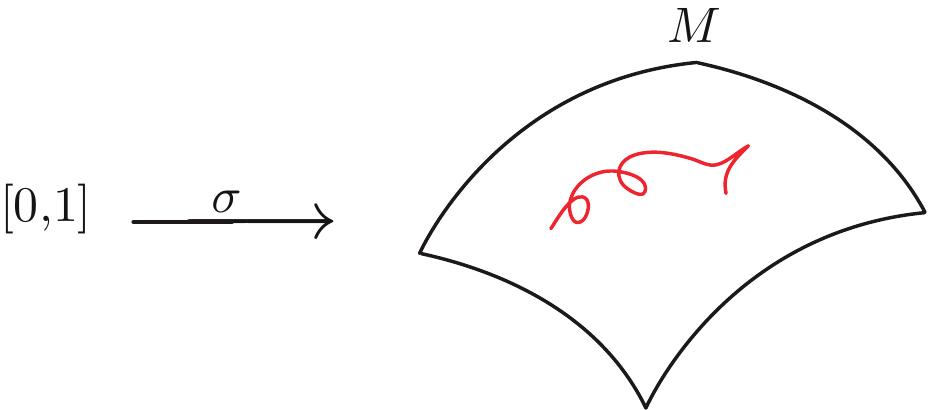}\end{array}$
\end{center}
\caption{Un $1$-simplexe singulier}
\label{fig:raisonsing}
\end{figure}

En s'appuyant sur les faces, on définit pour chaque entier $p \geq 1$ l'application \textit{bord} comme l'homomorphisme
\begin{align}\label{eqn:bord}
\partial_p \colon C_p(M) &\longrightarrow C_{p-1}(M) \\
\sigma &\longmapsto \sum_{i=0}^p (-1)^i (\sigma \circ \delta^i_{p-1}).
\end{align}
Ces applications sont bien définies car la composition
\[
\sigma \circ \delta^i_{p-1} \colon \Delta^{p-1} \to \Delta^p \to M,
\]
dont l'effet est de restreindre $\sigma$ au sous-espace $\{x_i=0\}$ de $\Delta^p$, est un $(p-1)$-simplexe singulier dans $M$. Par exemple, le bord d'un chemin~$\sigma \colon [0, 1] \to M$ est la différence $\sigma(1)-\sigma(0)$ entre son point final et son point initial; le bord d'un $2$\nobreakdash-simplexe singulier~$\sigma \colon \Delta^2 \to M$ est la somme alternée des chemins obtenus en restreignant~$\sigma$ à chacun des trois côtés du triangle, et ainsi de suite.

Grâce aux signes alternés dans la définition du bord \eqref{eqn:bord}, on vérifie par un calcul direct l'égalité
\begin{displaymath}
\partial_{p} \circ \partial_{p+1}=0
\end{displaymath}
pour tout entier $p \geq 0$. Par exemple, en calculant deux fois le bord d'un~$2$\nobreakdash-simplexe singulier~$\sigma \colon \Delta^2 \to M$, on trouve:
\begin{align*}
\partial_1(\partial_2 \sigma)&=\partial_1(\sigma \circ \delta_1^0-\sigma \circ \delta_1^1+\sigma \circ \delta_1^2) \\
&=\sigma \circ \delta_1^0 \circ \delta_0^0-\sigma \circ \delta_1^0 \circ \delta_0^1-\sigma \circ \delta_1^1\circ \delta_0^0 \\
&\hspace{10mm}+\sigma \circ \delta_1^1\circ \delta_0^1+\sigma \circ \delta_1^2\circ \delta_0^0-\sigma \circ \delta_1^2\circ \delta_0^1\\
&=\sigma(0, 0, 1)-\sigma(0, 1, 0)-\sigma(0, 0, 1) \\
&\hspace{10mm}+\sigma(1, 0, 0)+\sigma(0, 1, 0)-\sigma(1, 0, 0) \\
&=0.
\end{align*}
Il s'ensuit que l'image de l'application $\partial_{p+1}$, qui est un sous-groupe du groupe abélien $C_p(M)$, est contenue dans le noyau de l'application~$\partial_p$; on peut donc prendre le quotient. Pour avoir des notations uniformes dans la suite, il est utile de poser $C_{p}(M)=0$ pour tout~$p < 0$ et~\hbox{$\partial_p=0$ pour tout $p \leq 0$.} La suite des $(C_p(M), \partial_p)$ est appelée le \emph{complexe\footnote{Plus généralement, un \emph{complexe (homologique) de groupes abéliens} est une suite~$(A_p, \partial_p)_{p \in \ZZ}$ de groupes abéliens $A_p$ et d'homomorphismes \hbox{$\partial_p \colon A_p \to A_{p-1}$} tels que la composition~$\partial_p \circ \partial_{p+1}$ est nulle pour tout entier $p$.} des chaînes singulières} de $M$.

\begin{definition}\label{def:homologsing} Le groupe d'\textit{homologie singulière} en degré $p$ de l'espace topologique $M$ est le quotient
\begin{displaymath}
\rH_p(M, \ZZ)=\ker(\partial_p) \slash \mathrm{im}(\partial_{p+1}).
\end{displaymath}
\end{definition}

Les $p$-chaînes singulières dans le noyau de $\partial_p$ s'appellent des \textit{cycles} et les éléments dans l'image de $\partial_{p+1}$ des \textit{bords}; avec cette terminologie, l'homologie singulière est le groupe des cycles modulo les bords. On notera $[\sigma]$ la classe d'un cycle $\sigma$ dans le quotient $\rH_p(M, \ZZ)$. Par exemple, un chemin \hbox{$\sigma \colon [0, 1] \to M$} définit un cycle si et seulement si son point initial et son point final coïncident. Lorsqu'il n'y a pas d'ambiguïté, on omettra le sous-indice $p$ dans l'application bord.

\begin{exemple}[homologie en degré $0$]\label{exm:degre0} Calculons l'homologie singulière en degré $0$ d'un espace topologique $M$, c'est-à-dire le quotient \hbox{$C_0(M)\slash \mathrm{im}(\partial_1)$}. Le groupe $C_0(M)$ est engendré par les points de~$M$, et deux points $P, Q \in M$ représentent la même classe dans le quotient si et seulement si $P-Q$ est le bord d'une $1$-chaîne singulière. C'est le cas si et seulement si $P$ et $Q$ appartiennent à la même composante connexe par arcs, d'où un isomorphisme
\begin{displaymath}
\rH_0(M, \ZZ)\simeq \bigoplus_{\pi_0(M)} \ZZ,
\end{displaymath}
où~$\pi_0(M)$ désigne l'ensemble de ces composantes. En particulier, si~$M$ est connexe par arcs, le groupe $\rH_0(M, \ZZ)$ est isomorphe à $\ZZ$.
\end{exemple}

\begin{exemple}[homologie d'un point]\label{exp:point} Soit $M$ l'espace topologique réduit à un point. L'exemple~\ref{exm:degre0} donne un isomorphisme $\rH_0(M, \ZZ)\simeq \ZZ$. Le groupe~$C_p(M)$ est lui-même isomorphe à $\ZZ$ pour tout $p \geq 0$, car il n'y a qu'un seul $p$\nobreakdash-simplexe singulier $\sigma_p\colon\Delta^p \to M$, à savoir l'application constante de valeur le seul point de $M$. Pour $p \geq 1$, les applications bord sont donc déterminées par
\[
\partial_p(\sigma_p)=\sum_{i=0}^{p} (-1)^i \sigma_{p-1}=\begin{cases} \sigma_{p-1} & \text{si $p$ est pair,}\\ 0 & \text{si $p$ est impair.} \end{cases}
\]
Il vient que le noyau de $\partial_p$ est ou bien zéro ou bien égal à l'image de~$\partial_{p+1}$, et dans le deux cas $\rH_p(M, \ZZ)=\ker(\partial_p) \slash \mathrm{im}(\partial_{p+1})$ s'annule.
\end{exemple}

Soient $M$ et $M'$ des espaces topologiques et $f \colon M \to M'$ une application continue entre eux. En associant à chaque $p$\nobreakdash-simplexe~$\sigma$ dans~$M$ le~$p$\nobreakdash-simplexe $f \circ \sigma$ dans $M'$, on obtient un homomorphisme~$f_\ast \colon C_p(M) \to C_p(M')$. On vérifie sans peine que $f_\ast$ commute aux applications bord (c'est-à-dire que l'égalité $\partial_p \circ f_\ast=f_\ast \circ \partial_p$ est vraie pour tout $p$) et induit ainsi un homomorphisme entre les groupes d'homologie singulière, encore noté
\[
f_\ast \colon \rH_p(M, \ZZ) \longrightarrow \rH_p(M', \ZZ).
\]
Par construction, l'application identité sur $M$ induit l'identité en homologie et ces homomorphismes sont compatibles avec la composition au sens suivant: si $f \colon M \to M'$ et $g\colon M' \to M''$ sont des applications continues, alors l'égalité $(g \circ f)_\ast=g_\ast \circ f_\ast$ est vraie. Cela se résume en disant que~$\rH_p(-, \ZZ)$ est un \emph{foncteur} de la catégorie des espaces topo\-logiques vers la catégorie des groupes abéliens ou, de manière plus informelle, que l'homologie singulière est fonctorielle par rapport aux applications continues entre espaces topologiques.

On se placera plus tard dans le cadre où l'espace topologique $M$ est muni d'une structure de variété différentielle et on voudra intégrer des formes différentielles sur des chaînes singulières. Il ne suffira pas alors de travailler avec des applications continues \hbox{$\sigma \colon \Delta^p \to M$,} comme on l'a fait dans ce numéro, mais il faudra imposer des conditions de lissité pour que les intégrales existent, par exemple que~$\sigma$ s'étende en une fonction de classe $\mathcal{C}^\infty$ sur un voisinage ouvert de $\Delta^p$ dans~$\RR^{p+1}$. On peut démontrer que tout élément de $\rH_p(M, \ZZ)$ admet un tel représentant \cite[Th.\,18.13]{Lee}.

Par ailleurs, en remplaçant $\ZZ$ par $\QQ$ dans tout ce qui précède on aboutit à l'homologie singulière à coefficients rationnels $\rH_p(M, \QQ)$, qui est un $\QQ$-espace vectoriel. Pour certains propos tels que construire des produits (section \ref{sec:Legendre}), c'est son dual qui a des meilleures propriétés, d'où l'intérêt de la définition suivante.

\begin{defi} Soit $M$ un espace topologique. La \emph{cohomologie singulière} à coefficients rationnels de~$M$ en degré $p$ est le dual linéaire de l'homologie singulière, c'est-à-dire le $\QQ$-espace vectoriel
\[
\rH^p(M, \QQ)=\Hom(\rH_p(M, \QQ), \QQ)
\]
formé des applications linéaires $\rH_p(M, \QQ) \to \QQ$.
\end{defi}

Écrire le degré en exposant, plutôt qu'en indice, est une façon d'indiquer que la cohomologie singulière est \emph{contravariante}, c'est-à-dire qu'une application continue $f\colon M\to M'$ entre espaces topologiques induit pour tout $p$ une application linéaire
\[
f^\ast \colon \rH^p(M', \QQ)\longrightarrow \rH^p(M, \QQ),
\]
à savoir le dual de l'application $f_\ast$ induite en homologie. On dispose également d'une notion de cohomologie singulière à coefficients entiers, mais pour la définir il faut passer par le dual du complexe de chaînes singulières au lieu de simplement celui de l'homologie.

\subsection{Invariance par homotopie et suite de Mayer-Vietoris}\label{sec:homotopyMayerVietoris}

Pour aller plus loin dans les exemples, il est fort utile de connaître deux propriétés de l'homologie singulière, l'invariance par homotopie et la suite exacte de Mayer-Vietoris, qui permettent de calculer l'homo\-logie d'un espace en le décomposant en morceaux plus simples, comme les exemples \ref{exmp:lacet0} et \ref{exmp:torecomplexe} du numéro suivant l'illustreront.

\begin{invariance} On dit que deux applications continues $f_0, f_1 \colon M \to M'$ sont \textit{homotopes} si l'on peut déformer continûment l'une en l'autre, c'est-à-dire s'il existe une application continue
$
H \colon M \times [0, 1] \to M'
$
satisfaisant aux conditions
\[
H(x, 0)=f_0(x) \quad \text{et} \quad H(x, 1)=f_1(x)
\]pour tout $x \in M$. Une application continue $f \colon M \to M'$ est une \textit{équivalence d'homotopie} s'il existe une application continue \hbox{$g \colon M' \to M$} telle que $g \circ f$ et $f \circ g$ soient homotopes, respectivement, aux applications identité sur~$M$ et sur $M'$; on dit dans ce cas que $g$ est un inverse homotopique de $f$ et que les espaces $M$ et $M'$ ont \emph{même type d'homotopie}, ou encore qu'ils sont \emph{homotopiquement équivalents}.

\begin{prop}\label{prop:homotopie} Si $f \colon M \to M'$ est une équivalence d'homotopie, alors~$f_\ast \colon \rH_p(M, \ZZ) \to \rH_p(M', \ZZ)$ est un isomorphisme de groupes.
\end{prop}

Un cas particulier important est celui où $M$ a le même type d'homo\-topie que l'espace topologique réduit à un seul point; on dit alors que $M$ est \textit{contractile}. Au vu du calcul de l'homologie du point (exemple~\ref{exp:point}), il résulte de la proposition ci\nobreakdash-dessus que tout espace contractile $M$ a pour groupes d'homologie~\hbox{$\rH_0(M, \ZZ) \simeq \ZZ$} et \hbox{$\rH_p(M, \ZZ)=0$} pour tout $p\geq1$. Pour établir la proposition, on démontre que deux applications homotopes induisent le même homomorphisme en homologie; comme l'identité induit l'identité, cela implique que $f_\ast$ admet pour inverse l'application $g_\ast$ induite par un inverse homotopique $g$ de $f$ et est donc un isomorphisme. On renvoie le lecteur à \cite[Th.\,2.10]{hatcher} pour la preuve du premier énoncé.
\end{invariance}

\begin{mayer-vietoris} Si $M=U \cup V$ est la réunion de deux sous-espaces ouverts $U$ et $V$, on peut calculer l'homologie singulière de $M$ en fonction de celles de $U$, de $V$ et de l'intersection~$U \cap V$. Plus précisément, on dispose d'une \textit{suite exacte longue}
\begin{displaymath}
\xymatrix @R=1.25ex @C=1.1ex
{
& &&&\cdots \ar[rr] && \rH_2(M, \ZZ) \ar@{-}[r] & *{} \ar@{-}`r/9pt[d] `/9pt[l] \\
& *{} \ar@{-}`/9pt[d] `d/9pt[d]& && && & *{} \ar@{-}[llllll] & \\
& *{} \ar[r] & \ar[rr]^-{\alpha} \rH_{1}(U \cap V, \ZZ) &&\ar[rr]^-{\beta}\rH_1(U, \ZZ) \oplus \rH_1(V, \ZZ) && \rH_1(M, \ZZ) \ar@{-}[r] & *{} \ar@{-} `r/9pt[d] `[l]\\
& *{} \ar@{-}`/9pt[d] `d/9pt[d]& && && & *{} \ar@{-}[llllll] & \\
& *{} \ar[r] & \ar[rr]^-{\alpha} \rH_{0}(U \cap V, \ZZ) &&\ar[rr]^-{\beta}\rH_0(U, \ZZ) \oplus \rH_0(V, \ZZ)&& \rH_0(M, \ZZ)\ar[rr]&&0.
}
\end{displaymath}
Ci-dessus, les applications $\alpha$ et $\beta$ sont construites à partir des inclusions entre les différents ouverts: si~$\iota_{U \cap V, U}$ désigne l'inclusion de~$U \cap V$ dans $U$, et de même pour d'autres paires, on pose
\begin{displaymath}
\alpha=(-(\iota_{U \cap V, U})_\ast, (\iota_{U \cap V, V})_\ast) \quad \text{et} \quad \beta=(\iota_{U, M})_\ast+(\iota_{V, M})_\ast,
\end{displaymath}
le signe garantissant que la composition $\beta \circ \alpha$ est nulle. La flèche \textit{connectant} le dernier terme d'une ligne avec le premier terme de la ligne suivante est définie comme suit: soit~$[\sigma] \in \rH_p(M, \ZZ)$ la classe d'homologie représentée par une $p$-chaîne singulière $\sigma$ dans~$M$. Par un processus de \emph{subdivision barycentrique}\footnote{Il s'agit de voir le $p$-simplexe $\Delta^p$ comme une $p$-chaîne singulière en le divisant en de plus petites parties homéomorphes à $\Delta^p$, par exemple en écrivant le triangle~$\Delta^2$ comme la réunion des trois triangles tracés depuis son barycentre.}, on peut l'écrire comme
\[
\sigma=\sigma_U+\sigma_V,
\]
où~$\sigma_U$ et $\sigma_V$ sont des~$p$\nobreakdash-chaînes singulières dans $U$ et dans $V$. Comme le bord de $\sigma$ est nul, l'égalité $\partial_p \sigma_U=-\partial_p \sigma_V$ est satisfaite, et on voit ainsi que $\partial_p \sigma_U$ n'est pas seulement une chaîne dans~$U$, mais aussi dans $V$ puisqu'elle coïncide avec $-\partial_p \sigma_V$. Autrement dit,~$\partial_p \sigma_U$ est une $(p-1)$\nobreakdash-chaîne singulière dans l'intersection~$U \cap V$, et on peut alors considérer l'application
\begin{align}\label{eqn:morphismeconnectant}
\rH_p(M, \ZZ) &\longrightarrow \rH_{p-1}(U \cap V, \ZZ). \\
[\sigma] &\longmapsto [\partial_p \sigma_U]
\end{align}
Observons que celle-ci n'est pas toujours nulle: bien que $\partial_p\sigma_U$ soit le bord d'une chaîne dans $U$, il n'y a pas de raison pour que ce soit également le bord d'une chaîne dans $U \cap V$. 

Dire que le diagramme ci-dessus est une suite exacte longue est une manière concise d'exprimer le fait que n'importe quels trois termes consécutifs~$A \To{f} B \To{g} C$ satisfont à la relation~$\ker(g)=\mathrm{im}(f)$. En particulier, si $A$ est nul, l'application $g$ est injective; si $C$ est nul, l'application~$f$ est surjective; si $A$ et $C$ sont nuls, $B$ doit l'être aussi.
\end{mayer-vietoris}

\subsection{Plus d'exemples}

Armés de ces outils, nous sommes maintenant en état de calculer l'homologie du plan complexe épointé et des tores complexes.

\begin{exemple}[le plan complexe épointé]\label{exmp:lacet0} Soit $M=\CC^\times$ le plan complexe épointé. Cet espace topologique étant connexe par arcs, il y a un isomorphisme~\hbox{$\rH_0(M, \ZZ) \simeq \ZZ$}. Par ailleurs, l'inclusion
\begin{displaymath}
S^1=\{z \in \CC \mid |z|=1 \} \hookrightarrow M
\end{displaymath}
du cercle unité est une équivalence d'homotopie, car l'application~\hbox{$M \to S^1$} qui envoie un nombre complexe non nul $z$ vers~$z/|z|$ en est un inverse homotopique. Au vu de la proposition \ref{prop:homotopie}, il suffit donc de calculer l'homologie singulière de~$S^1$. Pour ce faire, on écrit~$S^1$ comme la réunion de deux arcs de cercle $U$ et $V$ tels que les montre la figure \ref{fig:homologiecercle}.
\begin{figure}[ht]
\centerline{\includegraphics[width=0.9\textwidth]{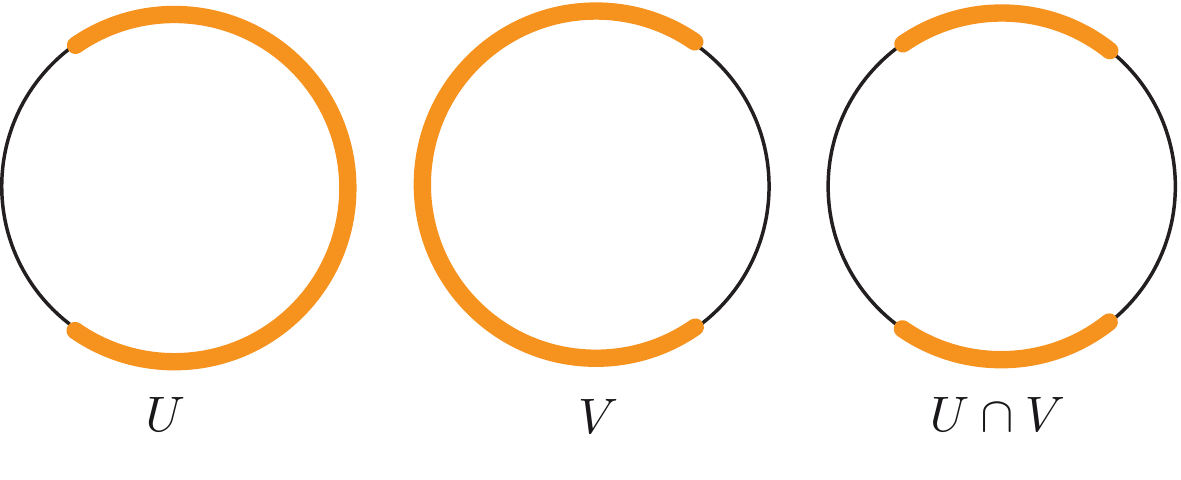}}
\caption{Recouvrement de $S^1$ par deux ouverts contractiles}
\label{fig:homologiecercle}
\end{figure}
Puisque les ouverts~$U$ et $V$, ainsi que les deux composantes connexes de~\hbox{$U \cap V$}, sont contractiles, la suite exacte de Mayer-Vietoris implique l'annulation~$\rH_p(S^1, \ZZ)=0$ pour tout $p \geq 2$. En effet, comme les termes à gauche et à droite dans la suite
\[
\rH_p(U, \ZZ) \oplus \rH_p(V, \ZZ) \longrightarrow \rH_p(S^1, \ZZ) \longrightarrow \rH_{p-1}(U \cap V, \ZZ)
\]
sont nuls pour $p \geq 2$, celui du milieu doit l'être aussi. De plus, l'homologie en degré $1$ s'inscrit dans la suite exacte longue
\begin{equation}
\xymatrix @R=1.25ex @C=1.1ex
{
& &&&0 \ar[rr] && \rH_1(S^1, \ZZ) \ar@{-}[r] & *{} \ar@{-}`r/9pt[d] `/9pt[l] \\
& *{} \ar@{-}`/9pt[d] `d/9pt[d]& && && & *{} \ar@{-}[llllll] & \\
& *{} \ar[r] & \ar[rr]^-{\alpha} \rH_{0}(U \cap V, \ZZ)\ar@{=}[d] &&\ar[rr]^-{\beta}\rH_0(U, \ZZ) \oplus \rH_0(V, \ZZ)\ar@{=}[d]&& \rH_0(S^1, \ZZ)\ar@{=}[d]\ar[rr]&&0,\\
& & \ZZ\oplus\ZZ &&\ZZ\oplus\ZZ &&\ZZ&&
}
\end{equation}
où l'application $\alpha$ est donnée par la matrice $\left(\begin{smallmatrix} -1 & -1 \\ 1 & 1\end{smallmatrix}\right)$ et l'application~$\beta$ est donnée par $(x, y) \mapsto x+y$. Comme la flèche connectante est injective car le terme précédent s'annule, le groupe $\rH_1(S^1, \ZZ)$ est isomorphe au noyau de $\alpha$, qui est le groupe libre de rang~$1$ engendré par la classe de~$P-Q$, où $P$ et $Q$ sont n'importe quels deux points dans des composantes connexes distinctes de~$U \cap V$. On a donc
\[
\rH_1(M, \ZZ) \simeq \rH_1(S^1, \ZZ) \simeq \ZZ.
\]
Un générateur explicite est donné par un lacet~\hbox{$\sigma \colon [0, 1] \to S^1 \subset M$} tel que l'application \eqref{eqn:morphismeconnectant} envoie $[\sigma]$ sur~$[P-Q]$, par exemple le cercle lui\nobreakdash-même, parcouru une seule fois dans le sens anti-horaire. Tout compte fait, le premier groupe d'homologie singulière de $\CC^\times$ est le groupe abélien libre engendré par le lacet $t \mapsto \exp(2\pi i t)$. 
\begin{figure}[ht]
\centerline{\includegraphics[width=0.25\textwidth]{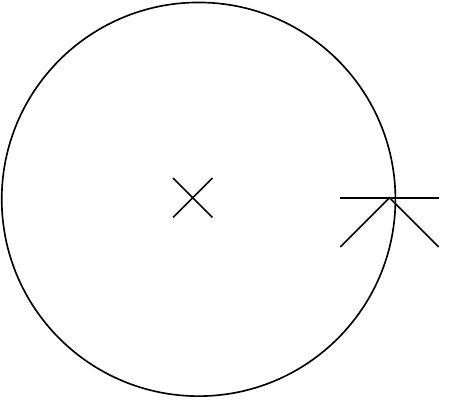}}
\caption{Un générateur de $\rH^1(\CC^\times, \ZZ)$}
\label{fig:lacetengendrant}
\end{figure}
\end{exemple}

\begin{exemple}[les tores complexes]\label{exmp:torecomplexe} Soient $\omega_1, \omega_2 \in \CC^\times$ deux nombres complexes non nuls qui ne sont pas $\RR$-proportionnels, c'est-à-dire dont le quotient $\sfrac{\omega_1}{\omega_2}$ n'est pas réel, et soit \hbox{$\Lambda=\ZZ \omega_1 \oplus \ZZ \omega_2$} le réseau du plan complexe qu'ils engendrent. Le groupe $\Lambda$ agit par translation sur $\CC$ et cette action est libre et proprement discontinue (tout élément non nul de $\Lambda$ agit sans point fixe et la préimage d'une partie compacte de $\CC$ par l'action $\Lambda \times \CC \to \CC$ est compacte); le quotient \hbox{$M=\CC \slash \Lambda$} est alors un espace topologique séparé. On l'appelle \emph{tore complexe}, car il est en même temps homéomorphe au produit de deux cercles~$\RR^2/\ZZ^2 \simeq S^1 \times S^1$ et muni d'une structure de variété complexe\footnote{En tant qu'espace topologique, $M$ est donc indépendant du choix de $\omega_1$ et $\omega_2$, mais sa structure de variété complexe en dépend.}. Le parallélogramme
\begin{displaymath}
\cF=\{x\omega_1+y\omega_2 \mid x, y \in \RR,\ 0 \leq x, y \leq 1 \}
\end{displaymath}
est presque un domaine fondamental pour l'action de $\Lambda$ sur $\CC$; l'espace topologique $M$ est obtenu en recollant deux à deux les côtés opposés de son bord, comme on le voit dans la figure \ref{fig:recollement-tore}.

\begin{figure}[h]
\begin{center}
	\begin{tikzpicture}
\node[inner sep=0pt] at (0,0) {\includegraphics[width=.9\textwidth]{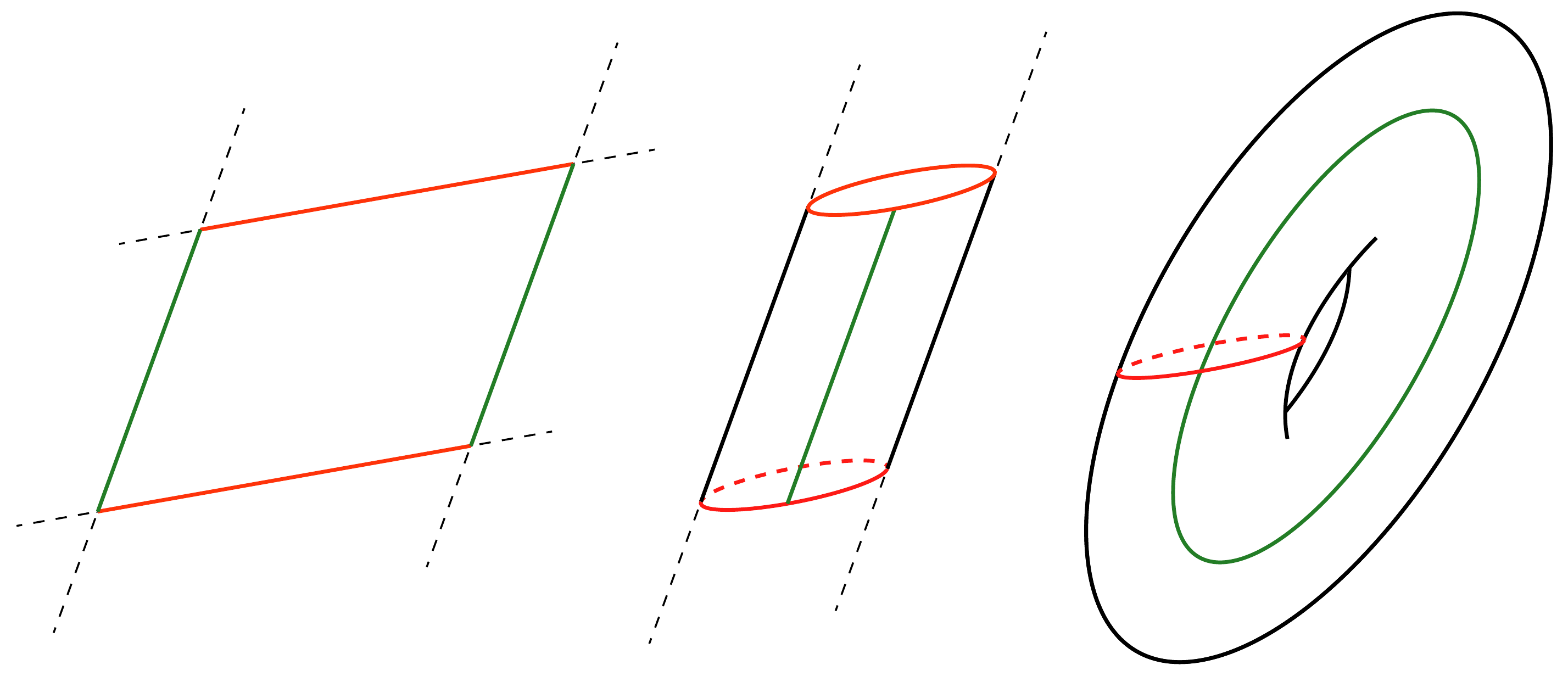}};	

\node at (-5,-1.5) {$0$};
\node at (-1.9,-1) {$\omega_1$};
\node at (-4.25,1) {$\omega_2$};
\node at (-3, 0) {$\mathcal{F}$};

\draw[->] (-1.5,0) --(-0.75, 0);
\draw[->] (1.5, 0)-- (2.25,0);
	\end{tikzpicture}
	\end{center}
	\caption{Le tore complexe $M=\CC\slash\Lambda$}
\label{fig:recollement-tore}
\end{figure}

Ces tores complexes étant connexes par arcs, on a un isomorphisme~$\rH_0(M, \ZZ)\simeq \ZZ$. Pour calculer l'homologie en degré supérieur, on écrit $M$ comme la réunion de deux coudes $U$ et $V$ tels que les montre la figure \ref{fig:homologietore}, et on utilise la suite exacte de Mayer\nobreakdash-Vietoris.
\begin{figure}[ht]
\centerline{\includegraphics[width=1\textwidth]{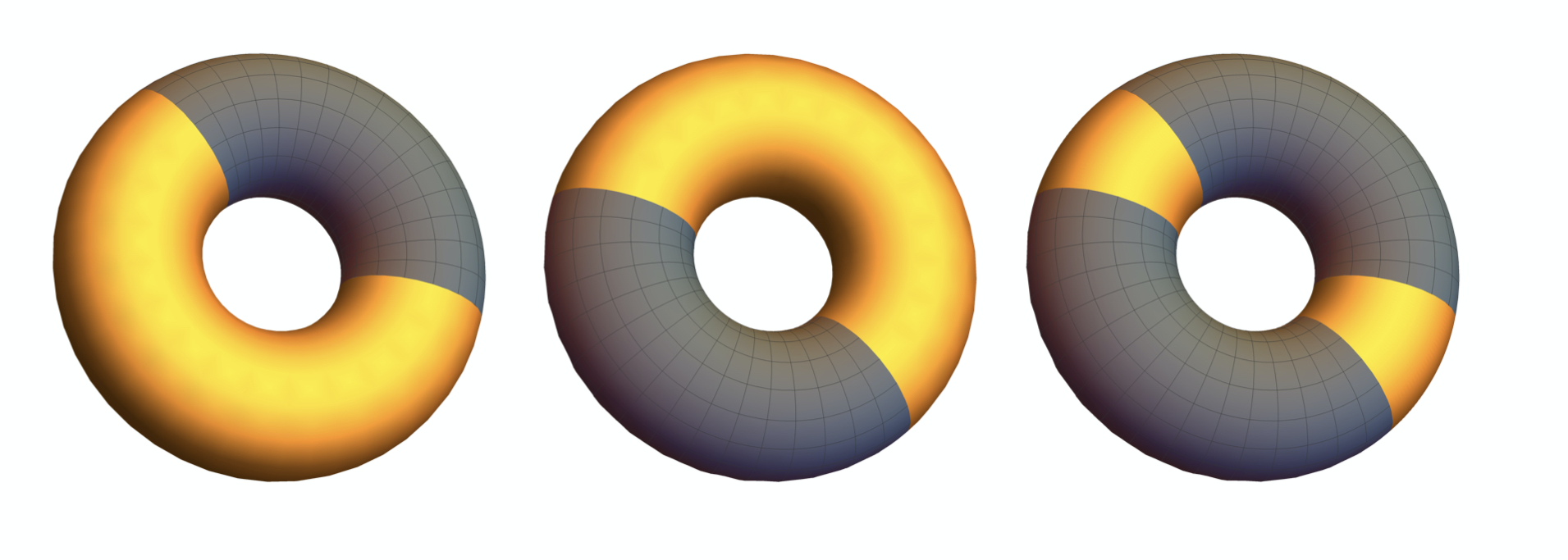}}
\caption{Recouvrement d'un tore complexe par deux ouverts avec même type d'homotopie que $S^1$}
\label{fig:homologietore}
\end{figure}
Comme les ouverts $U$, $V$ et les deux composantes connexes de $U \cap V$ ont même type d'homotopie qu'un cercle, dont l'homologie vient d'être calculée dans l'exemple \ref{exmp:lacet0}, les groupes~$\rH_p(M, \ZZ)$ s'annulent pour $p \geq 3$ par le même argument qu'avant. Les groupes restant s'inscrivent dans la suite exacte longue
\begin{displaymath}
\xymatrix @R=1.25ex @C=1.1ex
{
& &&&0 \ar[rr] && \rH_2(M, \ZZ) \ar@{-}[r] & *{} \ar@{-}`r/9pt[d] `/9pt[l] \\
& *{} \ar@{-}`/9pt[d] `d/9pt[d]& && && & *{} \ar@{-}[llllll] & \\
& *{} \ar[r] & \ar[rr]^-{\alpha} \rH_{1}(U \cap V, \ZZ)\ar@{=}[d] &&\ar[rr]^-{\beta}\rH_1(U, \ZZ) \oplus \rH_1(V, \ZZ)\ar@{=}[d]&& \rH_1(M, \ZZ) \ar@{-}[r] & *{} \ar@{-} `r/9pt[d] `[dd]\\
& & \ZZ\oplus\ZZ &&\ZZ\oplus\ZZ &&&&\\
& *{} \ar@{-}`/9pt[d] `d/9pt[d]& && && & *{} \ar@{-}[llllll] & \\
& *{} \ar[r] & \ar[rr]^-{\alpha} \rH_{0}(U \cap V, \ZZ)\ar@{=}[d] &&\ar[rr]^-{\beta}\rH_0(U, \ZZ) \oplus \rH_0(V, \ZZ)\ar@{=}[d] && \rH_0(M, \ZZ)\ar@{=}[d]\ar[rr]&&0.\\
& & \ZZ\oplus\ZZ &&\ZZ\oplus\ZZ &&\ZZ&&
}
\end{displaymath}
L'application $\alpha$ est encore donnée, tant en degré $0$ qu'en degré $1$, par la matrice $\left(\begin{smallmatrix} -1 & -1 \\ 1 & 1\end{smallmatrix}\right)$; son noyau et son image sont donc libres de rang $1$, d'où un isomorphisme $\rH_2(M, \CC) \simeq \ZZ.$ Enfin, comme~$\ker(\beta)$ et $\im(\alpha)$ sont égaux dans une suite exacte longue, ainsi que $\ker(\alpha)$ et l'image de la flèche connectante, on dispose de la suite exacte
\begin{displaymath}
\xymatrix @R=1.25ex @C=3.5ex
{
0 \ar[r] & (\rH_{1}(U, \ZZ) \oplus \rH_{1}(V, \ZZ))/\im(\alpha) \ar@{=}[d] \ar[r]^-{\beta} & \rH_1(M, \ZZ) \ar[r] & \ker(\alpha) \ar@{=}[d] \ar[r] & 0. \\
& \ZZ & & \ZZ &
}
\end{displaymath}
De là on trouve un isomorphisme $\rH_1(M, \CC) \simeq \ZZ \oplus \ZZ$. Explicitement, ce groupe est engendré par les images dans $M$ des applications
\begin{displaymath}
\begin{aligned}
\sigma_i \colon [0, 1] &\to \cF \\ t &\longmapsto t\omega_i
\end{aligned} \qquad (i=1, 2),
\end{displaymath}
qui sont représentées dans la figure \ref{fig:homologiedutore}.

\begin{figure}[ht]
\centerline{\includegraphics[width=0.4\textwidth]{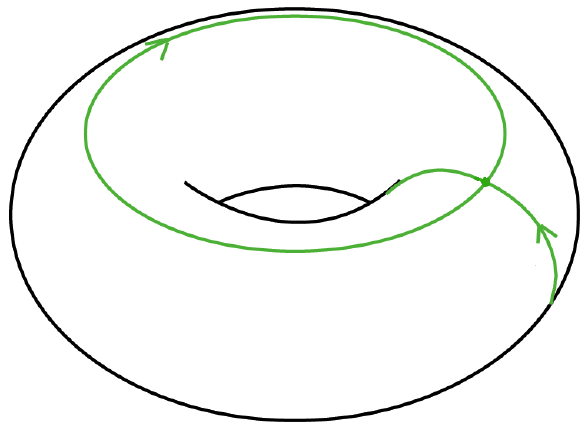}}
\caption{Homologie du tore $\rH_1(M, \CC)=\ZZ [\sigma_1] \oplus \ZZ [\sigma_2]$}
\label{fig:homologiedutore}
\end{figure}

On aura aussi besoin de connaître l'homologie d'un tore auquel on a ôté un point; notons-le $M^\ast$. Cet espace a le même type d'homotopie que le quotient du bord du parallélogramme $\cF$ par l'action de $\Lambda$ identifiant les côtés opposés, qui n'est rien d'autre que la réunion du cercle rouge et du cercle vert de la figure \ref{fig:recollement-tore}. On est ainsi réduit au calcul de l'homologie singulière d'un bouquet de deux cercles. Écrivons-le comme la réunion de deux ouverts $U$ et~$V$, chacun composé d'un des cercles en entier et de deux petits arcs sur l'autre cercle, de sorte que~$U$ et~$V$ aient le type d'homotopie d'un cercle, tandis que~$U \cap V$ soit contractile. La suite exacte de Mayer-Vietoris donne alors l'annulation de~$\rH_p(M^\ast, \ZZ)$ pour tout $p \geq 2$, par opposition au cas du tore complet, dont l'homologie en degré $2$ était isomorphe à $\ZZ$. En revanche, le premier groupe d'homologie reste inchangé, comme le montre la suite~exacte
\begin{displaymath}
\xymatrix @R=1.25ex @C=1.1ex
{
& & \ar[rr] 0 &&\ar[rr]^-{\beta}\rH_1(U, \ZZ) \oplus \rH_1(V, \ZZ)\ar@{=}[d]&& \rH_1(M^\ast, \ZZ) \ar@{-}[r] & *{} \ar@{-} `r/9pt[d] `[dd]\\
& & &&\ZZ\oplus\ZZ &&&&\\
& *{} \ar@{-}`/9pt[d] `d/9pt[d]& && && & *{} \ar@{-}[llllll] & \\
& *{} \ar[r] & \ar[rr]^-{\alpha} \rH_{0}(U \cap V, \ZZ)\ar@{=}[d] &&\ar[rr]\rH_0(U, \ZZ) \oplus \rH_0(V, \ZZ)\ar@{=}[d] && \rH_0(M^\ast, \ZZ)\ar@{=}[d]\ar[rr]&&0.\\
& & \ZZ &&\ZZ\oplus\ZZ &&\ZZ&&
}
\end{displaymath}
En effet, comme $\alpha$ est l'application injective $x \mapsto (-x, x)$, le morphisme connectant est nul, d'où un isomorphisme \hbox{$\rH_1(M^\ast, \ZZ) \simeq \ZZ\!\oplus\!\ZZ$}. Pour trouver des représentants explicites, il suffit de déformer les chemins $\sigma_1$ et $\sigma_2$ jusqu'à ce qu'ils ne rencontrent pas le point ôté.
\end{exemple}

\section{Cohomologie de de~Rham algébrique}\label{sec:5}

Sur les espaces topologiques sous-jacents aux variétés différentielles, on l'a dit, toute classe d'homologie singulière peut être représentée par une chaîne de classe $\mathcal{C}^\infty$. En intégrant des formes différentielles le long de ces chaînes, on obtient un accouplement à valeurs réelles entre l'homologie singulière et la cohomologie de de~Rham; c'est un accouplement parfait d'après un théorème dû à de~Rham lui-même. On pourrait avoir l'idée de définir les périodes comme les coefficients de cet accouplement par rapport à des bases quelconques de l'homologie singulière et de la cohomologie de de~Rham, mais dans ce cas tout nombre réel serait une période, car la cohomologie de de~Rham d'une variété différentielle n'a en général plus de structure que celle d'un~$\RR$\nobreakdash-espace vectoriel. De manière un peu surprenante, lorsque la variété en question est définie par des équations polynomiales, toute classe de cohomologie de de~Rham admet comme représentant une forme différentielle algébrique. L'espace de ces formes étant défini sur le même corps que les équations de la variété, on obtient un pendant algébrique de la cohomologie de de~Rham, encore en accouplement parfait avec l'homo\-logie singulière; c'est celui-là qui nous permettra de donner dans ce qui suit une interprétation cohomologique des périodes.

\subsection{Variétés algébriques affines}\label{sec:varietes}

Soient $n \geq 1$ un entier et $\mathbb{A}^n$ l'espace affine de dimension~$n$. On va s'intéresser aux \textit{variétés algébriques affines} $X \subset \mathbb{A}^n$ définies par l'annulation d'une famille de polynômes
\begin{displaymath}
f_1, \dots, f_m \in \QQ[x_1, \dots, x_n].
\end{displaymath}
Pour les propos de ce texte, il nous suffira la plupart du temps d'y penser comme les ensembles de leurs \emph{points complexes} 
\begin{displaymath}
X(\CC)=\{(z_1, \dots, z_n) \in \CC^n \mid f_j(z_1, \dots, z_n)=0 \text{ pour tout } j\}, 
\end{displaymath} sans oublier les équations les définissant, notamment le fait qu'elles sont toutes à coefficients rationnels. On appelle ces variétés \textit{algébriques} car les~$f_j$ sont des polynômes, par opposition à une variété disons différentielle, où l'on regarderait les zéros de fonctions de classe~$\mathcal{C}^\infty$ sur $\RR^n$ à la place; on les appelle \textit{affines} car on cherche les solutions aux équations~\hbox{$f_j=0$} dans l'espace affine $\CC^n$ plutôt que dans l'espace projectif~$\mathbb{P}^n(\CC)$, auquel cas il faudrait supposer que les~$f_j$ sont homogènes et on parlerait de variétés projectives.

Comme deux polynômes à coefficients rationnels prennent les mêmes valeurs sur l'ensemble $X(\CC)$ si leur différence est de la forme~$f_1P_1+\dots+f_mP_m$, pour des $P_i \in \QQ[x_1, \dots, x_n]$, il est naturel de définir l'\emph{anneau des fonctions} sur la variété $X$ comme le quotient
\begin{equation}\label{eqn:defnanneaufonctions}
A=\QQ[x_1, \dots, x_n] \slash (f_1, \dots, f_m)
\end{equation}
des fonctions sur l'espace affine $\mathbb{A}^n$ par l'idéal engendré par les poly\-nômes définissant $X$. Alors $X(\CC)$ s'identifie canoniquement à l'ensemble de morphismes d'anneaux de $A$ dans $\CC$:
\begin{equation}\label{eqn:foncteurpoints}
X(\CC)=\Hom(A, \CC).
\end{equation} En effet, pour tout point complexe $(z_1, \dots, z_n) \in X(\CC)$, le morphisme d'anneaux $\QQ[x_1, \dots, x_n] \to \CC$ qui envoie $x_i$ sur $z_i$ passe au quotient par l'idéal engendré par les $f_j$ et définit donc un morphisme \hbox{$A \to \CC$}. Réciproquement, comme les classes des $x_i$ dans $A$ satisfont aux équations polynomiales $f_j=0$, il en va de même pour leurs images par tout morphisme d'anneaux à valeurs dans $\CC$, d'où un élément de~$X(\CC)$. 

D'après le théorème des zéros (\textit{Nullstellensatz}) de Hilbert, si deux polynômes $P, Q \in \QQ[x_1, \ldots, x_n]$ prennent les mêmes valeurs partout sur~$X(\CC)$, alors il existe un entier~$r \geq 1$ tel que~$(P-Q)^r$ appartienne à l'idéal $(f_1, \ldots, f_m)$. Vu que les polynômes constants~$1$ et~$0$ prennent les mêmes valeurs sur l'ensemble vide, une conséquence remarquable de ce théorème est que, bien que les équations $f_j(x_1, \dots, x_n)=0$ puissent ne pas avoir de solutions rationnelles ou réelles, elles admettent toujours des solutions complexes dès que l'on exclut le cas trivial où l'idéal~$(f_1, \dots, f_m)$ est égal à l'anneau~$\QQ[x_1, \dots, x_n]$ tout entier \cite[Ch.\,I, \S 4]{perrin}. Par exemple, le polynôme~$x^2+1$ n'a pas de racines réelles, mais l'ensemble $X(\CC)=\{i, -i\}$ des points complexes de la variété~$X \subset \mathbb{A}^1$ qu'il définit est non vide. 

Soit $Y \subset \mathbb{A}^r$ une variété algébrique affine, définie par l'annulation d'une famille de polynômes $g_1, \dots, g_s \in \QQ[y_1, \dots, y_r]$, et soit
\[
B=\QQ[y_1, \dots, y_r] \slash (g_1, \dots, g_s)
\]
son anneau de fonctions. Un \emph{morphisme de variétés}~$h\colon X \to Y$ est une application $h \colon X(\CC) \to Y(\CC)$ dont les composantes 
\[
h_1, \dots, h_r \colon X(\CC) \longrightarrow \CC
\] sont les restrictions à $X(\CC)$ de polynômes dans $\QQ[x_1, \dots, x_n]$, vus comme des fonctions sur $\CC^n$. Concrètement, en choisissant de tels polynômes, encore notés \hbox{$h_1, \dots, h_r \in \QQ[x_1, \dots, x_n]$} par un léger abus, on peut penser au morphisme $h\colon X \to Y$ comme la donnée de~$r$ polynômes satisfaisant aux équations
\begin{equation}\label{eqn:annulation-morphisme}
g_i(h_1(x_1, \dots, x_n), \dots, h_r(x_1, \dots, x_n))=0
\end{equation}
pour tout $i=1, \dots, s$ et pour tout $n$-uplet $(x_1, \dots, x_n)$ dans le lieu des zéros des polynômes $f_j$. Une autre telle donnée $h_1', \dots, h_r'$ définit le même morphisme de variétés si et seulement si les différences~\hbox{$h_i'-h_i$} appartiennent à l'idéal $(f_1, \dots, f_m)$ pour tout~$i$. Par exemple, si \hbox{$X=\mathbb{A}^1$} est la droite affine et $Y \subset \mathbb{A}^2$ est la parabole d'équation $y=x^2$, l'application $h(z)=(z, z^2)$ est un morphisme de variétés $h \colon X \to Y$. De même, si $Z \subset \mathbb{A}^2$ est définie par $y^2=x^3$, l'application $m(z)=(z^2, z^3)$ est un morphisme de variétés $m \colon X \to Z$. 

On dit qu'un morphisme $h \colon X \to Y$ est un \emph{isomorphisme de variétés} s'il existe un morphisme de variétés $u \colon Y \to X$ tel que $u \circ h$ et~$h\circ u$ soient, respectivement, l'identité sur $X$ et sur $Y$. Par exemple, $h(z)=(z, z^2)$ ci-dessus est un isomorphisme de variétés, d'inverse la projection sur la première coordonnée $u(x, y)=x$; par contre, $m$ n'est pas un isomorphisme de variétés, bien que l'application $m \colon X(\CC) \to Z(\CC)$ soit bijective (pourquoi?). 

Par le biais de la bijection~\eqref{eqn:foncteurpoints}, un morphisme $h \colon X \to Y$ équivaut à la donnée d'un morphisme d'anneaux $h^\ast \colon B \to A$, à savoir celui qui envoie la classe de $y_i$ dans $B$ sur la classe du polynôme~$h_i$ dans~$A$. L'inversion des rôles de $A$ et~$B$ en passant des variétés à leurs anneaux des fonctions reflète l'intuition qu'une fonction sur $Y(\CC)$ donne lieu à une fonction sur $X(\CC)$ par précomposition avec $h$. Un morphisme de variétés $h$ est un isomorphisme si et seulement si le morphisme d'anneaux associé $h^\ast$ est un isomorphisme. 

\begin{definition}\label{def:varlisse} Soit $c$ un entier entre $0$ et $n$. On dit qu'une variété algébrique affine~\hbox{$X \subset \mathbb{A}^n$} définie par des polynômes~$f_1, \dots, f_m$ comme ci-dessus est \emph{lisse de dimension~$n-c$} si, pour tout point $P \in X(\CC)$, il existe un polydisque
\begin{displaymath}
D(P, \varepsilon)=\{z \in \CC^n \mid \|z-P\|< \varepsilon \}
\end{displaymath}
centré en $P$ et de rayon suffisamment petit $\varepsilon>0$ tel que l'intersection
\[
D(P, \varepsilon) \cap X(\CC)
\]
soit le lieu des zéros de $c$ polynômes parmi les $f_j$ (disons $f_1, \ldots, f_c$, quitte à les réordonner) et que la matrice jacobienne
\begin{displaymath}
\left(\frac{\partial f_i}{\partial x_j} \right)_{\substack{1 \leq i \leq c \\ 1 \leq j \leq n}}
\end{displaymath}
soit de rang $c$. Lorsque cette condition n'est pas satisfaite, on dit que la variété $X$ est \emph{singulière} en $P$.
\end{definition}

En particulier, une variété $X \subset \mathbb{A}^n$ définie par une seule équation non nulle~$f$ est lisse de dimension $n-1$ si et seulement si les polynômes
\begin{displaymath}
f, \frac{\partial f}{\partial x_1}, \dots, \frac{\partial f}{\partial x_n}
\end{displaymath}
n'ont pas de racine commune sur les nombres complexes. La condition de lissité dépend des équations définissant $X$ et pas seulement de l'ensemble $X(\mathbf{C})$. Par exemple, les polynômes $x$ et $x^2$ ont le même lieu des zéros dans $\mathbf{C}$, mais seulement le premier définit une variété variété lisse $X \subset \mathbb{A}^1$. Cet exemple illustre aussi le fait que, pour les variétés lisses, l'ensemble des polynômes s'annulant sur $X(\mathbf{C})$ coïncide exactement avec l'idéal $(f_1,\dots, f_m)$, c'est-à-dire que l'on peut prendre $r=1$ dans le théorème des zéros de Hilbert. 

\begin{exemple}[croisements normaux]\label{exmp:SNCD} La variété $X \subset \mathbb{A}^2$ définie par le polynôme $f(x, y)=xy$ est singulière car le point $(0, 0)$ est un zéro commun de $f$ et de ses dérivées $\sfrac{\partial f}{\partial x}=y$ et $\sfrac{\partial f}{\partial y}=x$. Cette variété est la réunion des lieux des zéros des polynômes $x$ et~$y$, qui sont chacun des variétés lisses de dimension $1$ s'intersectant au point~$(0, 0)$, lui-même une variété lisse de dimension $0$. C'est un exemple paradigmatique de \emph{diviseur à croisements normaux}, qui est plus généralement défini comme une réunion finie de variétés $X_i \subset \mathbb{A}^n$ lisses de dimension $n-1$, appelées les \emph{composantes irréductibles}, qui s'intersectent transversalement en des variétés lisses de la dimension \og attendue\fg, c'est\nobreakdash-à\nobreakdash-dire telles que toute intersection de $m$ composantes irréductibles est lisse de dimension $n-m$. Par exemple, trois droites dans~$\mathbb{A}^2$ forment un diviseur à croisements normaux si et seulement si elles ne contiennent pas de point d'intersection triple. La condition de transversalité signifie qu'en tout point $P$ dans l'intersection des points complexes de deux composantes irréductibles $X_i$ et $X_j$, la somme des espaces vectoriels tangents $T_P X_i(\CC)$ et $T_P X_j(\CC)$ est égale à $\CC^n$. Cela permet de définir algébriquement l'intersection comme le lieu des zéros $X_i \cap X_j \subset \mathbb{A}^n$ de la réunion des polynômes définissant $X_i$ et $X_j$, en excluant des situations, comme celle de deux courbes tangentes, où cette définition naïve ne donne pas lieu à une bonne théorie de l'intersection (par exemple, au théorème de Bézout). Grâce à un théorème de Hironaka connu sous le nom de \emph{résolution des singularités} \cite{Kollar}, on peut souvent réduire à l'étude des diviseurs à croisements normaux des questions ayant trait aux variétés algébriques singulières.\end{exemple}

\begin{remarque}\label{rem:Weilres} Pour l'étude des périodes, on pourrait considérer plus généralement le corps des nombres algébriques $\overline\QQ$ et des variétés algébriques affines $X$ définies par une famille de polynômes à coefficients dans  $\overline\QQ$. Pour la plupart de nos besoins, c'est commode mais pas strictement nécessaire: on peut toujours se ramener à des variétés définies sur $\QQ$ par un processus, appelé \emph{restriction de Weil}, similaire à celui qui permet d'identifier $\CC$ à $\RR^2$ en prenant des parties réelles et imaginaires, puis des conditions algébriques complexes du type \og $z \in \CC$ est non nul\fg à des conditions algébriques réelles du type \og la paire $(x, y) \in \RR^2$ est telle que $x^2+y^2$ est non nul\fg. \end{remarque}

Il n'a été question jusqu'à là que des points complexes d'une variété. Plus généralement, pour toute $\QQ$-algèbre (c’est-à-dire pour tout anneau commutatif avec unité $R$ contenant $\QQ$), on peut évaluer les polynômes $f_j$ en un point de $R^n$ puisque les seules opérations requises pour ce faire sont la multiplication et l’addition, présentes dans un anneau quelconque, et la multiplication par les coefficients des $f_j$, qui a un sens car $R$ contient $\QQ$. Les $R$-points de~$X$ sont alors définis comme le lieu des zéros dans $R^n$ des polynômes $f_j$, à savoir
\begin{displaymath}
X(R)=\{(z_1, \dots, z_n) \in R^n \mid f_j(z_1, \dots, z_n)=0 \text{ pour tout } j\}. 
\end{displaymath} Ces ensembles sont reliés entre eux comme suit: si $\varphi \colon R \to S$ est un morphisme de $\QQ$-algèbres, c’est-à-dire
un morphisme d’anneaux entre $\QQ$-algèbres, et si $(z_1, \dots, z_n) \in R^n$ appartient à $X(R)$, l'égalité 
\[
f_j(\varphi(x_1), \dots, \varphi(x_n))=\varphi(f_j(x_1, \dots, x_n))=0
\] est vraie pour tout $j$. Le point $(\varphi(x_1), \dots, \varphi(x_n)) \in S^n$ appartient donc à $X(S)$, d'où une application $\varphi_\ast\colon X(R) \to X(S)$. 

Avec ces notations, il y a encore une identification canonique 
\begin{equation}\label{eqn:foncteurpoints2}
X(R)=\Hom(A, R)
\end{equation} et la donnée de tous les ensembles $X(R)$ et de toutes les applications~$\varphi_\ast$ permet de retrouver l'anneau $A$ à isomorphisme près. Un point de vue plus puissant sur les variétés affines consiste alors à les définir par leurs anneaux de fonctions, sans mettre l'accent sur le plongement dans un espace affine donné $\mathbb{A}^n$. 

\subsection{Cohomologie de de~Rham algébrique} On dispose, sur l'anneau des polynômes à $n$ variables, des règles usuelles du calcul différentiel, notamment de l'application $\QQ$-linéaire
\begin{align}\label{eqn:diff}
d \colon \QQ[x_1, \ldots, x_n] &\longrightarrow \QQ[x_1, \dots, x_n]dx_1 \oplus \cdots \oplus \QQ[x_1, \dots, x_n]dx_n \nonumber \\
g &\longmapsto dg=\sum_{i=1}^n \Bigl(\frac{\partial g}{\partial x_i}\Bigr) dx_i
\end{align}
qui à un polynôme $g$ associe sa \textit{différentielle} $dg$. Cette application satisfait à la règle de Leibniz \begin{equation}\label{eqn:leibniz}
d(gh)=gdh+hdg.
\end{equation}
Pour qu'elle passe au quotient par l'idéal engendré par les polynômes $f_1, \ldots, f_m$ et définisse donc une application sur l'anneau $A$ des fonctions sur la variété $X$ correspondante, il faudrait que les relations obtenues en dérivant les $f_1, \ldots, f_m$ soient encore valables dans l'espace d'arrivée, ce qui bien sûr ne sera pas le cas sauf si on les force à l'être. On est ainsi amené à définir les $1$\nobreakdash-\textit{formes différentielles} sur la variété algébrique affine $X$ comme le $A$-module
\begin{displaymath}
\Omega^1_A=\left(A dx_1 \oplus \cdots \oplus Adx_n\right) \slash (df_1, \ldots, df_m),
\end{displaymath}
c'est-à-dire comme le quotient du $A$-module libre ayant pour base les symboles~$dx_1, \ldots, dx_n$ par le sous-module engendré par les différentielles des polynômes $f_j$. Une $1$-forme différentielle sur $X$ est donc représentée par une combinaison linéaire
\begin{equation}\label{eqn:writing1form}
\omega=\sum_{i=1}^n a_i dx_{i},
\end{equation}
où les $a_i \in A$ sont des fonctions sur $X$, et deux telles expressions repré\-sentent la même $1$-forme différentielle si et seulement si leur différence appartient à l'idéal engendré par les $df_j$. Par construction, la différentielle~\eqref{eqn:diff} induit une application $\QQ$-linéaire encore notée
\begin{displaymath}
d \colon A \to \Omega^1_A.
\end{displaymath}

On peut ensuite définir des formes différentielles de degré supérieur en considérant, pour chaque $p \geq 1$, la $p$\nobreakdash-ème puissance extérieure
\[
\Omega^p_A=\bw{p} \Omega^1_A
\]
du $A$-module $\Omega^1_A$. Concrètement, $\Omega^p_A$ est le $A$-module formé des combinaisons linéaires à coefficients dans $A$ de symboles
\begin{displaymath}
\omega_1 \wedge \cdots \wedge \omega_p, \quad \text{où }\ \omega_j \in \Omega^1_A \ \text{ pour tout }\ j=1, \ldots, p,
\end{displaymath}
et ces combinaisons linéaires sont soumises aux relations de bilinéarité
\begin{align}
\omega_1 \wedge &\cdots \wedge (a\omega_i+a' \omega_i') \wedge \cdots \wedge \omega_p \\
&=\omega_1 \wedge \cdots \wedge a\omega_i \wedge \cdots \wedge \omega_p+\omega_1 \wedge \cdots \wedge a'\omega_i' \wedge \cdots \wedge \omega_p 
\end{align}
pour tous $a, a' \in A$, ainsi qu'à la relation $\omega_1 \wedge \cdots \wedge \omega_p=0$ dès qu'il y a deux $1$-formes différentielles égales parmi les $\omega_j$. De ces contraintes on déduit aisément l'égalité
\[
\omega_{\sigma(1)} \wedge \cdots \wedge \omega_{\sigma(p)}=\varepsilon(\sigma)\omega_1\wedge \cdots \wedge \omega_p
\]
pour toute permutation $\sigma$ de l'ensemble $\{1, \dots, p\}$ de signature $\varepsilon(\sigma)$. On appelle \textit{$p$\nobreakdash-formes différentielles} sur $X$ les éléments de~$\Omega^p_A$.

Par récurrence, la différentielle $d \colon A \to \Omega^1_A$ s'étend de manière unique en des applications $\QQ$-linéaires
\begin{displaymath}
d^p \colon \Omega^p_A \to \Omega^{p+1}_A
\end{displaymath}
satisfaisant aux relations
\[
d^p \circ d^{p-1}=0 \qquad\text{et} \qquad d^p(\alpha \wedge \beta)=d^i\alpha \wedge \beta+(-1)^i \alpha \wedge d^{p-i}\beta
\]
pour tout $p$ et pour tous~$\alpha \in \Omega^i_A$ et $\beta \in \Omega^{p-i}_A$ (pour avoir des notations uniformes, on pose $\Omega^0=A$ et $d^0=d$, ainsi que \hbox{$\Omega^p=0$} et~\hbox{$d^p=0$} pour tout $p<0$). Notons que la seconde de ces relations généralise la règle de Leibniz~\eqref{eqn:leibniz}. En effet, pour construire $d^p$ on pose d'abord
\[
d^1\omega=\sum_{i=1}^n da_i \wedge dx_i
\] pour des $1$-formes différentielles comme dans \eqref{eqn:writing1form}, puis 
\[
d^p(\omega_1 \wedge \cdots \wedge \omega_p)=\sum_{i=1}^p (-1)^{i+1} \omega_1 \wedge \cdots \wedge d^1\omega_i \wedge \cdots \wedge \omega_p,  
\] et on vérifie que ces formules sont compatibles aux relations définissant $\Omega^p_A$. 
On obtient ainsi le \textit{complexe\footnote{Un \emph{complexe (cohomologique) d'espaces vectoriels} est une suite $(A^p, d^p)_{p \in \ZZ}$ d'espaces vectoriels $A^p$ et d'applications linéaires \hbox{$d^p\colon A^p \to A^{p+1}$} telles que la composition $d^{p}\circ d^{p-1}$ est nulle pour tout entier $p$. La cohomologie en degré $p$ d'un tel complexe est alors définie comme le quotient $\ker(d^p)/\mathrm{im}(d^{p-1})$.} de de~Rham algébrique}
\begin{equation}\label{eqn:complexedR}
\Omega^0_A \To{d^0} \Omega^1_A \To{d^1}\Omega^2_A \To{d^2}\cdots.
\end{equation}
Quand on voudra souligner les aspects géométriques plutôt qu'algébriques de la cohomologie de de Rham, on écrira $\Omega^p_X$ à la place de~$\Omega^p_A$; ces deux notations sont entièrement synonymes dans ces notes.

Même si tous les termes dans ce complexe sont des $A$\nobreakdash-modules, les différentielles sont seulement~$\QQ$\nobreakdash-linéaires et pas $A$-linéaires au vu de la règle de Leibniz. L'égalité $d^{p} \circ d^{p-1}=0$ implique, comme on l'a vu pour l'homologie singulière, que l'image de la différentielle $d^{p-1}$ est un sous-espace vectoriel du noyau de $d^{p}$. Il serait donc naturel de définir la cohomologie de de~Rham en degré de $p$ de $X$ comme l'espace quotient $\ker(d^p)/\im(d^{p-1})$. Il se trouve que cette définition n'est raisonnable\footnote{Sinon, la dimension de ce quotient peut être différente de celle de la cohomologie singulière des points complexes $X(\CC)$; d'après \cite{arapura}, c'est le cas par exemple pour le lieu des zéros $X \subset \AA^2$ du polynôme $x^5+y^5+x^2y^2$. Il n'y a donc pas d'isomorphisme de comparaison comme celui de la section \ref{sec:accoupl} avec cette définition.} que si la variété $X$ est lisse (définition \ref{def:varlisse}), auquel cas on peut démontrer que le~$A$-module~$\Omega^1_A$ est \emph{localement libre} de rang la dimension de $X$ (c'est le pendant algébrique de la notion de fibré vectoriel sur une variété différentielle). En particulier, le complexe de de~Rham est nul au\nobreakdash-delà de cette dimension. 

\begin{definition} Soit $X \subset \mathbb{A}^n$ une variété algébrique affine lisse. La \emph{cohomologie de de~Rham algébrique} en degré $p$ de $X$ est le quotient
\begin{displaymath}
\rH^p_{\dR}(X)=\frac{\ker(d^p \colon \Omega^p_A \to \Omega^{p+1}_A)}{\mathrm{im}(d^{p-1} \colon \Omega^{p-1}_A \to \Omega^p_A)}.
\end{displaymath}
\end{definition}

Les éléments dans le noyau de la différentielle s'appellent des formes \textit{fermées} et celles dans son image des formes \textit{exactes}; la cohomologie de de~Rham est donc l'espace des formes fermées modulo les formes exactes. On notera $[\omega]$ la classe d'une $p$-forme différentielle $\omega \in \Omega^p_A$ dans $\rH^p_{\dR}(X)$. La plupart du temps, on omettra l'exposant et on écrira $d$ au lieu de $d^p$. Bien que les termes dans le complexe de de~Rham algébrique soient, sauf dans des cas très particuliers comme celui de l'exemple~\ref{exmp:deRhamdim0}, des $\QQ$\nobreakdash-espaces vectoriels de dimension infinie, on peut démontrer par divers moyens (par exemple, en utilisant l'accouplement parfait du théorème \ref{thm:periods} et le fait que l'homologie singulière d'un CW-complexe de type fini est de dimension finie) que~$\rH^p_{\dR}(X)$ est toujours de dimension finie.

La cohomologie de de~Rham algébrique est fonctorielle (contravariante) par rapport aux morphismes de variétés affines lisses. Gardant les notations du numéro précédent, soient $X$ et $Y$ des telles variétés, d'anneaux de fonctions $A$ et $B$, et soit~$h\colon X \to Y$ un morphisme de variétés représenté par des polynômes $h_1, \dots, h_r$. Le morphisme d'anneaux~$h^\ast \colon B \to A$ donne lieu à une application $\QQ$-linéaire
\[
h^\ast \colon \Omega^1_B \longrightarrow \Omega^1_A
\]
déduite de l'application
\begin{align*}
Bdy_1\oplus \cdots\oplus Bdy_r &\longrightarrow \Omega^1_A \\
bdy_i &\longmapsto h^\ast(b)dh_i
\end{align*}
en observant que celle-ci factorise à travers le quotient par le sous-module engendré par les différentielles $dg_1, \dots, dg_s$. En effet, cette application envoie $dg_j$ sur l'élément
\[
\sum_{i=1}^r \frac{\partial g_j}{\partial y_i}\big(h_1(x_1, \dots, x_r), \dots, h_r(x_1, \dots, x_r)\big)dh_i,
\]
qui est nul dans $\Omega^1_A$ par la condition \eqref{eqn:annulation-morphisme} dans la définition de morphisme de variétés algébriques (en effet, comme le côté gauche de l'équation dans \loccit s'annule sur tous les points complexes $X(\CC)$, il est de la forme $f_1P_1+\dots+f_mP_m$, et sa dérivée, donnée par l'expression ci-dessus par la règle de la chaîne, est donc nulle dans $\Omega^1_A$ au vu des relations $f_j=0$ dans $A$ et $df_j=0$ dans $\Omega^1_A$). L'application ainsi obtenue s'étend, en prenant des puissances extérieures, en des applications $\QQ$-linéaires $h^\ast \colon \Omega^p_B \to \Omega^p_A$ pour tout $p$; elles commutent aux différentielles, d'où une application $\QQ$-linéaire
\begin{displaymath}
h^\ast \colon \rH^p_{\dR}(Y) \longrightarrow \rH^p_{\dR}(X)
\end{displaymath}
de la cohomologie de de~Rham algébrique de $Y$ vers celle de $X$.

\subsection{Premiers exemples}

En guise d'illustration, on reprend d'un point de vue algébrique les exemples \ref{exp:point} et \ref{exmp:lacet0} de la section précédente. On vérifie également que la cohomologie de la droite affine est concentrée en degré $0$, en accord avec le fait que l'espace topologique $\CC$ est contractile.

\begin{exemple}[les variétés de dimension $0$]\label{exmp:deRhamdim0} Soient $f \in \QQ[x]$ un poly\-nôme irréductible de degré $d \geq 1$ et $X \subset \mathbb{A}^1$ la variété algébrique affine définie par l'annulation de $f$; c'est une variété lisse de dimension $0$ car $f$ n'a pas de racine double. Les points complexes de $X$ sont les racines de $f$ dans $\CC$. L'anneau des fonctions sur $X$ est égal à
\begin{displaymath}
A=\QQ[x] \slash (f),
\end{displaymath}
qui est un corps car l'idéal engendré par un polynôme irréductible est maximal. Ce corps est une extension finie de $\QQ$, ayant pour base les fonctions $1, x, \dots, x^{d-1}$. Par définition, le $A$-module des formes différentielles sur~$X$ est le quotient $\Omega^1_A=Adx\slash (df)$. Or, au vu de l'égalité~$df=f' dx$ et du fait que~$f'$ est inversible dans $A$ par l'identité de Bézout, la~relation $df=0$ entraîne~$dx=0$ et donc $\Omega^p_A=0$ pour tout~$p \geq 1$ en prenant des puissances extérieures. Le complexe de de~Rham algébrique de~$X$ est ainsi réduit à $A$ et son seul groupe de cohomologie non nul est $\rH^0_{\dR}(X)=A$, vu comme $\QQ$-espace vectoriel.
\end{exemple}

\begin{exemple}[la droite affine]\label{exmp:cohomA1} Considérons le cas où $X=\mathbb{A}^1$ est la droite affine toute entière, c'est-à-dire qu'aucun polynôme ne s'annule. Le complexe de de~Rham algébrique est alors donné par
\begin{align}
d \colon \QQ[x] &\longrightarrow \QQ[x]dx \\
x^n &\longmapsto n x^{n-1} dx
\end{align}
et on doit calculer le noyau et le conoyau de cette flèche:
\begin{displaymath}
\rH^0_{\dR}(X)=\ker(d), \qquad \rH^1_{\dR}(X)=\QQ[x]dx \slash \im(d).
\end{displaymath}
Comme seuls les polynômes constants ont une dérivée nulle, en degré~$0$ on trouve $\rH^0_{\dR}(X)=\QQ$, et comme tout élément dans~$\QQ[x]$ y admet une primitive, la cohomologie en degré $1$ s'annule: $\rH^1_{\dR}(X)=0$. Ce résultat se généralise à l'espace affine de dimension $n$, dont le seul espace de cohomologie de~Rham non nul est à nouveau $\rH^0_{\dR}(\mathbb{A}^n)=\QQ$.
\end{exemple}

\begin{exemple}[la droite affine épointée]\label{exmp:Gm} Soit $X \subset \mathbb{A}^2$ le lieu des zéros du polynôme $xy-1$. L'égalité $xy=1$ est une façon algébrique d'exprimer le fait que $x$ n'est pas nul. Étant son inverse, $y$ est alors uniquement déterminé par $x$, de sorte que~$X$ n'est rien d'autre que la droite affine épointée~$\mathbb{G}_m=\mathbb{A}^1 \setminus \{0\}$, plongée dans le plan affine comme une hyperbole. Il s'agit d'une variété lisse de dimension $1$ dont l'anneau des fonctions est l'anneau des polynômes de Laurent
\begin{displaymath}
A=\QQ[x, y] \slash (xy-1) \simeq \QQ[x, x^{-1}]
\end{displaymath}
et le $A$-module des formes différentielles est donné par
\begin{displaymath}
\Omega^1_A=(Adx \oplus Ady) \slash (xdy+ydx) \simeq \QQ[x, x^{-1}]dx.
\end{displaymath}
En effet, les relations $xy=1$ dans $A$ et $xdy=-ydx$ dans $\Omega^1_A$ entraînent $dy=-x^{-2} dx$ et la différentielle~$dx$ suffit donc pour engendrer~$\Omega^1_A$. On en déduit que~$\Omega^p_A$ est nul pour tout $p \geq 2$ du fait qu'un produit extérieur de deux formes égales s'annule. Le complexe de de~Rham algébrique est donc le complexe à deux termes
\begin{align*}
d \colon \QQ[x, x^{-1}] &\longrightarrow \QQ[x, x^{-1}]dx \\
x^n &\longmapsto n x^{n-1} dx
\end{align*}
et on doit calculer le noyau et le conoyau de la différentielle:
\begin{displaymath}
\rH^0_{\dR}(X)=\ker(d), \qquad \rH^1_{\dR}(X)=\QQ[x, x^{-1}]dx \slash \im(d).
\end{displaymath}
La cohomologie en degré $0$ est réduite aux fonctions constantes. Quant au premier groupe de cohomologie, l'image de $d$ est l'espace vectoriel engendré par les formes $x^n dx$ pour tout $n \neq -1$; seule la forme différentielle $\sfrac{dx}{x}$, dont la primitive n'est pas un polynôme de Laurent, survit dans le quotient. On a donc des isomorphismes
\begin{displaymath}
\rH^0_{\dR}(X)\simeq \QQ, \qquad \rH^1_{\dR}(X)\simeq\QQ[\sfrac{dx}{x}].
\end{displaymath}
\end{exemple}

\subsection{Des variétés affines aux variétés projectives}\label{sec:varproj}

Soit $n \geq 1$ un entier. L'espace projectif complexe de dimension $n$ est le quotient $\PP^n(\CC)$ de $\CC^{n+1} \setminus \{0\}$ par l'action de $\CC^\times$ par homothétie, c'est-à-dire par la relation d'équivalence
\[
(x_0, \ldots, x_n) \sim (x_0', \ldots, x_n') \Longleftrightarrow \text{il existe $\lambda \in \CC^\times$ tel que $x_i'=\lambda x_i$}.
\]
On note $[x_0 \colon\cdots\colon x_n]$ la classe de $(x_0, \dots, x_n)$ dans ce quotient. On peut écrire $\PP^n(\CC)$ comme la réunion des $n+1$ sous-ensembles
\[
U_i=\{[x_0\colon\cdots\colon x_n] \in \PP^n(\CC) \,|\, x_i \neq 0 \} \qquad (i=0, \dots, n),
\]
qui sont chacun en bijection avec $\CC^n$ par les \emph{cartes}
\begin{align}
\varphi_i\colon U_i &\stackrel{\sim}{\longrightarrow} \CC^n \\
[x_0\colon\cdots\colon x_n] &\longmapsto \Bigl(\frac{x_0}{x_i}, \dots, \widehat{\frac{x_i}{x_i}}, \dots, \frac{x_n}{x_i}\Bigr),
\end{align}
où le chapeau indique que le terme $\sfrac{x_i}{x_i}$ est omis. Pour $i\neq j$, notons $(y_0, \dots, \hat{y}_i, \dots, y_n)$ les coordonnées sur~$\varphi_i(U_i)$. Le sous-ensemble 
\[
\varphi_i(U_i \cap U_j) \subset \CC^n
\] est alors formé des points où $y_j$ n'est pas nul et l'application $\varphi_j \circ \varphi_i^{-1}$ fournit une bijection, dite \emph{changement de cartes},
\begin{align*}
\varphi_i(U_i \cap U_j) &\stackrel{\sim}{\longrightarrow} \varphi_j(U_i \cap U_j) \\
(y_0, \dots, \hat{y}_i, \dots, y_j, \dots, y_n) &\longmapsto \Bigl(\frac{y_0}{y_j}, \dots, \frac{1}{y_j}, \dots, \widehat{\frac{y_j}{y_j}}, \dots, \frac{y_n}{y_j}\Bigr).
\end{align*}
En notant $(z_0, \dots, \hat{z}_j, \dots, z_n)$ les coordonnées sur $\varphi_j(U_j)$, ce changement de cartes permet d'identifier $z_k$ à $\sfrac{y_k}{y_j}$ pour $k \neq i$ et $z_i$ à $\sfrac{1}{y_j}$ sur le sous-espace $\varphi_j(U_i \cap U_j)$ défini par la condition $z_i \neq 0$.

Comme $\CC^n$ est l'ensemble des points complexes de l'espace affine~$\AA^n$, il est naturel de définir l'espace projectif $\PP^n$ sur $\QQ$ comme la variété algébrique obtenue en recollant $n+1$ copies de~$\AA^n$: partant de deux telles \emph{cartes affines}, d'anneaux de fonctions
\[
\QQ[y_0, \dots, \hat{y}_i, \dots, y_n] \quad \text{et} \quad \QQ[z_0, \dots, \hat{z}_j, \dots, z_n],
\]
on identifie les variétés affines définies par la non-annulation de $y_j$ et de $z_i$ par le biais de l'isomorphisme de variétés
\begin{align*}
\mathbb{A}^{j-1} \times \mathbb{G}_m \times \mathbb{A}^{n-j} &\stackrel{\sim}{\longrightarrow} \mathbb{A}^{i-1} \times \mathbb{G}_m \times \mathbb{A}^{n-i} \\ 
(\underbrace{y_0, \dots, \hat{y}_i, \dots,}_{j-1} y_j, \underbrace{\dots, y_n}_{n-j}) &\longmapsto (\underbrace{y_0 w, \dots,}_{i-1} w, \underbrace{\dots, \widehat{y_j w}, \dots, y_nw}_{n-i}), 
\end{align*} où $w$ désigne l'inverse de $y_j$. Sur les anneaux de fonctions, cet isomorphisme correspond à l'application
\[
\QQ[z_0, \dots, \hat{z}_j, \dots, z_n, t]/(z_it-1) \longrightarrow \QQ[y_0, \dots, \hat{y}_i, \dots, y_n, w]/(y_jw-1)
\]
qui envoie $z_k$ sur $y_kw$ pour $k \neq i$ et $z_i$ sur $w$.

Plus généralement, le lieu des zéros d'une famille de polynômes homogènes\footnote{La valeur d'un polynôme en un point de l'espace projectif n'est pas bien définie, car $(x_0, \dots, x_n)$ et $(\lambda x_0, \dots, \lambda x_n)$ représentent le même point. Par contre, si $f$ est un polynôme homogène de degré $d$, alors $f(\lambda x_0, \dots, \lambda x_n)=\lambda^d f(x_0, \dots, x_n)$ pour tout $\lambda$ et on peut donc parler du lieu des zéros de $f$.} à coefficients rationnels
\[
f_1, \dots, f_m \in \QQ[x_0, \dots, x_n]
\]
définit une \emph{variété algébrique projective} $X \subset \PP^n$, à laquelle on peut penser comme le résultat du recollement des $n+1$ cartes affines \hbox{$X_i \subset \AA^n$} définies par l'annulation des polynômes
\[
f_j(x_0, \dots, x_{i-1}, 1, x_{i+1}, \dots, x_n).
\]
Les points complexes de $X$ sont donnés par
\[
X(\CC)=\{[x_0 \colon \cdots\colon x_n] \in \PP^n(\CC) \,|\, f_j(x_0, \dots, x_n)=0 \text{ pour tout }j\}.
\]
Comme tout point $P \in X(\CC)$ appartient à une des cartes affines, la définition \ref{def:varlisse} permet d'étendre la notion de variété lisse aux variétés projectives. Par exemple, le lieu des zéros $X \subset \PP^n$ d'un polynôme homogène non nul $f$ est une variété projective lisse de dimension~$n-1$ si la seule racine commune des polynômes
\begin{displaymath}
\frac{\partial f}{\partial x_0}, \dots, \frac{\partial f}{\partial x_n}
\end{displaymath}
dans $\CC^{n+1}$ est le point $(0, \dots, 0)$; contrairement au cas affine, l'annulation de $f$ est alors automatique compte tenu de la relation d'Euler
\[
\deg(f)f=x_0 \frac{\partial f}{\partial x_0}+\cdots+x_n \frac{\partial f}{\partial x_n}.
\]

\subsection{Cohomologie de de~Rham d'une variété projective}

Soit $X \subset \PP^n$ une variété projective lisse. Si l'on disposait d'une définition de la cohomologie de de~Rham algébrique de $X$ pour laquelle il y aurait une suite exacte longue de Mayer-Vietoris, on pourrait la calculer à partir des cohomologies des cartes affines $X_i \subset \mathbb{A}^n$ et de leurs intersections. Une idée naturelle est alors de la \emph{définir} de sorte que Mayer-Vietoris soit satisfaite. Pour réduire les notations au minimum, on se contentera de traiter le cas où $X$ peut être recouverte par deux cartes affines~$U$ et $V$. C'est notamment le cas si la variété $X$ est de dimension $1$. Pour les courbes planes~\hbox{$X \subset \PP^2$}, on peut le démontrer en observant que la réunion de deux cartes affines de $\PP^2$ recouvre tout l'espace sauf un point où deux coordonnées s'annulent. Si $X(\CC)$ ne contient pas l'un des points $[1\colon 0\colon 0]$, $[0\colon 1\colon 0]$ ou $[0 \colon 0\colon 1]$, on prend pour recouvrement les traces sur $X$ des deux cartes en question; en cas contraire, on se ramène à cette situation par une translation. Pour une courbe générale, on peut choisir une fonction méromorphe non constante et prendre pour cartes le lieu où cette fonction n'a pas de pôles et le lieu où elle ne s'annule pas.\enlargethispage{-3\baselineskip}

Partant des complexes de de~Rham algébriques sur $U$, sur $V$ et sur~$U \cap V$, on est amené à considérer un \emph{complexe double}\footnote{Un complexe double d'espaces vectoriels est la donnée d'une famille $(A^{p, q})_{p, q \geq 0}$ d'espaces vectoriels et d'applications linéaires
\[
d^{p, q}_{\mathrm{ver}} \colon A^{p, q} \to A^{p, q+1} \quad \text{et} \quad d^{p, q}_{\mathrm{hor}} \colon A^{p, q} \to A^{p+1, q},
\]
les différentielles \emph{verticales} et les différentielles \emph{horizontales}, commutant entre elles et telles que les compositions $d^{p, q+1}_{\mathrm{ver}}\circ d^{p, q}_{\mathrm{ver}}$ et $d^{p+1, q}_{\mathrm{hor}}\circ d^{p, q}_{\mathrm{hor}}$ soient nulles pour tous~$p$ et~$q$ (autrement dit, chaque ligne et chaque colonne forme un complexe cohomologique). Le complexe total associé à un complexe double est obtenu en considérant les anti-diagonales, munies des applications
\[
\bigoplus_{p+q=n} A^{p, q} \longrightarrow \bigoplus_{r+s=n+1} A^{r, s}
\]
induites par $(-1)^p d^{p, q}_{\mathrm{ver}}$ et par $d^{p, q}_{\mathrm{hor}}$ sur le facteur $A^{p, q}$. Ce changement de signe garantit que la composition de deux telles applications successives est nulle.}.

C'est le complexe suivant:
\[
\xymatrix@R=.5cm{
\vdots & \vdots \\ \Omega^2_U \oplus \Omega^2_V \ar[u] \ar[r] & \Omega^2_{U \cap V} \ar[u] \\
\Omega^1_U \oplus \Omega^1_V \ar[u] \ar[r] & \Omega^1_{U \cap V} \ar[u] \\
\Omega^0_U \oplus \Omega^0_V \ar[r] \ar[u] & \Omega^0_{U \cap V} \ar[u]
}
\]
en degrés $p=1$ et $p=2$, où les flèches verticales sont données par les différentielles dans les complexes de de Rham algébriques et les flèches horizontales sont les différences $\iota_{U\cap V, U}^\ast-\iota_{U\cap V, V}^\ast$ des morphismes induits par les inclusions de $U \cap V$ dans $U$ et dans $V$. Le \emph{complexe total} associé à ce complexe double est alors égal à
\begin{equation}\label{eqn:complexetotal}
\Omega^0_U \oplus \Omega^0_V \longrightarrow \Omega^1_U \oplus \Omega^1_V \oplus \Omega^0_{U \cap V} \longrightarrow \Omega^2_U \oplus \Omega^2_V \oplus \Omega^1_{U \cap V} \longrightarrow \cdots,
\end{equation}
avec les signes des différentielles verticales changés une fois sur deux, c'est-à-dire que toutes les flèches $\Omega^q_U \oplus \Omega^q_V \to \Omega^{q+1}_U \oplus \Omega^{q+1}_V$ sont affectées d'un signe négatif. Par exemple, les deux premières flèches dans~\eqref{eqn:complexetotal} sont données par les matrices
\[
\setlength\arraycolsep{5pt}
\left(\begin{matrix}
-d^0 & 0 \\
0 & -d^0 \\
\iota_{U\cap V, U}^\ast & -\iota_{U\cap V, V}^\ast
\end{matrix}\right) \quad \text{et} \quad \left(\begin{matrix}
-d^1 & 0 & 0 \\
0 & -d^1 & 0 \\
\iota_{U\cap V, U}^\ast & -\iota_{U\cap V, V}^\ast & d^0
\end{matrix}\right),
\]
et on vérifie qu'avec ce choix de signes leur composition
\[
\setlength\arraycolsep{5pt}
\left(\begin{matrix}
d^1 \circ d^0 & 0 \\
0 & d^1 \circ d^0 \\
d^0\circ\iota_{U\cap V, U}^\ast -\iota_{U\cap V, U}^\ast\circ d^0 & \iota_{U\cap V, V}^\ast\circ d^0
-d^0\circ\iota_{U\cap V, V}^\ast
\end{matrix}\right)
\]
est bien nulle, vu que $d^1 \circ d^0$ l'est et que les applications induites par les restrictions commutent à la différentielle.

\begin{definition}\label{def:dRcasprojectif}
Soit $X \subset \PP^n$ une variété projective lisse obtenue en recollant deux cartes affines $U$ et $V$. La cohomologie de de~Rham algébrique de $X$ en degré $p$ est le quotient
\[
\rH^p_{\dR}(X)=\frac{\ker(\Omega^p_U \oplus \Omega^p_V \oplus \Omega^{p-1}_{U \cap V} \longrightarrow \Omega^{p+1}_U \oplus \Omega^{p+1}_V \oplus \Omega^p_{U \cap V})}{\mathrm{im}(\Omega^{p-1}_U \oplus \Omega^{p-1}_V \oplus \Omega^{p-2}_{U \cap V} \longrightarrow \Omega^p_U \oplus \Omega^p_V \oplus \Omega^{p-1}_{U \cap V})}.
\]
\end{definition}

Par définition, une classe en cohomologie de de~Rham peut donc être représentée par un triplet
\[
\omega=(\omega_U, \omega_V, \omega_{U\cap V}) \quad \text{avec}\quad \omega_U \in \Omega^p_U, \ \omega_V \in \Omega^p_V,\ \omega_{U \cap V} \in \Omega^{p-1}_{U\cap V}
\]
satisfaisant aux relations
\begin{align*}
&d\omega_U=d\omega_V=0, \quad d\omega_{U \cap V}+\iota^\ast_{U \cap V, U}(\omega_U)-\iota^\ast_{U \cap V, V}(\omega_V)=0,
\end{align*}
et deux tels représentants définissent la même classe si et seulement si leur différence est de la forme
\begin{equation}\label{eqn:ambigdR}
\bigl(d\eta_U, d\eta_V, d\eta_{U \cap V}+\iota^\ast_{U \cap V, V}(\eta_V)-\iota^\ast_{U \cap V, U}(\eta_U)\bigr)
\end{equation}
pour des éléments $\eta_U \in \Omega^{p-1}_U, \eta_V \in \Omega^{p-1}_V$ et $\eta_{U\cap V} \in \Omega^{p-2}_{U \cap V}$. 

C'est précisément ce qu'il faut pour avoir une suite exacte longue de Mayer-Vietoris 
\begin{displaymath}
\xymatrix @R=1.25ex @C=1.1ex
{
0 \ar[rr] && \rH^0_{\dR}(X) \ar[rr]^-{\beta} &&\rH_{\dR}^0(U) \oplus \rH_{\dR}^0(V) \ar[rr]^-{\alpha} && \rH^0_{\dR}(U\cap V) \ar@{-}[r] & *{} \ar@{-}`r/9pt[d] `/9pt[l] \\
& *{} \ar@{-}`/9pt[d] `d/9pt[d]& && && & *{} \ar@{-}[llllll] & \\
& *{} \ar[r] & \ar[rr]^-{\beta} \rH^1_{\dR}(X) &&\ar[rr]^-{\alpha}\rH_{\dR}^1(U) \oplus \rH_{\dR}^1(V) && \rH^1_{\dR}(U\cap V)\ar@{-}[r] & *{} \ar@{-} `r/9pt[d] `[l]\\
& *{} \ar@{-}`/9pt[d] `d/9pt[d]& && && & *{} \ar@{-}[llllll] & \\
& *{} \ar[r] & \ar[rr] \rH^2_{\dR}(X) && \cdots && &&
}
\end{displaymath}
dans laquelle les applications $\alpha$ et $\beta$ sont données par
\[
\alpha=\iota^\ast_{U \cap V, V}-\iota^\ast_{U \cap V, U}, \quad \beta([\omega_U, \omega_V, \omega_{U \cap V}])=([\omega_U], [\omega_V])
\]
et la flèche connectante envoie la classe d'une forme fermée \hbox{$\omega \in \Omega^{p-1}_{U \cap V}$} sur la classe du triplet $(0, 0, \omega)$ dans $\rH^p_{\dR}(X)$.

\begin{exemple}[la droite projective]\label{exmp:cohomologieP1} Dans le cas où $X=\PP^1$ est la droite projective, c'est-à-dire qu'aucun polynôme homogène ne s'annule, on peut choisir comme recouvrement les deux cartes affines standard. Écrivons donc $\PP^1=U \cup V$, avec
\begin{align*}
U&=\{[x \colon y] \in \PP^1 \mid y\neq 0 \}=\{[t \colon 1] \mid t \in \mathbb{A}^1\}, \\
V&=\{[x \colon y] \in \PP^1 \mid x\neq 0 \}=\{[1 \colon s] \mid s \in \mathbb{A}^1\}.
\end{align*}
Ces deux copies de la droite affine s'intersectent le long d'une droite affine épointée $\mathbb{G}_m$, sur laquelle les coordonnées~$t$ et $s$ sont reliées par l'isomorphisme~\hbox{$s \mapsto \sfrac{1}{t}$}; c'est ainsi que l'on recolle les deux cartes. En choisissant~$t$ comme coordonnée sur $U \cap V$, on trouve que le complexe total~\eqref{eqn:complexetotal} est donné dans cet exemple par
\begin{displaymath}
\QQ[t] \oplus \QQ[s] \To{a}\QQ[t]dt \oplus \QQ[s]ds \oplus \QQ[t, t^{-1}] \To{b} \QQ[t, t^{-1}]dt,
\end{displaymath}
où $a$ et $b$ sont les flèches\enlargethispage{-3\baselineskip}
\begin{align*}
a(f, g)&=(-f'(t)dt, -g'(s)ds, f(t)-g(\sfrac{1}{t})), \\
b(Pdt, Qds, h)&=(P(t)+Q(\sfrac{1}{t})t^{-2}+h'(t))dt.
\end{align*}
Il s'agit donc de calculer les espaces suivants:
\begin{gather*}
\rH^0_{\dR}(X)=\ker(a), \qquad \rH^1_{\dR}(X)=\ker(b)/\mathrm{im}(a), \\ \rH^2_{\dR}(X)=\QQ[t, t^{-1}]dt /\mathrm{im}(b).
\end{gather*}

Faisons-le d'abord à la main: 
\begin{itemize}
\item Si $(f, g) \in \QQ[t] \oplus \QQ[s]$ est dans le noyau de $a$, alors $f$ et $g$ sont des polynômes constants, vu que leurs dérivées sont nulles, qui prennent la même valeur, car $f(t)=g(1/t)$. On trouve donc $\rH^0_{\dR}(X)=\QQ$. 

\item Si $(Pdt, Qds, h) \in \QQ[t]dt \oplus \QQ[s]ds \oplus \QQ[t, t^{-1}]$ appartient au noyau de $b$, en choisissant pour $f$ une primitive de $-P$ et pour $g$ une primitive de $-Q$, on voit que $a(f, g)=(Pdt, Qds, h)$ en utilisant la condition $P(t)+Q(\sfrac{1}{t})t^{-2}+h'(t)=0$, d'où $\rH^1_{\dR}(X)=0$. 

\item Enfin, l'image de $b$ contient les formes $t^n dt=b(0, 0, \sfrac{t^{n+1}}{(n+1)})$ pour tout $n \neq -1$, mais pas la forme $dt/t$. En effet, l'équation 
\[
P(t)+Q(1/t)t^{-2}+h'(t)=\sfrac{1}{t} 
\] n'admet pas de solution $(P, Q, h) \in \QQ[t] \oplus \QQ[s] \oplus \QQ[t, t^{-1}]$, car aucun des trois termes dans le membre de gauche ne contient de monômes de degré $-1$. On trouve donc $\rH^2_{\dR}(X)=\QQ[dt/t]$. 

\end{itemize}

Une fois que l'on connaît les cohomologies de $\AA^1$ (exemple \ref{exmp:cohomA1}) et de $\mathbb{G}_m$ (exemple~\ref{exmp:Gm}), on peut retrouver ces groupes de cohomologie sans calcul par le biais de la suite de Mayer-Vietoris \begin{displaymath}
\xymatrix @R=1.25ex @C=1.1ex
{
0 \ar[rr] && \rH^0_{\dR}(\PP^1) \ar[rr] &&\rH_{\dR}^0(\AA^1) \oplus \rH_{\dR}^0(\AA^1)\ar@{=}[d] \ar[rr]^-{\alpha} && \rH^0_{\dR}(\mathbb{G}_m)\ar@{=}[d] \ar@{-}[r] & *{} \ar@{-}`r/9pt[d] `[dd] \\
& & &&\QQ \oplus\QQ &&\QQ&&\\
& *{} \ar@{-}`/9pt[d] `d/9pt[d]& && && & *{} \ar@{-}[llllll] & \\
& *{} \ar[r] & \ar[rr] \rH^1_{\dR}(\PP^1) &&\ar[rr] \rH_{\dR}^1(\AA^1) \oplus \rH_{\dR}^1(\AA^1)\ar@{=}[d] && \rH^1_{\dR}(\mathbb{G}_m)\ar@{=}[d] \ar@{-}[r] & *{} \ar@{-}`r/9pt[d] `[dd] \\
& & &&0 &&\QQ[\sfrac{dt}{t}]&& \\
& *{} \ar@{-}`/9pt[d] `d/9pt[d]& && && & *{} \ar@{-}[llllll] & \\
& *{} \ar[r] & \ar[rr] \rH^2_{\dR}(\PP^1) &&0. && &
}
\end{displaymath}
En effet, vu que $\alpha$ est l'application $(x, y) \mapsto y-x$ et que le second morphisme connectant est un isomorphisme, on trouve
\[
\rH^0_{\dR}(\PP^1)=\QQ, \quad \rH^1_{\dR}(\PP^1)=0, \quad \rH^2_{\dR}(\PP^1)=\QQ[\sfrac{dt}{t}].
\]
\end{exemple}

\subsection{L'exemple des courbes elliptiques}\label{exmp:courbeelliptique}

Comme on le verra dans la section \ref{exmp:int-elliptiques}, le pendant algébrique des tores complexes de l'exemple \ref{exmp:torecomplexe} sont les courbes planes définies par une équation de degré $3$. Considérons le polynôme\footnote{Les choix du coefficient dominant $4$ et des signes négatifs dans ce polynôme sont motivés par des raisons historiques, liées à l'uniformisation complexe des courbes elliptiques présentée dans la section \ref{exmp:int-elliptiques}.}
\begin{displaymath}
f(x)=4x^3-ax-b,
\end{displaymath}
où $a$ et $b$ sont des nombres rationnels, et la variété affine $X \subset \mathbb{A}^2$ définie par l'annulation de $y^2-f(x).$ Les dérivées partielles de ce dernier polynôme étant~$2y$ et $-f'$, la variété~$X$ est lisse de dimension~$1$ si et seulement si~$f$ et $f'$ n'ont pas de racine commune, c'est\nobreakdash-à\nobreakdash-dire si le discriminant de $f$ est non nul. On supposera dorénavant que c'est le cas. En termes des coefficients $a$ et~$b$, cette condition se traduit par la non-annulation du discriminant~$a^3-27b^2$ et on dit alors que~$X$ est une \textit{courbe elliptique affine}. Son anneau de fonctions est donné par 
\begin{displaymath}
A=\QQ[x, y] \slash (y^2-f(x)) \simeq \QQ[x] \oplus \QQ[x]y
\end{displaymath}
et son $A$-module des $1$-formes différentielles par
\begin{displaymath}
\Omega^1_A=(A dx \oplus A dy) \slash (2ydy-f'(x)dx).
\end{displaymath}

Pour calculer la cohomologie de de~Rham, on donne d'abord une présentation plus convenable de $\Omega^1_A$. Comme~$f$ et $f'$ sont premiers entre eux, par l'identité de Bézout il existe des polynômes $P, Q \in \QQ[x]$ satisfaisant à~\hbox{$Pf+Qf'=1$.} Regardons la forme différentielle
\begin{displaymath}
\omega=Pydx+2Qdy \in \Omega^1_A.
\end{displaymath}
En utilisant les relations $y^2=f(x)$ dans $A$ et $2ydy=f'(x)dx$ dans~$\Omega^1_A$, on trouve
\begin{align*}
y\omega&=Py^2dx+2yQdy=(Pf+Qf')dx=dx \\
f'\omega&=Pf'ydx+2Qf'dy=2(Py^2+Qf')dy=2dy.
\end{align*}
Introduire $\omega$ est ainsi la façon algébrique de parler de la forme différentielle $dx/y$, qui ne définirait pas \emph{a priori} un élément dans $\Omega^1_A$ car la fonction $y$ n'est pas inversible dans $A$. Les égalités ci-dessus montrent que les générateurs $dx$ et $dy$ de $\Omega^1_A$ s'expriment en termes de $\omega$ et, plus précisément, que toute $1$-forme différentielle s'écrit de manière unique comme $(R+Sy)\omega$ pour certains polynômes~$R, S \in \QQ[x]$. En particulier, $\Omega^p_A$ s'annule pour tout $p \geq 2$ et le complexe de de~Rham algébrique est formé par les deux termes
\begin{align*}
A &\To{d}\Omega^1_A \\ T+Uy &\longmapsto d(T+Uy).
\end{align*}

Il s'agit maintenant de déterminer quels éléments $(R+Sy)\omega \in \Omega^1_A$ ne sont pas de la forme \begin{align*}
d(T+Uy)&=T'dx+U'ydx+Udy \\
&=(T'y+U'y^2+\sfrac{Uf'}{2})\omega \\
&=(U'f+\sfrac{Uf'}{2}+T'y)\omega.
\end{align*}
En prenant $U=0$ et en choisissant pour $T$ une primitive de $S$, on voit d'abord que tous les éléments de la forme $Sy\omega$ appartiennent à l'image de $d$, puis qu'il en va de même pour tous les $R\omega$ avec $R$ un monôme de degré~$\geq 2$. En effet, pour $T=0$ et $U=x^r$, le terme dominant du polynôme qui multiplie $\omega$ est $(4r+6)x^{r+2}$, qui est de degré $\geq 2$ car~$r$ est un entier. Il s'ensuit que le premier groupe de cohomologie de de~Rham de $X$ est le $\QQ$-espace vectoriel de dimension~$2$ engendré par les formes~$\omega$ et $x\omega$:
\begin{displaymath}
\rH^1_{\dR}(X)\simeq \QQ\biggl[\frac{dx}{y}\biggr] \oplus \QQ\biggl[\frac{xdx}{y}\biggr].
\end{displaymath}

Ces deux formes différentielles se distinguent par leur comportement à l'infini. En général, ce que l'on appelle une \textit{courbe elliptique} tout court est le lieu des zéros $E \subset \mathbb{P}^2$ du polynôme homogène
\begin{equation}\label{eqn:projective}
x_1^2x_2-4x_0^3-ax_0x_2^2-bx_2^3
\end{equation}
dans le plan projectif de coordonnées $[x_0 \colon x_1 \colon x_2]$, qui s'obtient en rajoutant à la courbe elliptique affine le point à l'infini~$O=[0 \colon 1 \colon 0]$. Comme on l'a vu pour la droite projective, $E$ est obtenue en recollant deux cartes affines. Au nom des variables près, on récupère l'équation de départ, et donc la courbe elliptique affine $X$, en posant~$x_2=1$ dans~\eqref{eqn:projective}. D'un autre côté, en remplaçant $x_1=1$ et en renommant~$z$ et $t$ les variables $x_0$ et $x_2$, on trouve le polynôme
\begin{equation}\label{eqn:carte-affine-ell}
t-4z^3-azt^2-bt^3.
\end{equation}
Ce polynôme définit encore une variété affine lisse de dimension $1$, cette fois\nobreakdash-ci contenant le point à l'infini $O$ mais pas les trois zéros de la fonction $y$. Sur l'intersection de ces cartes affines, qui est donc la courbe elliptique privée du point à l'infini et des trois zéros de $y$, les deux systèmes de coordonnées sont reliées par le changement de carte
\begin{displaymath}
(x, y) \mapsto (z, t)=(x/y, 1/y).
\end{displaymath}

Une fonction rationnelle sur la courbe elliptique \textit{complète} $E$ a le même nombre de zéros que de pôles comptés avec multiplicité \hbox{\cite[Ch.\,VIII, Prop.\,2.7]{perrin}}. Comme les fonctions $x$ et $y$ ont, respectivement, deux et trois zéros simples sur~$X$, elles doivent avoir un pôle d'ordre $2$ et d'ordre $3$ au point à l'infini: en choisissant la variable $z$ comme coordonnée locale autour de $O$, il existe des séries entières
\begin{displaymath}
g(z)=c_0+c_1z+\cdots \quad\text{et}\quad h(z)=d_0+d_1z+\cdots,
\end{displaymath}
ne s'annulant pas en $z=0$ telles que $x=z^{-2} g(z)$ et $y=z^{-3} h(z)$. En remplaçant ces expressions dans $y^2=4x^3-ax-b$ on trouve que les premiers coefficients satisfont aux relations
\begin{equation}\label{eqn:coefficients}
4c_0^3=d_0^2 \quad \text{et} \quad 6c_0^2c_1=d_0d_1.
\end{equation}
Quitte à renormaliser $z$, on peut supposer $c_0=1$ et $d_0=2$. Dans cette coordonnée locale, la forme différentielle $\omega$ devient
\[
\frac{dx}{y}=\frac{d(z^{-2} g(z))}{z^{-3}h(z)}=\frac{-2g(z)+zg'(z)}{h(z)}dz=\bigl(-1+\cdots\bigr)dz,
\]
ce qui montre qu'elle n'a pas de pôle à l'infini; dans le langage classique, on dit que $\omega$ est une différentielle \textit{de première espèce}. Par contre, la forme $x\omega$ a un pôle double à l'infini, comme le calcul
\begin{displaymath}
x\frac{dx}{y}=\frac{-2g(z)^2+zg(z)g'(z)}{z^2h(z)}dz.
\end{displaymath}
le montre. Son résidu est le coefficient de $\sfrac{1}{z}$ dans la série
\begin{displaymath}
\frac{-2g(z)^2+zg(z)g'(z)g(z)}{z^2h(z)}=-\frac{1}{z^2}+\frac{d_1-3c_1}{2}\frac{1}{z}+\cdots,
\end{displaymath}
qui est nul par les relations \eqref{eqn:coefficients} ci-dessus. Ainsi, $x\sfrac{dx}{y}$ n'a pas de résidu; une telle forme s'appelle une différentielle de \textit{deuxième espèce}.

Calculons maintenant la cohomologie de de~Rham algébrique de la courbe elliptique projective $E \subset \PP^2$ en termes du recouvrement ci-dessus. On écrit donc~$E=U \cup V$, avec $U=E \setminus \{O\}$ et $V$ la courbe elliptique privée des trois points $[e_i \colon 0\colon1]$ avec $f(e_i)=0$. Les anneaux de fonctions sur $U$, sur $V$ et sur $U \cap V$ sont égaux à
\begin{align*}
A&=\QQ[x, y]/(y^2-4x^3+ax+b), \\
B&=\QQ[z, t]/(t-4z^3-azt^2-bt^3), \\
C&=\QQ[x, y, w]/(y^2-4x^3+ax+b, yw-1),
\end{align*}
et avec les notations du numéro précédent, il s'agit de calculer la cohomologie du complexe
\[
\xymatrix{
\Omega^0_U \oplus \Omega^0_V \ar[rrr]^-{\left(\begin{smallmatrix} -d & 0 \\ 0 & -d \\ \iota_{U\cap V, U}^\ast & -\iota_{U\cap V, V}^\ast \end{smallmatrix}\right)} & & & \Omega^1_U \oplus \Omega^1_V \oplus \Omega^0_{U \cap V} \ar[rr]^-{\left(\begin{smallmatrix} \iota_{U\cap V, U}^\ast \\ -\iota_{U\cap V, V}^\ast \\ d \end{smallmatrix}\right)} & & \Omega^1_{U \cap V}.
}
\]

Le noyau de la première flèche est isomorphe à $\QQ$, puisque les seules fonctions à dérivée nulle sur $A$ et $B$ sont les constantes et que leurs valeurs doivent coïncider pour que la paire appartienne au noyau de~$\iota_{U\cap V, U}^\ast-\iota_{U\cap V, V}^\ast$; on a donc un isomorphisme $\rH^0_{\dR}(E) \simeq \QQ$.

Pour déterminer la cohomologie en degré $1$, on démontrera l'analogue du fait que l'homologie singulière en degré $1$ d'un tore ne change pas en lui ôtant un point (exemple \ref{exmp:torecomplexe}), à savoir:

\begin{proposition}\label{prop:H1courbeelliptiqueaffine} L'application
\begin{equation}\label{eqn:1stand2ndkind}
\begin{aligned}
\rH^1_{\dR}(E) &\longrightarrow \rH^1_{\dR}(U) \\ [(\omega_U, \omega_V, g_{U \cap V})] &\longmapsto [\omega_U]
\end{aligned}
\end{equation}
est un isomorphisme.
\end{proposition}

\begin{proof} Si la classe de cohomologie dans $\rH^1_{\dR}(E)$ représentée par un triplet $(\omega_U, \omega_V, g_{U \cap V})$ appartient au noyau de cette application, alors il existe une fonction~$h_U$ avec \hbox{$dh_U=\omega_U$}. On déduit alors de la relation $dg_{U \cap V}+ \iota_{U\cap V, U}^\ast(\omega_U)- \iota_{U\cap V, V}^\ast(\omega_V)=0$ que la fonction $\varphi_{U \cap V}=g_{U \cap V}+\iota_{U\cap V, U}^\ast(h_U)$ satisfait à $d\varphi_{U \cap V}=\iota_{U\cap V, V}^\ast(\omega_V)$. Par continuité, la fonction~$\varphi_{U \cap V}$ s'étend en une fonction~$\varphi_V$ sur~$V$ satisfaisant à $d\varphi_V=\omega_V$ (en effet, si $\varphi_{U \cap V}$ avait un pôle sur~$V$, il en serait de même pour la forme $\omega_V$). La classe de départ est donc
\[
(\omega_U, \omega_V, g_{U \cap V})=(dh_U, d\varphi_V, \iota_{U\cap V, V}^\ast(\varphi_V)-\iota_{U\cap V, U}^\ast(h_U))
\] et les triplets de cette forme représentent la classe zéro en cohomologie. On a donc démontré que l'application est injective.

Pour établir la surjectivité, il suffit de démontrer que les générateurs $\sfrac{dx}{y}$ et $\sfrac{xdx}{y}$ de $\rH^1_{\dR}(U)$ sont dans l'image. Comme $\sfrac{dx}{y}$ n'a pas de pôle à l'infini, on peut considérer le triplet $(\sfrac{dx}{y}, \sfrac{dx}{y}, 0)$. Ce n'est plus possible pour la forme $\sfrac{xdx}{y}$, car elle a un pôle à l'infini, mais on peut prendre à la place un triplet
\[
(\sfrac{xdx}{y}, \sfrac{xdx}{y}-dg_{U \cap V}, g_{U \cap V})
\]
pour une fonction~$g_{U \cap V}$ telle que $\sfrac{xdx}{y}-dg_{U \cap V}$ n'ait pas de pôle à l'infini. Par exemple, $g_{U \cap V}=\sfrac{2x^2}{y}=\sfrac{1}{z}+\cdots$ fait l'affaire.
\end{proof}

Il nous reste à calculer la cohomologie en degré $2$; par analogie avec l'exemple \ref{exmp:torecomplexe}, on s'attend à trouver un espace de dimension~$1$. Vu que toute classe dans $\rH^2_{\dR}(E)$ est représentée par un élément de~$\Omega^1_C$, on commence par expliciter ce module. L'anneau $C$ est obtenu en inversant $y$ dans $A$; tous ses éléments s'écrivent donc de manière unique comme $\psfrac{P+Qy}{y^{2n}}$ pour des polynômes $P, Q \in \QQ[x]$ et un entier $n \geq 0$. De même, une forme différentielle admet une écriture unique $\psfrac{R+Sy}{y^{2n}}dx$. Par des manipulations similaires à celles qui nous ont permis de calculer la cohomologie de la courbe elliptique affine, on obtient alors un isomorphisme
\[
\rH^1_{\dR}(U \cap V)\simeq \QQ\biggl[\frac{dx}{y}\biggr] \oplus \QQ\biggl[\frac{xdx}{y} \biggr] \oplus \QQ\biggl[\frac{dx}{y^2}\biggr] \oplus \QQ\biggl[\frac{xdx}{y^2}\biggr]\oplus \QQ\biggl[\frac{x^2dx}{y^2}\biggr].
\]
Parmi ces classes, celles définies par une forme différentielle n'ayant pas de pôle sur $U$ ou sur $V$ ne contribuent pas à $\rH^2_{\dR}(E)$, puisqu'elles sont dans l'image de $\iota_{U\cap V, U}^\ast$ ou de $\iota_{U\cap V, V}^\ast$; c'est le cas, comme on le sait, pour les deux premières. Il en va de même pour les deux suivantes: en remplaçant $x=z^{-2}g(z)$ et $y=z^{-3}h(z)$, on voit qu'elles n'ont pas de pôle à l'infini. Par contre,
\begin{equation}\label{eqn:res-infty}
x^2\frac{dx}{y^2}=\frac{-2z^{-1}g(z)^3+g(z)^2g'(z)}{h(z)^2}dz=-\frac{1}{2}\frac{dz}{z}+\cdots
\end{equation}
a bien un pôle en ce point. L'espace $\rH^2_{\dR}(E)$ est, par conséquent, au plus engendré par la classe de $\sfrac{x^2dx}{y^2}$. Or, cette classe n'est pas nulle, car si la forme pouvait s'écrire comme
\[
dg_{U \cap V}+\iota_{U\cap V, V}^\ast(\omega_V)-\iota_{U\cap V, U}^\ast(\omega_U),
\]
alors elle aurait résidu zéro à l'infini. En effet, $\omega_V$ n'a pas de pôle en ce point, $\omega_U$ a résidu zéro car elle a au plus un pôle et que la somme des résidus d'une forme différentielle est toujours nulle, et $dg_{U \cap V}$ a résidu zéro en tout pôle, comme on le voit en dérivant le développement en série de Laurent de la fonction autour du point en question.

Ces propriétés de la cohomologie de de~Rham des courbes ellip\-tiques se généralisent comme suit aux courbes projectives lisses \hbox{$X \subset \PP^2$} définies par l'annulation d'un polynôme homogène $f$ de degré $d \geq 1$. Si $U$ désigne une des cartes affines recouvrant $X$, la formule \eqref{eqn:1stand2ndkind} définit encore une application injective\footnote{La démonstration de l'injectivité est identique à celle de la proposition \ref{prop:H1courbeelliptiqueaffine}, où le fait que $X$ soit une courbe elliptique n'était utilisé que pour la surjectivité.}
\begin{equation}\label{eqn:1stand2ndkind2}
\rH^1_{\dR}(X) \longrightarrow \rH^1_{\dR}(U).
\end{equation}
Cette application n'est pas surjective, à moins que $U$ soit obtenue en ôtant un seul point à $X$. En général, $D=X \setminus U$ est une variété lisse de dimension $0$; par exemple, si $U$ est la trace sur $X$ de la carte affine du plan projectif où $x_2$ est non nul, ce complémentaire s'identifie au lieu des zéros $D \subset \mathbb{A}^1$ du polynôme en une variable~$f(x, 1, 0)$. À l'aide du \emph{théorème de Riemann-Roch} \cite[Ch.\,VIII]{perrin}, on peut démontrer que l'espace $\rH^1_{\dR}(X)$ est de dimension $(d-1)(d-2)$ et caractériser l'image de~\eqref{eqn:1stand2ndkind2} comme le sous-espace des classes des formes différentielles de deuxième espèce, c'est-à-dire les formes sur~$U$ dont le résidu en tout point de $D(\CC)$ est nul. Parmi ces formes, la \og moitié\fg n'ont pas de pôles; ce sont les différentielles de première espèce. L'application~$\eqref{eqn:1stand2ndkind2}$ s'inscrit dans une suite exacte longue dite de~\emph{Gysin}
\begin{equation}\label{eqn:GysinDR}
0 \to \rH^1_{\dR}(X) \to \rH^1_{\dR}(U) \to \rH^0_{\dR}(D) \to \rH^2_{\dR}(X) \to 0,
\end{equation}
 où la flèche au milieu envoie la classe d'une forme différentielle $\omega$ sur~$U$ vers la fonction qui à un point de $D$ associe le résidu de $\omega$ en ce point. Celle-ci n'est pas surjective car la somme des résidus est toujours nulle. Si l'on considère ces variétés comme étant définies sur le corps des nombres complexes, la flèche de $\rH^0_{\dR}(D)$ dans~$\rH^2_{\dR}(X)$ peut être construite comme suit. On choisit un point auxiliaire $P \in X(\CC)$ en dehors de $D(\CC)$. Parmi les conséquences du théorème de Riemann-Roch, on trouve le fait que $V=X \setminus \{P\}$ est une courbe affine\footnote{C'est plus compliqué que cela n'en a l'air: bien que $X \subset \PP^2$ soit une courbe projective plane, $X \setminus \{P\}$ se plonge dans un espace affine $\mathbb{A}^n$ mais on ne peut pas en général choisir $n=2$; il faudrait pour ce faire être capable de trouver une droite dans $\PP^2$ qui n'intersecte $X$ qu'au point $P$, ce qui n'existe que très rarement...} et que, pour n'importe quelle fonction \hbox{$h \colon D(\CC) \to \CC$}, il existe une~$1$\nobreakdash-forme différentielle $\omega_h$ sur~$X$ n'ayant des pôles qu'en~$P$ et les points de $D(\CC)$, avec les résidus prescrits par $h$ le long de $D$. Prenant~$U \cup V$ pour recouvrement affine de $X$, le triplet $(0, 0, \omega_h)$ définit une classe dans $\rH^2_{\dR}(X)$ qui ne dépend pas du choix de~$P$; c'est l'image de $h$ par l'application $\rH^0_{\dR}(D) \to \rH^2_{\dR}(X)$. Enfin, l'espace~$\rH^2_{\dR}(X)$ est de dimension $1$; comme on le verra dans la remarque~\ref{rem:classefonda}, il peut être identifié canoniquement à~$\QQ$.

\section{L'accouplement de périodes}\label{sec:accoupl}

Ayant construit l'homologie singulière et la cohomologie de de Rham algébrique, qui sont le cadre naturel pour le domaine d'intégration et l'intégrande dans la définition des périodes, il ne nous reste plus qu'à interpréter le processus d'intégration comme un accouplement entre ces deux espaces vectoriels. 

\subsection{L'isomorphisme de comparaison de Grothendieck} 

Soit $X \subset \mathbb{A}^n$ une variété affine lisse, définie par l'annulation des polynômes à coefficients rationnels $f_1, \dots, f_m$. Muni de la topologie de sous-espace de $\CC^{n}$, l'ensemble de ses points complexes
\begin{displaymath}
X(\CC)=\{(z_1, \ldots, z_n) \in \CC^n \mid f_j(z_1, \ldots, z_n)=0 \text{ pour tout } j\}
\end{displaymath}
est un espace topologique (même une variété complexe) auquel on peut associer des groupes d'homologie singulière comme dans la section \ref{sec:homsing}. Ils portent dans ce contexte le nom suivant: 

\begin{definition} L'\textit{homologie de Betti} de~$X$ est l'homologie singulière de $X(\CC)$ à coefficients rationnels, c'est-à-dire le~$\QQ$-espace vectoriel
\[
\rH_p^\Betti(X)=\rH_p(X(\CC), \QQ).
\]
De même, la \textit{cohomologie de Betti} de $X$ est la cohomologie singulière de $X(\CC)$ à coefficients rationnels, c'est-à-dire
\[
\rH^p_\Betti(X)=\rH^p(X(\CC), \QQ)=\Hom(\rH_p^\Betti(X), \QQ).
\]
\end{definition}

D'après un théorème de Łojasiewicz \cite{Loj64}, l'espace topologique $X(\CC)$ a le type d'homotopie d'un CW-complexe de type fini; l'homo\-logie et la cohomologie de Betti d'une variété affine lisse~$X$ sont donc des $\QQ$-espaces vectoriels de dimension finie. 

Rappelons que tous les éléments de $\rH_p^\Betti(X)$ admettent pour représentants des $n$\nobreakdash-chaînes lisses le long desquelles on peut intégrer des formes différentielles: pour chaque combinaison linéaire $\sigma=\sum n_i f_i$ d'applications $f_i \colon \Delta^p \to X(\CC)$ de classe~$\mathcal{C}^\infty$, on définit
\begin{displaymath}
\int_\sigma \omega=\sum n_i \int_{\Delta^p} f_i^\ast \omega.
\end{displaymath} L'hypothèse de lissité sur $f_i$ garantit que le tiré en arrière $f_i^\ast \omega$ est bien défini et que son intégrale sur le simplexe $\Delta^p$ converge. Par exemple, sur la droite affine épointée $X=\mathbb{G}_m$, le tiré en arrière de la forme différentielle $\omega=dx/x$ par un lacet $f \colon [0, 1] \to X(\CC)$ de classe $\mathcal{C}^\infty$ est la forme $f^\ast \omega=f'(t)/f(t)dt$ sur $[0, 1]$. 

Il résulte de la formule de Stokes que la valeur de cette intégrale ne dépend de $\sigma$ et de $\omega$ qu'à travers leurs classes en homologie de Betti et en cohomologie de de~Rham. En effet, si l'on remplace la forme différentielle $\omega$ et le cycle $\sigma$ par d'autres représentants $\omega+d\eta$ et \hbox{$\sigma+\partial \tau$} des mêmes classes, l'intégrale ne change pas:
\begin{align*}
\int_{\sigma+\partial \tau} (\omega+d\eta)&=\int_\sigma \omega+\int_{\sigma} d\eta+\int_{\partial \tau} \omega+\int_{\partial \tau} d\eta \\
&=\int_\sigma \omega+\int_{\partial\sigma} \eta+\int_{\tau} d\omega+\int_{\partial^2 \tau} \eta \\
&=\int_\sigma \omega.
\end{align*}
Ci-dessus, la deuxième égalité découle de la formule de Stokes, et la troisième, des annulations $\partial \sigma=0$ (car la chaîne~$\sigma$ est un cycle),~\hbox{$d\omega=0$} (car la forme $\omega$ est exacte) et $\partial^2=0$.

\begin{thm}[Grothendieck]\label{thm:periods} Soit $X \subset \mathbb{A}^n$ une variété affine lisse. L'intégration induit un accouplement parfait
\begin{equation}\label{eqn:periodspairingGro}
\begin{aligned}
\rH^p_{\dR}(X) \otimes \rH_p^\Betti(X) &\longrightarrow \CC \\
([\omega], [\sigma]) &\longmapsto \int_{\sigma} \omega.
\end{aligned}
\end{equation}
\end{thm}

Être \emph{parfait} signifie que les $\QQ$-espaces vectoriels~$\rH^p_{\dR}(X)$ et $\rH_p^\Betti(X)$ ont même dimension et que, quel que soit le choix des bases $\{[\omega_i]\}$ de~$\rH^p_{\dR}(X)$ et~$\{[\sigma_j]\}$ de $\rH_p^\Betti(X)$, la matrice
\begin{displaymath}
\biggl(\int_{\sigma_j} \omega_i \biggr)
\end{displaymath}
est inversible; on l'appelle \textit{matrice des périodes}. La donnée d'un tel accouplement permet d'identifier canoniquement l'espace vectoriel complexe~$\rH^p_{\dR}(X) \otimes \CC$ aux applications~$\QQ$\nobreakdash-linéaires de $\rH_p^\Betti(X)$ dans~$\CC$, c'est-à-dire à l'espace $\rH^p_\Betti(X) \otimes \CC$, au sens où l'application
\begin{equation}\label{eqn:periodscompisom}
\begin{aligned}
\mathrm{comp}\colon \rH^p_{\dR}(X) \otimes \CC &\longrightarrow \rH^p_\Betti(X) \otimes \CC\\
[\omega] &\longmapsto \biggl([\sigma] \longmapsto \int_{\sigma} \omega\biggl)
\end{aligned}
\end{equation}
est un isomorphisme de $\CC$-espaces vectoriels; on l'appelle l'\emph{isomor\-phisme de comparaison} de Grothendieck. Comme les $\QQ$-espaces vectoriels $\rH^p_{\dR}(X)$ et~$\rH^p_\Betti(X)$ ont même dimension, ils sont aussi isomorphes, mais on ne sait écrire un isomorphisme naturel (par exemple, un qui commute avec les applications induites en cohomologie par tous les morphismes de variétés...) qu'après avoir étendu les scalaires aux nombres complexes.

Le théorème \ref{thm:periods} combine deux résultats: le premier est le \emph{lemme de Poincaré}, originellement dû à de~Rham \cite{deRham} pour les variétés différentielles et à Dolbeault \cite{Dolbeault} pour les variétés complexes, qui affirme que l'intégration fournit un accouplement parfait entre l'homologie de Betti et la cohomologie de de~Rham $\mathcal{C}^\infty$ ou \textit{analytique} de la variété complexe~$X(\CC)$, c'est\nobreakdash-à\nobreakdash-dire celle où les coefficients des formes différentielles ne sont pas des polynômes mais des fonctions de classe $\mathcal{C}^\infty$ ou holomorphes quelconques. Le deuxième est un théorème de Grothendieck~\cite{GrodR} d'après lequel la cohomologie de de~Rham analytique coïncide avec la cohomologie de de~Rham algébrique telle qu'elle a été définie dans ces notes. Cela n'a rien d'évident puisque le complexe de de~Rham analytique est bien plus gros. Pour le plan complexe épointé, par exemple, il s'agit du complexe à deux termes
\[
\mathcal{H}(\CC^\times) \to \mathcal{H}(\CC^\times)dz, 
\]
où $\mathcal{H}(\CC^\times)$ désigne l'espace des fonctions holomorphes sur $\CC^\times$, c'est\nobreakdash-à\nobreakdash-dire toutes les \emph{séries} de Laurent $\sum_{n \in \ZZ} a_n z^n$, par opposition au complexe de de~Rham algébrique qui ne fait intervenir que des polynômes de Laurent (celles où il n'y a qu'un nombre fini de coefficients en degré négatif). Les deux complexes calculent néanmoins la même cohomologie, vu qu'une fonction dont la dérivée s'annule est constante et que la seule obstruction pour qu'une série de Laurent admette une autre série de Laurent pour primitive est le terme $\sfrac{1}{z}$.

\begin{exemple}[les nombres algébriques]\label{exmp:algebriques} En degré $0$, l'accouplement de périodes associe à la classe d'une fonction $g$ sur la variété $X$ et à la classe d'un point complexe $P \in X(\CC)$ la valeur $g(P) \in \CC$. Le résultat est particulièrement intéressant lorsque $X \subset \mathbb{A}^1$ est le lieu des zéros d'un polynôme irréductible $f \in \QQ[x]$ de degré $d$. D'un côté, la cohomologie de de~Rham algébrique de $X$ est égale à
\begin{displaymath}
\rH^0_{\dR}(X)=\QQ[x] \slash (f)=\QQ[1] \oplus \QQ[x] \oplus \dots\oplus \QQ[x^{d-1}],
\end{displaymath}
comme on l'a vu dans l'exemple \ref{exmp:deRhamdim0}. De l'autre, $X(\CC)$ est l'ensemble fini $\{\alpha_1, \ldots, \alpha_d\}$ formé des $d$ racines distinctes de~$f$, de sorte que l'homologie de Betti $\rH_0^\Betti(X)$ est le $\QQ$-espace vectoriel de dimension~$d$ engendré par les classes de ces points. On voit déjà que les deux espaces vectoriels ont même dimension. La matrice de l'accouplement de périodes par rapport à ces bases est composée des valeurs des fonctions $1, x, \ldots, x^{d-1}$ aux points $\alpha_1, \ldots, \alpha_d$:
\begin{displaymath}
\left(\begin{matrix}
1 & & \alpha_1 & & \alpha_1^2 & & \cdots & & \alpha_1^{d-1} \\
1 & & \alpha_2 & & \alpha_2^2 & & \cdots & & \alpha_2^{d-1} \\
\vdots & & \vdots & & \vdots & & & & \vdots \\
1 & & \alpha_d & & \alpha_d^2 & & \cdots & & \alpha_d^{d-1}
\end{matrix}\right).
\end{displaymath}
C'est une matrice de Vandermonde, de déterminant $$\prod_{1 \leq i<j \leq d} (\alpha_j-\alpha_i),$$ qui est non nul car toutes les racines sont distinctes; il s'agit bien d'un accouplement parfait. Les nombres algébriques sont donc les périodes des variétés de dimension $0$ sur $\QQ$; ils ne se montrent pas tous seuls mais toujours en compagnie de leurs \textit{conjugués}!
\end{exemple}

\begin{exemple}[le nombre $2\pi i$]\label{example:nombrepi} Soit $X=\mathbb{G}_m$ la droite affine épointée. Comme on l'a vu dans l'exemple~\ref{exmp:Gm}, la cohomologie de de~Rham algébrique $\rH^1_{\dR}(X)$ est le $\QQ$-espace vectoriel de dimension $1$ engendré par la classe de la forme différentielle~$\sfrac{dx}{x}$. Par ailleurs, l'espace des points complexes est le plan complexe épointé~\hbox{$X(\CC)=\CC^\times$,} dont le premier groupe d'homologie de Betti~$\rH_1^\Betti(X)$ est, d'après l'exemple~\ref{exmp:lacet0}, le $\QQ$-espace vectoriel de dimension $1$ engendré par un lacet $\sigma$ tournant une fois autour de zéro dans le sens anti-horaire. Par rapport à ces bases, la matrice des périodes a pour seul coefficient
\begin{displaymath}
\int_\sigma \frac{dx}{x}=\int_0^1 \frac{d(e^{2\pi i t})}{e^{2\pi i t}}=2\pi i,
\end{displaymath}
qui est un nombre transcendant d'après le théorème de Lindemann. Bien que la cohomologie de de~Rham et l'homologie de Betti soient des $\QQ$-espaces vectoriels, l'accouplement de périodes~\eqref{eqn:periodspairingGro} ne prend donc pas de valeurs rationnelles, ni même algébriques.
\end{exemple}

\subsection{Le cas des variétés projectives} 

La définition de l'accouplement de périodes s'étend aux variétés projectives lisses $X \subset \PP^n$. À nouveau, on se contentera d'expliquer le cas où $X$ est recouverte par deux cartes affines $U$ et $V$. On a vu, après la définition \ref{def:dRcasprojectif}, que toute classe en cohomologie de de~Rham est alors représentée par un triplet $\omega=(\omega_U, \omega_V, \omega_{U \cap V})$. De même, on peut calculer l'homologie singulière de l'espace topologique $X(\CC)$ par le biais du complexe double\vspace*{-10pt}
\[
\xymatrix@R=.5cm{
\vdots \ar[d] & \vdots \ar[d] \\ C_2(U(\CC)) \oplus C_2(V(\CC)) \ar[d] & C_2(U(\CC) \cap V(\CC)) \ar[d] \ar[l] \\
C_1(U(\CC)) \oplus C_1(V(\CC)) \ar[d] & C_1(U(\CC) \cap V(\CC)) \ar[d] \ar[l] \\
C_0(U(\CC)) \oplus C_0(V(\CC)) & C_0(U(\CC) \cap V(\CC)), \ar[l]
}
\]
où les flèches verticales sont données par les applications bord dans les complexes de chaînes singulières et les flèches horizontales sont les différences $(\iota_{U \cap V, U})_\ast-(\iota_{U \cap V, V})_\ast$ des morphismes induits par les inclusions. Après passage au complexe total, n'importe quelle classe en homologie de Betti $\rH_p^\Betti(X)$ peut donc être représentée par un triplet $\sigma=(\sigma_U, \sigma_V, \sigma_{U \cap V})$ avec\vspace*{-3pt}
\begin{align}
& \sigma_U \in C_p(U(\CC)), \ \sigma_V \in C_p(V(\CC)),\ \sigma_{U \cap V} \in C_{p-1}(U(\CC) \cap V(\CC))
\end{align}
satisfaisant aux relations
\begin{equation}\label{eqn:repsBettiMV}
\begin{split}
\partial\sigma_{U \cap V}=0,\quad\partial\sigma_U&=-(\iota_{U \cap V, U})_\ast(\sigma_{U\cap V}),\\
\partial\sigma_V&=(\iota_{U \cap V, V})_\ast(\sigma_{U\cap V}),
\end{split}
\end{equation}
et deux tels représentants définissent la même classe d'homologie si et seulement si leur différence est de la forme\vspace*{-3pt}
\begin{equation}\label{eqn:BettiMVcoboundaries}
\bigl(\partial \tau_U+(\iota_{U \cap V, U})_\ast(\tau_{U\cap V}), \partial\tau_V-(\iota_{U \cap V, V})_\ast(\tau_{U \cap V}), \partial\tau_{U \cap V}\bigr)
\end{equation} pour des chaînes singulière $\tau_U \in C_{p+1}(U(\CC))$, $\tau_V \in C_{p+1}(V(\CC))$ et~$\tau_{U \cap V} \in C_p(U(\CC) \cap V(\CC))$. Comme d'habitude, on pourra supposer que toutes les chaînes intervenant dans ces formules sont lisses. 

En termes de ces représentants, la flèche $\beta$ et la flèche connectante dans la suite exacte de Mayer-Vietoris\vspace*{-5pt}
\begin{displaymath}
\xymatrix @R=1.25ex @C=1.1ex
{
& &&&\cdots \ar[rr] && \rH_{p+1}^\Betti(X) \ar@{-}[r] & *{} \ar@{-}`r/9.5pt[d] `/9.5pt[l] \\
& *{} \ar@{-}`/9.5pt[d] `d/9.5pt[d]& && && & *{} \ar@{-}[llllll] & \\
& *{} \ar[r] & \ar[rr]^-{\alpha} \rH_{p}^\Betti(U \cap V) &&\ar[rr]^-{\beta}\rH_p^\Betti(U) \oplus \rH_p^\Betti(V) && \rH_p^\Betti(X) \ar[r] &\cdots
}
\end{displaymath} de la section \ref{sec:homotopyMayerVietoris} sont, respectivement, données par\vspace*{-3pt}
\[
\beta([\sigma_U, \sigma_V])=[(\sigma_U, \sigma_V, 0)] \quad \text{et par} \quad [(\sigma_U, \sigma_V, \sigma_{U \cap V})] \mapsto [\sigma_{U \cap V}], 
\] qui sont bien définies grâce aux relations \eqref{eqn:repsBettiMV}. 

Munis de cette description de l'homologie de Betti d'une variété projective recouverte par deux cartes affines, on définit l'intégrale de~$\omega=(\omega_U, \omega_V, \omega_{U \cap V})$ le long de $\sigma=(\sigma_U, \sigma_V, \sigma_{U \cap V})$ par la formule
\[
\int_\sigma \omega=\int_{\sigma_U} \omega_U+\int_{\sigma_V} \omega_V+\int_{\sigma_{U \cap V}} \omega_{U \cap V}.
\]
Le fait que cette définition ne dépende que de la classe de~$\omega$ en coho\-mologie, c'est-à-dire que l'intégrale d'un élément du type~\eqref{eqn:ambigdR} le long de n'importe quelle classe $\sigma$ soit nulle, résulte du calcul
\begin{align*}
\int_{\sigma_U} d\eta_U&+\int_{\sigma_V} d\eta_V+\int_{\sigma_{U \cap V}} d\eta_{U \cap V}\\
&\hspace*{1.8cm}+\int_{\sigma_{U \cap V}}\iota^\ast_{U \cap V, U}(\eta_U)-\int_{\sigma_{U \cap V}}\iota^\ast_{U \cap V, V}(\eta_V) \\
&=\int_{\partial \sigma_U} \eta_U+\int_{\partial\sigma_V} \eta_V+\int_{\partial\sigma_{U \cap V}} \eta_{U \cap V} \\
&\hspace*{1.8cm}+\int_{(\iota_{U \cap V, U})_\ast(\sigma_{U \cap V})} \eta_U-\int_{(\iota_{U \cap V, V})_\ast(\sigma_{U \cap V})} \eta_V \\
&=0,
\end{align*}
où l'on a utilisé la formule de Stokes, la formule du changement de variables et les identités \eqref{eqn:repsBettiMV}. Un calcul similaire montre que la valeur de l'intégrale ne dépend de $\sigma$ qu'à travers sa classe en homologie, c'est-à-dire que l'intégrale de $\omega$ le long de tout élément du type~\eqref{eqn:BettiMVcoboundaries} est nulle, d'où un accouplement bien défini
\begin{align*}
\rH^p_{\dR}(X) \otimes \rH_p^\Betti(X) &\longrightarrow \CC \\
([\omega], [\sigma]) &\longmapsto \int_{\sigma} \omega.
\end{align*}
On peut déduire du théorème \ref{thm:periods} pour les variétés affines que c'est également un accouplement parfait; c'est l'approche originale de Grothendieck. En pratique, on procède souvent à l'inverse: on démontre que l'accouplement de périodes sur les variétés affines est parfait en les \emph{compactifiant} par des variétés projectives.

\begin{exemple}[périodes de la droite projective]\label{exmp:classecanonique} Soit $X=\PP^1$ la droite projective, recouverte par les cartes $U=X \setminus \{0\}$ et $V=X \setminus \{\infty\}$. D'après l'exemple \ref{exmp:cohomologieP1}, l'espace $\rH^2_{\dR}(X)$ est engendré par la classe du triplet $(0, 0, \sfrac{dt}{t})$. D'un autre côté, comme $U(\CC)$ et $V(\CC)$ sont contractiles et comme leur intersection est $\CC^\times$, l'application
\begin{align*}
\rH_2(X(\CC), \ZZ) &\longrightarrow \rH_1(\CC^\times, \ZZ) \\
[(\sigma_U, \sigma_V, \sigma_{U \cap V})] &\longmapsto [\sigma_{U \cap V}]
\end{align*}
est un isomorphisme par la suite exacte de Mayer-Vietoris. Notant~$\sigma$ le lacet de la figure \ref{fig:lacetengendrant}, la classe du triplet $(0, 0, \sigma)$ est donc un générateur de $\rH_2(X(\CC), \ZZ)$ et \emph{a fortiori} de l'homologie de Betti $\rH_2^\Betti(X)$. Par rapport à ces bases, la période vaut
\[
\int_\sigma \frac{dt}{t}=2\pi i.
\]
\end{exemple}

\begin{remarque}[la classe fondamentale]\label{rem:classefonda}
Soit $X \subset \PP^2$ une courbe projective lisse. Comme toute surface topologique compacte, l'espace des points complexes $X(\CC)$ peut être trian\-gulé, par exemple en le réalisant comme le quotient d'un polygone régulier à $4g$ côtés par une relation d'équivalence identifiant des côtes opposés\footnote{La vidéo \url{https://www.youtube.com/watch?v=ZrQT-I78YPs} montre ce recollement en action dans le cas d'un octogone.} et en triangulant celui-ci depuis son centre \cite[Ex.\,3.31]{hatcher}. L'orientation de $\CC$ induit une orientation canonique sur chaque triangle par rapport à laquelle une combinaison linéaire avec des signes bien choisis est une $2$-chaîne singulière sans bord, dont la classe d'homologie ne dépend pas du choix de la triangulation. Il y a donc une classe canonique dans le groupe~$\rH_2(X(\CC), \ZZ)$ que l'on appellera la \emph{classe fondamentale}; c'est celle que l'on considère dans l'exemple \ref{exmp:classecanonique}. On peut ensuite choisir le générateur de $\rH^2_{\dR}(X)$ de sorte que l'intégrale le long de cette classe soit égale à $2\pi i$, d'où un isomorphisme canonique $\rH^2_{\dR}(X) \simeq \QQ$.
\begin{figure}[ht]
\centering
\includegraphics[width=0.6\textwidth]{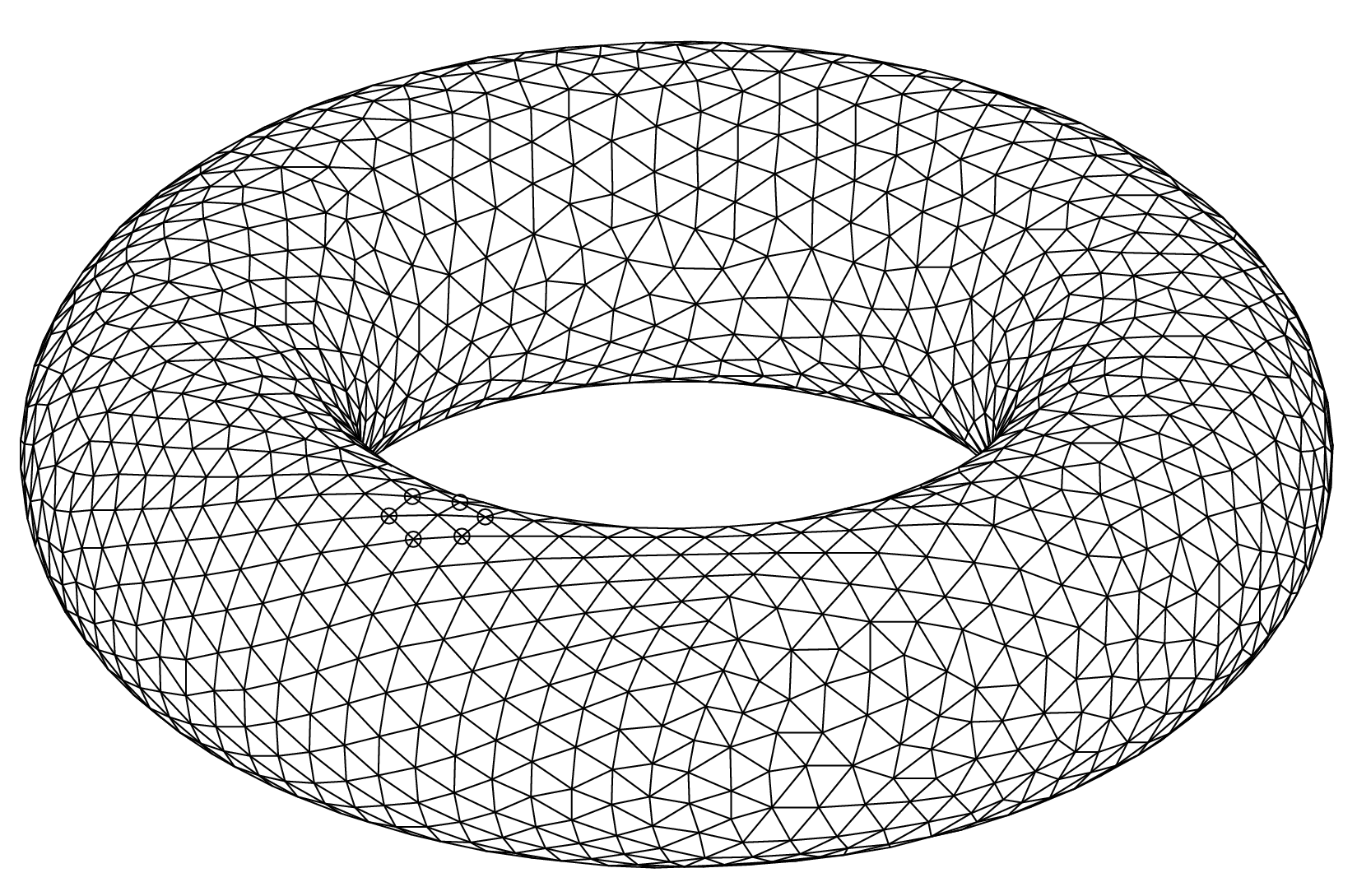}
\caption{Une triangulation d'un tore complexe}
\label{fig:triang}
\end{figure}
\end{remarque}

\begin{remarque}[torsion à la Tate]\label{rem:torsionTate} Soient $X \subset \mathbb{P}^2$ une courbe projective lisse et $U=X\setminus D$ le complémentaire d'un ensemble fini non vide $D$. L'inclusion de $U$ dans $X$ induit une application
\[
\rH_1^\Betti(U) \longrightarrow \rH_1^\Betti(X)
\]
qui est surjective (car on peut toujours bouger un $1$-cycle dans $X(\CC)$ pour qu'il ne rencontre pas $D(\CC)$ sans changer sa classe d'homologie) et s'inscrit dans une suite exacte longue
\[
0 \to \rH_2^{\Betti}(X) \to \rH_0^{\Betti}(D)\to \rH_1^{\Betti}(U) \to \rH_1^{\Betti}(X) \to 0,
\]
l'analogue en homologie de Betti de la suite exacte de Gysin~\eqref{eqn:GysinDR}. Rappelons que l'espace $\rH_0^\Betti(D)$ est une somme de copies de $\QQ$ indexées par les points de $D(\CC)$ et qu'une flèche $\rH_2^{\Betti}(X) \to \rH_0^{\Betti}(D)$ est caractérisée par l'image de la classe fondamentale de la remarque~\ref{rem:classefonda}, en l'occurrence l'élément $(1, \dots, 1)$. Quant à la flèche $\rH_0^{\Betti}(D)\to \rH_1^{\Betti}(U)$, elle envoie une combinaison linéaire de points de $D(\CC)$ vers la combinaison linéaire correspondante de classes de petits lacets, orientés positivement, autour de ces points. À ce stade, on constate que les suites exactes de Gysin en cohomologie de de~Rham et en homologie de Betti ne sont \emph{pas} compatibles avec l'accouplement de périodes: l'application $\rH^0_{\dR}(D) \to \rH^2_{\dR}(X)$ envoie une fonction $h$ vers la classe d'une forme différentielle $\omega_h$ ayant les résidus prescrits par $h$ le long de $D$. Or, le théorème des résidus de Cauchy donne l'égalité
\[
\int_{X(\CC)} \omega_h=2\pi i \sum_{z\in D(\CC)} h(z)
\]
et les deux accouplements diffèrent donc par un facteur de $2\pi i$! Pour y remédier, on introduit la \emph{torsion à la Tate}: les espaces vectoriels restent les mêmes, mais on écrit $\rH^0_{\dR}(D)(-1)$ et $\rH^0_{\Betti}(D)(-1)$ pour se rappeler que l'accouplement de périodes est multiplié par\footnote{En termes de l'isomorphisme de comparaison \eqref{eqn:periodscompisom}, on change la position du $\QQ$\nobreakdash-espace vectoriel $\mathrm{comp}^{-1}(\rH^0_{\Betti}(D))$ dans le $\CC$-espace vectoriel $\rH^0_{\dR}(D) \otimes \CC$. } $2\pi i$.
\end{remarque}

\subsection{Les intégrales elliptiques}\label{exmp:int-elliptiques}

Soit $X=E \setminus \{O\}$ la courbe elliptique affine d'équation \begin{displaymath}
y^2=4x^3-ax-b.
\end{displaymath}
Le premier groupe de cohomologie de de~Rham $\rH^1_{\dR}(X)$ est, d'après les calculs de la section \ref{exmp:courbeelliptique}, le $\QQ$-espace vectoriel de dimension $2$ engendré par les classes des formes différentielles
\[
\omega=\frac{dx}{y} \quad \text{et}\quad \eta=\frac{xdx}{y}, 
\]
dont seule la première s'étend en une différentielle sans pôle sur la courbe elliptique complète $E \subset \PP^2$. Ceci signifie que $E$ est une courbe projective lisse de \textit{genre} $1$~\cite{popescu}. D'après le théorème d'uniformisation pour les surfaces de Riemann, pour lequel on renvoie au beau livre collectif~\cite{unif}, le revêtement universel de
\begin{displaymath}
E(\CC)=\{(x, y) \in \CC^2 \mid y^2=4x^3-ax-b \} \cup O
\end{displaymath}
est le plan complexe $\CC$ et on récupère $E(\CC)$ comme le quotient par un réseau $\Lambda=\ZZ \omega_1 \oplus \ZZ \omega_2.$ Le premier groupe d'homologie de Betti de $\rH_1^\Betti(E)$ est, par conséquent, le $\QQ$-espace vectoriel engendré par les lacets $\sigma_1$ et $\sigma_2$ de l'exemple \ref{exmp:torecomplexe} et on trouve la matrice des périodes
\begin{displaymath}
\left(\begin{matrix}
\int_{\sigma_1} \omega & & \int_{\sigma_1} \eta \vspace{2mm} \\
\int_{\sigma_2} \omega & & \int_{\sigma_2} \eta
\end{matrix}\right).
\end{displaymath}

Le réseau $\Lambda$ ci-dessus est étroitement lié aux périodes de la courbe elliptique. En effet, étant donné un point $P \in E(\CC)$, on peut considérer l'intégrale $\int_O^P \omega \in \CC$ de la forme différentielle de première espèce le long d'un chemin $\gamma$ joignant le point à l'infini et $P$. Si~$\gamma'$ est un autre tel chemin, la différence $\gamma'-\gamma$ est une $1$-chaîne singulière sans bord dans $E(\CC)$ et définit ainsi une classe en homologie $\rH_1(E(\CC), \ZZ)$. Or, ce groupe étant le groupe libre de rang $2$ engendré par $\sigma_1$ et $\sigma_2$, il existe des entiers $n_1, n_2 \in \ZZ$ tels que
\begin{displaymath}
[\gamma'-\gamma]=n_1 \sigma_1+n_2\sigma_2.
\end{displaymath}
Compte tenu de l'égalité
\begin{displaymath}
\int_{\gamma'} \omega=\int_\gamma \omega+n_1\int_{\sigma_1} \omega+n_2\int_{\sigma_2} \omega,
\end{displaymath}
pour un autre choix de chemin allant de $O$ à $P$ la valeur de l'intégrale change par un élément du réseau
\begin{displaymath}
\Lambda=\ZZ \biggl(\int_{\sigma_1} \omega \biggr) \oplus \ZZ \biggl(\int_{\sigma_2} \omega \biggr),
\end{displaymath}
de sorte que l'application
\begin{align*}
E(\CC) &\longrightarrow \CC \slash \Lambda \\
P &\longmapsto \int_0^P \frac{dx}{y} \pmod\Lambda
\end{align*}
est bien définie. Il s'agit en fait d'un isomorphisme! Toute fonction méromorphe sur le tore $E(\CC)$ se relève en une fonction doublement périodique sur $\CC$ de \textit{périodes} $\int_{\sigma_i} \sfrac{dx}{y}.$

Pour construire l'application inverse, on commence par associer à un réseau $\Lambda$ de $\CC$ sa fonction $\wp$ de Weierstrass, définie par la série
\begin{displaymath}
\wp(z)=\frac{1}{z^2}+\sum_{\lambda \in \Lambda \setminus \{0\}} \biggl[\frac{1}{(z-\lambda)^2} -\frac{1}{\lambda^2}\biggr].
\end{displaymath}
Cette fonction et sa dérivée $\wp'$ sont reliées par l'équation algébrique \begin{displaymath}
\left(\wp'\right)^2=4\wp^3-g_2(\Lambda) \wp-g_3(\Lambda),
\end{displaymath}
dont les coefficients sont les nombres complexes
\begin{equation}\label{eqn:modularg2g3}
g_2(\Lambda)=60\sum_{\lambda \in \Lambda \setminus \{0\}} \frac{1}{\lambda^4}, \quad g_3(\Lambda)=140\sum_{\lambda \in \Lambda \setminus \{0\}}\frac{1}{\lambda^6}.
\end{equation}
On en déduit une application
\begin{align}
e \colon \CC \slash \Lambda &\longrightarrow \{(x, y) \in \CC^2 \mid y^2=4x^3-g_2(\Lambda) x-g_3(\Lambda) \} \cup O. \\
z \ (\text{mod } \Lambda) &\longmapsto \begin{cases} (\wp(z), \wp'(z)) & z \neq 0, \\ O & z=0. \end{cases}
\end{align}

Observons que le tiré en arrière par cette application de la différentielle de première espèce $\omega=dx/y$ n'est rien d'autre que
\begin{displaymath}
e^\ast(\sfrac{dx}{y})=\sfrac{d(\wp(z))}{\wp'(z)}=dz.
\end{displaymath}
Par conséquent, si l'on choisit pour générateurs $\sigma_1, \sigma_2 \in \rH_1(E(\CC), \ZZ)$ les projections dans $\CC/\Lambda$ des chemins $t \mapsto \omega_1 t$ et $t \mapsto \omega_2 t$ respectivement, la formule de changement de variables donne
\begin{displaymath}
\int_{\sigma_1} \omega =\int_0^{\omega_1} dz=\omega_1, \qquad \int_{\sigma_2} \omega=\int_0^{\omega_2} dz=\omega_2.
\end{displaymath}
Quant à la différentielle de deuxième espèce $\eta=x dx/y$, son tiré en arrière par l'application $e$ vaut
\begin{displaymath}
e^\ast(xdx/y)=\wp(z)\sfrac{d(\wp(z))}{\wp'(z)}=\wp(z)dz.
\end{displaymath}
Cette forme différentielle ayant un pôle au point à l'infini (qui est la classe de n'importe quel point du réseau dans le quotient $\CC/\Lambda$), il~faut choisir pour l'intégrer des représentants de $\sigma_1$ et $\sigma_2$ qui l'évitent, par exemple les projections des chemins
\[
t \mapsto t\omega_1+ \sfrac{\omega_2}{2} \quad \text{et} \quad t \mapsto t\omega_2+\sfrac{\omega_1}{2}.
\]
Les égalités suivantes sont alors vraies:
\begin{align*}
\int_{\sigma_1} x\frac{dx}{y}&=\int_{\sfrac{\omega_2}{2}}^{\omega_1+\sfrac{\omega_2}{2}} \wp(z)dz=\zeta(\sfrac{\omega_2}{2})-\zeta(\omega_1+\sfrac{\omega_2}{2}) \\
\int_{\sigma_2} x\frac{dx}{y} &=\int_{\sfrac{\omega_1}{2}}^{\omega_2+\sfrac{\omega_1}{2}} \wp(z)dz=\zeta(\sfrac{\omega_1}{2})-\zeta(\omega_2+\sfrac{\omega_1}{2}),
\end{align*}
où $\zeta(z)$ est une primitive de $-\wp(z)$. Le choix standard pour cette primitive est la fonction $\zeta$ de Weierstrass
\begin{equation}\label{eqn:defWeierstrass}
\zeta(z)=\frac{1}{z}+\sum_{\lambda \in \Lambda \setminus \{0\}}\biggl(\frac{1}{z-\lambda}+\frac{1}{\lambda}+\frac{z}{\lambda^2}\biggr).
\end{equation}
Il s'agit d'une fonction holomorphe sur $\CC \setminus \Lambda$, à pôles simples de résidu~$1$ aux points de~$\Lambda$. Puisque l'on peut substituer $-\lambda$ à $\lambda$ dans l'expression ci-dessus sans changer la somme, on a affaire à une fonction impaire. En intégrant la relation de périodicité $\wp(z+\lambda)=\wp(z)$ pour chaque $\lambda \in \Lambda$, on trouve que la fonction $\zeta(z)-\zeta(z+\lambda)$ est constante. Notant $\eta(\lambda)$ sa valeur, il vient
\begin{displaymath}
\int_{\sigma_1} x\,\frac{dx}{y}=\eta(\omega_1), \qquad \int_{\sigma_2} x\,\frac{dx}{y}=\eta(\omega_2).
\end{displaymath}
Les nombres $\eta_1=\eta(\omega_1)$ et $\eta_2=\eta(\omega_2)$ s'appellent classiquement les \emph{quasi-périodes} de la fonction $\wp$ de Weierstrass, mais ils sont tout autant des périodes que $\omega_1$ et $\omega_2$ au sens de ces notes.

\begin{proposition}[relation de Legendre]\label{prop:legendre} L'égalité
\begin{displaymath}
\omega_1\eta_2-\omega_2\eta_1=2\pi i
\end{displaymath}
est satisfaite. En particulier, l'accouplement de périodes sur une courbe elliptique est parfait.
\end{proposition}

\begin{proof} La relation de Legendre résulte du calcul de l'intégrale de $\zeta(z)$ le long du parallélogramme de la figure \ref{fig:Leg}.
\begin{figure}[ht]
\centering
\includegraphics[width=0.5\textwidth]{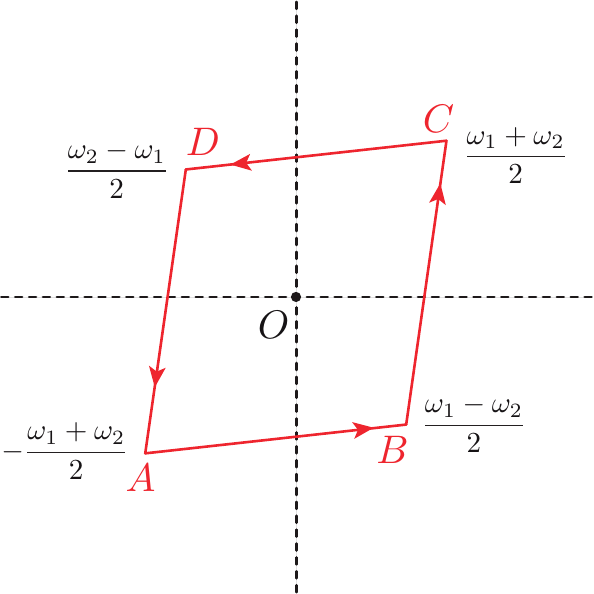}
\caption{Domaine d'intégration pour démontrer la relation de Legendre}
\label{fig:Leg}
\end{figure}
En effet, le seul pôle de $\zeta(z)$ à l'intérieur de ce contour étant $z=0$, on a \begin{align*}
2\pi i&=\int_{ABCD} \zeta(z) dz \\
&=\int_{AB} \zeta(z) dz+\int_{BC} \zeta(z) dz+\int_{CD} \zeta(z) dz+\int_{DA} \zeta(z) dz.
\end{align*}
d'après le théorème des résidus. Ces quatre termes sont reliés par les identités $\zeta(z+\omega_i)=\zeta(z)-\eta_i$. En prenant garde que l'orientation des côtés change, il vient
\begin{displaymath}
\int_{BC} \zeta(z)dz=-\int_{DA} \zeta(z+\omega_1)dz=-\int_{DA} \zeta(z)dz-\omega_2\eta_1
\end{displaymath}
et, de même,
\begin{displaymath}
\int_{CD} \zeta(z)dz=-\int_{AB} \zeta(z+\omega_2)dz=-\int_{AB} \zeta(z)dz+\omega_1\eta_2.
\end{displaymath}
Par conséquent,
\begin{displaymath}
2\pi i=\omega_1\eta_2 -\omega_2 \eta_1,
\end{displaymath}
ce que l'on voulait démontrer.
\end{proof}

\begin{remarque}\label{remq:isomcan} Comme on l'expliquera dans la section \ref{sec:Legendre}, la raison pour laquelle $2\pi i$ apparaît dans la relation de Legendre est que l'accouplement des périodes entre $\rH^2_{\dR}(E)$ et $\rH_2^\Betti(E)$ est donné par $2\pi i$ dans des bases bien choisies. Par rapport au recouvrement de $E$ par les cartes affines $U$ et $V$ de la section \ref{exmp:courbeelliptique}, la classe fondamentale (remarque \ref{rem:classefonda}) dans~$\rH_2(E(\CC), \ZZ)$ est représentée par le triplet $(0, 0, \sigma)$, où~$\sigma$ est un petit lacet, orienté au sens anti-horaire, autour du point à l'infini qui n'encercle aucun des trois zéros de la fonction $y$. En effet, la classe d'un tel lacet appartient au noyau de la flèche
\[
\rH_1(U(\CC)\cap V(\CC), \ZZ) \longrightarrow \rH_1(U(\CC), \ZZ) \oplus \rH_1(V(\CC), \ZZ)
\]
et est donc un générateur de $\rH_2(E(\CC), \ZZ)$ par la suite exacte de Mayer-Vietoris. Par ailleurs, $\rH^2_{\dR}(E)$ est le $\QQ$-espace vectoriel de dimension $1$ engendré par la classe du triplet $(0, 0, \sfrac{x^2dx}{y^2}$). Comme le résidu de cette forme à l'infini est $-\sfrac{1}{2}$ d'après le calcul \eqref{eqn:res-infty}, il est plus naturel de la renormaliser et de choisir $ \sfrac{-2x^2dx}{y^2}$ comme générateur. Par le théorème des résidus, on trouve alors la période
\[
\int_\sigma -2x^2\,\frac{dx}{y^2}=2\pi i.
\]
En général, une fois que l'on sait que la cohomologie de de~Rham en degré $2$ d'une courbe projective lisse est un $\QQ$-espace vectoriel de dimension $1$, son générateur canonique est la seule forme différentielle dont l'intégrale le long de la classe fondamentale est $2\pi i$.
\end{remarque}

Pour chaque entier $k \geq 2$, la \textit{série d'Eisenstein}
\begin{equation}\label{eqn:Eisenstein}
G_{2k}(\tau)=\sum_{\substack{m, n \in \ZZ \\ (m, n) \neq (0, 0)} } \frac{1}{(m+n\tau)^{2k}}
\end{equation}
converge absolument vers une fonction holomorphe sur le demi-plan de Poincaré $\mathfrak{h}=\{\tau \in \CC \mid \mathrm{Im}(\tau)>0\}$. C'est une forme modulaire de poids $2k$: pour toute matrice $\left(\begin{smallmatrix} a & b \\ c & d \end{smallmatrix}\right) \in \mathrm{SL}_2(\ZZ)$, agissant sur $\mathfrak{h}$ par transformation de Möbius, la relation
\begin{align}\label{eqn:modulariteG}
G_{2k}\Bigl(\frac{a\tau+b}{c\tau+d}\Bigr)&=(c\tau+d)^{2k} \sum_{\substack{m, n \in \ZZ \\ (m, n) \neq (0, 0)}} \frac{1}{[(md+nb)+(mc+na)\tau]^{2k}} \nonumber \\
&=(c\tau+d)^{2k} G_{2k}(\tau)
\end{align}
est satisfaite (la seconde égalité résulte du fait que, pour toute paire \hbox{$(r, s) \in \ZZ^2$,} il existe une et une seule solution $(m, n) \in \ZZ^2$ aux équations $r=md+nb$ et~\hbox{$s=mc+na$}). Si $k=1$, la série
\begin{displaymath}
G_2(\tau)=2\sum_{m=1}^{\infty} \frac{1}{m^2} +2\sum_{n=1}^\infty \sum_{m \in \ZZ} \frac{1}{(m+n\tau)^2}
\end{displaymath}
converge, mais pas absolument, vers une fonction holomorphe sur $\mathfrak{h}$ qui n'est, cette fois-ci, que quasi-modulaire:
\begin{equation}\label{eqn:modulariteG2}
G_{2}\Bigl(\frac{a\tau+b}{c\tau+d}\Bigr)=(c\tau+d)^{2}G_{2}(\tau)-2\pi i c(c\tau+d).
\end{equation}

Les valeurs de ces fonctions en $\tau=\sfrac{\omega_2}{\omega_1}$ sont reliées aux périodes de la courbe elliptique associée au réseau $\Lambda=\ZZ\omega_1 \oplus \ZZ\omega_2$ par 
\begin{displaymath}
G_4(\tau)=\frac{g_2(\Lambda)}{60}\,\omega_1^4, \quad G_6(\tau)=\frac{g_3(\Lambda)}{140}\,\omega_1^6, \quad G_2(\tau)=-\omega_1\eta_1.
\end{displaymath}
Si $g_2(\Lambda)$ et $g_3(\Lambda)$ sont des nombres algébriques, ce qui est notre cas, le corps engendré par les quatre périodes de la courbe elliptique est donc une extension finie du corps
\[
\QQ(2\pi i, G_2(\tau), G_4(\tau), G_6(\tau)).
\]
Pour $\tau$ fixé, on ne sait pas calculer en général le degré de transcendance de ce corps, mais on peut démontrer (théorème \ref{thm:transfonct}) que les \emph{fonctions} $G_2(\tau), G_4(\tau)$ et $G_6(\tau)$ sont algébriquement indépendantes

\begin{exemple}[valeurs bêta]\label{exmp:Fermat} Soit $d \geq 2$ un entier. La \emph{courbe de Fermat} affine de degré $d$ est la variété $X \subset \mathbb{A}^2$ définie par l'équation
\[
x^d+y^d=1.
\]
Comme les polynômes $dx^{d-1}$, $dy^{d-1}$ et $x^d+y^d-1$ n'ont pas de zéro commun, c'est une variété lisse de dimension $1$, dont l'anneau de \hbox{fonctions} et le module de $1$-formes différentielles sont donnés par
\begin{align*}
A&=\QQ[x, y]/(x^d+y^d-1), \\ \Omega^1_A&=(Adx \oplus A dy)/(x^{d-1}dx+y^{d-1}dy).
\end{align*}

Les relations définissant $\Omega^1_A$ entraînent que l'élément $dx$ est divisible par $y^{d-1}$, vu que l'égalité
\[
y^{d-1}(ydx-xdy)=(y^d+x^d)dx=dx
\]
y est satisfaite. Pour des entiers $1 \leq r, s \leq d-1$, considérons les formes différentielles
\[
\omega_{r, s}=x^{r-1}y^{s-1}\frac{dx}{y^{d-1}} \in \Omega^1_A.
\]
Le choix de ces formes est motivé par la formule \eqref{eqn:betafunction} pour les valeurs de la fonction bêta en des arguments rationnelles, qui à un multiple rationnel près sont leur intégrale le long du chemin
\begin{align*}
\sigma \colon [0, 1] &\longrightarrow X(\CC) \\
t &\longmapsto (t^{\sfrac{1}{d}}, (1-t)^{\sfrac{1}{d}}),
\end{align*}
les racines $d$-ièmes étant celles réelles positives:
\begin{align*}
\int_{\sigma} \omega_{r, s}&=\int_0^1 t^{\psfrac{r-1}{d}} (1-t)^{\psfrac{s-d}{d}} d(t^{\sfrac{1}{d}}) \\
&=\frac{1}{d}\int_0^1 t^{\sfrac{r}{d}-1}(1-t)^{\sfrac{s}{d}-1}dt \\
&=\frac{1}{d}\,\mathrm{B}(\sfrac{r}{d}, \sfrac{s}{d}).
\end{align*}

Le chemin $\sigma$ ne définit pas une classe en homologie singulière puisque ce n'est pas un lacet. Pour en faire un sans trop modifier les périodes, on se servira du fait que la courbe de Fermat a beaucoup d'automorphismes. Posons~\hbox{$\zeta=\exp(\sfrac{2\pi i}{d})$} et notons $\mu_d(\CC)$ le groupe des racines $d$-ièmes de l'unité, c'est-à-dire le groupe cyclique d'ordre $d$ engendré par $\zeta$. Les applications $f, g \colon \CC^2 \to \CC^2$ données~par
\[
f(x, y)=(\zeta x, y) \quad\text{et}\quad g(x, y)=(x, \zeta y)
\]
laissent le sous-espace $X(\CC) \subset \CC^2$ stable. Si l'on considère la courbe de Fermat comme étant définie sur le corps cyclotomique $\QQ(\zeta)$ plutôt que sur $\QQ$, pour pouvoir ainsi multiplier les coordonnées par des racines $d$-ièmes de l'unité, ce sont même des automorphismes de la variété algébrique $X$. Considérons la chaîne singulière
\[
\kappa=\sigma-g_\ast \sigma+f_\ast g_\ast \sigma-f_\ast \sigma,
\]
où $f_\ast$ et $g_\ast$ désignent les applications induites par $f$ et $g$ sur~$C_1(X(\CC))$. Par exemple, $g_\ast \sigma$ est le chemin $t \mapsto (t^{\sfrac{1}{d}}, \zeta(1-t)^{\sfrac{1}{d}})$. Les signes ci-dessus sont choisis pour que le bord de $\kappa$ s'annule:
\[
\scalebox{.95}{$\partial \kappa=(1, 0)-(0, 1)+(0, \zeta)-(1, 0)+(\zeta, 0)-(0, \zeta)+(0, 1)-(\zeta, 0)=0.$}
\]
Par conséquent, $\kappa$ représente une classe en homologie de Betti:
\[
[\kappa] \in \rH_1^\Betti(X).
\]

Les automorphismes $f$ et $g$ n'agissent pas seulement sur les chaînes singulières, mais aussi sur les formes différentielles, et le calcul
\begin{align*}
(f^i \circ g^j)^\ast \omega_{r, s}&=(\zeta^ix)^{r-1}(\zeta^j y)^{s-1}\frac{d(\zeta^i x)}{(\zeta^j y)^{d-1}} \\
&=\zeta^{ir+js} \omega_{r, s}
\end{align*}
montre que l'on a affaire à des vecteurs propres pour cette action. Plus précisément, les éléments $(\zeta^i, \zeta^j)$ du groupe $G=\mu_d(\CC) \times \mu_d(\CC)$ agissent sur $X$ au travers des automorphismes $f^i\circ g^j$ et l'action induite sur les formes $\omega_{r, s}$ est la multiplication par le caractère
\[
\chi_{r, s}\colon G \to \CC^\times, \qquad \chi_{r, s}(\zeta^i, \zeta^j)=\zeta^{ir+js}.
\]
Jointe à la formule du changement de variables, cette propriété permet de calculer les intégrales sur le lacet $\kappa$:
\begin{equation}\label{eqn:periodesFermat}
\begin{aligned}
\int_{\kappa} \omega_{r, s}&=\int_\sigma \omega_{r, s}-\int_{g_\ast \sigma}\omega_{r, s}+\int_{f_\ast g_\ast \sigma}\omega_{r, s}-\int_{f_\ast \sigma}\omega_{r, s} \\
&=\int_\sigma \omega_{r, s}-\int_{\sigma} g^\ast \omega_{r, s}+\int_{\sigma}f^\ast g^\ast \omega_{r, s}-\int_{\sigma}f^\ast \omega_{r, s} \\
&=\frac{(1-\zeta^r)(1-\zeta^s)}{d} \,\mathrm{B}(\sfrac{r}{d}, \sfrac{s}{d}).
\end{aligned}
\end{equation}
Comme ces périodes sont toutes non nulles, il s'ensuit que $[\omega_{r, s}]$ est un élément non nul de $\rH^1_{\dR}(X)$. De plus, les classes $[\omega_{r, s}]$ sont linéairement indépendantes pour $1 \leq r, s \leq d-1$ car le groupe $G$ y agit par multiplication par des caractères \emph{distincts}.

Soit $\bar X \subset \PP^2$ la complétion projective de la courbe de Fermat affine, c'est-à-dire la courbe projective plane d'équation
\[
x_0^d+x_1^d=x_2^d.
\]
C'est une courbe lisse, reliée à celle de départ par le changement de coordonnées $x=\sfrac{x_0}{x_2}$ et $y=\sfrac{x_1}{x_2}$. Le complémentaire de~$X$ dans~$\overline{X}$ est donc la variété affine $D$ de dimension $0$ définie par l'équation $x^d+y^d=0$, dont les points complexes sont les $[1\colon \xi^a \colon 0]$, avec $\xi=\exp(\sfrac{\pi i}{d})$ et~$a=1, 3, \dots, 2d-1$. La dimension de $\rH^1_{\dR}(\overline X)$ étant égale à $(d-1)(d-2)$ par le théorème de Riemann-Roch, la suite exacte longue de Gysin
\[
0 \to \rH^1_{\dR}(\overline{X}) \to \rH^1_{\dR}(X) \to \rH^0_{\dR}(D) \to \rH^2_{\dR}(X) \to 0
\]
montre que la cohomologie de de~Rham de $X$ est de dimension
\begin{align*}
\dim \rH^1_{\dR}(X) &=\dim \rH^1_{\dR}(\overline{X})+\dim \rH^0_{\dR}(D)-\dim \rH^2_{\dR}(X) \\
&=(d-1)(d-2)+d-1\\
&=(d-1)^2.
\end{align*}
Vu que l'on a trouvé autant de classes de cohomologie linéairement indépendantes, les $[\omega_{r, s}]$ forment une base de~$\rH^1_{\dR}(X)$.

Lesquelles s'étendent à $\overline X$? Dans les coordonnées $u=\sfrac{x_1}{x_0}=\sfrac{y}{x}$ et~\hbox{$v=\sfrac{x_2}{x_0}=\sfrac{1}{x}$,} le point à l'infini $[1\colon \xi^a\colon 0]$ est donné par $u=\xi^a$ et $v=0$, et il s'agit de calculer le résidu de
\[
\omega_{r, s}=x^{r-1}y^{s-1}\frac{dx}{y^{d-1}}=-v^{d-r-s}u^{s-d}\frac{dv}{v}
\]
en ce point. Si $r+s<d$, la forme n'a pas de pôle (première espèce), tandis que si $r+s>d$, elle a un pôle de résidu zéro (deuxième espèce); les formes soumis à ces contraintes constituent donc une base de~$\rH^1_{\dR}(\bar X)$. Par contre, si $r+s=d$, la forme $\omega_{r, s}$ a pour résidu~\hbox{$\xi^{as}\neq 0$}. On remarquera que, pour ces valeurs des paramètres, la formule de réflexion \eqref{eqn:functionalgamma} pour la fonction gamma implique que les périodes~\eqref{eqn:periodesFermat} sont des multiples algébriques de $2\pi i$, à savoir
\[
\int_{\kappa} \omega_{r, s}=-\frac{\xi^r+\xi^s}{d} 2\pi i.
\]
\end{exemple}

\section{Une variante relative}\label{sec:variante-relative}

L'homologie singulière, telle qu'elle a été introduite dans la définition~\ref{def:homologsing}, ne prend compte que des chaînes sans bord, par exemple les lacets. Or, dans la définition de période on s'autorise également à intégrer sur des chemins qui joignent deux points distincts; c'est notam\-ment le cas des logarithmes~$\log(q)$, représentés par l'intégrale de la forme différentielle~$\sfrac{dx}{x}$ le long d'un chemin allant de $1$ à $q$. Pour pouvoir également traiter cette situation, on introduit une variante \textit{relative} de l'homologie singulière et de la cohomologie de de~Rham.

\subsection{Homologie singulière relative}

Soient $M$ un espace topologique et $\iota \colon N \hookrightarrow M$ l'inclusion d'un sous-espace. Les homomorphismes $\iota_\ast \colon C_p(N) \to C_p(M)$ obtenus en regardant les chaînes singulières dans $N$ comme des chaînes singulières dans l'espace plus grand~$M$ sont injectifs pour tout $p$ et commutent aux applications bord, donnant lieu à un \textit{complexe~double}\vspace*{-5pt}\enlargethispage{-2\baselineskip}
\begin{displaymath}
\xymatrix@R=.6cm{
\vdots \ar[d]_{\partial_2} & \vdots \ar[d]_{\partial_2} \\
C_1(N) \ar[r]^{\iota_\ast} \ar[d]_{\partial_1} & C_1(M)\ar[d]_{\partial_1} \\
C_0(N) \ar[r]^{\iota_\ast} & C_0(M).
}
\end{displaymath}

Comme d'habitude, on le transforme en un complexe simple en considérant les applications
\begin{align}
C_p(M) \oplus C_{p-1}(N) &\longrightarrow C_{p-1}(M) \oplus C_{p-2}(N) \\
(\sigma, \tau) &\longmapsto ( \partial_p(\sigma)-\iota_\ast(\tau), -\partial_{p-1}(\tau)).
\end{align}

\begin{definition} L'\emph{homologie singulière de $M$ relative à $N$} en degré $p$ est le quotient
\begin{displaymath}
\rH_p(M, N; \ZZ)=\frac{\ker\bigl(C_p(M) \oplus C_{p-1}(N) \to C_{p-1}(M)\oplus C_{p-2}(N)\bigl)}{\mathrm{im}\bigl(C_{p+1}(M) \oplus C_{p}(N) \to C_{p}(M)\oplus C_{p-1}(N)\bigr)}.
\end{displaymath}
\end{definition}

Les éléments de ce groupe sont représentés par des paires $(\sigma, \tau)$ où~$\tau$ est un $(p-1)$-cycle dans $N$ et $\sigma$ est une $p$-chaîne singulière dans~$M$ satisfaisant à $\partial_p(\sigma)=\iota_\ast(\tau)$. Or, comme les applications $\iota_\ast$ sont injectives, cette condition détermine uniquement $\tau$; les classes d'homologie relative sont dont représentées par des chaînes singulières~$\sigma$ dans~$M$ dont le bord ne s'annule pas nécessairement mais est soumis à la contrainte d'être inclus dans $N$. Par exemple, des chemins qui ne sont pas des lacets peuvent donner lieu à des éléments dans~$\rH_1(M, N; \ZZ)$, mais il faut alors que les points initiaux et finaux du chemin appartiennent au sous-espace~$N$.

Ces groupes d'homologie relative sont reliés à l'homologie de $M$ et de $N$ par une suite exacte longue
\begin{equation}\label{eqn:suiteexactelonguerelative}
\xymatrix @R=1.25ex @C=1.1ex
{
& &&&\cdots \ar[rr] && \rH_2(M, N; \ZZ) \ar@{-}[r] & *{} \ar@{-}`r/9pt[d] `/9pt[l] \\
& *{} \ar@{-}`/9pt[d] `d/9pt[d]& && && & *{} \ar@{-}[llllll] & \\
& *{} \ar[r] & \ar[rr]^-{\alpha} \rH_{1}(N, \ZZ) &&\ar[rr]^-{\beta}\rH_1(M, \ZZ) && \rH_1(M, N; \ZZ) \ar@{-}[r] & *{} \ar@{-} `r/9pt[d] `[l]\\
& *{} \ar@{-}`/9pt[d] `d/9pt[d]& && && & *{} \ar@{-}[llllll] & \\
& *{} \ar[r] & \ar[rr]^-{\alpha} \rH_{0}(N, \ZZ) &&\ar[rr]^-{\beta}\rH_0(M, \ZZ)&& \rH_0(M, N; \ZZ)\ar[rr]&&0,
}
\end{equation}
où les flèches connectant une ligne avec la suivante sont les applications bord, à savoir:
\begin{align*}
\rH_p(M, N; \ZZ) &\longrightarrow \rH_{p-1}(N, \ZZ). \\
[\sigma] &\longmapsto [\partial_p \sigma]
\end{align*}
En effet, $\partial_p \sigma$ est, par définition de l'homologie relative \hbox{$\rH_p(M, N; \ZZ)$}, une $(p-1)$-chaîne singulière dans $N$, et cette chaîne est un cycle car la composition~$\partial_{p-1} \circ \partial_p$ s'annule. Les applications $\alpha$ sont celles induites par l'inclusion $\iota\colon N \hookrightarrow M$ et les applications $\beta$ consistent à regarder les chaînes sans bord dans $M$ comme un cas particulier des chaînes dont le bord prend des valeurs dans $N$.

Soient $(M, N)$ et $(M', N')$ des paires formées d'un espace topologique et d'un sous-espace. Un \emph{morphisme de paires} est une fonction continue $f\colon M \to M'$ satisfaisant à $f(N) \subset N'$. Un tel morphisme induit des homomorphismes 
\[
f_\ast\colon C_p(M) \to C_p(M') \quad \text{et} \quad f_\ast \colon C_p(N) \to C_p(N')
\] pour tout $p$ qui commutent aux applications bord. Si le bord de $\sigma$ appartient à $N$, le bord de $f_\ast(\sigma)$ appartient donc à $N'$, d'où un homomorphisme entre groupes d'homologie relative
\begin{equation}\label{eqn:fonctorBettirel}
f_\ast \colon \rH_p(M, N; \ZZ) \longrightarrow \rH_p(M', N'; \ZZ).
\end{equation}
Autrement dit, l'homologie relative est fonctorielle par rapport aux morphismes de paires d'espaces topologiques.

\subsection{Cohomologie de de~Rham relative}

On dispose d'une construction similaire en cohomologie de~de~Rham. Soit $X \subset \mathbb{A}^n$ une variété affine lisse définie par l'annulation des poly\-nômes $f_1, \ldots, f_m$ et soit $D \subset \mathbb{A}^n$ une autre variété lisse définie par ces mêmes équations plus d'autres; on dit alors que~$D$ est une \emph{sous-variété fermée} de~$X$ et on note $\iota \colon D \hookrightarrow X$ l'inclusion, qui est un morphisme de variétés algébriques. Les applications de restriction $\iota^\ast \colon \Omega^p_X \to \Omega^p_D$ se combinent dans un complexe double\vspace*{-5pt}
\begin{equation}\label{eqn:complexedRrelatif}
\begin{array}{c}
\xymatrix@R=.6cm{
\vdots & \vdots \\
\Omega^1_X \ar[r]^{\iota_\ast} \ar[u]^{d} & \Omega^1_D \ar[u]_{d} \\
\Omega^0_X \ar[r]^{\iota_\ast} \ar[u]^{d} & \Omega^0_D \ar[u]_{d},
}
\end{array}
\end{equation}
à partir duquel on peut former un complexe simple en considérant les applications
\begin{align*}
\Omega^p_X \oplus \Omega^{p-1}_D &\longrightarrow \Omega^{p+1}_X \oplus \Omega^{p}_D \\
(\omega, \eta) &\longmapsto (d\omega, \iota^\ast(\omega)-d\eta).
\end{align*}

\begin{definition} Pour $X$ et $D$ comme ci-dessus, la \emph{cohomologie de de~Rham de $X$ relative à $D$} en degré~$p$ est le quotient
\[
\rH^p_{\dR}(X, D)=\frac{\ker\bigl(\Omega^p_X \oplus \Omega^{p-1}_D \to \Omega^{p+1}_X \oplus \Omega^{p}_D\bigr)}{\mathrm{im}\bigl(\Omega^{p-1}_X \oplus \Omega^{p-2}_D \to \Omega^{p}_X \oplus \Omega^{p-1}_D\bigr)}.
\]
\end{definition} Une classe de cohomologie relative est donc représentée par une forme fermée $\omega$ sur $X$ dont la restriction à la sous-variété $D$ est exacte (c'est par exemple le cas pour toute forme de degré la dimension de~$X$) et par une \og primitive\fg $\eta$ de cette restriction. Comme l'homologie relative, ces espaces s'inscrivent dans une suite exacte longue
\begin{displaymath}
\xymatrix @R=1.25ex @C=1.1ex
{
0 \ar[rr] && \rH^0_{\dR}(X, D) \ar[rr] &&\rH_{\dR}^0(X) \ar[rrr]^-{\iota^\ast} &&& \rH^0_{\dR}(D) \ar@{-}[r] & *{} \ar@{-}`r/9pt[d] `/9pt[l] \\
& *{} \ar@{-}`/9pt[d] `d/9pt[d] & && && && *{} \ar@{-}[lllllll] & \\
& *{} \ar[r] & \ar[rr] \rH^1_{\dR}(X, D) &&\ar[rrr]^-{\iota^\ast}\rH_{\dR}^1(X) &&& \rH^1_{\dR}(D) \ar@{-}[r] & *{} \ar@{-} `r/9pt[d] `[l]\\
& *{} \ar@{-}`/9pt[d] `d/9pt[d]& && && && *{} \ar@{-}[lllllll] & \\
& *{} \ar[r] & \ar[rr] \rH^2_{\dR}(X, D) && \cdots && &&&
}
\end{displaymath}
dans laquelle les flèches $\rH^p_{\dR}(X, D) \to \rH^p_{\dR}(X)$ envoient la classe d'une paire $(\omega, \eta)$ vers la classe de $\omega$, et les morphismes connectants envoient la classe d'une forme $\eta$ sur $D$ vers celle de la paire $(0, \eta)$.

Soient $(X, D)$ et $(X', D')$ des paires formées d'une variété affine et d'une sous-variété fermée, toutes les deux supposées lisses. Un morphisme de variétés \hbox{$f\colon X \to X'$} satisfaisant à $f(D) \subset D'$ induit, pour tout~$p$, des applications 
\[
f^\ast\colon \Omega^p_{X'} \to \Omega^p_{X} \quad \text{et}  \quad f_\ast \colon \Omega^p_{D'} \to \Omega^p_{D}
\] qui commutent à la différentielle. Si la restriction à $D'$ d'une forme fermée $\omega'$ sur $X'$ est exacte, la restriction à $D$ de la forme fermée $f^\ast \omega$ est donc exacte aussi, d'où une application linéaire
\begin{equation}\label{eqn:fonctorDRrel}
f^\ast \colon \rH^p_{\dR}(X', D') \longrightarrow \rH^p_{\dR}(X, D).
\end{equation}
Autrement dit, la cohomologie de de~Rham relative est fonctorielle par rapport aux morphismes de paires de variétés affines lisses.

\subsection{L'accouplement de périodes}

Par analogie avec le cas déjà traité, on définit l'homologie de Betti relative comme le $\QQ$-espace vectoriel 
\[
\rH_p^\Betti(X, D)=\rH_p(X(\CC), D(\CC); \QQ)
\] et la cohomologie de Betti relative comme son dual
\[
\rH^p_\Betti(X, D)=\mathrm{Hom}(\rH_p^\Betti(X, D), \QQ). 
\]L'accouplement des périodes s'étend alors en un accouplement parfait
\begin{equation}\label{eqn:accouplementrelatif}
\begin{aligned}
\rH^p_{\dR}(X, D) \otimes \rH_p^\Betti(X, D) &\longrightarrow \CC \\
([(\omega, \eta)], [\sigma])&\longmapsto \int_{\sigma} \omega+\int_{\partial \sigma} \eta, 
\end{aligned} 
\end{equation} que l'on peut encore réinterpréter comme un isomorphisme 
\begin{equation}\label{eqn:isomcomprel}
\mathrm{comp}\colon \rH^p_{\dR}(X, D) \otimes \CC \stackrel{\sim}{\longrightarrow} \rH^p_\Betti(X, D) \otimes \CC.  
\end{equation}
Voyons quelques exemples.

\begin{exemple}[les logarithmes]\label{exmp:log-cohomologique} Soit~$q >1$ un nombre rationnel. Considérons la droite affine épointée $X=\mathbb{G}_m$ et la sous-variété~\hbox{$D \subset X$} formée des points $1$ et~$q$, c'est\nobreakdash-à\nobreakdash-dire le lieu des zéros du polynôme $(x-1)(x-q)$ dans $\mathbb{A}^1$. C'est une sous-variété lisse de dimension~$0$ avec anneau de fonctions
\begin{align*}
B&=\QQ[x] \slash (x-1)(x-q) \\ &=\QQ[x] \slash (x-1) \oplus \QQ[x] \slash (x-q) \\ &=\QQ \oplus \QQ.
\end{align*}
Dans ce cas, le complexe double \eqref{eqn:complexedRrelatif} est donné par
\begin{displaymath}
\xymatrix{
\QQ[x, x^{-1}]dx \\
\QQ[x, x^{-1}] \ar[r] \ar[u]^d & \QQ \oplus \QQ,
}
\end{displaymath}
où la flèche verticale est la différentielle et la flèche horizontale envoie un polynôme de Laurent $g$ sur la paire $(g(1), g(q))$ des valeurs qu'il prend en $1$ et en $q$. Le complexe simple associé est alors
\begin{align*}
\QQ[x, x^{-1}] &\To{a} \QQ[x, x^{-1}]dx \oplus \QQ \oplus \QQ \\
g &\longmapsto (g'(x)dx, g(1), g(q))
\end{align*}
et il faut calculer le noyau et le conoyau de cette flèche. Si $g$ est dans le noyau, on a \hbox{$g'(x)=0$,} donc $g$ est constant, mais comme cette valeur constante doit être égale à $g(1)=g(q)=0$, la seule possibilité est~\hbox{$g=0$}, d'où $\rH^0_{\dR}(X, D)=0$. Pour calculer le conoyau, on est amené à déterminer quels éléments de la base
\begin{displaymath}
(x^n dx, 0, 0), \quad (0, 1, 0), \quad (0, 0, 1)
\end{displaymath}
sont linéairement indépendants modulo l'image de $a$. D'un côté, les calculs $a(1)=(0, 1, 1)=(0, 1, 0)+(0, 0, 1)$ et
\begin{displaymath}
a(\spfrac{x^{n+1}}{n+1})=(x^n dx, 0, 0)+\frac{1}{n+1}\,(0, 1, 0)+\frac{q^{n+1}}{n+1}\,(0, 0, 1)
\end{displaymath}
pour $n \neq -1$ montrent que le conoyau est engendré par $(dx/x, 0, 0)$ et $(0, 1, 0)$. Comme visiblement aucune combinaison linéaire de ces éléments n'est dans l'image de $a$, ils en forment une base. En calculant l'image de $(x-q)/(q-1)$ on peut enfin remarquer que $(dx/(q-1), 0, 0)$ et $(0, 1, 0)$ définissent la même classe de cohomologie relative, d'où la présentation suivante:
\begin{displaymath}
\rH^1_{\dR}(X, D)=\QQ\biggl[\frac{dx}{q-1}\biggr] \oplus \QQ\biggl[\frac{dx}{x}\biggr].
\end{displaymath}

Par ailleurs, comme $X(\CC)$ est le plan complexe épointé, dont l'homologie singulière a été calculée dans l'exemple \ref{exmp:lacet0}, et comme \hbox{$D(\CC)$} est formé de deux points, l'homologie de Betti relative $\rH_1^\Betti(X, D)$ se déduit de la suite exacte longue
\begin{displaymath}
\xymatrix @R=1.25ex @C=1.1ex
{
& *{} & 0 \ar[rr] &&\ar[rr]^-{\beta}\rH_1(\CC^\times, \ZZ) \ar@{=}[d]&& \rH_1(\CC^\times, \{1, q\}; \ZZ) \ar@{-}[r] & *{} \ar@{-} `r/9pt[d] `[dd]\\
& & &&\ZZ &&&&\\
& *{} \ar@{-}`/9pt[d] `d/9pt[d]& && && & *{} \ar@{-}[llllll] & \\
& *{} \ar[r] & \ar[rr]^-{\alpha} \rH_{0}(\{1, q\}, \ZZ)\ar@{=}[d] &&\ar[rr]^-{\beta}\rH_0(\CC^\times, \ZZ) \ar@{=}[d] && \rH_0(\CC^\times, \{1, q\}; \ZZ)\ar[rr]&&0, \\
& & \ZZ\oplus\ZZ &&\ZZ && &&
}
\end{displaymath}
qui montre qu'elle est engendrée par le lacet~$\sigma_1$ qui tourne une fois autour de zéro, comme dans l'exemple \ref{exmp:lacet0}, et par le chemin droit~$\sigma_0$ allant de $1$ à $q$ (c'est pour pouvoir considérer ce dernier que l'on a introduit la variante relative). La matrice de l'accouplement de périodes par rapport à ces bases est donnée par
\begin{equation}\label{eqn:periodslog}
\def\arraystretch{1.3}
\begin{array}{c|ccc}
\int & \frac{1}{q-1}\,dx & & \sfrac{dx}{x} \\ \hline
\sigma_0 & 1 & & \log(q) \\ \hline
\sigma_1 & 0 & & 2\pi i \end{array}
\end{equation}
On voit que c'est un accouplement parfait car le déterminant vaut~$2\pi i$.

Cette matrice reflète le fait que le logarithme est une fonction \emph{multiforme}: pour obtenir un nombre complexe bien déterminé, on n'a pas défini $\log(q)$ comme une fonction de $q$, mais plutôt comme une fonction d'un chemin $\gamma$ allant de $1$ vers $q$, à savoir
\begin{displaymath}
\gamma \longmapsto \int_\gamma \frac{dx}{x}.
\end{displaymath}
(L'espace de ces chemins n'est rien d'autre que le \emph{revêtement universel} de $\CC^\times$ pointé en $1$.) Un choix standard pour $q>1$ consiste à prendre la ligne droite~$\sigma_0$ entre $1$ et $q$, mais on aurait pu aussi bien choisir n'importe quel autre chemin $\gamma$. Vu que celui-ci définit une classe d'homologie relative dans~$\rH_1(\CC^\times, \{1, q\}; \ZZ)$ et que les classes de $\sigma_0$ et $\sigma_1$ en forment une base, la classe $[\gamma]$ est une combinaison linéaire à coefficients entiers de $[\sigma_0]$ et~$[\sigma_1]$, et on peut montrer que cette combinaison linéaire est
\[
[\gamma]=[\sigma_0]+n[\sigma_1],
\]
où $n$ désigne le nombre de fois que $\gamma$ tourne autour de $0$. Il vient
\begin{displaymath}
\int_\gamma \frac{dx}{x}=\int_{\sigma_0} \frac{dx}{x}+n\int_{\sigma_1}\frac{dx}{x}=\log(q)+n2\pi i,
\end{displaymath}
où l'on voit apparaître la \textit{monodromie} du logarithme. 
\end{exemple}

Cette variante relative de l'homologie singulière et de la cohomologie de de~Rham n'est toujours pas suffisante pour engendrer toutes les périodes de Kontsevich-Zagier, mais on n'en est pas loin! Il faut s'autoriser à prendre pour~$D$ des sous-variétés qui ne sont pas lisses; heureusement, les singularités les plus bénignes s'avèrent suffisantes. Rappelons que $D \subset X$ est un diviseur à croisements normaux s'il est réunion de variétés lisses $D_1, \dots, D_m$ de codimension~$1$ dont toutes les intersections sont lisses de la dimension attendue (exemple \ref{exmp:SNCD}). On peut alors former le complexe double
\[
\xymatrix@R=.5cm{
\vdots & \vdots & \vdots \\
\Omega^1_X \ar[r] \ar[u] & \bigoplus_i \Omega^1_{D_i} \ar[u] \ar[r] & \bigoplus_{i<j} \Omega^1_{D_i \cap D_j} \ar[u] \\
\Omega^0_X \ar[u] \ar[r] & \bigoplus_i \Omega^0_{D_i} \ar[r] \ar[u] & \bigoplus_{i<j} \Omega^0_{D_i \cap D_j} \ar[r] \ar[u] & \cdots
}
\]
dans lequel les flèches verticales sont les différentielles dans le complexe de de~Rham de chaque composante irréductible $D_i$ et les flèches horizontales sont des sommes alternées des applications induites par les différentes inclusions. Par exemple, tous les signes sont positifs dans les premières flèches~$\Omega^p_X \to \bigoplus \Omega^p_{D_i}$, mais~$\Omega^p_{D_j} \to \Omega^p_{D_i \cap D_j}$ a signe positif et $\Omega^p_{D_i} \to \Omega^p_{D_i \cap D_j}$ a signe négatif dans les deuxièmes. De là on passe au complexe total
\[
\Omega^0_X \longrightarrow \Omega^1_X \oplus \bigoplus_i \Omega^0_{D_i} \longrightarrow \Omega^2_X \oplus \bigoplus_i \Omega^1_{D_i} \oplus \bigoplus_{i<j} \Omega^0_{D_i \cap D_j}\longrightarrow \cdots
\]
en modifiant, comme d'habitude, une fois sur deux le signe des différentielles horizontales. La \emph{cohomologie de de~Rham de $X$ relative à~$D$} est la cohomologie $\rH^p_{\dR}(X, D)$ de ce complexe. Ses éléments sont donc représentés par une $p$-forme sur $X$, des $(p-1)$-formes sur les composantes irréductibles $D_i$, des $(p-2)$-formes sur les intersections doubles~$D_i \cap D_j$, et ainsi de suite. En degré maximal, c'est\nobreakdash-à\nobreakdash-dire pour~$p$ égal à la dimension de $X$, toute $p$-forme sur $X$ définit une classe de cohomologie relative; réciproquement, toute classe de cohomologie relative peut être représentée par une telle forme \cite[Prop.\,3.3.19]{huber-muller}. On a déjà rencontré un exemple de ce phénomène en étudiant les logarithmes: la classe de cohomologie relative $(0, 1, 0)$, où la fonction constante $1$ est vue comme une $0$-forme sur un point, peut être représentée par la $1$-forme $dx/(q-1)$ sur $X=\mathbb{G}_m$. 

\begin{exemple}[les dilogarithmes]\label{exmp:dilogs} Le \emph{dilogarithme} est la fonction définie sur le disque $|z|<1$ par la série entière
\[
\mathrm{Li}_2(z)=\sum_{n=1}^\infty \frac{z^n}{n^2},
\]
où elle converge absolument. La représentation intégrale
\begin{equation}\label{eqn:integraldilog}
\mathrm{Li}_2(z)=\int_{[0, 1]^2} \frac{zdxdy}{1-zxy}
\end{equation}
montre que ses valeurs en des arguments algébriques $\alpha$ sont des périodes. Supposons, pour simplifier, que $\alpha$ est un rationnel non nul de module $|\alpha|<1$ et étudions la géométrie que cette intégrale suggère.

Soit $Z \subset \mathbb{A}^2$ le lieu des zéros du polynôme $f(x, y)=1-\alpha xy$, c'est\nobreakdash-à\nobreakdash-dire la droite affine épointée $\mathbb{G}_m$ plongée dans $\AA^2$ comme l'ensemble des $(x, \sfrac{1}{\alpha x})$. Vu que l'intégrande a des pôles le long de~$Z$, il est naturel de considérer comme espace ambiant la variété
\[
X=\mathbb{A}^2 \setminus Z.
\]
Il s'agit d'une variété affine lisse de dimension $2$ sur $\QQ$ car, quitte à introduire une nouvelle variable $t$, on peut y penser comme le lieu des zéros dans $\AA^3$ du polynôme $tf(x, y)-1$. Son anneau de fonctions est donc donné par la \emph{localisation}
\[
A=\{\sfrac{g}{f^n} \,|\, g \in \QQ[x, y],\, n \geq 0 \}
\]
et toute $2$-forme différentielle sur $X$ s'écrit comme $\Psfrac{g}{f^n} dx\wedge dy$. En~particulier, l'intégrande définit une classe de cohomologie 
\begin{equation}\label{eqn:integrandDR}
[\omega] \in \rH^2_{\dR}(X), \qquad \omega=\Psfrac{\alpha}{f} dx \wedge dy.
\end{equation}

Grâce à l'hypothèse $|\alpha|<1$, le domaine d'intégration $[0, 1]^2$ est inclus dans~$X(\CC)$, et on peut le voir comme une $2$-chaîne singulière en découpant le carré en deux triangles. Ce n'est pas un cycle, mais son bord est contenu dans les points complexes de la sous-variété
\[
D=\{xy(1-x)(1-y)=0\} \cap X,
\]
qui est un diviseur à croisements normaux sur $X$: la réunion de deux droites affines $D_1$ et $D_2$ et de deux droites affines épointées $D_3$ et $D_4$ sans intersection triple, telles que les montre la figure \ref{fig:16}. 

\begin{figure}[htb]
\centering\smaller\smaller
	\begin{tikzpicture}[scale=.5]
\def\a{3};
\draw[domain={\a/2+\a/32}:{2*\a-\a/8}, smooth, variable=\x, dashed] plot ({\x+\a/8}, {\a*\a/\x+\a/8});
\draw (0*\a,-.25*\a) -- (0,2*\a);	
\draw (1*\a,-.25*\a) -- (1*\a,2*\a);

\draw (-.25*\a,0) -- (2*\a,0);	
\draw (-.25*\a,1*\a) -- (2*\a,1*\a);
\path[pattern=north east lines, pattern color=gray] (0,0) -- (\a, 0) -- (\a, \a) -- (0, \a) -- (0,0);

\draw[red] (\a,\a+\a/4) circle (2*\a pt);
\draw[red] (\a+\a/4,\a) circle (2*\a pt);

\node[anchor=north] at (0*\a,-.25*\a) {$D_1$};
\node[anchor=north] at (1*\a,-.25*\a) {$D_3$};

\node[anchor=east] at (-.25*\a,0) {$D_2$};
\node[anchor=east] at (-.25*\a,1*\a) {$D_4$};

\node[anchor=north] at (1.5*\a,.73*\a) {$\alpha x y = 1$};

\end{tikzpicture}
\caption{Géométrie suggérée par la représentation intégrale du dilogarithme}
\label{fig:16}
\end{figure}
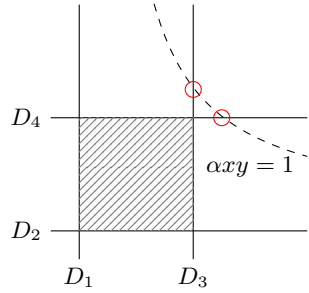

Le domaine d'intégration définit donc une classe en homologie de Betti relative~$\rH_2^\Betti(X, D)$. Comme la forme $\omega$ est de degré maximal, elle définit aussi une classe en cohomologie de de~Rham relative $\rH_{\dR}^2(X, D)$, la période associée à ces deux classes étant précisément
\[
\int_{[0, 1]^2} \omega=\mathrm{Li}_2(\alpha).
\]

Calculons le reste de la matrice des périodes, en commençant par la cohomologie de $X$. Comme $\AA^2$ n'a de la cohomologie qu'en degré~$0$, on déduit de la suite exacte longue
\[
\cdots \rightarrow \rH^2_{\dR}(\mathbb{A}^2) \to \rH^2_{\dR}(X) \xrightarrow{\mathrm{Res}_Z} \rH^1_{\dR}(Z)(-1) \to \rH^3_{\dR}(\mathbb{A}^2) \rightarrow \cdots
\] (analogue en dimension supérieure de la suite exacte longue de Gysin de la section \ref{exmp:courbeelliptique}) 
que le \emph{résidu le long de $Z$} est un isomorphisme
\[
\mathrm{Res}_Z\colon \rH^2_{\dR}(X) \stackrel{\sim}{\longrightarrow} \rH^1_{\dR}(Z)(-1).
\]
Il s'ensuit que $\rH^2_{\dR}(X)$ est de dimension $1$ et que la classe~\eqref{eqn:integrandDR} en est une base. Pour calculer son résidu, on écrit 
\[
\frac{\alpha dx\wedge dy}{f}=-\frac{dx}{x} \wedge \frac{df}{f} 
\] en utilisant l'égalité $df=-\alpha y dx-\alpha x dy$ et l'annulation de $dx \wedge dx$, puis on pose 
\[
\mathrm{Res}_Z([\omega])=\text{restriction à $Z$ de }-[\sfrac{dx}{x}].
\]
De même, on déduit de l'isomorphisme $\rH^1_{\dR}(X) \simeq \rH^0_{\dR}(Z)(-1)$ que la cohomologie $\rH^1_{\dR}(X)$ est engendrée par la classe de la forme différentielle $\sfrac{df}{f}$, dont le résidu le long de $Z$ vaut $1$.

L'analogue en homologie de Betti du calcul ci-dessus repose sur la suite exacte longue de Gysin
\[
\cdots \to \rH_3^{\Betti}(\mathbb{A}^2) \to \rH_1^{\Betti}(Z)(-1) \To{\mathrm{T}} \rH_2^{\Betti}(X) \to \rH_2^{\Betti}(\mathbb{A}^2) \to \cdots,
\]
qui induit un isomorphisme
\[
\mathrm{T} \colon \rH_1^{\Betti}(Z)(-1) \stackrel{\sim}{\longrightarrow} \rH_2^{\Betti}(X).
\]
C'est l'application qui envoie la classe d'un lacet engendrant~$\rH_1^{\Betti}(Z)$ vers la classe de son \emph{tube} dans $X$, que l'on peut par exemple représenter par la $2$-chaîne singulière
\[
T=\{(x, y) \in \CC^2 \mid |x|>\varepsilon \text{ et } |\alpha y-1/x|>\varepsilon\}
\]
pour $\varepsilon >0$ fixé. Comme les suites exactes de Gysin sont compatibles avec l'accouplement de périodes (c'est à cet effet que l'on avait introduit la torsion à la Tate dans la remarque \ref{rem:torsionTate}), on trouve\footnote{On peut aussi calculer directement l'intégrale en faisant le changement de variables $u=x$ et $v=f/x$ et en appliquant le théorème de Fubini à l'intégrale de la $2$-forme~$du/u \wedge dv/v$ sur le domaine défini par $|u|>\varepsilon$ et $|v|>\varepsilon$.}
\begin{equation}\label{eqn:coeff2pii}
\int_{T} \omega=2\pi i \int_{\sigma} \mathrm{Res}_Z(\omega)=2\pi i \int_{\sigma} \frac{dx}{x}=(2\pi i)^2.
\end{equation}
De même, $\rH_1^\Betti(X)\simeq \rH_0^{\Betti}(Z)(-1) $ est l'espace de dimension~$1$ engendré par le tube d'un point dans $Z$, c'est-à-dire par un lacet dans $X$ tournant autour de l'hyperbole pointillée de la figure \ref{fig:16}.

Passons maintenant à la cohomologie de $D$. Puisque les composantes irréductibles de $D$ sont des variétés affines de dimension $1$ sans intersection triple, le complexe de de~Rham de $D$ est égal à
\[
\bigoplus_i \Omega^0_{D_i} \longrightarrow \bigoplus_i \Omega^1_{D_i} \oplus \bigoplus_{i<j} \Omega^0_{D_i \cap D_j},
\]
où les flèches vers la première somme sont $-d$, l'opposé de la différentielle, et les flèches vers la seconde sont les restrictions des fonctions avec signe positif si l'on restreint de $D_i$ à $D_i \cap D_j$ et négatif si l'on restreint de $D_j$ à $D_i \cap D_j$. Toute classe dans $\rH^1_{\dR}(D)$ est donc représentée par un $8$-uplet dans
\[
\bigoplus_{i=1}^4 \rH^1_{\dR}(D_i) \oplus \bigoplus_{i<j} \rH^0_{\dR}(D_i\cap D_j).
\]
Comme on l'avait fait pour les variétés projectives recouvertes par deux cartes affines, on en déduit une suite exacte longue
\begin{displaymath}
\xymatrix @R=1.25ex @C=1.1ex
{
0 \ar[rr] & & \ar[rr] \rH_{\dR}^0(D)\ar@{=}[d] &&\ar[rr] \bigoplus_{i=1}^4 \rH^0_{\dR}(D_i) \ar@{=}[d]&& \bigoplus_{i<j} \rH^0_{\dR}(D_i\cap D_j)\ar@{=}[d] \ar@{-}[r] & *{} \ar@{-} `r/9pt[d] `[dd]\\
& & \QQ &&\QQ^4 && \QQ^4 && \\
& *{} \ar@{-}`/9pt[d] `d/9pt[d]& && && & *{} \ar@{-}[llllll] & \\
& *{} \ar[r] & \ar[rr] \rH_{\dR}^1(D) &&\ar[rr] \bigoplus_{i=1}^4 \rH^1_{\dR}(D_i) \ar@{=}[d] && 0 && \\
& & &&\QQ^2 && &&
}
\end{displaymath}
du type Mayer-Vietoris. Dans cette suite, la flèche $\QQ^4 \to \QQ^4$ est de rang $3$, car elle s'identifie à $(a, b, c, d) \mapsto (a-b, a-d, b-c, d-c)$, dont l'image est contenue dans l'hyperplan $r-s+t-u=0$. Il s'ensuit que l'espace $\rH^1_{\dR}(D)$ est de dimension $3$ et qu'il est engendré par les générateurs de la cohomologie de $D_3$ et de $D_4$ et par n'importe quel élément de~$\bigoplus_{i<j} \rH^0_{\dR}(D_i\cap D_j)$ qui n'est pas dans l'image de la flèche ci-dessus. Les classes des $8$-uplets
\begin{align*}
&(0, 0, \eta_3, 0, 0, 0, 0, 0),\ (0, 0, 0,\eta_4, 0, 0, 0, 0),\ (0, 0, 0, 0, 0, 0, 0, 1)
\end{align*}
avec $\eta_3=\sfrac{dy}{(y-\alpha^{-1})}$ et $\eta_4=\sfrac{dx}{(x-\alpha^{-1})}$ sont donc des générateurs; avec un léger abus, on les notera $[\eta_3]$, $[\eta_4]$ et~$[g]$. De plus, comme $\sfrac{df}{f}$ se restreint aux formes $\eta_3$ et~$\eta_4$ sur $D_3$ et sur~$D_4$ respectivement, l'application $\rH^1_{\dR}(X) \to \rH^1_{\dR}(D)$ induite par les inclusions des composantes irréductibles de $D$ dans~$X$ est injective. L'homologie de Betti $\rH_1^{\Betti}(D)$ est aussi un espace de dimension $3$, une base étant donnée par les classes du bord du carré~$[0, 1]^2$ et des deux lacets $\tau_3$ et $\tau_4$ autour des points manquants dans~$D_3(\CC)$ et~$D_4(\CC)$ qui, rappelons-le, sont tous les deux des plans complexes épointés.

Combiné avec la suite exacte de cohomologie relative
\[
0 \to \rH^1_{\dR}(X) \to \rH^1_{\dR}(D) \to \rH^2_{\dR}(X, D) \to \rH^2_{\dR}(X) \to 0,
\]
ce qui précède montre que $\rH^2_{\dR}(X, D)$ est l'espace de dimension $3$ ayant pour base les classes $[\omega]$, $[g]$ et $[\eta_3]$ (comme $[\eta_3]+[\eta_4]$ appartient à l'image de $\rH^1_{\dR}(X)$, les classes de $\eta_3$ et de $\eta_4$ sont l'opposée l'une de l'autre dans la cohomologie relative). Comme on l'a dit dans la discussion générale, on peut trouver des $2$-formes sur $X$ représentant les classes de $[g]$ et de $[\eta_3]$. Pour ce faire, on considère le diagramme\enlargethispage{-2\baselineskip}
\begin{equation}\label{eqn:complexedoublezigzag}
\begin{aligned}
\xymatrix{
\Omega^2_X & & \\
\Omega^1_X \ar[u]^{d} \ar[r] & \bigoplus_i \Omega^1_{D_i} & \\
& \bigoplus_i \Omega^0_{D_i} \ar[r] \ar[u]^{-d} & \bigoplus_{i<j} \Omega^0_{D_i \cap D_j}.
}
\end{aligned}
\end{equation}
Vu que l'élément $(0, 0, 0, -x)$ de $ \bigoplus_i \Omega^0_{D_i}$ s'envoie par restriction sur l'élément $(0, 0, 0, 1)$ de $\bigoplus_{i<j} \Omega^0_{D_i \cap D_j}$ et que son image  $(0, 0, 0, dx)$ dans~$\bigoplus_i \Omega^1_{D_i}$ par l'opposé de la différentielle est la restriction à $D_4$ de la $1$-forme $ydx$ sur $X$, dont la différentielle vaut $-dx \wedge dy$, on trouve que la~$2$\nobreakdash-forme $-dx\wedge dy$ représente la même classe que $g$ en cohomologie relative. Un argument similaire, avec l'élément $-\alpha x dy/(1-\alpha xy)$ de~$\Omega^1_X $ à la place, montre que $\eta_3$ représente la même classe de cohomologie relative que $\alpha dx\wedge dy/(1-\alpha xy)^2$. Une base de $\rH^2_{\dR}(X, D)$ est donc donnée par les classes de
\begin{equation}\label{eqn:basesdRrelatif}
dx\wedge dy, \quad \frac{\alpha dx\wedge dy}{(1-\alpha xy)^2}, \quad \frac{\alpha dx\wedge dy}{1-\alpha xy}.
\end{equation}

Par ailleurs, la suite exacte longue d'homologie relative
\[
0 \to \rH_2^{\Betti}(X) \to \rH_2^{\Betti}(X, D) \to \rH_1^{\Betti}(D) \to \rH_1^{\Betti}(X) \to 0
\]
montre que $\rH_2^{\Betti}(X, D)$ est de dimension $3$: il est engendré par les classes du tube $T$, du carré $[0, 1]^2$ et d'un troisième cycle $\sigma$ dont le bord intersecte $D_3(\CC)$ et $D_4(\CC)$ le long des lacets $\tau_3$ et $\tau_4$ avec des orientations opposées. Par rapport aux bases que l'on vient de décrire, la matrice des périodes est donnée par
\begin{equation}\label{eqn:periodsdilog}
\def\arraystretch{1.2}
\begin{array}{c|ccc}
\int & dx\wedge dy & \frac{\alpha dx\wedge dy}{(1-\alpha xy)^2} & \frac{\alpha dx\wedge dy}{1-\alpha xy} \\ \hline
[0, 1]^2 & 1 & -\log(1-\alpha) & \mathrm{Li}_2(\alpha) \\ \hline
\sigma & 0 & 2\pi i & 2\pi i\log(\alpha) \\ \hline
T & 0 & 0 & (2\pi i)^2
\end{array}
\end{equation}
Il y a divers moyens\footnote{Une autre possibilité consiste à observer que l'intégrande est l'image de $\eta_3$ par la suite exacte longue de cohomologie relative et que le bord du domaine d'intégration intersecte $D_3(\CC)$ en l'intervalle $[0, 1]$, d'où 
\[
\int_{[0, 1]^2}  \frac{\alpha dx\wedge dy}{(1-\alpha xy)^2}=\pm \int_0^1\frac{dy}{y-\alpha^{-1}}=\pm\log\big(\frac{1-\alpha^{-1}}{-\alpha^{-1}}\big)=\pm \log(1-\alpha),
\] avec un signe $\pm$ qui dépend des orientations, mais qui doit être négatif puisque l'intégrale est un nombre réel positif et que $\log(1-\alpha)$ est négatif pour $|\alpha|<1$.}  de calculer le coefficient $-\log(1-\alpha)$, par exemple en utilisant les identités 
\[
z \frac{d}{dz} \mathrm{Li}_2(z)=\int_{[0,1]^2} \frac{z dx\wedge dy}{(1-\alpha xy)^2}=-\log(1-z), 
\] obtenues en comparant les dérivées terme à terme de la série entière $\sum_{n=1}^\infty z^n/n^2$ et sous le signe intégral dans la représentation \eqref{eqn:integraldilog}. On avait déjà calculé le coefficient $(2\pi i)^2$ dans \eqref{eqn:coeff2pii}, et les autres intégrales le long de $T$ sont nulles puisque ce cycle est dans l'image de $\rH_2^{\Betti}(X)$ par la suite exacte longue d'homologie relative et que les formes différentielles en question ont une image nulle dans $\rH^2_{\dR}(X)$.

Plutôt que calculer les coefficients restant en paramétrant explicitement la chaîne $\sigma$, comme il est fait dans \cite[Prop.\,15,6.3]{huber-muller} dans un cas très proche, je vais conclure cet exemple en expliquant pourquoi la matrice des périodes est triangulaire supérieure, avec un bloc de taille $1$ de période $1$ et un bloc de taille $2$, qui est au facteur $2\pi i$ près la même matrice des périodes que dans l'exemple \ref{exmp:log-cohomologique}, sur la diagonale. L'explication repose sur la suite exacte longue de Gysin
\[
\cdots \to \rH^2_{\dR}(\mathbb{A}^2, D) \to \rH^2_{\dR}(X, D) \to \rH^1_{\dR}(Z, Z \cap D)(-1) \to \cdots,
\] cette fois-ci en cohomologie relative. Rappelons que $Z$ est isomorphe à la droite affine épointée $\mathbb{G}_m$ par le morphisme $(x, y) \mapsto (x, y/\alpha)$ et que l'intersection $Z \cap D$ est formée des points $(1, 1/\alpha)$ et $(1/\alpha, 1)$, d'où des isomorphismes de $\QQ$-espaces vectoriels
\[
\rH^p_{\dR}(Z, Z \cap D) \simeq \rH_{\dR}^p(\mathbb{G}_m, \{1, 1/\alpha\}). 
\] En particulier, le terme $\rH^0_{\dR}(Z, Z \cap D)(-1)$ est nul par l'exemple~\ref{exmp:log-cohomologique}. Pour des raisons de dimension, on sait que $\rH^3_{\dR}(\mathbb{A}^2, D)$ s'annule aussi, et comme ceux-ci sont les termes par lesquels la suite de Gysin continuerait à gauche et à droite, on trouve en fait une suite exacte courte 
\[
0 \to \rH^2_{\dR}(\mathbb{A}^2, D) \to \rH^2_{\dR}(X, D) \to \rH^1_{\dR}(\mathbb{G}_m, \{1, 1/\alpha\})(-1) \to 0. 
\] Un calcul sans malice avec le complexe double analogue de~\eqref{eqn:complexedoublezigzag} montre que le premier terme est de dimension $1$, avec base 
\[
\rH^2_{\dR}(\mathbb{A}^2, D)=\QQ[dx\wedge dy].
\] Les deux autres classes dans la base \eqref{eqn:basesdRrelatif} de $\rH^2_{\dR}(\mathbb{A}^2, D)$ s'envoient sur les classes\footnote{En effet, le deuxième élément de la base représente la même classe que~$\eta_3$, qui a résidu $0$ en $1$ et $1$ en $1/\alpha$, et on sait d'après l'exemple~\ref{exmp:log-cohomologique} que $(0, 0, 1)$ et $-dx/(\alpha-1)$ représentent la même classe dans $\rH^1_{\dR}(\mathbb{G}_m, \{1, 1/\alpha\})$. } $[dx/(\alpha-1)]$ et $-[dx/x]$ dans $\rH^1_{\dR}(\mathbb{G}_m, \{1, 1/\alpha\})$. 

L'homologie de Betti relative $\rH_2^\Betti(\mathbb{A}^2, D)$ est aussi de dimension~$1$, une base étant donnée par la classe du carré $[0, 1]^2$; la matrice des périodes a donc pour seul coefficient $\int_{[0, 1]^2} dx\wedge dy=1$. On peut ensuite vérifier que les classes $\sigma$ et $T$ sont, au signe près, les images par la première flèche dans la suite exacte courte
\[
0 \to \rH_1^{\Betti}(\mathbb{G}_m, \{1, 1/\alpha\})(-1)  \to \rH_2^{\Betti}(X, D) \to \rH_2^{\Betti}(\mathbb{A}^2, D) \to 0
\] des classes $[\sigma_0]$ et $[\sigma_1]$ dans la base standard de l'exemple~\ref{exmp:log-cohomologique}. La compatibilité de ces deux suites exactes avec l'accouplement de périodes, pourvu que l'on tienne compte du facteur $2\pi i$ dans la torsion à la Tate, explique la forme de la matrice \eqref{eqn:periodsdilog}. 
\end{exemple} 

\subsection{Comparaison des définitions}

Après avoir étendu la cohomologie de de~Rham relative aux diviseurs à croisements normaux, on a atteint notre but de donner une interprétation cohomologique des périodes.

\begin{theoreme}\label{thm:comparison} Le sous-ensemble des nombres complexes qui apparaissent comme coefficients de l'accouplement~\eqref{eqn:accouplementrelatif} coïncide avec l'ensemble des périodes de Kontsevich-Zagier (définition \ref{defn:periodeKZ}).
\end{theoreme}

Ce résultat, énoncé en passant au début de l'article de Kontsevich et Zagier, est démontré dans \cite[Th.\,12.2.1]{huber-muller}. Pour voir que les coefficients de l'accouplement sont des périodes au sens élémentaire, on identifie $X(\CC) \subset \CC^n$ à un sous-ensemble de $\RR^{2n}$ et on utilise des résultats de Łojasiewicz et de Hironaka \cite{hironaka} pour démontrer que toute classe en homologie de Betti relative $\rH_p^{\Betti}(X, D)$ peut être représentée par un $p$-simplexe singulier $\QQ$-semi-algébrique $\sigma \colon \Delta^p \to \RR^{2n}$, c'est-à-dire dont le graphe est un sous-ensemble $\QQ$-semi-algébrique de $\RR^{p+1} \times \RR^{2n}$. C'est en quelque sorte une élaboration du théorème d'approximation de Weierstrass 

Pour donner une idée d'où gît la difficulté pour obtenir l'inclusion réciproque et de comment la contourner, prenons $\alpha=1$ dans l'exemple \ref{exmp:dilogs}. L'intégrale~\eqref{eqn:integraldilog} est encore convergente, de valeur
\[
\zeta(2)=\int_{[0, 1]^2} \frac{dxdy}{1-xy},
\]
mais on ne peut plus la représenter comme période d'un groupe de cohomologie relative car le bord du domaine d'intégration rencontre le lieu des pôles de l'intégrande au point $(1, 1)$. Un moyen de les séparer est d'effectuer une construction géométrique appelée \emph{éclatement}: on remplace le plan affine $\mathbb{A}^2$ par la sous-variété $Y \subset \mathbb{A}^2 \times \PP^1$ d'équation
\begin{equation}\label{eqn;blowup}
(x-1)s=(y-1)t,
\end{equation}
où $[s\colon t]$ désignent les coordonnées homogènes sur $\PP^1$. Comme cette équation admet une unique solution pour tout $(x, y)$ distinct de $(1, 1)$ fixé, la projection sur le premier facteur $\pi \colon Y \to \AA^2$ est un isomorphisme en dehors de ce point; en revanche, la condition étant vide pour $(x, y)=(1, 1)$, sa préimage consiste en une copie de la droite projective, que l'on appelle le \emph{diviseur exceptionnel} et on note $E$. Outre~$E$, la préimage dans~$Y$ du lieu des pôles de l'intégrande contient la droite affine $F$ d'équation~\hbox{$s=-yt$} (en remplaçant dans \eqref{eqn;blowup} on trouve $(1-x)yt=(y-1)t$, puis $xy=1$ car~$t$ ne peut pas être nul) et leur réunion est un diviseur à croisements normaux. On montre ensuite que le tiré en arrière par~$\pi$ de la forme $dxdy/(1-xy)$ n'a pas de pôles le long de $E$ et que le bord de la préimage du domaine d'intégration ne rencontre pas $F$. Par conséquent, on peut voir l'intégrale de départ comme une période de $Y \setminus F$ relative au diviseur à croisement normaux de composantes irréductibles $E$ et les préimages des composantes du diviseur de départ.

\begin{figure}[htb]
\centering\smaller
	\begin{tikzpicture}[scale=.6]
\def\a{2.5};
\def\b{3};

\draw[domain={\a/2}:2*\a, smooth, variable=\x, dashed] plot ({\x}, {\a*\a/\x});
\draw (0*\a,-.25*\a) -- (0,2*\a);	
\draw (1*\a,-.25*\a) -- (1*\a,2*\a);

\draw (-.25*\a,0) -- (2*\a,0);	
\draw (-.25*\a,1*\a) -- (2*\a,1*\a);
\path[pattern=north east lines, pattern color=gray] (0,0) -- (\a, 0) -- (\a, \a) -- (0, \a) -- (0,0);

\draw[red] (\a,\a) circle (2*\a pt);

\draw[domain={\a/4}:2*\a, smooth, variable=\x] plot ({\x+\b*\a}, {(\a/2)*\a/\x});

\draw[domain={7*\a/32}:3*\a/2, smooth, variable=\x, dashed] plot ({\x+\b*\a + \a/2}, {(\a/2)*\a/\x-\a/4});

\draw (\b*\a,-.25*\a) -- (\b*\a,2*\a);	
\draw (\a+\b*\a,-.25*\a) -- (\a+\b*\a,.625*\a);

\draw (-.25*\a+\b*\a,0) -- (2*\a+\b*\a,0);	
\draw (-.25*\a+\b*\a,1*\a) -- (.625*\a+\b*\a,1*\a);

\draw[pattern=north east lines, pattern color=gray, domain={\a/2}:\a, smooth, variable=\x] (\b*\a,0) -- (\b*\a,\a) -- plot ({\x+\b*\a}, {(\a/2)*\a/\x}) -- (\a+\b*\a, 0);



\draw[->] (7/8*\b*\a, 3/4*\a) -- (11/16*\b*\a, 3/4*\a);

\node[anchor=south] at (0.78*\b*\a,.8*\a) {$\pi$};

\node[anchor=north] at (2*\a,.5*\a) {\small $xy = 1$};
\node[anchor=north] at (1.14*\b*\a,2*\a) {$\textcolor{red}{E}$};
\node[anchor=north] at (1.30*\b*\a,2*\a) {$F$};

\end{tikzpicture}
\caption{Éclatement du point $(1, 1)$ dans le plan affine $\AA^2$}
\label{fig:eclatement}
\end{figure}
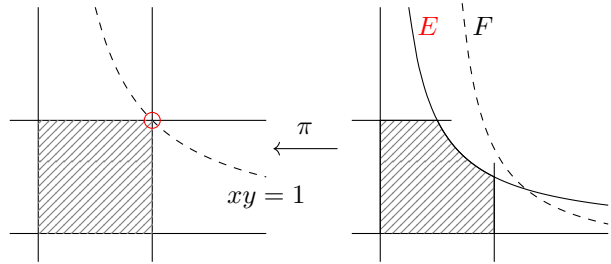

\section{Périodes exponentielles et cohomologie}\label{sec:periodexp}

De même que l'on peut réaliser les périodes comme les coefficients de l'accouplement d'intégration entre la cohomologie de de~Rham et l'homologie de Betti d'une variété algébrique, les périodes exponentielles interviennent lorsque l'on compare une cohomologie de nature algébrique avec une homologie de nature topologique associées à des variétés munies de la donnée supplémentaire d'une fonction.

Soient donc $X\subset \AA^n$ une variété affine lisse sur $\QQ$ et~$f$ une fonction sur $X$, c'est-à-dire un élément de l'anneau $A$ défini dans~\eqref{eqn:defnanneaufonctions}. Le~\emph{complexe de de~Rham tordu}
\begin{equation}\label{eqn:complexedRtordu}
\Omega^0_A \To{d_f}\Omega^1_A \To{d_f}\Omega^2_A \To{d_f}\cdots
\end{equation}
est le complexe ayant les mêmes termes que le complexe de de~Rham algébrique \eqref{eqn:complexedR}, mais dans lequel la différentielle usuelle $d$ a été modifiée tenant compte de la fonction en
\[
d_f(\omega)=d\omega-df \wedge \omega.
\]
Grâce aux propriétés du complexe de de~Rham usuel, notamment les égalités $d\circ d=0$ et $d(\alpha \wedge \beta)=d\alpha \wedge \beta-\alpha \wedge d\beta$ pour $\alpha \in \Omega^1_A$ et $\beta \in \Omega^p_A$ quelconques et le fait qu'un produit extérieur de deux formes égales s'annule, le calcul
\begin{align*}
(d_f \circ d_f)(\omega)&=d(d\omega-df \wedge \omega)-df \wedge (d\omega-df \wedge \omega) \\
&=d^2\omega-d^2f \wedge \omega+df\wedge d\omega-df \wedge d\omega+df \wedge df \wedge \omega \\
&=0
\end{align*}
montre que \eqref{eqn:complexedRtordu} est bien un complexe. Introduire la différentielle $d_f$ est un moyen algébrique de parler de la fonction exponentielle $e^{-f}$: formellement, si à la place de $\Omega^p_A$ on considère le module $e^{-f}\Omega^p_A$ et on y étend la différentielle en préservant la règle de Leibniz et la formule usuelle pour la dérivée de $e^{-f}$, il vient
\[
d(e^{-f}\omega)=e^{-f}d_f(\omega).
\]
\begin{definition} La \emph{cohomologie de de~Rham tordue} en degré $p$ de la paire $(X, f)$ est le quotient
\begin{displaymath}
\rH^p_{\dR}(X, f)=\frac{\ker(d_f \colon \Omega^p_A \to \Omega^{p+1}_A)}{\mathrm{im}(d_f \colon \Omega^{p-1}_A \to \Omega^p_A)}.
\end{displaymath}
\end{definition}

Passons maintenant à l'analogue topologique de cette construction. En prenant les points complexes, la fonction $f$ donne lieu à une application \hbox{$f \colon X(\CC) \to \CC$.} Pour chaque réel $r \geq 0$, posons
\[
S_r=\{z \in \CC \,|\, \mathrm{Re}(z)\geq r\},
\]
de sorte que $f^{-1}(S_r)$ est le sous-espace fermé de $X(\CC)$ où $f$ prend des valeurs de partie réelle supérieure ou égale à $r$.

\begin{definition} L'\emph{homologie à décroissance rapide} en degré $p$ de la paire $(X, f)$ est la limite de groupes d'homologie relative
\begin{equation}\label{eqn:defirapiddecay}
\rH_p^\rrd(X, f)=\lim_{r \to +\infty} \rH_p(X(\CC), f^{-1}(S_r); \ZZ).
\end{equation}
\end{definition}

Cette formule mérite une explication. D'abord, comme pour tous réels $r' \geq r$ le sous-espace $f^{-1}(S_{r'})$ est contenu dans $f^{-1}(S_r)$, l'application identité induit un morphisme de paires
\[
(X(\CC), f^{-1}(S_{r'})) \longrightarrow (X(\CC), f^{-1}(S_r)),
\]
d'où des applications dites \emph{de transition}
\[
\pi_{r'}^r \colon \rH_p(X(\CC), f^{-1}(S_{r'}); \ZZ) \longrightarrow \rH_p(X(\CC), f^{-1}(S_{r}); \ZZ)
\]
par fonctorialité de l'homologie relative \eqref{eqn:fonctorBettirel}. La limite dans la définition de l'homologie à décroissance rapide est alors le groupe des
\[
\sigma=(\sigma_r)_{r \geq 0} \in \prod_{r\geq 0} \rH_n(X(\CC), f^{-1}(S_r); \ZZ)
\]
qui sont compatibles aux applications de transition, c'est-à-dire qui satisfont à la condition~$\pi_{r'}^r(\sigma_{r'})=\sigma_r$ pour tous $r' \geq r$. Ce n'est pas une \og vraie\fg limite, en ce sens que l'on peut démontrer que les inclusions $f^{-1}(S_{r'})\hookrightarrow f^{-1}(S_r)$ sont des équivalences d'homotopie pour~$r$ assez grand et les applications de transition des isomorphismes. Si l'on se donne une $p$-chaîne singulière sur $X(\CC)$ dont le bord est contenu dans un sous-espace où la partie réelle de~$f$ est très positive, il y a donc un seul moyen de la prolonger en une classe d'homologie à décroissance rapide; introduire la limite est une manière de contourner le choix d'un tel sous-espace et de pouvoir parler de cycles \emph{non compacts} dans $X(\CC)$. Comme d'habitude, on pourra représenter les classes d'homologie à décroissance rapide par des chaînes lisses. 

\begin{thm}[Deligne, Bloch-Esnault, Hien-Roucairol] Soient $X \subset \mathbb{A}^n$ une variété affine lisse et $f$ une fonction sur $X$. L'intégration induit un accouplement parfait
\begin{equation}\label{eqn:periodspairing}
\begin{aligned}
\rH^p_{\dR}(X, f) \otimes \rH_p^\rrd(X, f) &\longrightarrow \CC \nonumber \\
([\omega], [\sigma]) &\longmapsto \lim_{r \to +\infty} \int_{\sigma_r} e^{-f} \omega.
\end{aligned}
\end{equation}
\end{thm}

Ce théorème remonte à une lettre de Deligne à Malgrange en~1976, où il explique comment le déduire, pour $X=\mathbb{A}^1$ la droite affine, de l'exemple \ref{exmp:gamma} ci-dessous \cite[p.\,17]{deligne-malgrange}; il est dû à Bloch-Esnault \cite{bloch-esnault} en dimension $1$ et à Hien-Roucairol \cite{hien-roucairol} en dimension quelconque. La~définition de l'homologie à décroissance rapide est faite pour que ces intégrales convergent. Il n'y aurait évidemment pas de problème si~$\sigma$ était sans bord, mais il n'y a pas en général assez de cycles; comme le montre l'exemple ci-dessous, la cohomologie $\rH^n_{\dR}(X, f)$ peut être non triviale même si $X(\CC)$ est contractile! Or, comme la partie réelle de~$f$ est très positive sur le bord $\partial\sigma$, la fonction $e^{-f}$ décroît plus rapidement que n'importe quel polynôme le long de ce bord; la forme $\omega$ étant algébrique, cela assure la convergence. On dispose également de variantes relatives de la cohomologie de de~Rham tordue et de l'homologie à décroissance rapide \cite{jossen}, qui sont par exemple nécessaires pour interpréter cohomologiquement la représentation intégrale \eqref{eqn:gamma-per-exponentielle} de la constante $\gamma$ d'Euler; avec ces variantes relatives et une version légèrement modifiée de la définition~\ref{def:perexp}, Commelin, Habegger et Huber ont démontré que les coefficients de l'accouplement \eqref{eqn:periodspairing} sont exactement les périodes exponentielles~\cite{ominimal}.

\subsection{Exemples}

Si $f$ est la fonction nulle, la cohomologie de de~Rham tordue et l'homologie à décroissance rapide de la paire $(X, f)$ coïncident avec la cohomologie de de~Rham et l'homologie singulière de $X$ vu que, dans ce cas, $df$ est nul et $f^{-1}(S_r)$ est vide pour $r$ assez grand. C'est encore le cas si $f$ est constante, mais lorsque cette constante est non nulle, l'accouplement \eqref{eqn:periodspairing} est affecté d'un facteur exponentiel; pour cette raison, les exponentielles des nombres algébriques sont des périodes exponentielles de variétés de dimension $0$.

\begin{exemple}[valeurs gamma]\label{exmp:gamma} Lorsque $X=\mathbb{A}^1$ est la droite affine, une fonction $f$ sur $X$ est un élément de l'anneau~\hbox{$A=\QQ[x]$,} autrement dit un polynôme à coefficients rationnels. Considérons l'exemple où~$f=\nobreak x^n$ pour un entier $n \geq 2$. Le complexe de de~Rham tordu \eqref{eqn:complexedRtordu} est dans ce cas le complexe à deux termes
\begin{align*}
d_f \colon \QQ[x] &\longrightarrow \QQ[x]dx \\
P &\longmapsto (P'-nx^{n-1}P) dx
\end{align*}
et il s'agit de calculer le noyau et le conoyau de $d_f$:
\begin{displaymath}
\rH^0_{\dR}(X, f)=\ker(d_f), \qquad \rH^1_{\dR}(X, f)=\QQ[x]dx \slash \im(d_f).
\end{displaymath}
Comme l'équation différentielle $P'=nx^{n-1}P$ n'a pour solutions que les fonctions~$c\exp(x^n)$, qui ne sont pas de polynômes si $c$ est non nul, l'application $d_f$ est injective. Pour calculer son conoyau, on remarque d'abord qu'un élément dans l'image de $d_f$ est de la forme $Qdx$ pour un polynôme $Q$ de degré au moins $n-1$; il s'ensuit que $dx, xdx, \dots, x^{n-2}dx$ définissent des classes linéairement indépendantes dans le conoyau de $d_f$. Il s'agit également d'une famille génératrice: pour chaque entier $m \geq n-1$, l'égalité
\[
d_f(-\sfrac{x^{m-n-1}}{n})=x^m dx-\frac{m-n-1}{n}\,x^{m-n-2}dx
\]
montre que la classe de $x^m dx$ modulo l'image de $d_f$ est un multiple de celle de $x^{m-n-2}dx$, ce qui par récurrence permet de l'écrire comme une combinaison linéaire des classes de $dx, \dots, x^{n-2}dx$. On a donc:
\begin{equation}\label{eqn:basegamma}
\rH^1_{\dR}(\mathbb{A}^1, x^n)=\QQ [dx] \oplus \QQ [xdx] \oplus \dots \oplus \QQ [x^{n-2}dx].
\end{equation}

Passons maintenant à l'homologie à décroissance rapide. Pour un nombre réel $r$ suffisamment grand, le sous-espace de $\CC$ où $\mathrm{Re}(x^n)\geq r$ est la réunion de $n$ régions disjointes concentrées autour des demi\nobreakdash-droites de pente les racines de l'unité d'ordre $n$ (elles sont dessinées en bleue pour $n=5$ dans la figure \ref{fig:rapid-decay}). Ces régions étant contractiles, le groupe d'homologie relative intervenant dans la limite \eqref{eqn:defirapiddecay} est isomorphe à celui du plan complexe relatif à $n$ points.
\begin{figure}[ht]
\centerline{\includegraphics[width=0.6\textwidth]{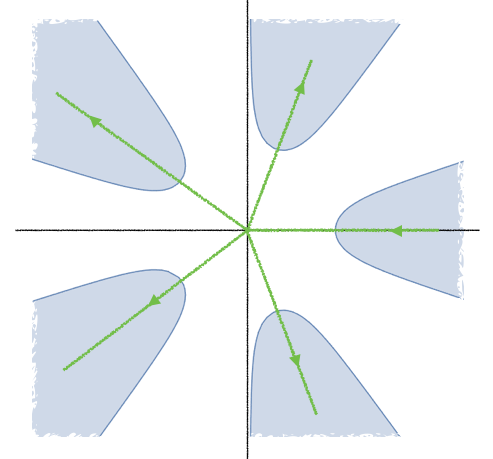}}
\caption{Cycles à décroissance rapide}
\label{fig:rapid-decay}
\end{figure}
Ce dernier peut être calculé par le biais de la suite exacte longue
\begin{displaymath}
\xymatrix @R=1.25ex @C=1.1ex
{
& *{} & &&\ar[rr]
\rH_1(\CC, \ZZ)\ar@{=}[d]&& \rH_1(\CC, n\text{ pts}; \ZZ) \ar@{-}[r] & *{} \ar@{-} `r/9pt[d] `[dd]\\
& & &&0 &&&&\\
& *{} \ar@{-}`/9pt[d] `d/9pt[d]& && && & *{} \ar@{-}[llllll] & \\
& *{} \ar[r] & \ar[rr]
\rH_0(n\text{ pts}, \ZZ)\ar@{=}[d] &&\ar[rr] \rH_0(\CC, \ZZ) \ar@{=}[d] && \rH_0(\CC, n\text{ pts}; \ZZ)\ar@{=}[d]\ar[rr]&&0,\\
& & \ZZ^n &&\ZZ &&0&&
}
\end{displaymath}
qui montre qu'il est isomorphe à $\ZZ^{n-1}$. Il s'ensuit que $\rH_n^\rrd(\mathbb{A}^1, x^n)$ est également de dimension $n-1$. Posons $\zeta=\exp(\sfrac{2\pi i}{n})$ et considérons, pour chaque réel $r \geq 0$, les 1-simplexes singuliers
\begin{equation}
\begin{aligned}
\tau_{i, r}\colon [0, 1] &\longrightarrow \CC \\
t &\longmapsto \zeta^i t r
\end{aligned} \qquad (i=0, \dots, n-1),
\end{equation}
dont le bord est $(\zeta^i r)-(0)$. Le bord de la chaîne singulière
\[
\sigma_{i, r}=\tau_{i, r}-\tau_{0, r}
\]
(dessinée en vert dans la figure \ref{fig:rapid-decay}) est donc contenu dans $\mathrm{Re}(x^n)\geq r$, d'où des classes d'homologie à décroissance rapide
\[
[\sigma_i]=([\sigma_{i, r}])_{r \geq 0} \in \rH_1^{\rrd}(\mathbb{A}^1, x^n).
\]

À l'aide du calcul
\begin{align*}
\lim_{r\to+\infty}\int_0^{\zeta^i r} e^{-x^n} x^{j-1}dx&=\int_0^{\zeta^i \infty} e^{-x^n} x^{j-1}dx\\
&=\int_0^\infty e^{-x^n}(\zeta^i x)^{j-1}d(\zeta^i x) \\
&=\frac{\zeta^{ij}}{n} \int_0^\infty e^{-s} s^{\sfrac{j}{n}-1} ds \\
&=\frac{\zeta^{ij}}{n}\Gamma(\sfrac{j}{n}),
\end{align*}
on trouve que la matrice de l'accouplement de périodes exponentielles par rapport à la base \eqref{eqn:basegamma} de la cohomologie de de~Rham tordue et aux classes $[\sigma_i]$ d'homologie à décroissance rapide est
\begin{displaymath}
P=\left(\begin{matrix}
\Psfrac{(\zeta^{ij}-1)}{n}\Gamma(\sfrac{j}{n})
\end{matrix}\right)_{1 \leq i, j \leq n-1}.
\end{displaymath}

Vérifions par un calcul de déterminant que cette matrice est inversible, ce qui montrera en même temps que les classes $[\sigma_i]$ forment une base de $\rH_1^{\rrd}(\mathbb{A}^1, x^n)$ et que l'accouplement \eqref{eqn:periodspairing} est parfait dans le cas qui nous occupe:
\begin{align*}
\det P&=n^{1-n}\cdot \det\left(\begin{matrix} \zeta^{ij}-1
\end{matrix}\right)_{1 \leq i, j \leq n-1}\cdot \prod_{j=1}^{n-1} \Gamma(\sfrac{j}{n}) \\
&=n^{\sfrac{1}{2}-n}\cdot (2\pi)^{\sfrac{(n-1)}{2}} \cdot \det\left(\begin{matrix} \zeta^{ij}-1
\end{matrix}\right)_{1 \leq i, j \leq n-1},
\end{align*}
où la seconde égalité découle de la formule de multiplication pour la fonction gamma \eqref{eqn:functionalgamma}. Pour montrer que le déterminant restant n'est pas nul, on observe qu'il est égal au déterminant de Vandermonde
\[
\det\left(\begin{matrix} \zeta^{ij} \end{matrix}\right)_{0 \leq i, j \leq n-1}=\prod_{0 \leq i < j \leq n-1} (\zeta^j -\zeta^i).
\]
C'est un exercice amusant de calculer sa valeur exacte.
\end{exemple}

\begin{exemple}[fonctions de Bessel] Soit $X=\mathbb{G}_m$ la droite affine épointée, comme dans l'exemple \ref{exmp:Gm}; les fonctions sur $X$ sont les polynômes de Laurent $f \in \QQ[x, x^{-1}]$. Prenons~$f=x+\sfrac{1}{x}$. Le complexe de de~Rham tordu de la paire $(X, f)$ est le complexe à deux termes
\begin{equation}\label{eqn:diffDRtordu}
\begin{aligned}
d_f \colon \QQ[x, x^{-1}] &\longrightarrow \QQ[x, x^{-1}]dx \\
P &\longmapsto (P'-P+\sfrac{P}{x^2}) dx.
\end{aligned}
\end{equation}
La différentielle $d_f$ étant injective, comme dans l'exemple précédent, la seule cohomologie potentiellement non nulle est en degré $1$:
\[
\rH^1_{\dR}(X, f)=\mathrm{coker}(d_f).
\]

D'un côté, tout élément de $\QQ[x, x^{-1}]dx$ peut s'écrire comme une combinaison linéaire de $dx/x$ et $dx$ modulo l'image de~$d_f$, car l'égalité
\[
x^ndx=nx^{n-1}dx+x^{n-2}dx-d_f(x^n)
\]
est vraie pour tout entier $n$. D'un autre côté, les classes de ces deux formes différentielles sont linéairement indépendantes dans le conoyau de~$d_f$, puisque les seules valeurs $a, b \in \QQ$ pour lesquelles l'équation
\[
\frac{a}{x}+b=P'-P+\frac{P}{x^2}
\]
a une solution dans les polynômes de Laurent sont $a=b=0$. On a ainsi trouvé une base:
\begin{displaymath}
\rH^1_{\dR}(X, f)=\QQ\biggl[\frac{dx}{x}\biggr] \oplus \QQ[dx].
\end{displaymath}

Calculons l'homologie à décroissance rapide. Pour $r$ assez grand, le sous-espace $f^{-1}(S_r) \subset \CC^\times$ consiste en deux régions contractiles disjointes autour des points $\sfrac{1}{r}$ et $r$ sur l'axe réel positif. On est donc réduit à calculer l'homologie du plan complexe épointé relative à deux points; la suite exacte longue
\begin{displaymath}
\xymatrix @R=1.25ex @C=1.1ex
{
& *{} & \rH_1(2\text{ pts}, \ZZ)\ar@{=}[d]\ar[rr] &&\ar[rr]
\rH_1(\CC^\times, \ZZ)\ar@{=}[d]&& \rH_1(\CC^\times, 2\text{ pts}; \ZZ) \ar@{-}[r] & *{} \ar@{-} `r/9pt[d] `[dd]\\
& & 0 &&\ZZ &&&&\\
& *{} \ar@{-}`/9pt[d] `d/9pt[d]& && && & *{} \ar@{-}[llllll] & \\
& *{} \ar[r] & \ar[rr]
\rH_0(2\text{ pts}, \ZZ)\ar@{=}[d] &&\ar[rr]^-{\beta} \rH_0(\CC^\times, \ZZ) \ar@{=}[d] && \rH_0(\CC^\times, 2\text{pts}; \ZZ)\ar@{=}[d]\ar[rr]&&0\\
& & \ZZ^2 &&\ZZ &&0&&
}
\end{displaymath}
montre que ce groupe est isomorphe à $\ZZ^2$. On en déduit que $\rH^{\rrd}_1(X, f)$ est de dimension $2$. Le lacet $\sigma$ engendrant $\rH_1(\CC^\times, \ZZ)$ y fournit un élément, et un autre est donné par la collection des $1$-simplexes
\begin{equation}
\begin{aligned}
\tau_{i, r}\colon [0, 1] &\longrightarrow \CC^{\times} \\
t &\longmapsto rt+\psfrac{1-t}{r}
\end{aligned}
\end{equation}
pour $r>0$, qui est la manière de regarder la droite réelle positive~$\RR_{>0}$ comme un cycle à décroissance rapide.

Les intégrales le long de $\sigma$ se calculent facilement à l'aide du théorème des résidus:
\begin{align}
\oint e^{-(x+\sfrac{1}{x})} \frac{dx}{x}&=\sum_{n=0}^{\infty} \frac{(-1)^n}{n!} \oint \Bigl(x+\frac{1}{x}\Bigr)^n \frac{dx}{x}=2\pi i \sum_{n=0}^\infty \frac{1}{(n!)^2} \\
\oint e^{-(x+\sfrac{1}{x})} dx &=\sum_{n=0}^{\infty} \frac{(-1)^n}{n!} \oint \Bigl(x+\frac{1}{x}\Bigr)^n dx=-2\pi i \sum_{n=1}^\infty \frac{1}{n!(n-1)!}.
\end{align}
Une manière de montrer que le déterminant de la matrice des périodes est non nul consiste à exprimer ses coefficients comme des valeurs spéciales des \emph{fonctions de Bessel modifiées}
\begin{align*}
I_0(z)&=\frac{1}{2\pi i}\oint e^{-\frac{z}{2}(x+\sfrac{1}{x})} \frac{dx}{x}=\sum_{n=0}^\infty \frac{(\sfrac{z}{2})^{2n}}{(n!)^2} \\
K_0(z)&=\frac{1}{2}\int_0^{\infty} e^{-\frac{z}{2}(x+\sfrac{1}{x})} \frac{dx}{x}
\end{align*}
et de leurs dérivées. Ces fonctions sont des solutions $\CC$-linéairement indépendantes de l'équation différentielle du second ordre
\[
z^2\frac{d^2u}{dz}+z\frac{du}{dz}-z^2u=0
\]
pour une fonction $u$ de la variable $z$. Par conséquent, leur wronskien
\[
W(z)=I_0(z)K_0'(z)-I_0(z)'K_0(z)
\]
satisfait à l'équation différentielle
\begin{align*}
W'(z)&=I_0(z)K_0''(z)-I_0''(z)K_0(z)\\
&=I_0(z)[K_0(z)-z^{-1}K_0(z)']-[I_0(z)-z^{-1}I_0(z)']K_0(z)\\
&=-\frac{1}{z}W(z).
\end{align*}
Il s'ensuit qu'il existe un nombre complexe non nul $c$ (non nul car les solutions sont linéairement indépendantes) tel que $W(z)=\sfrac{c}{z}$, et en regardant le développement asymptotique des fonctions de Bessel à l'infini on trouve que ce nombre vaut $c=-1$.

En termes des fonctions de Bessel modifiées, la matrice de l'accouplement des périodes exponentielles est donnée par
\begin{equation}
\def\arraystretch{1.2}
\setlength\arraycolsep{3pt}
\begin{array}{c|ccc}
\int & \sfrac{dx}{x} & & dx \\ \hline
\sigma & 2\pi i I_0(2) & & -2\pi i I_0'(2) \\ \hline
 \RR_{>0} & 2K_0(2) & & -2K_0'(2)
 \end{array}
\end{equation}
Seule l'identité faisant intervenir la dérivée de $K_0(z)$ n'est pas évidente: pour la démontrer, on peut par exemple dériver sous le signe intégral et utiliser le fait que $dx$ et $\sfrac{dx}{x^2}$ définissent la même classe en cohomologie de de~Rham tordue (prendre $P=1$ dans \eqref{eqn:diffDRtordu}):
\[
K_0'(z)=-\frac{1}{4}\int_0^\infty e^{-\frac{z}{2}(x+\sfrac{1}{x})} \Bigl(x+\frac{1}{x}\Bigr)\frac{dx}{x}=-\frac{1}{2}\int_0^\infty e^{-\frac{z}{2}(x+\sfrac{1}{x})}dx.
\]
Par conséquent, le déterminant de la matrice des périodes vaut
\begin{equation}\label{eqn:detBessel}
\det \begin{pmatrix}
2\pi i I_0(2) & & -2\pi i I_0'(2) \\
2K_0(2) & \hspace{2mm} & -2K_0'(2)
\end{pmatrix}=-4\pi i W(2)=2\pi i.
\end{equation}
C'est l'analogue de la relation de Legendre (proposition \ref{prop:legendre}). La transcendance d'aucun des coefficients de cette matrice n'est connue, mais on sait que le quotient $\sfrac{I_0(2)}{I_0'(2)}$ des périodes associées au cycle $\sigma$ est transcendant d'après le théorème de Siegel-Shidlovskii \cite{bertrand}.
\end{exemple}

\section{Retour sur la relation de Legendre}\label{sec:Legendre}

Comme on le verra dans la section suivante, munis de l'interprétation cohomologique des périodes, on peut reformuler la conjecture de Kontsevich-Zagier en disant que toutes les relations algébriques entre périodes sont \emph{d'origine géométrique}. Pour illustrer ce que cela peut signifier en pratique, on revient suivant Deligne \cite[pp.\,24-25]{deligne900} sur la relation de Legendre pour les périodes d'une courbe elliptique. Cette relation n'étant pas linéaire, on doit d'abord comprendre à quelle structure sur la cohomologie correspond le produit de périodes.

\subsection{Le cup-produit}\label{sec:cup-produit}

Soient~$M$ et $M'$ des espaces topologiques et~$M \times M'$ leur produit. Un $p$-simplexe singulier \hbox{$\sigma\colon \Delta^p \to M \times M'$} donne lieu, par projection, à des $p$-simplexes $\sigma_M$ et~$\sigma_{M'}$ dans~$M$ et dans~$M'$. Si $a$ et $b$ sont des entiers positifs avec $a+b=p$, en composant avec les applications
\begin{align*}
(t_0, \dots, t_a) &\longmapsto (t_0, \dots, t_a, 0, \dots, 0) \text{ de $\Delta^a$ dans~$\Delta^p$} \\ (t_0, \dots, t_b) &\longmapsto (0, \dots, 0, t_0, \dots, t_b) \text{ de $\Delta^b$ dans~$\Delta^p$},
\end{align*}
on en déduit un $a$\nobreakdash-simplexe $^a\sigma_M$ dans $M$ (la \textit{face avant} de $\sigma_M$) et un $b$\nobreakdash-simplexe $\sigma_{M'}^b$ dans $M'$ (la \textit{face arrière} de $\sigma_{M'}$). On vérifie ensuite que l'application
\begin{align*}
\rH_p(M \times M', \ZZ) &\longrightarrow \rH_a(M, \ZZ) \otimes \rH_b(M', \ZZ) \\
[\sigma] &\longmapsto [^a\sigma_M] \otimes [\sigma_{M'}^b]
\end{align*}
est bien définie, c'est-à-dire que si le $p$-simplexe $\sigma$ est un cycle, il en va de même pour le $a$-simplexe $^a\sigma_M$ et pour le $b$-simplexe $\sigma_{M'}^b$, et que les bords $\partial \tau$ sont envoyés sur $0$. Par dualité, la cohomologie singulière à coefficients rationnels est munie des applications \emph{produit externe}
\begin{equation}\label{eqn:exterior}
\rH^a(M, \QQ) \otimes \rH^b(M', \QQ) \longrightarrow \rH^{a+b}(M \times M', \QQ).
\end{equation}

\begin{theoreme}[formule de Künneth] Pour tout $p$, les applications produit externe induisent un isomorphisme
\begin{displaymath}
\bigoplus_{a+b=p} \rH^a(M, \QQ) \otimes \rH^b(M', \QQ) \simeq \rH^p(M \times M', \QQ).
\end{displaymath}
\end{theoreme}

Ce théorème est, par exemple, démontré dans \cite[Thm.\,3.16]{hatcher}. À ce stade, un avantage de dualiser pour travailler en cohomologie plutôt qu'en homologie est que, dans le cas où $M$ et $M'$ sont égaux, on dispose d'une opération de \emph{cup-produit}
\begin{displaymath}
\smile\colon \rH^a(M, \QQ) \otimes \rH^b(M, \QQ) \longrightarrow \rH^{a+b}(M, \QQ),
\end{displaymath}
définie par composition du produit externe \eqref{eqn:exterior} avec l'application
\[
\mathrm{diag}^\ast \colon \rH^{a+b}(M \times M, \QQ) \to \rH^{a+b}(M, \QQ)
\]
déduite par fonctorialité de la diagonale $\mathrm{diag} \colon M \longrightarrow M \times M$, qui envoie un point $x$ de $M$ vers le point $\mathrm{diag}(x)=(x, x)$ de $M \times M$. 

\begin{theoreme}[dualité de Poincaré]\label{thm:dualitePoincare} Si $M$ est une variété complexe compacte connexe de dimension $d$, la composition
\begin{displaymath}
\smile\colon \rH^p(M, \QQ) \otimes \rH^{2d-p}(M, \QQ) \longrightarrow \QQ(-d)
\end{displaymath}
du cup-produit avec l'isomorphisme canonique \hbox{$\rH^{2d}(M, \QQ) \simeq \QQ(-d)$} induit par l'analogue de la classe fondamentale (remarque \ref{rem:classefonda}) en dimension supérieure est un accouplement parfait pour tout $p$.
\end{theoreme}

On renvoie à \cite[Chap.\,3.2]{hatcher} pour la construction de la classe fondamentale et la démonstration de ce théorème. La dualité de Poincaré est notamment valable pour les points complexes d'une variété \emph{projective} lisse $X \subset \PP^n$ pourvu que $X(\CC)$ soit connexe, et permet d'identifier canoniquement le dual $\QQ$-linéaire de~\hbox{$\rH^p_{\Betti}(X)=\rH^p(X(\CC), \QQ)$} avec la cohomologie $\rH^{2d-p}_{\Betti}(X)(d)$. La torsion à la Tate est là pour nous rappeler que, pour que ces identifications soient compatibles à l'isomorphisme de comparaison, il faut multiplier les périodes par~$(2\pi i)^{-d}$. 

Il y a aussi une variante relative de la formule de Künneth:
\begin{equation}\label{eqn:Kunnethrelative}
\begin{aligned}
\bigoplus_{a+b=p} \rH^a(M, &N; \QQ) \otimes \rH^b(M', N'; \QQ) \\
&\simeq \rH^p(M \times M', (N \times M') \cup (M \times N'); \QQ),
\end{aligned}
\end{equation}
grâce à laquelle on peut étendre la dualité de Poincaré aux points complexes des variétés affines lisses. Plus précisément, on associe à toute paire $(X, D)$ formée d'une variété lisse~$X$ de dimension $d$ et d'un diviseur à croisements normaux $D$ sur $X$ une autre paire $(X', D')$ et un accouplement parfait
\begin{equation}\label{eqn:Poincarevarietesaffines}
\rH^d_{\Betti}(X, D) \otimes \rH^{2d-p}_{\Betti}(X', D') \longrightarrow \QQ(-d).
\end{equation}
Pour ce faire, on utilise le théorème de résolution des singularités de Hironaka pour trouver une variété projective lisse $\overline{X}$ et un morphisme de variétés $j\colon X \hookrightarrow \overline{X}$ identifiant $X(\CC)$ à un ouvert dans~$\overline{X}(\CC)$ de complémentaire les points complexes d'un diviseur à croisements normaux $Z \subset \overline{X}$. Si $D$ est vide, on pose $(X', D')=(\overline{X}, Z)$; en général, on considère l'adhérence $\overline{D} \subset \overline{X}$ de $D$ dans $X$ et on pose 
\[
(X', D')=(\overline{X} \setminus \overline{D}, Z \setminus (Z \cap \overline{D})), 
\] voir par exemple \cite[Thm.\,2.4.5]{huber-muller}. Par le biais de~\eqref{eqn:Poincarevarietesaffines}, le dual de la cohomologie relative $\rH^d_{\Betti}(X, D)$ s'identifie à $ \rH^{2d-p}_{\Betti}(X', D')(d)$, ce qui est encore une forme de dualité de Poincaré pour les variétés affines. Par exemple, si~\hbox{$X=\mathbb{G}_m$} est la droite affine épointé, on prend pour~$\overline{X}=\mathbb{P}^1$ la droite projective et pour $Z$ le diviseur formé des points $0$ et $\infty$, et on identifie le dual de $\rH^1_\Betti(\mathbb{G}_m)$ à $\rH^1_\Betti(\mathbb{P}^1, \{0, \infty\})(1)$. 

Toutes ces constructions ont des analogues en cohomologie de de~Rham algébrique. Soient $X \subset \mathbb{A}^n$ et $Y \subset \mathbb{A}^r$ des variétés affines lisses, d'anneaux de fonctions
\[
A=\QQ[x_1, \dots, x_n]/(f_1, \dots, f_m), \quad B=\QQ[x_1, \dots, x_r]/(g_1, \dots, g_s).
\]
Le produit de $X$ et $Y$ est la variété affine lisse $X \times Y \subset \mathbb{A}^{n+r}$ définie par l'annulation des polynômes
\[
f_i(x_1,\dots, x_n) \quad \text{et}\quad g_j(x_{n+1},\dots, x_{n+r}).
\]
Son anneau de fonctions s'identifie canoniquement au produit tensoriel $A \otimes B$ et les projections sur chaque facteur
\[
\mathrm{pr}_X \colon X \times Y \to X \quad \text{et}\quad \mathrm{pr}_Y \colon X \times Y \to Y
\]
correspondent aux inclusions $a \mapsto a \otimes 1$ et $b \mapsto 1 \otimes b$. Étant données des formes différentielles $\omega \in \Omega^a_X$ sur $X$ et $\eta \in \Omega^b_Y$ sur $Y$, on obtient par fonctorialité des formes différentielles
\begin{displaymath}
\mathrm{pr}_X^\ast \omega \in \Omega^a_{X \times Y} \quad \text{et}\quad \mathrm{pr}_Y^\ast \eta \in \Omega^b_{X \times Y}
\end{displaymath}
sur~$X \times Y$. Comme l'application $\mathrm{pr}_X^\ast$ commute à la différentielle, si les formes~$\omega$ et $\eta$ sont fermées, alors $\mathrm{pr}_X^\ast \omega \wedge \mathrm{pr}_Y^\ast \eta$ l'est aussi, est si $\omega$ ou $\eta$ sont exactes, $\mathrm{pr}_X^\ast \omega \wedge \mathrm{pr}_Y^\ast \eta$ l'est aussi; l'application
\begin{align*}
\rH_{\dR}^a(X) \otimes \rH_{\dR}^b(Y) &\longrightarrow \rH_{\dR}^{a+b}(X\times Y) \\
[\omega] \otimes [\eta] &\longmapsto \bigl[\mathrm{pr}_X^\ast \omega \wedge \mathrm{pr}_Y^\ast \eta\bigl]
\end{align*}
est donc bien définie. Comme dans le cas de la cohomologie singulière, si~$X$ et~$Y$ sont égales, on en déduit un cup-produit
\[
\smile \colon \rH_{\dR}^a(X) \otimes \rH_{\dR}^b(X) \longrightarrow \rH_{\dR}^{a+b}(X)
\]
en composant avec $\mathrm{diag}^\ast$. Cette application n'est pas très intéressante, il faut l'avouer, pour les variétés affines, car leur cohomologie s'annule au-delà de la dimension, mais on peut l'étendre par recollement en un cup-produit non trivial sur les variétés projectives lisses. On se contentera d'en donner une recette pour les~courbes.

\subsection{Le cas des courbes}

Soit $X \subset \PP^2$ une courbe projective lisse. En combinaison avec l'isomorphisme canonique $\rH^2_{\dR}(X) \simeq \QQ(-1)$ de la remarque \ref{remq:isomcan}, le cup\nobreakdash-produit en cohomologie de de~Rham fournit un accouplement
\[
\smile \colon \rH^1_{\dR}(X) \otimes \rH^1_{\dR}(X) \longrightarrow \QQ(-1).
\]
Après avoir représenté les éléments de~$\rH^1_{\dR}(X)$ par des différentielles de deuxième espèce, c'est\nobreakdash-à\nobreakdash-dire par des formes différentielles sur une carte affine de~$X$ ayant résidu zéro en chaque pôle à l'infini, il peut être calculé comme suit. Soient $\omega$ et $\eta$ de telles formes, avec des pôles dans l'ensemble fini $D \subset X$. Dans une coordonnée locale $z$ autour de chaque point~$P$ de $D$, la forme $\omega$ s'écrit
\[
\omega=\biggl(\frac{a_{-r}}{z^r}+\cdots+\frac{a_{-2}}{z^2}+a_0+a_1z+\cdots\biggr)dz.
\]
Il existe donc une primitive méromorphe de $\omega$, bien définie à une constante près. En notant $\int \omega$ n'importe quel choix d'une telle primitive, le cup-produit est donné par la formule
\[
[\omega]\smile [\eta]=\sum_{P \in D} \mathrm{Res}_P \bigl(\int \omega)\eta,
\]
qui est indépendante de ces choix car $\eta$ n'a pas de résidu.

\begin{exemple}\label{exmp:cupproduitdRell} Pour calculer le cup-produit de $\omega=\sfrac{dx}{y}$ et $\eta=x\sfrac{dx}{y}$ sur une courbe elliptique $E$, le seul pôle à considérer est d'après la section~\ref{exmp:courbeelliptique} le point~$O$ à l'infini. Au vu des expressions locales
\[
\omega=\bigl(-1+\cdots\bigr)dz, \quad \eta=\Bigl(-\frac{1}{z^2}+\cdots\Bigr)dz
\]
autour de ce point, en choisissant la primitive $\int \omega=-z+\cdots$ on trouve le cup-produit
\[
[\omega]\smile [\eta]=\mathrm{Res}_O\Bigl(\frac{dz}{z}+\cdots\Bigr)=1.
\]
\end{exemple}

D'un autre côté, l'évaluation en la classe fondamentale fournit un isomorphisme canonique $\rH_\Betti^2(X)\simeq \QQ(-1)$, à travers lequel le cup-produit en cohomologie de Betti donne lieu à une application
\begin{equation}\label{eqn:CupProductCurve}
\smile \colon \rH^1_{\Betti}(X) \otimes \rH^1_{\Betti}(X) \longrightarrow \QQ(-1)
\end{equation}
qui s'avère être le dual du produit d'intersection topologique: si~$[\sigma]^\ast$ et $[\tau]^\ast$ sont les duaux des classes de lacets $\sigma$ et $\tau$ dans~$X(\CC)$, on peut bouger $\sigma$ et $\tau$ dans leur classe d'homologie jusqu'à ce qu'ils se rencontrent transversalement et compter le nombre de points d'intersection, avec un signe positif $i_P(\sigma, \tau)=1$ si autour de ce point $\sigma$ et~$\tau$ définissent la même orientation que $1$ et $i$ dans le plan complexe et négatif $i_P(\sigma, \tau)=-1$ sinon. Le cup-produit est alors donné par
\begin{equation}\label{eqn:cup-prod-top}
[\sigma]^\ast\smile [\tau]^\ast=\sum_{P \in \sigma \cap \tau} i_P(\sigma, \tau), 
\end{equation} voir par exemple \cite[A.2.3]{Bost}. 

\begin{exemple}\label{exmp:cupproduitBettiell} Écrivons $E(\CC)=\CC/\ZZ\omega_1 \oplus \ZZ \omega_2$ et soient $\sigma_1$ et $\sigma_2$ les lacets engendrant $\rH_1(E(\CC), \ZZ)$ décrits dans l'exemple \ref{exmp:torecomplexe}. Comme la formule \eqref{eqn:cup-prod-top} définit visiblement une forme bilinéaire alternée, pour déterminer la matrice du cup-produit il suffit de calculer l'intersection topologique de $\sigma_1$ et $\sigma_2$. Supposons que $\tau=\sfrac{\omega_2}{\omega_1}$ est dans le demi-plan de Poincaré. Alors les cycles $\sigma_1$ et $\sigma_2$ définissent la même orientation que $1$ et $i$ autour de leur seul point d'intersection et le cup-produit est donné par la matrice
\begin{displaymath}
\left(\begin{matrix}
[\sigma_1]\smile[\sigma_1] & \vspace{2mm} & [\sigma_1]\smile [\sigma_2] \\
[\sigma_2]\smile [\sigma_1] & & [\sigma_2]\smile [\sigma_2]
\end{matrix}\right)=\left(\begin{matrix}
0 & 1 \\
-1 & 0\end{matrix}\right).
\end{displaymath}
En accord avec la dualité de Poincaré (théorème \ref{thm:dualitePoincare}), le déterminant de cette matrice n'est pas nul.
\end{exemple}

De plus, les cup-produits en cohomologie de de~Rham et en cohomologie de Betti sont compatibles avec l'isomorphisme de comparaison~\eqref{eqn:periodscompisom}, en ce sens que le diagramme
\begin{equation}\label{eqn:commutative-cup-produit}
\xymatrix{
(\rH^1_{\dR}(X) \otimes \CC) \otimes (\rH^1_{\dR}(X) \otimes \CC) \ar[d] \ar[rr]^{\hspace{2cm}\smile} && \CC \ar[d]^{\cdot 2\pi i} \\ (\rH^1_{\Betti}(X) \otimes \CC) \otimes (\rH^1_{\Betti}(X) \otimes \CC) \ar[rr]^{\hspace{2.5cm}\smile} && \CC
}\end{equation}
déduit par extension des scalaires aux nombres complexes est commutatif. Examinons le contenu de cette commutativité dans le cas où~$X$ est une courbe elliptique. En composant la première flèche horizontale avec la multiplication par $2\pi i$ on trouve
\[
2\pi i \left([\omega]\smile [\eta]\right)=2\pi i
\]
d'après l'exemple \ref{exmp:cupproduitdRell}. D'un autre côté, l'isomorphisme de comparaison envoie $[\omega]$ sur $\omega_1[\sigma_1]^\ast+\omega_2[\sigma_2]^\ast$, et $[\eta]$ sur $\eta_1[\sigma_1]^\ast+\eta_2[\sigma_2]^\ast$, de~sorte que l'autre manière de parcourir le diagramme donne
\[
\left(\omega_1[\sigma_1]^\ast+\omega_2[\sigma_2]^\ast \right) \smile \left(\eta_1[\sigma_1]^\ast+\eta_2[\sigma_2]^\ast \right)=\omega_1\eta_2-\omega_2\eta_1
\]
vu l'exemple \ref{exmp:cupproduitBettiell}. Par conséquent, cette quantité vaut $2\pi i$\footnote{Des constructions similaires, avec la cohomologie à décroissance rapide et la cohomologie de de~Rham tordue, donnent une interprétation géométrique de la relation \eqref{eqn:detBessel} entre les valeurs spéciales des fonctions de Bessel modifiées.}!

\section{Vers une théorie de Galois pour les périodes}\label{sec:period-conjecture}

Si l'on se donne une période par une représentation intégrale à la Kontsevich-Zagier, la première étape pour prédire ses propriétés de transcendance consiste à trouver une variété algébrique $X$ telle que la matrice de l'accouplement entre un groupe de cohomologie de de~Rham et un groupe d'homologie de Betti de $X$, éventuellement relatives à une sous-variété, contienne la période en question parmi ses coefficients (c'est toujours possible grâce au théorème \ref{thm:comparison}). Ces espaces vectoriels étant rarement de dimension $1$, on trouve d'autres périodes chemin faisant; pour les nombres algébriques, ces compagnons n'étaient rien d'autre que leurs conjugués au sens de la théorie de Galois, c'est-à-dire les autres racines du polynôme minimal (exemple~\ref{exmp:algebriques}). Par le biais de la théorie des motifs, on peut associer aux périodes des \textit{groupes de Galois motiviques} qui conjecturalement permutent ces nombres en respectant toutes leurs relations algébriques. Puisque les périodes sont en général des nombres transcendants, on doit s'attendre à ce que ces groupes ne soient plus des groupes finis agissant sur des ensembles finis, comme en théorie de Galois classique, mais plutôt des groupes de matrices agissant sur des espaces vectoriels: des groupes algébriques linéaires.

\subsection{Le groupe de Galois motivique}

Soit $n \geq 1$ un entier. Le \emph{groupe général linéaire} $\GL_n$ est le lieu des zéros dans l'espace affine $\mathbb{A}^{n^2+1}$ de coordonnées $x_{11}, x_{12}, \dots, x_{nn}, y$ du polynôme $\det(x_{ij})y-1$, qui est la façon algébrique d'exprimer la contrainte que la matrice $(x_{ij})$ soit inversible. C'est une variété affine lisse de dimension $n^2$, d'anneau des fonctions
\[
A=\QQ[x_{11}, \dots, x_{nn}, y]/(\det(x_{ij})y-1).
\]
De plus, cette variété est munie d'une structure de \emph{groupe algébrique}, notamment des morphismes de variétés
\[
m \colon \GL_n \times \GL_n \longrightarrow \GL_n \quad \text{et}\quad \iota \colon \GL_n \longrightarrow \GL_n
\]
le produit et l'inversion des matrices, satisfaisant aux axiomes usuels dans la définition de groupe. Par exemple, la multiplication correspond au morphisme entre anneaux de fonctions
\[
m^\ast\colon A \longrightarrow A \otimes A, \qquad x_{ij} \longmapsto \sum_{\ell=1}^n x_{i\ell} \otimes x_{\ell j}, \quad y \mapsto y \otimes y,
\]
où l'on reconnaît la formule classique pour les coefficients de la matrice produit et le fait que le déterminant est multiplicatif.

En utilisant le point de vue plus abstrait sur les variétés affines évoqué à la fin de la section \ref{sec:varietes}, on associe à tout $\QQ$\nobreakdash-espace vectoriel de dimension finie~$V$ un groupe algébrique affine $\GL_V$ sur~$\QQ$; pour une~$\QQ$-algèbre $R$, ses $R$-points sont les automorphismes $R$\nobreakdash-linéaires
\[
V \otimes R \longrightarrow V \otimes R.
\]
Moyennant le choix d'une base de $V$, on retrouve la description précédente. Sans ce choix, on peut décrire les fonctions sur~$\GL_V$ comme la $\QQ$\nobreakdash-algèbre obtenue en inversant le déterminant dans l'anneau engendré par les fonctions constantes et par les \emph{coefficients matriciels}
\[
[V, v, \varphi] \quad \text{pour tous } v \in V\text{ et } \varphi \in V^\ast=\Hom(V, \QQ),
\]
modulo les relations de bilinéarité
\begin{align*}
[V, \lambda v+\lambda' v', \varphi]=\lambda[V, v, \varphi]+\lambda'[V, v', \varphi] \\
[V, v, \lambda\varphi+\lambda'\varphi']=\lambda[V, v, \varphi]+\lambda'[V, v, \varphi']
\end{align*}
pour tous $\lambda, \lambda' \in \QQ$. Pour ce faire, on interprète le coefficient matriciel associé à $v$ et à $\varphi$ comme la fonction
\begin{align*}
\GL_V(R) &\longrightarrow R, \\ g &\longmapsto \varphi(g(v)). 
\end{align*}
Si l'on identifie $\GL_V$ à $\GL_n$ par le choix d'une base $v_1, \dots, v_n$, les fonctions $x_{ij}$ correspondent aux symboles $[V, v_j, v_i^\ast]$. Les coefficients matriciels $[V, v, \varphi]$ sont donc les polynômes de degré $1$ en les $x_{ij}$. 

Pour obtenir des polynômes de degré supérieur, on remplace l'espace vectoriel $V$ dans les coefficients matriciels par des constructions tensorielles. Si~$g$ est un automorphisme de $V$, alors $g \otimes g$ est un automorphisme du produit tensoriel~$V \otimes V$; un élément $v_1 \otimes v_2$ de $V \otimes V$ et un élément $\varphi_1 \otimes \varphi_2$ de $(V \otimes V)^\ast=V^\ast \otimes V^\ast$ définissent la fonction 
\begin{align*}
\GL_V(R) &\longrightarrow R, \\ g &\longmapsto \varphi_1(g(v_1))\varphi(g(v_2)), 
\end{align*} que l'on note $[V \otimes V, v_1 \otimes v_2, \varphi_1 \otimes \varphi_2]$. La structure d'anneau sur les coefficients matriciels est alors donnée par
\[
[V, v_1, \varphi_1]\cdot [V, v_2, \varphi_2]=[V \otimes V, v_1 \otimes v_2, \varphi_1 \otimes \varphi_2].
\]
De ce point de vue, le déterminant est la fonction
\[
\Bigl[V^{\otimes n}, \sum_{\sigma \in \mathfrak{S}_n} \varepsilon(\sigma) v_{\sigma(1)} \otimes \cdots \otimes v_{\sigma(n)},\,v_1^\ast\otimes\cdots\otimes v_n^\ast\Bigr]
\]
pour n'importe quel choix de base $v_1, \dots, v_n$ de $V$ (le résultat n'en dépend pas). Cette description n'est pas encore tout à fait satisfaisante, parce que ni les fonctions constantes ni l'inverse du déterminant ne sont pour l'instant des coefficients matriciels. Pour traiter les constantes, on peut faire jouer au symbole $[\QQ, 1, 1]$, où le second $1$ désigne l'application identité, le rôle de la fonction constante de valeur~$1$. Pour inverser le déterminant, on recourt au dual. Tout espace vectoriel $H$ est muni d'applications
\[
H \otimes H^\ast \to \QQ \quad \text{et}\quad \QQ \to H^\ast \otimes H,
\]
définies par évaluation et en envoyant $1$ vers l'élément correspondant à l'application identité à travers l'isomorphisme $H^\ast \otimes H\simeq \mathrm{End}(H)$, et que ces applications sont duales l'une de l'autre. Il est alors naturel de considérer parmi les coefficients matriciels tous les produits tensoriels de $V$ et de $V^\ast$, et pas seulement de $V$, et d'imposer aux symboles obtenus les relations de compatibilité avec les applications ci\nobreakdash-dessus
\[
[H \otimes H^\ast, h \otimes \varphi, a\cdot \mathrm{Id}_H]=[\QQ, \varphi(h), a]
\]
pour tous $h \in H, \varphi \in H^\ast$ et $a \in \QQ$. Par exemple, si l'espace vectoriel de départ $V$ est de dimension $1$ et $v$ est un vecteur non nul, en combinant la règle du produit avec cette relation on trouve
\[
[V, v, v^\ast]\cdot[V^\ast, v^\ast, v]=[V \otimes V^\ast, v \otimes v^\ast, v^\ast \otimes v]=[\QQ, 1, 1],
\]
ce qui signifie précisément que la fonction $[V, v, v^\ast]$ est inversible!

Soient $p$ un entier, $X$ une variété algébrique affine lisse et $D \subset X$ un diviseur à croisements normaux (éventuellement vide), tout étant défini par des polynômes à coefficients rationnels. Notons
\[
V_{\dR}=\rH_{\dR}^p(X, D) \quad\text{et}\quad V_\Betti=\rH_\Betti^p(X, D)
\]
la cohomologie de de~Rham et la cohomologie de Betti de $X$ relatives à~$D$. Dans une approximation grossière, le \emph{motif}~$\rH^p(X, D)$ est le triplet formé par ces deux espaces vectoriels, que l'on appelle la \emph{réalisation de Betti} et la \emph{réalisation de de~Rham} du motif, et par l'isomorphisme de comparaison $\mathrm{comp}\colon V_{\dR} \otimes \CC \to V_{\Betti} \otimes \CC$ entre leurs complexifiés. On dit que deux tels triplets sont isomorphes lorsqu'il existe des isomorphismes \og de nature géométrique\fg entre les espaces vectoriels sous-jacents compatibles à l'isomorphisme de comparaison. 

Le \emph{lego des motifs} consiste alors à identifier des briques simples et à comprendre comment elles se combinent pour donner lieu à des motifs plus compliqués. La plus simple est sans doute le motif de la variété $X \subset \AA^1$ de dimension $0$ définie par l'équation $x=0$, dont la cohomologie de Betti et la cohomologie de de~Rham sont isomorphes à $\QQ$, avec isomorphisme de comparaison l'identité; on le note $\QQ(0)$ et on l'appelle le \emph{motif trivial}. Tout de suite après vient le motif de la droite affine épointée $X=\mathbb{G}_m$, dont la cohomologie de Betti et la cohomologie de~Rham sont encore isomorphes à $\QQ$, mais dont l'isomorphisme de comparaison est, cette fois-ci, la multiplication par~$2\pi i$; on le note~$\QQ(-1)$ et on l'appelle le \emph{motif de Tate}. Ce motif a des nombreuses incarnations: il est, par exemple, isomorphe à $\rH^1(\mathbb{G}_m, \{1\})$ ou à $\rH^2(\PP^1)$ si l'on inclut les variétés projectives dans cette discussion. Pour donner une idée de comment ces motifs se recombinent, le fait que la matrice des périodes \eqref{eqn:periodslog} soit triangulaire supérieure avec les nombres~$1$ et $2\pi i$ sur la diagonale reflète le fait que le motif~$\rH^1(\mathbb{G}_m, \{1, q\})$ s'insère dans une suite exacte
\[
0 \to \QQ(0) \to \rH^1(\mathbb{G}_m, \{1, q\}) \to \QQ(-1) \to 0,
\]
où les morphismes proviennent des suites exactes longues de cohomologie relative et sont donc de nature géométrique; on dit que le motif~$\rH^1(\mathbb{G}_m, \{1, q\})$ est une \emph{extension} de $\QQ(-1)$ par $\QQ(0)$. De même, la discussion à la fin de l'exemple \eqref{exmp:dilogs} montre que le motif $\rH^2(X, D)$ issue de la représentation intégrale du dilogarithme est une extension de~$\rH^1(\mathbb{G}_m, \{1, 1/\alpha\})$ par $\QQ(0)$.  

La théorie des motifs permet d'associer à $\rH^p(X, D)$ un \emph{groupe de Galois motivique}, qui est un sous-groupe\vspace*{-3pt}
\begin{equation}\label{eqn:groupedeGaloismotivique}
G \subset \GL_{\rH_\Betti^p(X, D)}
\end{equation}
défini par des équations polynomiales. Plutôt que le groupe, on construit son algèbre de fonctions, à partir de laquelle on peut le retrouver grâce à la relation \eqref{eqn:foncteurpoints2}. Comme pour le groupe général linéaire, outre l'objet $\rH^p(X, D)$, on considère toutes ses puissances tensorielles $\rH^p(X, D)^{\otimes a}$, qui dans la description grossière correspondent aux espaces vectoriels $V_{\dR}^{\otimes a}$ et $V_{\Betti}^{\otimes a}$ et à l'isomorphisme~$\mathrm{comp}^{\otimes a}$. Le point clé est que cette construction est encore \og géométrique\fg, en ce sens que la formule de Künneth \eqref{eqn:Kunnethrelative} permet de voir ces espaces comme des morceaux de la cohomologie d'une variété algébrique; par exemple, le carré tensoriel est un facteur direct
\[
\rH^p(X, D)^{\otimes 2} \subset \rH^{2p}(X\times X, D \times X \cup X \times D).
\]
De même, en utilisant la dualité de Poincaré \eqref{eqn:Poincarevarietesaffines}, on peut donner un sens géométrique au dual $\rH^p(X, D)^\ast$, et donc à toutes ses puissances tensorielles en la combinant avec la formule de Künneth. La seule subtilité est que le dual s'identifie à $\rH^{2d-p}(X', D')(d)$ et qu'il faudrait savoir interpréter géométriquement cette torsion à la Tate par des puissances négatives de $(2\pi i)^{-d}$, ce qui est un peu problématique si ce nombre, comme on le croit, n'est pas une période effective (remarque~\ref{noneffectif}). La solution consiste à l'inverser formellement, c'est\nobreakdash-à\nobreakdash-dire à \og inventer\fg un motif $\QQ(1)$ dont les périodes sont les multiples rationnels de $\sfrac{1}{2\pi i}$. Plus généralement, on appelle \emph{construction tensorielle} sur le motif $\rH^p(X, D)$ toute somme directe finie
\[
\bigoplus_i \rH^p(X, D)^{\otimes a_i} \otimes (\rH^p(X, D)^\ast)^{\otimes b_i}
\]
pour des entiers $a_i, b_i \geq 0$. La théorie des motifs permet alors de sélectionner une classe $\mathcal{C}$ de sous-quotients de ces constructions tensorielles et de morphismes entre eux. Par analogie avec $\GL_V$, on considère le $\QQ$-espace vectoriel engendré par des symboles
\[
[H, v, \varphi],
\]
où $H$ est un objet dans $\mathcal{C}$ (en particulier, la donnée d'un espace vectoriel $H_{\Betti}$) et $v$ et $\varphi$ sont des vecteurs dans $H_\Betti$ et dans $H_\Betti^\ast$, modulo les relations de bilinéarité et les relations
\begin{equation}\label{eqn:relations-tannakiennes}
[H, v, \transp{f_B}(\varphi)]=[H', f_B(v), \varphi]
\end{equation}
pour tout morphisme $f \colon H \to H'$ dans la classe $\mathcal{C}$ (un tel morphisme est en particulier la donnée d'une application linéaire $f_B \colon H_\Betti \to H'_\Betti$ et $\transp{f_B}\colon (H_\Betti')^\ast \to (H_\Betti')^\ast$ désigne sa transposée).

Tout le sel est dans le choix de la classe $\mathcal{C}$: si, par exemple, $\mathcal{C}$ contenait \emph{tous} les sous-espaces et tous les morphismes entre eux, la relation~\eqref{eqn:relations-tannakiennes} impliquerait que toute fonction est constante et on trouverait le groupe trivial. Grosso modo, la solution consiste à autoriser seulement des constructions d'algèbre linéaire (noyau, conoyau, dual, etc.) sur des morphismes d'origine géométrique tels que les applications~\eqref{eqn:fonctorBettirel} induites par des morphismes de paires de variétés affines ou les duaux des morphismes connectant dans la suite exacte longue \eqref{eqn:suiteexactelonguerelative}. Par exemple, le motif d'une variété qui, comme la courbe de Fermat (exemple \ref{exmp:Fermat}), a beaucoup d'automorphismes aura un petit groupe de Galois car les relations imposent que ce groupe soit contenu dans le commutant des matrices induites par tous ces automorphismes. Un autre exemple de relation, si l'on discutait le motif d'une courbe projective $X \subset \PP^2$ de genre $g$, serait donné par la compatibilité au cup-produit \eqref{eqn:CupProductCurve}, qui implique que les éléments du groupe de Galois motivique de $\rH^1(X)$ doivent respecter la forme symplectique standard à un scalaire près (ce scalaire étant relié au fait que le diagramme \eqref{eqn:commutative-cup-produit} n'est commutatif que si l'on multiplie par $2\pi i$ à droite). Ce groupe est donc contenu dans le groupe 
\[
\mathrm{GSp}_{2g}=\{g \in \GL_{2g} \mid \transp{g}Jg=\mu(g)J \} 
\]
des similitudes symplectiques, avec $J=\left(\begin{smallmatrix} 0 & I_g \\ -I_g  & 0 \end{smallmatrix}\right)$ la matrice de la forme symplectique standard et $\mu(g)$ un scalaire; pour une courbe \emph{générique} c'est la seule contrainte.

\subsection{La conjecture des périodes de Grothendieck}

La dimension du groupe de Galois motivique est déjà censée être un invariant riche, si l'on croit à cette conjecture de Grothendieck qui, comme on l'expliquera à la fin de la section \ref{sec:periodesformelles}, est à peu de choses près équivalente à celle de Kontsevich-Zagier:

\begin{conjecture}[Grothendieck]\label{conj:perGroth} Le degré de transcendance du corps engendré par les coefficients de l'accouplement de périodes entre la cohomologie de de~Rham $\rH_\dR^p(X, D)$ et l'homologie de Betti $\rH^\Betti_p(X, D)$ est égal à la dimension du groupe de Galois motivique \eqref{eqn:groupedeGaloismotivique}.
\end{conjecture}

La seule trace de cette conjecture dans les écrits \emph{publiés} de Grothendieck est une note de bas de page dans l'article \cite{GrodR} où il démontre l'isomorphisme entre les cohomologies de de~Rham algébrique et analytique: après avoir mentionné le théorème de Schneider sur la transcendance des périodes $\omega_1$ et $\omega_2$ d'une courbe elliptique définie sur $\overline\QQ$, il dit qu'on s'attend à ce que ces deux périodes soient algébriquement indépendantes si la courbe elliptique n'a pas multiplication complexe et que cette conjecture s'étend de manière évidente aux quatre périodes et plus généralement aux $4g^2$ périodes d'une courbe projective lisse de genre~$g$. Pour un historique de la conjecture, je renvoie le lecteur à la lettre de André en appendice à \cite{bertolin}.

\begin{exemple}[le groupe de Galois de $2\pi i$]\label{exmp:Galoispi} Pour $X=\mathbb{G}_m$, la cohomologie $\rH^1_{\Betti}(X)$ est de dimension $1$. Le groupe de Galois motivique est donc un sous-groupe de $\mathrm{GL}_1$, avec coordonnée la fonction
\[
x=\bigl[\rH^1_{\Betti}(X), [\sigma]^\ast, [\sigma]\bigr],
\]
où $\sigma$ est le générateur usuel de l'homologie de $\CC^\times$. Hormis lui-même, les seuls sous-groupes algébriques de~$\mathrm{GL}_1$ sont les racines de l'unité d'un certain ordre, définies par l'annulation du polynôme $x^n-1$. Puisque le produit des fonctions est donné par le produit tensoriel de motifs et que la fonction $1$ correspond au motif trivial, il faut donc décider si $\rH^1(X)^{\otimes n}$ peut être isomorphe à $\QQ(0)$ pour un entier~$n \geq 1$. Or, vu que les périodes de $\rH^1(X)^{\otimes n}$ sont les multiples rationnels de~$(2\pi i)^n$ et que les périodes de $\QQ(0)$ sont les nombres rationnels, un tel isomorphisme impliquerait que $(2\pi i)^n$ est rationnel; la transcendance de $\pi$ entraîne donc que le groupe de Galois motivique est~$\GL_1$ en entier. Ce qui rend la théorie intéressante est que l'on peut s'en passer pour établir ce fait; dans une approximation plus fine de la notion de motif, les espaces $V_B$ et $V_{\dR}$ sont munis d'une \emph{filtration par le poids} compatible au produit tensoriel, pour laquelle~$\rH^1(X)$ est pur de poids $2$, tandis que $\QQ(0)$ est pur de poids zéro. Quoi qu'il en soit, on conclut que le groupe de Galois motivique de $\rH^1(X)$ est~$\GL_1$; sa dimension est égale au degré de transcendance du corps $\QQ(2\pi i)$, comme le prédit la conjecture de Grothendieck, et ses~$\QQ$\nobreakdash-points sont $\GL_1(\QQ)=\QQ^\times$, en accord avec l'heuristique de la section \ref{sec:Galoistheoryheuristics}.
\end{exemple}

\begin{exemple}[les logarithmes]\label{eqn:Galoislogs} Soit $q>1$ un nombre rationnel. D'après l'exemple \ref{exmp:log-cohomologique}, le logarithme de $q$ est une période du motif $\rH^1(\mathbb{G}_m, \{1, q\})$ de dimension $2$; son groupe de Galois motivique est donc un sous-groupe de $\GL_2$. Le fait que ce motif soit une extension de $\QQ(-1)$ par $\QQ(0)$, dont les groupes de Galois motiviques sont~$\GL_1$ et le groupe trivial respectivement, impose la contrainte que le groupe consiste en des matrices de la forme $\begin{psmallmatrix}1 & a \\0 & b \end{psmallmatrix}$, avec $b$ non nul. On remarque qu'il n'y a que deux tels groupes: ou bien $a$ est nul ou bien il prend n'importe quelle valeur. Le premier cas signifierait que~$\rH^1(\mathbb{G}_m, \{1, q\})$ est isomorphe à la somme directe de~$\QQ(0)$ et~$\QQ(-1)$ (on dit alors que l'extension est \emph{scindée}), donc que $\log(q)$ peut s'écrire comme une combinaison linéaire à coefficients rationnels de $1$ et de $2\pi i$, pour que la matrice des périodes devienne $\begin{psmallmatrix}1 & 0 \\0 & 2\pi i \end{psmallmatrix}$ dans une base convenable. Or, comme $\log(q)$ est réel et irrationnel, ce n'est pas possible. Le groupe de Galois de~$\log(q)$ est donc le groupe des matrices $\begin{psmallmatrix}1 & a \\0 & b \end{psmallmatrix}$, et ses conjugués devraient être les périodes
\begin{displaymath}
\begin{psmallmatrix}1 & a \\0 & b \end{psmallmatrix} \cdot \log(q)=b\log(q)+a 2\pi i,
\end{displaymath}
avec $a$ et $b$ rationnels et $b$ non nul. En accord avec l'heuristique basée sur la monodromie, on trouve $\log(q)+n2\pi i$ parmi les conjugués de~$\log(q)$. Comme le groupe de Galois motivique est de dimen\-sion~$2$, la conjecture des périodes de Grothendieck prédit que les nombres~$2\pi i$ et $\log q$ sont algébriquement indépendants.
\end{exemple}

\subsection{Le cas des courbes elliptiques}\label{sec:grothendieck-elliptic}

Soit $E \subset \PP^2$ une courbe elliptique d'équation affine
\[
y^2=4x^3-g_2x-g_3,
\]
comme dans la section \ref{exmp:int-elliptiques}, sauf qu'on s'autorise à prendre pour $g_2$ et~$g_3$ des nombres \emph{algébriques}. L'ensemble des points complexes
\begin{displaymath}
E(\CC)=\CC\slash \Lambda
\end{displaymath}
est muni de la loi de groupe abélien:
\begin{displaymath}
z \, (\text{mod } \Lambda)+ z'\, (\text{mod } \Lambda)=z+z'\, (\text{mod } \Lambda).
\end{displaymath}
On s'intéresse aux applications holomorphes \hbox{$E(\CC) \to E(\CC)$} qui sont aussi des homomorphismes de groupes. Par exemple, si $\alpha$ est un nombre complexe stabilisant le réseau $\Lambda$, c'est-à-dire pour lequel on a l'inclusion \hbox{$\alpha \Lambda \subset \Lambda$}, alors la fonction \og multiplication par $\alpha$\fg de $\CC$ dans $\CC$ induit une telle application sur le quotient. En utilisant que~$\CC$ est le revêtement universel de~$E(\CC)$, on peut montrer qu'elles sont toutes de cette forme; la question de les caractériser devient ainsi: quels sont les nombres complexes $\alpha \in \CC$ satisfaisant à $\alpha \Lambda \subset \Lambda$?

Les nombres entiers ont toujours cette propriété; pour la plupart des réseaux, ce sont les seuls. En effet, écrivons $\Lambda=\ZZ \oplus \ZZ\tau$ avec~$\tau$ dans le demi-plan de Poincaré et supposons qu'il y existe~\hbox{$\alpha \in \CC \setminus \ZZ$} satisfaisant à~$\alpha \Lambda \subset \Lambda$, c'est-à-dire tel que $\alpha$ et $\alpha \tau$ appartiennent à~$\Lambda$. On peut alors trouver des entiers $a, b, c, d \in \ZZ$ tels~que
\begin{equation}\label{eqn:definingCM}
\alpha=a+b\tau, \qquad \alpha \tau=c+d\tau,
\end{equation}
avec $b$ non nul car autrement $\alpha$ serait entier. De là, en divisant la seconde égalité par la première, il vient que $\tau$ satisfait à l'équation
\begin{equation}\label{eqn:quad1}
b\tau^2+(a-d)\tau-c=0,
\end{equation}
c'est-à-dire qu'il s'agit d'un nombre algébrique de degré $2$. Réciproquement, si $\tau$ est un tel nombre, le réseau $\Lambda=\ZZ \oplus \ZZ \tau$ est stable par multiplication par $\tau$. Par ailleurs, en remplaçant $\tau=\psfrac{\alpha-a}{b}$ dans la seconde égalité dans \eqref{eqn:definingCM} on obtient
\begin{equation}\label{eqn:quad2}
\alpha^2-(a+d)\alpha+ad-bc=0,
\end{equation}
ce qui implique que $\alpha$ lui-même doit être un nombre algébrique de degré $2$ et, qui plus est, un entier algébrique (racine d'un polynôme unitaire à coefficients entiers). Il n'est pas réel, puisque les équations quadratiques \eqref{eqn:quad1} et~\eqref{eqn:quad2} ont même discriminant et que celui de la première est négatif car~$\tau$ n'est pas réel. On dit alors que la courbe elliptique $E$ est à \emph{multiplication complexe} par le corps quadratique imaginaire $\QQ(\tau)$.

Supposons que la courbe elliptique $E$ soit à multiplication complexe et revenons à l'étude de ses périodes. Comme les points complexes~$E(\CC)$ peuvent être uniformisés par le réseau $\Lambda=\ZZ\oplus \ZZ\tau$ avec
\begin{equation}\label{eqn:reseauCM}
\tau=\frac{\int_{\sigma_2} \sfrac{dx}{y}}{\int_{\sigma_1} \sfrac{dx}{y}}=\frac{\omega_2}{\omega_1},
\end{equation}
et comme ce quotient est algébrique, les deux périodes ci-dessus sont algébriquement dépendantes. Masser \cite[Lem.\,3.1]{masser} a remarqué qu'il y a une deuxième relation linéaire à coefficients algébriques parmi les périodes de $E$, à savoir
\begin{equation}\label{eqn:masser}
c\eta_2+b\tau \eta_1=\beta \omega_2
\end{equation}
pour les entiers $b$ et $c$ dans \eqref{eqn:definingCM} et un nombre algébrique $\beta$ dans le même corps que les coefficients $g_2, g_3$ de l'équation de $E$. De plus, $\beta$ est nul si et seulement si $g_2$ ou $g_3$ est nul.

\begin{thm}[Chudnovsky]\label{thm:chudnovsky} Soit $E$ une courbe elliptique définie sur~$\overline\QQ$. Le degré de transcendance du corps engendré par les périodes de $E$ est au moins $2$:
\begin{displaymath}
\mathrm{deg tr}\,\overline\QQ(\omega_1, \omega_2, \eta_1, \eta_2) \geq 2.
\end{displaymath}
En particulier, lorsque la courbe elliptique $E$ est à multiplication complexe, ce degré de transcendance est égal à $2$.
\end{thm}

Par exemple, la courbe elliptique $y^2=4x^3-4x$ est à multiplication complexe par le corps $\QQ(i)$. D'un point de vue algébrique, la multiplication par $i$ est donnée par le morphisme de variétés $(x, y) \mapsto (-x, iy)$, dont le carré au sens de la composition est $(x, y) \mapsto (x, -y)$, la multiplication par $-1$ pour la loi de groupe sur la courbe. On peut dans ce cas calculer directement les périodes, en considérant la projection des cycles $\sigma_1$ et $\sigma_2$ vers la coordonnée~$x$. En utilisant l'identité \eqref{eqn:Eulerformula} reliant la fonction bêta et la fonction gamma et la formule de réflexion \eqref{eqn:functionalgamma} pour la fonction gamma, on trouve
\begin{equation}\label{eqn:CSd=4}
\omega_1=\int_1^\infty \frac{dx}{\sqrt{4x^3-4x}}=\int_0^1 t^{\sfrac{1}{4}-1}(1-t)^{\sfrac{1}{2}-1}=\frac{\Gamma(\sfrac{1}{4})^2}{2\sqrt{2\pi}},
\end{equation}
puis $\omega_2=i\omega_1$. La relation de Legendre et \eqref{eqn:masser} montrent alors que le corps engendré par les périodes est égal à $\QQ(2\pi i, \Gamma(\sfrac{1}{4}))$. Par le théorème de Chudnovsky, les nombres~$2\pi i$ et $\Gamma(\sfrac{1}{4})$ sont donc algébriquement indépendants \cite{walds, walds-ell}, comme le prédit la conjecture de Lang\nobreakdash-Rohrlich (exemple \ref{exmp:valeursgamma}). De même, en considérant la courbe elliptique $y^2=4x^3-1$, à multiplication complexe par le corps $\QQ(\rho)$, on en déduit l'indépendance algébrique des nombres $2\pi i$ et $\Gamma(\sfrac{1}{3})$.

La formule \eqref{eqn:CSd=4} est un cas particulier d'une \emph{formule de Lerch-Chowla-Selberg} \cite{gross} exprimant les périodes des courbes elliptiques à multiplication complexe $E$ en termes de valeurs gamma. Pour l'énoncer, on aura besoin de quelques rappels sur l'arithmétique des corps quadratiques imaginaires. Soit $K$ un corps quadratique imaginaire de discriminant $d$ et soit $\mathcal{O}_K$ l'anneau des entiers de $K$. Un tel corps est de la forme $K=\QQ(\sqrt{-D})$ pour un unique entier $D \geq 1$ sans facteur carré, et son discriminant vaut~$d=D$ si $D$ est congru à~3 modulo 4 et $d=4D$ autrement. L'anneau des entiers est égal~à
\[
\mathcal{O}_K=\begin{cases} \ZZ[\sqrt{-D}] & \text{si }D\equiv 1, 2 \text{ mod }4,\\
\ZZ\bigl[\frac{1+\sqrt{-D}}{2}\bigl] & \text{si } D\equiv 3\text{ mod }4.\end{cases}
\]
On note~$w$ le cardinal du groupe des unités~$\mathcal{O}_K^\times$, qui vaut $w=2$ sauf dans les cas exceptionnels $d=3$ (entiers de Gauss) et $d=4$ (entiers d'Eisenstein), pour lesquels on a $w=6$ et $w=4$ respectivement. Le corps $K$ se plonge dans le corps cyclotomique $\QQ(e^{\sfrac{2\pi i}{d}})$, d'où un homomorphisme de restriction
\[
\varepsilon\colon \mathrm{Gal}(\QQ(e^{\sfrac{2\pi i}{d}})/\QQ) \longrightarrow \mathrm{Gal}(K/\QQ).
\]
La source et la cible s'identifiant canoniquement à $(\ZZ/d\ZZ)^\times$ et $\{\pm 1\}$ respectivement, on en déduit un \emph{caractère de Dirichlet}
\[
\varepsilon\colon (\ZZ/d\ZZ)^\times \longrightarrow \{\pm 1\}.
\]
Il peut être calculé en termes des symboles de Legendre \cite[p.\,20]{saitokato}; par exemple, si $d$ est impair et si $d=p_1\cdots p_r$ est sa décomposition en facteurs premiers, $\varepsilon$ est donné par $(\frac{\cdot}{p_1})\cdots (\frac{\cdot}{p_r})$. On notera encore $\varepsilon \colon \ZZ \to \{\pm 1\}$ l'application qui à un entier $a$ associe $\varepsilon(a\,\mathrm{ mod }\,d)$ si $a$ et $d$ sont premiers entre eux et zéro sinon.

Par ailleurs, un \emph{idéal fractionnaire} de $K$ est un sous-ensemble~\hbox{$\mathfrak{a}\!\subset\! K$} de la forme~$x^{-1}I$, où~$x$ est un élément non nul de $K$ et $I$ est un idéal de~$\mathcal{O}_K$. Deux idéaux fractionnaires non nuls $\mathfrak{a}$ et $\mathfrak{b}$ sont dits équivalents s'il existe des éléments non nuls $a, b \in \mathcal{O}_K$ vérifiant~\hbox{$a\mathfrak{a}=b\mathfrak{b}$}. Le produit d'idéaux munit l'ensemble $\mathrm{Cl}(K)$ des classes d'équivalence pour cette relation d'une structure de groupe. D'après un théorème fondamental en théorie algébrique des nombres, ce groupe est fini; son cardinal $h=|\mathrm{Cl}(K)|$ s'appelle le \emph{nombre de classes} de $K$ et mesure combien la propriété de factorisation unique est mise en défaut dans l'anneau $\mathcal{O}_K$. Par exemple, $h$ vaut $1$ pour les corps $\QQ(i)$ et $\QQ(\rho)$, mais $2$ pour le corps $\QQ(\sqrt{-5})$, car l'élément $6$ y admet les factorisations non équivalentes $2\cdot 3$ et $(1+\sqrt{-5})\cdot(1-\sqrt{-5})$.

En regardant un idéal fractionnaire $\mathfrak{a}$ comme un réseau dans le plan complexe, on peut lui associer l'invariant
\[
\Delta(\mathfrak{a})=g_2(\mathfrak{a})^3-27g_3(\mathfrak{a})^2,
\]
où $g_2(\mathfrak{a})$ et $g_3(\mathfrak{a})$ sont les invariants définis dans \eqref{eqn:modularg2g3}. Le nombre complexe $\Delta(\mathfrak{a})\Delta(\mathfrak{a}^{-1})$ ne dépend que de la classe de $\mathfrak{a}$ dans $\mathrm{Cl}(K)$, car si l'on change de représentant $\Delta(\mathfrak{a})$ est multiplié par $a^{-12}$ pour un élément non nul $a \in\mathcal{O}_K$, et $\Delta(\mathfrak{a}^{-1})$ par $a^{12}$.

\begin{thm}[Lerch-Chowla-Selberg] Pour tout corps quadratique imaginaire $K$ de discriminant $d$, l'égalité suivante est vraie:
\begin{displaymath}
\prod_{[\mathfrak{a}] \in \mathrm{Cl}(K)} \Delta(\mathfrak{a})\Delta(\mathfrak{a}^{-1})=\biggl(\frac{2\pi i}{d}\biggr)^{12h}
\prod_{a=1}^{d-1} \Gamma(\sfrac{a}{d})^{6w \varepsilon(a)}.
\end{displaymath}
\end{thm}

Examinons maintenant les conséquences de cette formule pour les périodes des courbes elliptiques. D'après~\eqref{eqn:reseauCM}, le réseau des périodes d'une courbe elliptique $E$ à multiplication complexe par $K$ est de la forme $\ZZ\omega_1 \oplus \ZZ\omega_2=\omega_1\mathfrak{a}$ pour un idéal fractionnaire $\mathfrak{a}$, et le fait que ces courbes soient définies sur $\overline\QQ$ entraîne l'algébricité de $\Delta(\omega_1\mathfrak{a})$. Les nombres $\omega_1^{12}$ et $\Delta(\mathfrak{a})$ diffèrent donc par un multiple algébrique non nul et on déduit de la formule de Lerch-Chowla-Selberg la relation
\begin{displaymath}
\int_{\sigma} \frac{dx}{y} \sim_{\overline \QQ^\times} \sqrt{\pi} \prod_{a=1}^d \Gamma(\sfrac{a}{d})^{w \varepsilon(a)/4h}
\end{displaymath}
pour toute classe non nulle d'homologie de Betti $[\sigma] \in \rH_1^\Betti(E)$, où le symbole $\sim_{\overline \QQ^\times}$ indique que le côté droit et le côté gauche diffèrent par un multiple algébrique non nul.

Lorsque $E$ est sans multiplication complexe, il n'y a pas de source évidente de relations algébriques entre les périodes de~$E$.

\begin{conjecture}\label{conj:periodsnonCM} Les quatre périodes d'une courbe elliptique sur $\overline \QQ$ sans multiplication complexe sont algébriquement indépendantes.
\end{conjecture}

Le seul résultat connu dans cette direction est un théorème de Masser \cite{masser}, qui dit que le $\overline\QQ$-espace vectoriel engendré par
\[
1, \omega_1, \omega_2, \eta_1, \eta_2, 2\pi i
\]
est de dimension $4$ ou $6$ selon que la courbe elliptique est à multiplication complexe ou pas; il s'agit d'un précurseur du théorème \ref{thm:HW} caractérisant toutes les relations linéaires entre périodes d'une variable. Le théorème \ref{thm:chudnovsky} et la conjecture \ref{conj:periodsnonCM} sont bien des instances de la conjecture des périodes de Grothendieck puisque le groupe de Galois motivique de $\rH^1(E)$ est de dimension $2$ si~$E$ est à multiplication complexe (il s'agit du groupe de matrices de la forme $\left(\begin{smallmatrix} a & -Db \\ b & a \end{smallmatrix} \right)$ si $K=\QQ(\sqrt{-D}$)) et qu'il est égal à $\GL_2$ sinon.

\subsection{L'anneau des périodes formelles}\label{sec:periodesformelles}

Il est conjecturé, on l'a dit, que le groupe de Galois motivique agit sur les périodes en préservant toutes les relations algébriques parmi ces nombres. Dans leur expression $\int_\sigma \omega$, seul l'accouplement d'intégration~$\int$ n'est pas de nature algébrique; en l'oubliant, on arrive à la notion de \emph{période formelle}. Il s'agit d'une variante de la construction des coefficients matriciels pour le groupe $\GL_V$. Si, à la place d'un seul espace vectoriel $V$, on s'en donne deux $V$ et $W$, on peut considérer la variété affine $\mathrm{Isom}_{W, V}$ dont les points à valeurs dans une $\QQ$-algèbre~$R$ sont les isomorphismes $R$-linéaires
\[
W \otimes R \longrightarrow V \otimes R.
\]
Cette variété est munie d'une action librement transitive
\[
\mathrm{Isom}_{W, V} \times \GL_V \longrightarrow \mathrm{Isom}_{W, V}
\]
donnée par composition sur les $R$-points; on dit que $\mathrm{Isom}_{W, V}$ est un \emph{torseur} sous $\GL_V$. L'analogue des coefficients matriciels sont les symboles $[W, V, \omega, \varphi]$, avec $w \in W$ et $\varphi \in V^\ast$, vus comme des fonctions
\begin{align*}
\mathrm{Isom}_{W, V}(R) &\longrightarrow R \\ f &\longmapsto \varphi(f(w)).
\end{align*}

L'idée est d'appliquer cette construction au cas où $W$ et $V$ sont la réalisation de de~Rham et la réalisation de Betti d'un motif $\rH^p(X, D)$, puis de définir une sous-variété de $\mathrm{Isom}_{W, V}$, le \emph{torseur des périodes}, en imposant aux coefficients matriciels des relations de nature géométrique. La conjecture de Kontsevich-Zagier suggère de considérer le $\QQ$-espace vectoriel $\mathbf{P}$ défini par des générateurs et des relations comme suit. Les générateurs sont les symboles
\begin{displaymath}
[\rH^p(X, D), \omega, \sigma],
\end{displaymath}
consistant en
\begin{itemize}
\item un motif $\rH^p(X, D)$, où $X$ est une variété affine lisse et $D \subset X$ est un diviseur à croisements normaux (qui peut être vide) définis sur~$\QQ$,
\item une classe de cohomologie de de~Rham relative $\omega \in \rH^p_{\dR}(X, D)$,
\item une classe d'homologie singulière relative $\sigma \in \rH_p^\Betti(X, D)$.
\end{itemize}
Quant aux relations, il y en a de trois types:
\begin{enumeratei}
\item\label{enum:i}
\emph{Additivité:} ces symboles sont bilinéaires en $\omega$ et~$\sigma$.

\item\label{enum:ii}
\emph{Changement de variables:} pour chaque morphisme de variétés $f \colon X \to X'$ tel que $f(D)$ soit contenu dans $D'$, on a la relation
\begin{equation}\label{eqn:KZformel2}
[\rH^p(X, D), f^\ast \omega, \sigma]=[\rH^p(X', D'), \omega, f_\ast \sigma],
\end{equation}
où $f^\ast\colon \rH^n_{\dR}(X', D') \to \rH^n_{\dR}(X, D)$ et $f_\ast\colon \rH_n^{\Betti}(X, D) \to \rH_n^{\Betti}(X', D')$ désignent les applications induites par $f$ en cohomologie de de~Rham relative~\eqref{eqn:fonctorDRrel} et en homologie singulière relative \eqref{eqn:fonctorBettirel}.

\item\label{enum:iii}
\emph{Formule de Stokes:} on a la relation
\begin{displaymath}
[\rH^p(X, D), d\omega, \sigma]=[\rH^{p-1}(D^{(1)}, D^{(2)}), \omega, \partial\sigma],
\end{displaymath}
où $D^{(1)}$ désigne la réunion disjointe des composantes irréductibles $D_i$ de~$D$ et $D^{(2)}\subset D^{(1)}$ la réunion disjointe, indexée par $i$, des réunions des intersections $D_i \cap D_j$ pour tout $j \neq i$.
\end{enumeratei}

\begin{exemple}[retour sur les logarithmes] Soient $a, b \geq 1$ des nombres rationnels et $X=\mathbb{G}_m$. Compte tenu de l'exemple~\ref{exmp:log-cohomologique}, le logarithme de $ab$ correspond à la période formelle
\[
\log^{\mathfrak{m}}(ab)=[\rH^1(X, \{1, ab\}), \sfrac{dx}{x}, [1, ab]].
\]
Voici ce que deviennent les relations de l'exemple \ref{exmp:KZ-log} en ces termes. Comme $\{1, ab\}$ est contenu dans $\{1, a, ab\}$, en appliquant \eqref{eqn:KZformel2} au morphisme identité $f \colon X \to X$, une représentation équivalente est 
\[
\log^{\mathfrak{m}}(ab)=[\rH^1(X, \{1, a, ab\}), \sfrac{dx}{x}, [1, ab]],
\]
l'avantage étant que l'on peut alors écrire la relation
\[
[1, ab]=[1, a]+[a, ab]
\]
dans l'homologie de Betti $\rH_1^\Betti(X, \{1, a, ab\})$. Par bilinéarité, on trouve
\begin{align*}
\log^{\mathfrak{m}}(ab)&=[\rH^1(X, \{1, a, ab\}), \sfrac{dx}{x}, [1, a]]\\
&\hspace{2cm}+[\rH^1(X, \{1, a, ab\}), \sfrac{dx}{x}, [a, ab]] \\
&=[\rH^1(X, \{1, a\}), \sfrac{dx}{x}, [1, a]]\\
&\hspace{2cm}+[\rH^1(X, \{ a, ab\}), \sfrac{dx}{x}, [a, ab]].
\end{align*}
Il ne reste qu'à transformer le dernier terme en $\log^{\mathfrak{m}}(b)$. Pour ce faire, on applique \eqref{eqn:KZformel2} au morphisme $f\colon X \to X$ donné par la multiplication par $a$: comme $f$ envoie $\{1, b\}$ sur $\{a, ab\}$ et qu'il satisfait à $f^\ast(\sfrac{dx}{x})=\sfrac{dx}{x}$ et à~$f_\ast[1, b]=[a, ab]$, on trouve
\[
[\rH^1(X, \{ a, ab\}), \sfrac{dx}{x}, [a, ab]]=[\rH^1(X, \{1, b\}), \sfrac{dx}{x}, [1, b]]
\]
puis, tout compte fait, $\log^{\mathfrak{m}}(ab)=\log^{\mathfrak{m}}(a)+\log^{\mathfrak{m}}(b)$.
\end{exemple}

Un point central dans l'approche motivique des périodes est que l'espace vectoriel~$\mathbf{P}$ est muni d'une structure d'anneau, reflétant le fait que l'ensemble des \emph{vraies} périodes est une $\QQ$-algèbre (lemme~\ref{lem:periodalgebra}). La définition du produit repose sur la formule de Künneth \eqref{eqn:Kunnethrelative} et sa variante en cohomologie de de~Rham relative, qui permettent d'interpréter géométriquement des classes de cohomologie ou d'homologie:
\begin{align*}
[\rH^p(X, D),\,&\omega, \sigma]\cdot [\rH^{p'}(X', D'),\,\omega', \sigma'] \\ &=[\rH^{p+p}(X \times X', D \times X' \cup X \times D'), \omega \otimes \omega', \sigma \otimes \sigma'].
\end{align*}
En inversant la version formelle de $2\pi i$, comme on l'a fait pour les périodes non effectives (remarque \ref{noneffectif}), on arrive à l'algèbre $\widehat{\mathbf{P}}$. Comme les intégrales satisfont aux relations \eqref{enum:i}, \eqref{enum:ii}, \eqref{enum:iii} ci-dessus, l'application qui à une période formelle $[\rH^p(X, D),\,\omega, \sigma]$ associe le nombre complexe $\int_\sigma \omega$ induit une application $\QQ$-linéaire
\begin{displaymath}
\mathrm{per} \colon \widehat{\mathbf{P}} \longrightarrow \CC.
\end{displaymath}
La conjecture de Kontsevich-Zagier se traduit alors en l'énoncé:

\begin{conjecture}[Kontsevich-Zagier]\label{conj:KZbis} L'application $\mathrm{per}$ est injective.
\end{conjecture}

Le groupe de Galois motivique de $\rH^p(X, D)$ agit sur les périodes \emph{formelles} modelées sur ce motif par\footnote{L'action est définie comme $\transp{g^{-1}}(\sigma)$, plutôt que $g(\sigma)$ comme dans~\eqref{eqn:actionGalois}, parce que l'on travaille maintenant avec des automorphismes de la cohomologie de Betti, par opposition à l'homologie. La transposée permet d'agir sur $\sigma$, qui est un élément de l'espace vectoriel dual, et l'inverse d'avoir l'égalité $\transp{h^{-1}}\circ \transp{g^{-1}}=\transp{(hg)^{-1}}$.}
\[
g \cdot [\rH^p(X, D), \omega, \sigma]=[\rH^p(X, D), \omega, \transp{g^{-1}}(\sigma)].
\]
Si la conjecture de Kontsevich-Zagier était vraie, on obtiendrait une action bien définie du groupe de Galois motivique sur les vraies périodes, c'est-à-dire sur les nombres complexes obtenus à partir du symbole $[\rH^p(X, D), \omega, \sigma]$ en calculant l'intégrale $\int_\sigma \omega$. Sans supposer la conjecture connue, on peut dans certaines situations exploiter des calculs explicites de cette action pour déduire des résultats sur les périodes; c'est par exemple la démarche de Brown pour démontrer son théorème sur les valeurs zêta multiples (exemple \ref{exmp:zeta(2)}).

Expliquons enfin le lien entre la conjecture de Kontsevich-Zagier et celle de Grothendieck. Pour chaque motif $\rH^p(X, D)$, il y a un point complexe distingué dans le torseur des périodes, à savoir l'isomorphisme de comparaison \eqref{eqn:isomcomprel} de Grothendieck 
\begin{equation}\label{eqn:comp-pointcomplexe-torseur}
\mathrm{comp} \in \mathrm{Isom}_{\rH^p_{\dR}(X, D), \rH^p_\Betti(X, D)}(\CC), 
\end{equation} en lequel la fonction $[\rH^p(X, D), \omega, \sigma]$ prend la valeur $\int_\sigma \omega$. Dire que l'application $\mathrm{per}$ est injective (conjecture \ref{conj:KZbis}) équivaut donc à dire que l'évaluation des fonctions sur le torseur des périodes en le point $\mathrm{comp}$ est injective. Sur une variété connexe $Y$, les points complexes ayant cette propriété sont ceux qui ne sont pas contenus dans aucune sous-variété fermée stricte $Z \subsetneq Y$ définie par des polynômes à coefficients rationnels; on les appelle \emph{génériques}. Par exemple, les fonctions sur la droite affine $Y=\mathbb{A}^1$ sont les polynômes $\QQ[x]$ et les points $z \in Y(\CC)=\CC$ pour lesquels l'application $P \mapsto P(z)$ de~$\QQ[x]$ dans~$\CC$ est injective sont les nombres transcendants, ceux qui ne sont pas contenus dans des sous-variétés $\{f=0\}$ avec $f \in \QQ[x]$. Encore une caractérisation équivalente des points génériques consiste à dire que le degré de transcendance de $\QQ(z)$ est égal à $1$, la dimension de~$Y$. Dans le cas du point complexe \eqref{eqn:comp-pointcomplexe-torseur}, le corps engendré par les coordonnées de $\mathrm{comp}$ est précisément le corps engendré par les coefficients de l'accouplement de périodes entre la cohomologie de de~Rham $\rH_\dR^p(X, D)$ et l'homologie de Betti $\rH^\Betti_p(X, D)$. Dire que c'est un point générique équivaut donc à dire que le torseur des périodes est connexe et que sa dimension est égale au degré de transcendance de ce corps. Or, comme le groupe de Galois motivique agit librement transitivement sur le torseur des périodes, les dimensions de ces deux variétés sont les mêmes et on retrouve la conjecture \ref{conj:perGroth}.

\subsection{Transcendance fonctionnelle}\label{sec:functionaltrans}

Par comparaison avec le théorème d'Hermite-Lindemann sur la transcendance des \emph{valeurs spéciales} de l'exponentielle $e^\alpha$, pour un nombre algébrique non nul $\alpha$, il est très simple de voir que la \emph{fonction} exponentielle est transcendante. Supposons \emph{a contrario} qu'il existe un polynôme non nul $P \in \CC[x, y]$ satisfaisant à $P(z, e^z)=0$ pour tout~$z \in \CC$ et choisissons-le de degré minimal en~$y$. Pour aboutir à une contradiction, une possibilité consiste à dériver par rapport à~$z$ à de multiples reprises, jusqu'à arriver à une relation de plus petit degré; il s'agit là d'une preuve formelle, en ce sens qu'elle n'utilise que la définition de~$e^z$ comme série entière et les propriétés de la dérivation. Un autre argument utilise le fait que l'exponentielle est périodique de période $2\pi i$, d'où en particulier l'égalité $e^{2\pi i n}=1$ pour tout entier~$n$. Le polynôme en une variable $P(x, 1) \in \CC[x]$ ayant une infinité de racines, il est forcément nul, et $P$ peut donc s'écrire sous la forme $P=(y-1)Q$ avec $Q \in \CC[x, y]$ de degré strictement plus petit en~$y$. La fonction entière $Q(z, e^z)$ s'annule alors sur l'ouvert du plan complexe où $e^z$ ne prend pas la valeur $1$, d'où $Q(z, e^z)=0$ pour tout $z \in \CC$, contredisant la minimalité de $P$. Le lecteur averti n'aura pas manqué de remarquer que cette seconde preuve repose sur le même principe que celle du lemme XXVIII: si le graphe de la fonction exponentielle était une courbe algébrique dans $\CC^2$, elle ne pourrait couper la droite $y=1$ qu'en un nombre fini de points!

Encore plus frappant est le cas des périodes d'une courbe elliptique: la conjecture \ref{conj:periodsnonCM} est complètement ouverte à l'heure actuelle, mais on peut démontrer de manière élémentaire son analogue pour les séries d'Eisenstein introduites dans~\eqref{eqn:Eisenstein}.

\begin{thm}\label{thm:transfonct} Les fonctions $G_2(\tau)$, $G_4(\tau)$ et $G_6(\tau)$ sont algébriquement indépendantes sur $\CC(\tau)$.
\end{thm}

\begin{proof} Démontrons dans un premier temps que~$G_4$ et~$G_6$ sont algébriquement indépendantes, c'est-à-dire qu'il n'existe pas de polynôme non nul $P \in \CC[x_0, x_1, x_2]$ satisfaisant à
\begin{equation}\label{eqn:annul}
P(\tau, G_4(\tau), G_6(\tau))=0.
\end{equation}
On utilisera la propriété de modularité \eqref{eqn:modulariteG}. Par l'absurde, supposons qu'il existe un polynôme non nul satisfaisant à \eqref{eqn:annul} et choisissons\nobreakdash-en un de degré minimal en $x_0$. Comme les fonctions $G_4$ et~$G_6$ sont invariantes par la translation $\tau \mapsto \tau+1$, le polynôme
\begin{displaymath}
P(x_0+1, x_1, x_2)-P(x_0, x_1, x_2)
\end{displaymath}
satisfait à \eqref{eqn:annul} également. Le degré en $x_0$ de ce dernier étant strictement plus petit, on en déduit qu'il est nul.  Pour $x_1$ et $x_2$ fixés, le polynôme en une variable $P(x_0, x_1, x_2)$ prend donc les mêmes valeurs en $x_0$ et $x_0+1$, ce qui implique qu'il est constant; il s'ensuit que $P$ ne dépend pas de $x_0$. Par ailleurs, au vu des identités
\[
G_4(\sfrac{-1}{\tau})=\tau^4 G_4(\tau) \quad \text{et} \quad G_6(\sfrac{-1}{\tau})=\tau^6 G_6(\tau),
\]
on peut supposer que tous les monômes $c x_1^n x_2^m$ dans $P$ ont le même \og degré \fg $3n+2m$ (c'est-à-dire que $P$ est homogène si l'on donne poids $3$ à la variable $x_1$ et poids $2$ à la variable $x_2$) et le choisir de \og degré\fg  minimal. Parmi ces monômes, il ne peut pas y avoir de multiples de~$x_1^n$: si tel était le cas, on aurait une relation du type
\begin{displaymath}
G_4(\tau)^n+G_6(\tau)Q(G_4(\tau), G_6(\tau))=0,
\end{displaymath}
ce qui n'est pas possible car $G_6(i)=0$ mais $G_4(i) \neq 0$. De même, en évaluant en l'argument $\rho=\exp(2\pi i/3)$, pour lequel $G_4(\rho)=0$ mais $G_6(\rho)\neq 0$, on voit que $P$ ne contient pas de multiples de~$x_2^n$. Le polynôme $P$ est ainsi de la forme $x_1x_2 R(x_1, x_2)$ et on obtient une relation $R(G_4(\tau), G_6(\tau))=0$ avec $R$ homogène de degré strictement plus petit que celui de $P$. Cette contradiction achève la preuve de l'indépendance algébrique de $G_4$ et $G_6$.

Pour démontrer que $G_2$, $G_4$ et $G_6$ sont algébriquement indépendants, supposons qu'il existe $P \in \CC[x_0, x_1, x_2, x_3]$ non nul tel~que
\begin{equation}\label{eqn:annul2}
P(\tau, G_2(\tau), G_4(\tau), G_6(\tau))=0
\end{equation}
et choisissons-le de degré minimal en $x_0$ parmi ceux de degré minimal en $x_1$. On fera usage de la relation de quasi-modularité \eqref{eqn:modulariteG2} satisfaite par $G_2$. Comme $G_2(\tau+1)=G_2(\tau)$, le même raisonnement qu'auparavant montre que~$P$ ne dépend pas de~$x_0$. Il suffit alors de voir que $P$ ne dépend pas de~$x_1$ non plus, car un tel $P$ donnerait lieu à une relation algébrique non triviale entre $G_4$ et $G_6$ et il n'y en a pas d'après la première partie de la preuve. Écrivons
\begin{displaymath}
P=P_n x_1^n+P_{n-1} x_1^{n-1}+\cdots+P_0,
\end{displaymath}
avec des polynômes $P_i \in \CC[x_2, x_3]$ et $P_n$ non nul, et posons
\begin{multline*}
Q=P_n(x_2, x_3)P(\sfrac{-1}{x_0}, x_0^2x_1-2\pi i x_0, x_0^4 x_2, x_0^6 x_3) \\
-x_0^{2n}P_n(x_0^4 x_2, x_0^6 x_3)P(x_0, x_1, x_2, x_3).
\end{multline*}
La relation $G_2(-1/\tau)=G_2(\tau)-2\pi i \tau$ implique que le polynôme $Q$ satisfait encore à \eqref{eqn:annul2}, puisque chacun des deux termes s'annule. Comme $\deg_{x_1}(Q)<\deg_{x_1}(P)$, il vient \hbox{$Q=0$}. Si $n \geq 1$, le calcul des puissances impaires de $x_0$ dans le coefficient de $x_1^{n-1}$ dans $Q$ donne
\begin{displaymath}
2\pi i n P_n(x_2, x_3) P_n(x_0^4 x_2, x_0^6 x_3)x_0^{2n-1}=0,
\end{displaymath}
en contradiction avec la non-nullité de $P_n$. On a ainsi $n=0$, ce qui revient à dire que $P$ ne dépend que de $x_2$ et $x_3$.
\end{proof}

Ce théorème peut aussi être interprété comme une égalité entre le degré de transcendance du corps engendré par des \emph{fonctions de périodes} et la dimension d'un groupe algébrique, en l'occurrence le groupe de Galois différentiel \cite[Ch.\,6]{chambi} de l'équation différentielle dont elles sont solution. La raison pour laquelle cette dimension est~$3$ et pas~$4$, comme il est attendu pour les nombres, est la relation de Legendre: du point de vue fonctionnel, la valeur $2\pi i$ du déterminant de la matrice des périodes est tout autant une constante que n'importe quel autre nombre complexe et force le groupe de Galois différentiel à être égal à $\mathrm{SL}_2$. On peut en fait toujours relier le degré de transcendance du corps engendré par des fonctions de périodes d'une famille de motifs à un paramètre à la dimension du groupe de Galois de l'équation différentielle à laquelle elles satisfont \cite{Yves1, Bourbaki}; cette idée joue un rôle important dans la preuve par Ayoub de la variante géométrique de la conjecture de Kontsevich-Zagier (théorème \ref{thm:Ayoub}).\enlargethispage{-2\baselineskip}

\backmatter
\bibliographystyle{smfplain}
\bibliography{xups19-01}

\providecommand{\eprint}[1]{\href{http://arxiv.org/abs/#1}{\texttt{arXiv\string:\allowbreak#1}}}\providecommand{\doi}[1]{\href{http://dx.doi.org/#1}{\texttt{doi\string:\allowbreak#1}}}
\providecommand{\bysame}{\leavevmode ---\ }
\providecommand{\og}{``}
\providecommand{\fg}{''}
\providecommand{\smfandname}{\&}
\providecommand{\smfedsname}{\'eds.}
\providecommand{\smfedname}{\'ed.}
\providecommand{\smfmastersthesisname}{M\'emoire}
\providecommand{\smfphdthesisname}{Th\`ese}
\begin{thebibliography}{10}

\bibitem{motifs}
{\scshape Y.~Andr\'{e}} -- \emph{Une introduction aux motifs (motifs purs,
  motifs mixtes, p\'eriodes)}, Panoramas et Synth\`eses, vol.~17, Soci\'et\'e
  Math\'ematique de France, Paris, 2004.

\bibitem{Yves2}
\bysame , {\og Galois theory, motives and transcendental numbers\fg}, in
  \emph{Renormalization and {G}alois theories}, IRMA Lect. Math. Theor. Phys,
  vol.~15, Eur. Math. Soc., 2009, p.~165--177.

\bibitem{Yves4}
\bysame , {\og Id\'{e}es galoisiennes\fg}, in \emph{Histoire de
  math\'{e}matiques}, Ed. \'{E}c. Polytech., Palaiseau, 2012, p.~1--16.

\bibitem{Yves1}
\bysame , {\og Galois theory beyond algebraic numbers and algebraic
  functions\fg}, in \emph{Colloquium {D}e {G}iorgi 2010{\nobreakdash-}2012},
  Colloquia, vol.~4, Ed. Norm., Pisa, 2013, p.~1--7.

\bibitem{Bourbaki}
\bysame , {\og Groupes de {G}alois motiviques et p\'{e}riodes\fg},
  \emph{Ast\'{e}risque} (2017), no.~390, p.~Exp. No. 1104, 1--26, S\'{e}minaire
  Bourbaki. Vol. 2015/2016. Expos\'{e}s 1104--1119.

\bibitem{apery}
{\scshape R.~Ap\'{e}ry} -- {\og Irrationalit\'{e} de {$\zeta(2)$} et
  {$\zeta(3)$}\fg}, in \emph{Luminy Conference on Arithmetic}, Ast\'{e}risque,
  vol.~61, Soci{\'e}t{\'e} Math{\'e}matique de France, 1979, p.~11--13.

\bibitem{arapura}
{\scshape D.~Arapura {\normalfont \smfandname} S.-J. Kang} -- {\og
  K\"{a}hler-de {R}ham cohomology and {C}hern classes\fg}, \emph{Comm. Algebra}
  \textbf{39} (2011), no.~4, p.~1153--1167.

\bibitem{Arnold}
{\scshape V.~I. Arnold} -- \emph{Huygens and {B}arrow, {N}ewton and {H}ooke},
  Birkh\"{a}user Verlag, Basel, 1990, Pioneers in mathematical analysis and
  catastrophe theory from evolvents to quasicrystals, Translated from the
  Russian by Eric J. F. Primrose.

\bibitem{AyoubEMS}
{\scshape J.~Ayoub} -- {\og Periods and the conjectures of {G}rothendieck and
  {K}ontsevich--{Z}agier\fg}, \emph{Eur. Math. Soc. Newsl.} \textbf{91} (2014),
  p.~12--18.

\bibitem{Ayoub}
\bysame , {\og Une version relative de la conjecture des p{\'e}riodes de
  {K}ontsevich--{Z}agier\fg}, \emph{Ann. of Math.} \textbf{181} (2015), no.~3,
  p.~905--922.

\bibitem{AyoubTohoku}
\bysame , {\og La version relative de la conjecture des p\'{e}riodes de
  {K}ontsevich--{Z}agier revisit\'{e}e\fg}, \emph{Tohoku Math. J.} \textbf{71}
  (2019), no.~3, p.~465--485.

\bibitem{baker}
{\scshape A.~Baker} -- \emph{Transcendental number theory}, Cambridge
  University Press, London-New York, 1975.

\bibitem{baker-wustholz}
{\scshape A.~Baker {\normalfont \smfandname} G.~W\"{u}stholz} --
  \emph{Logarithmic forms and {D}iophantine geometry}, New Mathematical
  Monographs, vol.~9, Cambridge University Press, Cambridge, 2007.

\bibitem{ball-rivoal}
{\scshape K.~Ball {\normalfont \smfandname} T.~Rivoal} -- {\og
  Irrationalit\'{e} d'une infinit\'{e} de valeurs de la fonction z\^{e}ta aux
  entiers impairs\fg}, \emph{Invent. math.} \textbf{146} (2001), no.~1,
  p.~193--207.

\bibitem{belkale}
{\scshape P.~Belkale {\normalfont \smfandname} P.~Brosnan} -- {\og Periods and
  {I}gusa local zeta functions\fg}, \emph{Int. Math. Res. Not.} (2003), no.~49,
  p.~2655--2670.

\bibitem{bertolin}
{\scshape C.~Bertolin} -- {\og Third kind elliptic integrals and 1-motives\fg},
  \emph{J. Pure Appl. Algebra} \textbf{224} (2020), no.~10, p.~106396, 28, avec
  une lettre de Y. Andr\'{e} et un appendice de M. Waldschmidt.

\bibitem{bertrand}
{\scshape D.~Bertrand} -- {\og Multiplicity and vanishing lemmas for
  differential and {$q$}-difference equations in the
  {S}iegel-{S}hidlovski\u{\i} theory\fg}, \emph{Fundam. Prikl. Mat.}
  \textbf{16} (2010), no.~5, p.~19--30.

\bibitem{BW}
{\scshape F.~Beukers {\normalfont \smfandname} J.~Wolfart} -- {\og Algebraic
  values of hypergeometric functions\fg}, in \emph{New advances in
  transcendence theory ({D}urham, 1986)}, Cambridge Univ. Press, Cambridge,
  1988, p.~68--81.

\bibitem{bloch-esnault}
{\scshape S.~Bloch {\normalfont \smfandname} H.~Esnault} -- {\og Homology for
  irregular connections\fg}, \emph{J. Th\'{e}or. Nombres Bordeaux} \textbf{16}
  (2004), no.~2, p.~357--371.

\bibitem{real-algebraic-geometry}
{\scshape J.~Bochnak, M.~Coste {\normalfont \smfandname} M.-F. Roy} --
  \emph{Real algebraic geometry}, Ergebnisse der Mathematik und ihrer
  Grenzgebiete, vol.~36, Springer-Verlag, Berlin, 1998.

\bibitem{Bost}
{\scshape J.-B. Bost} -- {\og {Introduction to compact Riemann surfaces,
  Jacobians, and abelian varieties.}\fg}, in \emph{{From number theory to
  physics (Les Houches, 1989)}}, Springer, Berlin, 1992, p.~64--211.

\bibitem{brown}
{\scshape F.~Brown} -- {\og Mixed {T}ate motives over {$\Bbb Z$}\fg},
  \emph{Ann. of Math.} \textbf{175} (2012), no.~2, p.~949--976.

\bibitem{mzvbook}
{\scshape J.~I. Burgos~Gil {\normalfont \smfandname} J.~Fres{\'a}n} --
  \emph{Multiple zeta values: from numbers to motives}, Clay Mathematics
  Proceedings, {\`a} para{\^\i}tre.

\bibitem{cantor}
{\scshape G.~Cantor} -- {\og Ueber eine {E}igenschaft des {I}nbegriffs aller
  reellen algebraischen {Z}ahlen\fg}, \emph{J. reine angew. {M}ath.}
  \textbf{77} (1874), p.~258--262, Traduction fran{{\c c}}aise: G.~Cantor, Sur
  une propri{\'e}t{\'e} du syst{\`e}me de tous les nombres alg{\'e}briques
  r{\'e}els, Acta Math. 2 (1883), 305--310.

\bibitem{cantor2}
\bysame , {\og Ueber eine elementare {F}rage der {M}annigfaltigkeitslehre\fg},
  \emph{Jahresber. Dtsch. Math.-Ver} \textbf{1} (1891), p.~75--78.

\bibitem{chambi}
{\scshape A.~Chambert-Loir} -- \emph{Alg{\`e}bre corporelle}, Ed. \'{E}c.
  Polytech., Palaiseau, 2005.

\bibitem{colmez}
{\scshape P.~Colmez} -- {\og Arithm\'{e}tique de la fonction z\^{e}ta\fg}, in
  \emph{La fonction z\^{e}ta}, Ed. \'{E}c. Polytech., Palaiseau, 2003,
  p.~37--164.

\bibitem{ominimal}
{\scshape J.~Commelin, P.~Habegger {\normalfont \smfandname} A.~Huber} -- {\og
  Exponential periods and o-minima\-lity\fg}, \eprint{2007.08280}.

\bibitem{deligne-goncharov}
{\scshape P.~Deligne {\normalfont \smfandname} A.~B. Goncharov} -- {\og Groupes
  fondamentaux motiviques de {T}ate mixte\fg}, \emph{Ann. Sci. \'{E}cole Norm.
  Sup. (4)} \textbf{38} (2005), no.~1, p.~1--56.

\bibitem{deligne-malgrange}
{\scshape P.~Deligne, B.~Malgrange {\normalfont \smfandname} J.-P. Ramis} --
  \emph{Singularit\'{e}s irr\'{e}guli{\`e}res. {C}orrespondance et documents},
  Documents Math\'{e}matiques, vol.~5, Soci\'{e}t\'{e} Math\'{e}matique de
  France, Paris, 2007.

\bibitem{deligne900}
{\scshape P.~Deligne} -- {\og Hodge cycles on abelian varieties\fg}, in
  \emph{Hodge cycles, motives, and {S}himura varieties} (Berlin-New York)
  (P.~Deligne, J.~S. Milne, A.~Ogus {\normalfont \smfandname} K.-y. Shih,
  \smfedsname), Lecture Notes in Mathematics, vol. 900, Springer-Verlag, 1982,
  p.~9--100.

\bibitem{Dolbeault}
{\scshape P.~Dolbeault} -- {\og {Sur la cohomologie des vari\'{e}t\'{e}s
  analytiques complexes}\fg}, \emph{{C. R. Acad. Sci. Paris}} \textbf{236}
  (1953), p.~175--177.

\bibitem{dupont}
{\scshape C.~Dupont} -- {\og Valeurs z{\^e}ta multiples\fg}, in
  \emph{P{\'e}riodes et transcendence}, Ed. \'{E}c. Polytech., Palaiseau, 2022.

\bibitem{jossen}
{\scshape J.~Fres{\'a}n {\normalfont \smfandname} P.~Jossen} -- {\og
  Exponential motives\fg}, en pr{\'e}paration.

\bibitem{Gray}
{\scshape R.~Gray} -- {\og Georg {C}antor and transcendental numbers\fg},
  \emph{Amer. Math. Monthly} \textbf{101} (1994), no.~9, p.~819--832.

\bibitem{gross}
{\scshape B.~H. Gross} -- {\og On an identity of {C}howla and {S}elberg\fg},
  \emph{J. Number Theory} \textbf{11} (1979), no.~3, p.~344--348.

\bibitem{GrodR}
{\scshape A.~Grothendieck} -- {\og On the de {R}ham cohomology of algebraic
  varieties\fg}, \emph{Inst. Hautes \'Etudes Sci. Publ. Math.} \textbf{{}}
  (1966), no.~29, p.~95--103.

\bibitem{hatcher}
{\scshape A.~Hatcher} -- \emph{Algebraic topology}, Cambridge University Press,
  Cambridge, 2002.

\bibitem{hermite}
{\scshape C.~Hermite} -- {\og Sur la fonction exponentielle\fg}, \emph{C. R.
  Math. Acad. Sci. Paris} \textbf{77} (1873), p.~18--24, 74--79, 226--233,
  285--293.

\bibitem{hien-roucairol}
{\scshape M.~Hien {\normalfont \smfandname} C.~Roucairol} -- {\og Integral
  representations for solutions of exponential {G}auss-{M}anin systems\fg},
  \emph{Bull. Soc. Math. France} \textbf{136} (2008), no.~4, p.~505--532.

\bibitem{hironaka}
{\scshape H.~Hironaka} -- {\og Triangularions of algebraic sets\fg}, in
  \emph{{Algebraic Geometry}}, Proceedings of Symposium in Pure Mathematics,
  vol.~29, Amer. Math. Soc., 1975, p.~165--185.

\bibitem{huber-muller}
{\scshape A.~Huber {\normalfont \smfandname} S.~M{\"u}ller-Stach} --
  \emph{Periods and {N}ori motives}, Ergebnisse der Mathematik und ihrer
  Grenzgebiete. 3. Folge., vol.~65, Springer, Cham, 2017, With contributions of
  Benjamin Friedrich and Jonas von Wangenheim.

\bibitem{huber-wustholz}
{\scshape A.~Huber {\normalfont \smfandname} G.~W\"ustholz} --
  \emph{Transcendence and linear relations of $1$-periods}, Cambridge Tracts in
  Mathematics, vol. 227, Cambridge University Press, 2022.

\bibitem{hurwitz}
{\scshape A.~Hurwitz} -- {\og {{\"U}ber best{\"a}ndig convergirende
  Potenzreihen mit rationalen Zahlencoefficienten und vorgeschriebenen
  Nullstellen}\fg}, \emph{Acta Math.} \textbf{14} (1890-1891), p.~211--215.

\bibitem{saitokato}
{\scshape K.~Kato, N.~Kurokawa {\normalfont \smfandname} T.~Saito} --
  \emph{Number theory. 2}, Translations of Mathematical Monographs, vol. 240,
  American Mathematical Society, Providence, RI, 2011, Introduction to class
  field theory, Translated from the 1998 Japanese original by Masato Kuwata and
  Katsumi Nomizu, Iwanami Series in Modern Mathematics.

\bibitem{Kollar}
{\scshape J.~Koll{\'a}r} -- \emph{Lectures on resolution of singularities},
  Annals of Mathematics Studies, vol. 166, Princeton University Press,
  Princeton, NJ, 2007.

\bibitem{KZ}
{\scshape M.~Kontsevich {\normalfont \smfandname} D.~Zagier} -- {\og
  Periods\fg}, in \emph{Mathematics unlimited---2001 and beyond}, Springer,
  Berlin, 2001, p.~771--808.

\bibitem{lairez}
{\scshape P.~Lairez {\normalfont \smfandname} E.~C. Sert{\"o}z} -- {\og
  Separation of periods of quartic surfaces\fg}, \emph{Algebra Number Theory}
  ({\`a} para{\^\i}tre).

\bibitem{Lee}
{\scshape J.~M. Lee} -- \emph{Introduction to smooth manifolds}, 2\ieme
  \smfedname, Graduate Texts in Mathematics, vol. 218, Springer, New York,
  2013.

\bibitem{Lin-pi}
{\scshape F.~Lindemann} -- {\og \"{U}ber die {Z}ahl $\pi$\fg}, \emph{Math.
  Ann.} \textbf{20} (1882), no.~2, p.~213--225.

\bibitem{Loj64}
{\scshape S.~{{\L}}ojasiewicz} -- {\og Triangulation of semi-analytic sets\fg},
  \emph{Ann. Scuola Norm. Sup. Pisa Cl. Sci. (3)} \textbf{18} (1964),
  p.~449--474.

\bibitem{masser}
{\scshape D.~Masser} -- \emph{Elliptic functions and transcendence}, Lecture
  Notes in Mathematics, vol. 437, Springer-Verlag, Berlin-New York, 1975.

\bibitem{chatelet}
{\scshape I.~Newton} -- \emph{Principes math{\'e}matiques de la philosophie
  naturelle}, Centre international d'{\'e}tude du XVIIIe si{\`e}cle, 2015, La
  traduction fran{\c c}aise des Philosophiae naturalis principia mathematica
  par Emilie du Ch{\^a}telet.

\bibitem{perrin}
{\scshape D.~Perrin} -- \emph{{G{\'e}om{\'e}trie alg{\'e}brique. Une
  introduction}}, Savoirs Actuels, InterEditions, Paris; CNRS \'{E}ditions,
  Paris, 1995.

\bibitem{vdP}
{\scshape A.~J. Van~der Poorten} -- {\og On the arithmetic nature of definite
  integrals of rational functions\fg}, \emph{Proc. Amer. Math. Soc.}
  \textbf{29} (1971), p.~451--456.

\bibitem{popescu}
{\scshape P.~Popescu-Pampu} -- {\og Qu'est-ce que le genre?\fg}, in
  \emph{Histoire de math\'{e}matiques}, Ed. \'{E}c. Polytech., Palaiseau, 2012,
  p.~55--195.

\bibitem{pourciau}
{\scshape B.~Pourciau} -- {\og The integrability of ovals: {N}ewton's {L}emma
  28 and its counterexamples\fg}, \emph{Arch. Hist. Exact Sci.} \textbf{55}
  (2001), no.~5, p.~479--499.

\bibitem{deRham}
{\scshape G.~de~Rham} -- {\og {Sur l'analysis situs des vari{\'e}t{\'e}s {\`a}
  $n$ dimensions}\fg}, \smfphdthesisname, 1931.

\bibitem{rivoal}
{\scshape T.~Rivoal} -- {\og Les {$E$}-fonctions et {$G$}-fonctions de
  {S}iegel\fg}, in \emph{P{\'e}riodes et transcendence}, Ed. \'{E}c. Polytech.,
  Palaiseau, 2022.

\bibitem{unif}
{\scshape H.~P. de~Saint-Gervais} -- \emph{Uniformisation des surfaces de
  {R}iemann}, ENS \'{E}ditions, Lyon, 2010, Retour sur un th\'{e}or{\`e}me
  centenaire.

\bibitem{Schneider}
{\scshape T.~Schneider} -- {\og Zur {T}heorie der {A}belschen {F}unktionen und
  {I}ntegrale\fg}, \emph{J. reine angew. Math.} \textbf{183} (1941),
  p.~110--128.

\bibitem{siegel}
{\scshape C.~L. Siegel} -- \emph{Transcendental {N}umbers}, Annals of
  Mathematics Studies, no.~16, Princeton University Press, Princeton, N. J.,
  1949.

\bibitem{TZ}
{\scshape K.~Tent {\normalfont \smfandname} M.~Ziegler} -- {\og Computable
  functions of reals\fg}, \emph{M\"{u}nster J. Math.} \textbf{3} (2010),
  p.~43--65.

\bibitem{terasoma}
{\scshape T.~Terasoma} -- {\og Mixed {T}ate motives and multiple zeta
  values\fg}, \emph{Invent. Math.} \textbf{149} (2002), no.~2, p.~339--369.

\bibitem{viu-sos}
{\scshape J.~Viu-Sos} -- {\og A semi-canonical reduction for periods of
  {K}ontsevich-{Z}agier\fg}, \emph{Int. J. Number Theory} \textbf{17} (2021),
  no.~1, p.~147--174.

\bibitem{walds-hist}
{\scshape M.~Waldschmidt} -- {\og Les d\'{e}buts de la th\'{e}orie des nombres
  transcendants ({\`a} l'occasion du centenaire de la transcendance de {$\pi
  $})\fg}, in \emph{Proceedings of the {S}eminar on the {H}istory of
  {M}athematics, 4}, Inst. Henri Poincar\'{e}, Paris, 1983, p.~93--115.

\bibitem{walds}
\bysame , {\og Transcendence of periods: the state of the art\fg}, \emph{Pure
  Appl. Math. Q.} \textbf{2} (2006), no.~2, p.~435--463.

\bibitem{walds-ell}
\bysame , {\og Elliptic functions and transcendence\fg}, in \emph{Surveys in
  number theory}, Dev. Math., vol.~17, Springer, New York, 2008, p.~143--188.

\bibitem{WusICM}
{\scshape G.~W\"{u}stholz} -- {\og Algebraic groups, {H}odge theory, and
  transcendence\fg}, in \emph{Proceedings of the {I}nternational {C}ongress of
  {M}athematicians, {V}ol. 1, 2 ({B}erkeley, {C}alif., 1986)}, Amer. Math.
  Soc., Providence, RI, 1987, p.~476--483.

\bibitem{wust-leib}
\bysame , {\og Leibniz' conjecture, periods \& motives\fg}, in \emph{Colloquium
  {D}e {G}iorgi 2009}, Colloquia, vol.~3, Ed. Norm., Pisa, 2012, p.~33--42.

\bibitem{yoshinaga}
{\scshape M.~Yoshinaga} -- {\og Periods and elementary real numbers\fg},
  \eprint{0805.0349}, 2008.

\end{thebibliography}
\end{document}